\documentclass{report}

\usepackage[utf8]{inputenc}
\usepackage{amsmath, amsthm, amssymb, amsfonts}
\usepackage[francais]{babel}
\usepackage{xspace}
\usepackage[T1]{fontenc}
\usepackage{xcolor}
\usepackage{mathrsfs}
\usepackage{tensor}
\usepackage{nicefrac}
\usepackage{extarrows}
\usepackage{tikz-cd}
\usepackage{tensor}
\usepackage{enumerate}
\usepackage{bm}
\usepackage{multirow}
\usepackage[hyphens]{url}

\usepackage{pdfpages}
\usepackage[hyperfootnotes=false, colorlinks, linkcolor={black}, citecolor={black}, filecolor={black}, urlcolor={black}]{hyperref}
\hypersetup{breaklinks=true}

\newtheorem{defi}{Définition}[section]
\newtheorem{thm}[defi]{Théorème}
\newtheorem{lemme}[defi]{Lemme}
\newtheorem{prop}[defi]{Proposition}
\newtheorem{ex}[defi]{Exemple}
\newtheorem{cor}[defi]{Corollaire}
\newtheorem{conj}[defi]{Conjecture}
\newtheorem{prop-def}[defi]{Proposition-Définition}

\newtheorem*{conj_nn}{Conjecture}
\newtheorem*{prop_nn}{Proposition}
\newtheorem{thmintro}{Th\'eor\`eme}

\newtheorem{conjintro}[thmintro]{Conjecture}

\newcommand\numberthis{\addtocounter{equation}{1}\tag{\theequation}}

\newcommand{\aq}{\mathbb{A}_\mathbb{Q}}
\newcommand{\af}{\mathbb{A}_f}
\newcommand{\re}{\mathrm{Re} \:}

\newcommand{\qp}{\mathbb{Q}_p}
\newcommand{\zp}{\mathbb{Z}_p}

\newcommand{\sgocs}{sous-groupes compacts ouverts}
\newcommand{\Gal}{\mathrm{Gal}}
\newcommand{\ab}{\mathrm{ab}}
\newcommand{\eps}{\varepsilon}
\newcommand{\C}{\mathbb{C}}
\newcommand{\R}{\mathbb{R}}
\newcommand{\Z}{\mathbb{Z}}
\newcommand{\Q}{\mathbb{Q}}
\newcommand{\F}{\mathrm{F}}
\newcommand{\Ll}{\mathscr{L}}
\newcommand{\Ker}{\mathrm{Ker} \:}
\newcommand{\Imm}{\mathrm{Im} \:}
\newcommand{\Ok}{\mathcal{O}}
\newcommand{\pk}{\mathfrak{p}}
\newcommand{\Okk}{\mathcal{O}_K}
\newcommand{\pkk}{\mathfrak{p}_K}

\newcommand{\1}{\mathbf{1}}
\newcommand{\GL}{\mathrm{GL}}
\newcommand{\SL}{\mathrm{SL}}
\newcommand{\PGL}{\mathrm{PGL}}
\newcommand{\PGSp}{\mathrm{PGSp}}
\newcommand{\aar}{\mathrm{a}}
\newcommand{\SO}{\mathrm{SO}}
\newcommand{\SU}{\mathrm{SU}}

\newcommand{\Sp}{\mathrm{Sp}}
\newcommand{\Fr}{\mathrm{Fr}}
\newcommand{\Ind}{\mathrm{Ind}}
\newcommand{\St}{\mathrm{St}}
\newcommand{\Oo}{\mathrm{O}}
\newcommand{\Vect}{\mathrm{Vect}}
\newcommand{\val}{\mathrm{val}}
\newcommand{\J}{\mathrm{J}}
\newcommand{\K}{\mathrm{K}}
\newcommand{\Jac}{\mathrm{Jac}}
\newcommand{\ii}{\boldsymbol{i}}
\newcommand{\rr}{\boldsymbol{r}}
\newcommand{\Hom}{\mathrm{Hom}}
\newcommand{\dd}{\mathrm{d}}
\newcommand{\Cent}{\mathrm{Cent}}
\newcommand{\Jord}{\mathrm{Jord}}
\newcommand{\WD}{\mathrm{WD}}
\newcommand{\W}{\mathrm{W}}
\newcommand{\B}{\mathrm{B}}
\newcommand{\E}{\mathrm{E}}
\newcommand{\GSp}{\mathrm{GSp}}
\newcommand{\ps}{\par\smallskip}

\newcommand{\longhookrightarrow}{}
\DeclareRobustCommand{\longhookrightarrow}{\lhook\joinrel\relbar\joinrel\rightarrow}

\newcommand{\ie}{\emph{i.e.}\xspace}
\newcommand{\cf}{\emph{cf.}\xspace}

\newcommand{\m}{\mathrm{m}}

\newcommand{\isomo}{\overset{\sim}{\rightarrow}}

\newcommand{\AAA}{\mathbb{A}}
\newtheorem{fact}[defi]{Fait}

\def\got{\mathfrak}

\begin{document}

\begin{titlepage}
	\parindent0pt
	\begin{minipage}{\dimexpr\linewidth\relax}
		\centering
		{\huge Université Paris-Saclay \par}\medskip \Large
		Laboratoire de Mathématiques d'Orsay\par
	\end{minipage}%
	\vskip1cm plus 1fil
	\begin{center}
	\Huge Autour de l'énumération des représentations automorphes cuspidales algébriques de $\GL_n$ \linebreak sur $\Q$ en conducteur $>1$
	\end{center}
	\vskip1.5cm
	\begin{center}
	\Large Guillaume \bsc{Lachaussée}
	\end{center}
	\vskip0pt plus 1fil
	\vskip1cm
	\leavevmode
	\vskip0pt plus 1fil

	{\centering\Large\it
Thèse de doctorat réalisée sous la direction de  \newline
Gaëtan Chenevier
\par
	}
	
	\vskip0pt plus 2fil
\end{titlepage}


\emph{Abstract}\newline

The cuspidal automorphic representations of the linear group over the rationals are, in a certain sense, "the final objects" in the theory of automorphic forms. Among these, a distinguished subclass are the algebraic representations. The complexity of such a representation is measured by two numbers, its motivic weight $w$ and its conductor $N$. It is then natural to try to make lists of automorphic algebraic representations with small conductor and small weight. Chenevier and his coauthors succeeded in making such lists for motivic weight up to 23 and conductor $N = 1$. The next logical case to consider is that of conductor $N = p$, a single prime.

The first main result of this thesis is an explicit list of all such representations with motivic weight up to 17 and conductor $N = 2$; there are 10 of them. This result can be extended under the additional hypothesis of self-duality up to motivic weight 19. There are similar results for prime conductor up to 17 (in which the weight bound becomes lower as the conductor becomes higher).
Making exhaustive lists of automorphic representations (up to a certain motivic weight and conductor) involves two steps: firstly find the \emph{footprints} of the representations in question; secondly, prove that the list is complete.

For the first step, we use the theory of Arthur, which allows for the construction of many relevant representations from classical modular forms. (First the classical object leads to a representation of an orthogonal or symplectic group, which can then be transferred to a general linear group.) For the second step, we use an analytic method known as the explicit formula of Riemann-Weil-Mestre. For small weight and conductor, the lower bound provided by the constructive method coincides with the upper bound provided by the explicit formula, and hence one has obtained a complete list of automorphic representations. 

Along Arthur's theory, the relevant transfer for this thesis is that of split orthogonal groups ${\rm SO}_{2n + 1}$ to ${\rm GL}_{2n}$. Since the goal is to construct representations of ${\rm GL}_{2n}$ with prime conductor, a precise understanding of the representations of ${\rm SO}_{2n + 1}$ with prime conductor is required.
This is where the local part of the thesis comes in. We are able to classify the irreducible, admissible, tempered representations of ${\rm SO}_{2n + 1}(F)$ of prime conductor, where $F$ is a $p$-adic field. We are furthermore able to characterize such representations according to a conjecture of Gross (which is then proven in the given case). This is the second main result of this thesis.\newpage

\emph{Résumé}\newline

Les représentations automorphes cuspidales du groupe linéaire sur le corps des rationnels sont, en un certain sens, \og les objets finaux \fg de la théorie des formes automorphes. On s’intéresse ici à une sous-classe, celle des représentations algébriques. La complexité d'une telle représentation est mesurée par deux entiers, son poids motivique $w$ et son conducteur $N$. Il est alors naturel d'essayer d’établir une liste de représentations automorphes cuspidales algébriques de petit conducteur et de petit poids. Chenevier et ses coauteurs ont réussi à établir une telle liste en poids motivique inférieur à $23$ et en conducteur $N = 1$. Le cas suivant à considérer est celui du conducteur $N = p$, où $p$ est un nombre premier.

Le premier résultat principal de cette thèse est une liste explicite de toutes les représentations de ce type, en poids motivique inférieur à $17$ et en conducteur $N = 2$ (il y en a 10). Sous l'hypothèse supplémentaire d'autodualité, le résultat peut être étendu jusqu'au poids motivique $19$. On obtient des résultats similaires pour des conducteurs premiers jusqu’à $17$ (la borne de poids diminuant à mesure que le conducteur augmente).
Ces listes exhaustives de représentations automorphes sont obtenues en deux étapes : il faut trouver \emph{l'empreinte} des représentations en question puis prouver que la liste ainsi obtenue est complète.

Pour la première étape, nous utilisons la théorie d'Arthur, qui permet de construire de nombreuses représentations pertinentes à partir de formes modulaires classiques : l'objet classique conduit à une représentation d'un groupe orthogonal ou symplectique, qui peut ensuite être transférée à un groupe linéaire général. Pour la deuxième étape, nous utilisons une méthode analytique à savoir la formule explicite de Riemann-Weil-Mestre. Pour les petits poids et petits conducteurs, la limite inférieure fournie par la méthode constructive coïncide avec la limite supérieure fournie par la formule explicite, et on obtient ainsi une liste exhaustive des représentations automorphes.

Le transfert qui nous intéresse ici, selon la théorie d'Arthur, est celui des groupes orthogonaux déployés ${\rm SO}_{2n + 1}$ vers ${\rm GL}_{2n}$. Puisque le but est de construire des représentations de ${\rm GL}_{2n}$ de conducteur premier, il s’agit de comprendre les représentations de ${\rm SO}_{2n+1}$ de conducteur premier.
C'est là qu'intervient la partie locale de la thèse. Nous sommes en mesure de classifier les représentations irréductibles, admissibles et tempérées de ${\rm SO}_{2n+1}(F)$ de conducteur premier, où $F$ est un corps $p$-adique. Nous sommes en outre capable de caractériser de telles représentations selon une conjecture de Gross (qui est donc prouvée dans le cas donné). Cela constitue le deuxième résultat principal de cette thèse.\newpage

\newpage

\tableofcontents

\chapter*{Introduction}

\section*{Énoncés locaux}

Nos énoncés locaux porteront sur une extension finie $F$ du corps $\qp$ des nombres $p$-adiques. Considérons pour l'instant le cas de $F=\qp$, qui nous permettra d'introduire les choses plus simplement.

\subsection*{Conducteur d'une représentation locale de $\GL_n$}
Commençons avec un caractère $\chi$ de $\GL_1(\qp)=\qp^\times$, \ie un morphisme continu de $\qp^\times$ dans $\C^\times$. On a une filtration naturelle de $\qp^\times$ par les sous-groupes $U_m$ pour $m \in \Z_{\geq 0}$ avec $U_0=\zp^\times$ et $U_m=1+p^m\zp$ si $m\geq 1$.
La continuité de $\chi$ impose que $\chi_{|U_m}$ est trivial pour $m$ assez grand. On définit alors l'\emph{exposant} de $\chi$, $\aar(\chi)$, comme le plus petit entier $m$ tel que $\chi_{|U_m}$ soit trivial. Il est d'ailleurs intéressant à plus d'un titre de considérer plutôt la quantité $p^{\aar(\chi)}$, dite \emph{conducteur} de $\chi$.

Soit maintenant $\pi$ une représentation lisse de $\GL_n(\qp)$, \ie un morphisme continu de $\GL_n(\qp)$ dans $\GL(V)$ (où $V$ un $\C$-espace vectoriel et où $\GL(V)$ est muni de la topologie discrète). On cherche à étendre à ce contexte la notion de conducteur. On suppose de plus que $\pi$ est irréductible. 

Si $\pi$ est de dimension finie (\ie si $\dim V < \infty$), alors on montre simplement qu'elle est en fait la composée du déterminant et d'un caractère de $\qp^\times$ (et donc de dimension 1). Hormis ce cas, on a donc affaire à une représentation de dimension infinie et, plutôt que de chercher un sous-groupe $K$ de $\GL_n(\qp)$ tel que $\pi_{|K}$ soit triviale, 
on s'intéresse aux sous-groupes compacts $K$ les plus gros possibles de $\GL_n(\Q_p)$ tels que l'espace des $K$-invariants $\pi^K$ soit non trivial.\ps 

Le fait de trouver alors une famille de sous-groupes compacts \og avec de bonnes propriétés \fg{}\footnote{Nous restons volontairement vague pour l'instant. À noter que l'idée naturelle d'utiliser la filtration de $\GL_n(\qp)$ par $\{\GL_n(\zp),(1+p^m {\rm M}_n(\zp))_{m\geq 1}\}$ ne donne pas de résultats satisfaisants.} est délicat. Le premier résultat en ce sens historiquement est le travail d'Arthur Atkin et Joseph Lehner concernant les formes modulaires pour des sous-groupes arithmétiques de $\SL_2(\Z)$ 
dont le but était de détecter, parmi les formes modulaires, celles qui sont \emph{nouvelles} et celles qui sont \emph{anciennes}. 
Ce travail, réinterprété par William Casselman, fournit une famille intéressante (en fait, une suite décroissante) de tels sous-groupes compacts de $\GL_2(\Q_p)$.
Hervé Jacquet, Ilia Piatetski-Shapiro et Joseph Shalika généralisent le point de vue de Casselman dans \cite{JPSS} et utilisent la famille de sous-groupes de congruences suivante\footnote{Nous adoptons des notations légèrement différentes dans un souci de cohérence, ainsi ce que nous notons $K(p^m)$ est noté $K_m$ \emph{loc. cit.}} :
%
 $K(1)=\GL_n(\zp)$ et, pour $m \geq 1$,

\[
K(p^m)=\left\lbrace \begin{pmatrix}
A & B \\
C & d
\end{pmatrix} \in \GL_n(\zp) \middle| {\begin{array}{l} A \in \GL_{n-1}(\zp) \\ C \in {\rm M}_{1,n-1}(p^m\zp)\\ d \in 1+p^m\zp \end{array}}
\right\rbrace.
\]

Sous des hypothèses de généricité de $\pi$ que nous ne détaillons pas (mais qui sont par exemple satisfaites si $\pi$ est la composante locale d'une représentation automorphe cuspidale de $\GL_n$ sur $\Q$), ces auteurs montrent qu'il existe un plus petit entier $m$, noté $\aar(\pi)$, tel que $\pi$ ait des $K(p^m)$-invariants. Mieux, ces invariants sont alors de dimension 1 (\cite{JPSS}, §5 Théorème). Là encore, il est intéressant de considérer la quantité $p^{\aar(\pi)}$, le \emph{conducteur} de $\pi$.

On remarque d'ailleurs que, lorsque $n=1$, les sous-groupes $K(p^m)$ sont bien égaux aux sous-groupes $U_m$, assurant la cohérence de ces définitions.\newline

\color{black}
Les conducteurs locaux ainsi définis ont une manifestation globale remarquable et importante. Soit $\Pi$ une représentation automorphe cuspidale de $\GL_n$ sur $\Q$ : pour chaque nombre premier $p$, sa composante locale (en $p$) $\Pi_p$ est bien définie, c'est une représentation lisse irréductible de $\GL_n(\qp)$. Notons ${\rm N}(\Pi)$ le produit des $p^{\aar(\Pi_p)}$ lorsque $p$ parcourt l'ensemble de tous les nombres premiers (ce produit est bien défini car on a $\aar(\Pi_p)=0$ pour $\pi_p$ non ramifiée et donc pour presque tout $p$). Alors ce conducteur global intervient dans l'équation fonctionnelle qui relie la fonction $L$ de $\Pi$, définie par Roger Godement et Hervé Jacquet \cite{GJ}, à celle de sa représentation contragrédiente $\Pi^\vee$ :
\begin{equation*}
L(s,\Pi)=\eps(\Pi)\,{\rm N}(\pi)^{(\frac{1}{2}-s)}L(1-s,\Pi^\vee),
\end{equation*}
où $\eps(\Pi)$ est une constante (c'est en fait la valeur de la \emph{fonction} epsilon en $s=\frac{1}{2}$, selon l'écriture de \cite{JPSS}). C'est d'ailleurs un des objectifs de \cite{JPSS} que de montrer que les entiers $\aar(\Pi_p)$ définis par la filtration $K(p^m)$ (pour tout $p$) sont bien ceux qui interviennent dans l'équation fonctionnelle (pour être exact, ils raisonnent dans l'autre sens). \newline 
\color{black}

Reprenons notre caractère $\chi$ de $\qp^\times$. Par la théorie du corps de classes local, on peut aussi le voir comme un caractère de $\W_{\qp}$, le \emph{groupe de Weil} de $\qp$, un substitut \emph{ad hoc} du groupe de Galois $\Gal(\overline{\qp}/\qp)$. Quitte à tordre notre caractère $\chi$ par un caractère non ramifié, on peut même le voir comme 
un caractère du groupe de Galois d'une extension galoisienne finie $L/\qp$ de $\qp$. Il est alors possible de définir son exposant d'Artin qui mesure (également via une filtration) sa ramification \og côté galoisien \fg{} et c'est une des propriétés de la théorie du corps de classes (voir \cite{Se}, XV §2) que la filtration \og galoisienne \fg correspond, via l'application de réciprocité d'Artin, à la filtration $U_m$ introduite ci-dessus : on compte bien la même chose.\ps 

En dimension supérieure, la \emph{correspondance de Langlands locale} pour $\GL_n$ \footnote{Prouvée dans le cas non archimédien par Michael Harris et Richard Taylor \cite{HT}, avec une seconde preuve par Guy Henniart \cite{Henniart2000}.} associe, à une représentation lisse irréductible $\pi$ de $\GL_n(\qp)$, une (classe de conjugaison de) représentation semi-simple continue $\varphi=\Ll(\pi)$ de $\WD_{\qp}$ de dimension $n$, où $\WD_{\qp}$ est le groupe de Weil-Deligne de $\qp$, produit direct de $\W_{\qp}$ et du groupe compact $\SU(2)$.

On peut encore définir l'exposant d'Artin de $\varphi$ et se pose alors la question de savoir si les conducteurs (ou les exposants) sont les mêmes pour $\pi$ et pour son paramètre de Langlands $\Ll(\pi)$. La réponse est oui \emph{par construction} de la correspondance de Langlands, dont on veut qu'elle respecte ce genre de propriétés. Il faut noter que cette compatibilité est d'ailleurs utilisée pour \emph{rigidifier} la correspondance (et elle est \og cachée \fg dans la compatibilité des {\it facteurs epsilon de paires} dont le rôle crucial a été identifié par Henniart \cite{Hen85}, \cite{Hen93}). \newline

\subsection*{Conducteur d'une représentation locale de $\SO_{2n+1}$}
Fort de cette mécanique bien huilée, on peut s'interroger sur ce qu'il se passe pour un autre groupe (algébrique) que $\GL_n$. Peu de choses sont connues en général.
Toutefois, dans le cas particulier du groupe ${\rm GSp}_4$, une théorie comparable à celle de Atkin-Lehner/Casselman a été mise en place par Brooks Roberts et Ralf Schmidt \cite{RS_book}. Une famille de sous-groupes dits {\it paramodulaires} y joue le rôle des $K(p^m)$ ci-dessus : nous y reviendrons plus loin.

Partant des isomorphismes exceptionnels ${\rm PGSp}_2 \simeq \SO_3$ et ${\rm PGSp}_4 \simeq \SO_5$, ainsi que d'une remarque d'Armand Brumer \cite{Brumer}, Benedict Gross a par la suite observé que la famille des groupes paramodulaires peut se définir plus généralement pour tous les groupes spéciaux orthogonaux impairs, et il a formulé une conjecture dans ce contexte général, que nous allons rappeler.\smallskip


Précisons tout de suite à quels groupes on s'intéresse.
 On munit l'espace $\qp^{2n+1}$ de la forme quadratique $q:\underline{x} \mapsto x_1x_2+\cdots+x_{2n-1}x_{2n}+x_{2n+1}^2$ et on considère alors les automorphismes de $\qp^{2n+1}$, qui préservent $q$ et qui sont de déterminant 1. Cela définit un groupe (algébrique) réductif et déployé sur $\qp$, noté $\SO_{2n+1}$.

%
%

Soit donc $\pi$ une représentation lisse irréductible de $\SO_{2n+1}(\qp)$ et reprenons les choses \emph{à l'envers} par rapport au cas de $\GL_n$. La correspondance de Langlands locale pour $\SO_{2n+1}$ est connue par les travaux de James Arthur \cite{Art13} et de Colette M\oe{}glin. 
À $\pi$, Arthur associe son paramètre de Langlands $\varphi=\Ll(\pi)$ qui est un(e classe de conjugaison de) morphisme continu semi-simple de $\WD_{\qp}$ à valeurs dans $\Sp_{2n}(\C)$. En effet, les conjectures de Langlands précisent que les paramètres d'un groupe réductif $G$ sont à valeurs dans le dual de Langlands de ce groupe ${}^L G$. Ceci était \og invisible \fg{} pour le cas de $\GL_n$ puisque le dual de Langlands de $\GL_n$ est $\GL_n$ lui-même, tandis que le dual de $\SO_{2n+1}$ est $\Sp_{2n}$. 

On peut alors en utilisant la représentation tautologique (fidèle) $\tau : \Sp_{2n}(\C)\rightarrow \GL_{2n}(\C)$ voir notre paramètre $\varphi$ comme une représentation continue semi-simple (autoduale symplectique) de $\WD_{\qp}$ et définir son exposant d'Artin. Cela nous fournit alors une \emph{définition} de l'exposant de $\pi$ noté $\aar(\pi,\tau)$ et, partant, de son conducteur.\ps 

Il n'existe pas de théorie similaire à celle de Godement-Jacquet (en tout cas pas aussi générale) pour associer une fonction $L$ à une représentation $\pi$ lisse irréductible de $\SO_{2n+1}(\qp)$.
On peut en revanche à ce dessein utiliser son paramètre de Langlands $\varphi$ et la construction classique d'Artin. De même que pour le conducteur, 
 les objets sont bien définis à valeurs dans le groupe linéaire, il est donc encore nécessaire de faire appel à une représentation de $\Sp_{2n}(\C)$. Si l'on utilise la même représentation tautologique $\tau$, alors on \emph{définit} $L(s,\pi,\tau)$ comme étant $L(s,\tau \circ \varphi)$.
 
Ces définitions étant posées, l'interprétation analytique globale du conducteur en termes de fonctions $L$ standard (dans l'esprit de ce qui a été fait \emph{supra}) découle alors tautologiquement du cas des groupes linéaires, du moins si l'on admet l'existence de transferts vers $\GL_{2n}$ (comme ceux fournis par la théorie d'Arthur). 
La question restante importante est donc de trouver une interprétation des $\aar(\pi_p,\tau)$ en termes de sous-groupes compacts ouverts de $\SO_{2n+1}(\qp)$.
 
\subsubsection{La conjecture de Gross} 
On en vient donc à ce qui était le premier point de notre discussion sur $\GL_n$ : la question est de savoir si cette ramification est mesurable par une famille de sous-groupes et des invariants de $\pi$ associés.

Benedict Gross définit dans \cite{Gross} une famille de sous-groupes compacts ouverts $K(p^m)$ de $\SO_{2n+1}(\qp)$ pour $m\in \Z_{\geq 0}$, dits sous-groupes paramodulaires, ainsi qu'une famille de sur-groupes $J(p^m)$ (en fait plus naturels), vérifiant $J(p^m)/K(p^m)=\Z/2\Z$.


\begin{conj_nn} {\rm (Gross)}

Soit $\pi$ une représentation lisse irréductible générique de $\SO_{2n+1}(\qp)$, de conducteur $p^m$.
Alors pour tout $x<m$, $\pi^{K(p^x)}=\{0\}$ et $\pi^{K(p^m)}\neq\{0\}$. \smallskip

Mieux, $\pi^{K(p^m)}$ est de dimension $1$ et, si $m\geq 1$, le groupe $J(p^m)$ agit sur cette droite par un signe, ce signe étant exactement le facteur epsilon de $\pi$ (en $s=\frac{1}{2}$).
\end{conj_nn}

Dans le cas $\SO_3\simeq \PGL_2$, la famille de sous-groupes $K(p^m)$ est conjuguée à (l'image par $\SL_2(\Z)\rightarrow \PGL_2(\Z)$ de) la famille de sous-groupes de congruence de $\SL_2(\Z)$ classiquement notés $\Gamma_0(p^m)$. La conjecture est alors un théorème, dont la preuve complète est due à William Casselman (\cite{Casselman-AL}).

Dans le cas $\SO_5 \simeq {\rm PGSp}_4$, la famille de sous-groupes $K(p^m)$ est conjuguée à la famille des sous-groupes étudiés par Brooks Roberts et Ralf Schmidt dans \cite{RS_book}. Ces deux auteurs démontrent \emph{loc. cit.} la conjecture.\ps 

Indépendamment de la dimension, on remarque que dans le cas $m=0$, on a affaire à une représentation non ramifiée. La conjecture est alors très classiquement vraie (voir paragraphes \ref{Représentations non ramifiées} et \ref{Paramètres non ramifiés}).

Dans cette partie locale, on s'intéresse uniquement au cas $m=1$, pour des raisons qui s'éclaireront lorsque nous introduirons la partie globale.

Le seul résultat disponible à notre connaissance en toute dimension est celui de la thèse de doctorat de Pei-Yu Tsai \cite{Tsai-phd}, qui démontre la conjecture pour $\pi$ supercuspidale (mais elle le démontre \emph{pour tout} $m$).

Nous démontrons la conjecture dans le cas $m=1$ sous l'hypothèse que $\pi$ est tempérée. Puisque nous l'avons démontrée dans le cadre général d'un corps local non archimédien de caractéristique nulle $F$ (d'anneau des entiers $\Ok$ dont l'idéal maximal est $\pk$), énonçons le résultat dans ce cadre.
Nous adoptons maintenant les notations que nous avons utilisées dans le corps du texte : nous notons respectivement $\K_0,\J^+,\J$ ce que Gross note $K(1),K(\pk),J(\pk)$ et nous parlons de sous-groupe compact hyperspécial, de sous-groupe \emph{paramodulaire}, de sous-groupe \emph{épiparamodulaire}.\ps 

%
%

\begin{thmintro}\label{thm_A}
Soit $F$ un corps local non archimédien de caractéristique nulle. Soit $\pi$ une représentation irréductible tempérée de $\SO_{2n+1}(F)$. Soit $\tau$ la représentation tautologique (fidèle) $\Sp_{2n}(\C)\rightarrow \GL_{2n}(\C)$. Alors les propositions suivantes sont équivalentes :
\begin{enumerate}
\item $\pi^{\K_0}=\{0\}$ et $\pi^{\J^+}\neq\{0\}$ ;
\item $\pi$ est de conducteur $\pk$.
\end{enumerate}
De plus, l'espace des invariants $\pi^{\J^+}$ est une droite, sur laquelle le groupe $\J$ agit par un signe. Ce signe est égal à $\eps(\pi,\tau)$.
\end{thmintro}

On renvoie aux Théorèmes \ref{thm_pcpl_temp} et \ref{thm_J_variants_eps} (et à tout le paragraphe \ref{Lien avec les facteurs epsilon} pour la définition correcte de $\eps(\pi,\tau)$ qui fait intervenir un caractère additif de $F$). Nous ne savons pas démontrer la conjecture dans le cas où $\pi$ n'est pas tempérée, plus exactement nous ne savons pas démontrer l'implication $\mathit{2}. \Rightarrow \mathit{1}.$ Une piste potentielle serait de faire l'analogue en conducteur $\pk$ de ce que fait Casselman en conducteur $\Ok$ dans \cite{Cass-art}, pour s'assurer que le quotient de Langlands a bien les invariants souhaités. 
Ces représentations non tempérées n'interviennent cependant pas dans notre étude globale.

Précisons maintenant la structure de notre travail et les étapes de la démonstration. \newline

Nous commençons par des rappels généraux sur les représentations lisses des groupes réductifs $p$-adiques, qui nous permettent en outre de fixer certaines notations pour la suite (Chapitre \ref{Chapitre_Représentations}). \smallskip

Nous formulons également des rappels concernant la correspondance de Langlands (Chapitre \ref{Chapitre_Paramétrisation_Lgl}), notamment dans les cas que nous utiliserons des groupes\footnote{Notre convention est de noter les groupes algébriques par des lettres droites et grasses. Pour la lisibilité, nous faisons une exception pour les familles de groupes $\GL_n, \, \Sp_{n},\, \SO_n$.} $\GL_n$ et $\SO_{2n+1}$. Nous analysons de façon plus approfondie au §\ref{Paramètres discrets} la forme que prennent les paramètres de Langlands discrets de $\SO_{2n+1}$. \smallskip

Nous définissons au Chapitre \ref{Le groupe paramodulaire} les groupes hyperspéciaux, épiparamodulaires et paramodulaires, suivant \cite{Gross} et \cite{Tsai-phd}. Plusieurs éléments nouveaux et déterminants pour la démonstration interviennent.
\begin{itemize}
\item On remarque que les groupes épiparamodulaires et paramodulaires ont les mêmes propriétés de \og bonne intersection \fg{} avec les sous-groupes de Levi de $\SO_{2n+1}(F)$, que les sous-groupes compacts hyperspéciaux (\cf Propositions \ref{Hypersp_inter_Levi}, \ref{Param_inter_Levi}, \ref{Param+_inter_Levi}), \ie une telle intersection fait apparaître des sous-groupes \emph{du même type} permettant un raisonnement par récurrence sur la dimension $2n+1$.
\item Le groupe paramodulaire $\J^+$ est défini comme un sous-groupe d'indice 2 du groupe épiparamodulaire $\J$ (Proposition-Définition \ref{prop-def_J+}). Nous donnons une autre caractérisation (à notre connaissance nouvelle) de $\J^+$ à l'intérieur de $\J$ par la norme spinorielle (§\ref{Norme spinorielle}). L'avantage de la norme spinorielle est qu'elle est définie sur $\SO_{2n+1}(F)$ tout entier, si bien qu'une torsion adéquate nous permettra de raisonner sur des représentations avec des $\J$-invariants plutôt que sur des représentations avec des $\J^+$-invariants.
\item Les factorisations \emph{d'Iwasawa} faisant intervenir les groupes $\K_0$ et $\J$ (voir paragraphe \ref{Factorisations d'Iwasawa}) et un sous-groupe de Borel de $\SO_{2n+1}(F)$ se trouvent dans la thèse de Tsai §7.2. Nous en donnons une nouvelle démonstration, géométrique avec des réseaux (là où elle utilisait des techniques immobilières).
\end{itemize}
\smallskip

Le Chapitre \ref{Représentations avec des invariants paramodulaires} caractérise les représentations qui vérifient la propriété $\mathit{1.}$ du Théorème et fait intervenir le résultat-clé suivant (c'est la Proposition \ref{ser_disc_>4_pas_d'inv}).
\begin{prop_nn}
Soit $F$ un corps local non archimédien de caractéristique nulle.
Il n'existe pas de série discrète de $\SO_{2n+1}(F)$ pour $2n+1 \geq 5$ ayant des $\J^+$-invariants.
\end{prop_nn}
{\noindent Outre les trois ingrédients du Chapitre \ref{Le groupe paramodulaire} sus-mentionnés, la démonstration de cette Proposition utilise les propriétés des foncteurs de Jacquet des séries discrètes, en lien avec la paramétrisation de Langlands, démontrées par Colette M\oe{}glin et Marko Tadi\'c (\cite{Moeg_ser_dis}, \cite{Moeg-Tad}). Cela nous permet de nous ramener au fait suivant : \emph{si $\pi$ est une série discrète non ramifiée pour $\GL_n$, alors $n=1$ et $\pi$ est un caractère (non ramifié)}.}

Indiquons également que nous n'utilisons pas le résultat de Tsai (dont les techniques de démonstration sont différentes des nôtres). En effet, une représentation ayant des $\J^+$-invariants possède \emph{a fortiori} des ${\rm I}$-invariants où ${\rm I}$ désigne un sous-groupe d'Iwahori (c'est la Proposition \ref{prop_I_inclus_J+}), elle est donc classiquement dans la série principale non ramifiée et ne peut donc pas être supercuspidale (sauf bien sûr si $2n+1=1$).

\smallskip

Enfin, au Chapitre \ref{Conducteur}, après avoir rappelé comment définir le conducteur d'une représentation locale et redémontré \emph{l'inégalité d'Henniart} \eqref{exposant_prod_tens_rep_locales}, nous montrons que les représentations lisses irréductibles tempérées de $\SO_{2n+1}(F)$ de conducteur $\pk$ sont exactement celles que nous avons identifiées au Chapitre précédent. Cela démontre ainsi le Théorème de façon plus explicite que nécessaire (puisqu'on sait exactement de quelles représentations il s'agit). Nous terminons par des rappels sur les facteurs epsilon (§\ref{Lien avec les facteurs epsilon}) qui nous permettent de vérifier la bonne propriété de \emph{détection du signe local} par l'action du groupe $\J$ sur l'espace des $\J^+$-invariants (Théorème \ref{thm_J_variants_eps}).

\newpage

\section*{Énoncés globaux}
Dans la seconde partie de cette thèse, de nature globale, notre objectif
est de généraliser au cas d'un conducteur $>1$ les résultats de
classification de représentations automorphes cuspidales obtenus par
Chenevier-Lannes (\cite{Chen-Lannes}, Théorème F) et Chenevier-Taïbi (\cite{Chen-Ta}, Theorems 3, 4, 5) dans le
cas du conducteur 1. Avant d'énoncer nos résultats, discutons plus en
détail le type de représentations que l'on cherche à classifier.\ps 

On s'intéresse ici aux représentations automorphes cuspidales des groupes linéaires $\GL_n$, définis sur le corps $\Q$ des rationnels. Nous ne commentons pas le choix du corps de base (d'autres cas seraient intéressants, mais ne sont pas étudiés ici), mais nous attardons sur le choix des groupes linéaires (plutôt que de groupes généraux).

Les représentations automorphes cuspidales du groupe linéaire disposent en effet de meilleures propriétés.
\begin{itemize}
\item Leurs fonctions $L$ (et plus généralement celles de \emph{paires} de telles représentations) ont des propriétés analytiques bien comprises et démontrées (c'est l'objet du §\ref{Fonctions $L$ et facteurs epsilon} que de les rappeler).
\item Il y a une rigidité très forte\footnote{À laquelle est d'ailleurs liée le fait que la correspondance de Langlands locale pour $\GL_n$ soit une bijection.} sur ces représentations comme l'indique le Théorème de multiplicité 1 forte pour $\GL_n$ (\cite{PS_mult_1}, \cite{JS}).\ps 
\end{itemize}

Au-delà de ce seul \emph{confort d'étude}, le cas de $\GL_n$ est \emph{primordial} (au sens de \og devant être étudié en premier \fg{}). En effet, selon les conjectures globales de Langlands et d'Arthur, \emph{toutes} les représentations automorphes cuspidales de \emph{tous} les groupes réductifs s'expriment en termes d'objets pour le groupe linéaire : il faut donc \emph{d'abord} comprendre les représentations automorphes cuspidales de $\GL_n$ avant de pouvoir s'intéresser à celles des autres groupes.

Le lecteur peut alors être surpris que l'étude locale ait porté sur le groupe $\SO_{2n+1}$ : c'est précisément que nous utilisons cette intrication \emph{en sens inverse}. 
Cette approche est rendue possible par le travail monumental de James Arthur (\cite{Art13}), et par le complément apporté par Olivier Taïbi (\cite{Taibi_cpctmult}), cas particuliers des conjectures globales susmentionnées.

\smallskip

Nous poursuivons donc ici un objectif de classification des représentations automorphes cuspidales du groupe linéaire sur $\Q$.
Cette question est évidemment encore trop générale et il faut préciser à quelle sous-classe de représentations on va s'intéresser. \newline


Notre première restriction concerne le fait que l'on s'intéresse uniquement aux représentations \emph{algébriques}. Ces représentations sont d'un intérêt particulier car, selon des conjectures de Fontaine-Mazur et Langlands (\cite{Fontaine-Mazur}, \cite{Langlands_funct}), elles sont reliées à la cohomologie des variétés algébriques sur $\Q$. Rappelons leur définition.

Soit $\pi$ une représentation automorphe cuspidale de $\GL_n$ sur $\Q$. Alors sa composante archimédienne $\pi_\infty$ admet un caractère infinitésimal que l'on peut voir  suivant Harish-Chandra et Langlands (on renvoie à \cite{Lgl_Euler_products}), comme une classe de conjugaison semi-simple de $\mathrm{M}_n(\C)$. On appelle \emph{poids} de $\pi_\infty$ (ou de $\pi$ par abus de langage) les valeurs propres de cette classe de conjugaison semi-simple. Une représentation automorphe cuspidale est dite \emph{strictement algébrique} si ses poids sont des entiers relatifs $w_1 \geq \cdots \geq w_n$. Quitte à tordre $\pi$ par une puissance de la norme, on peut supposer que le plus petit des poids $w_n$ est 0, le plus grand étant alors appelé poids motivique de $\pi$ et noté ${\rm w}(\pi)$. Selon les conjectures susmentionnées, $\pi$ est ainsi associée à une représentation $\rho : \Gal(\overline{\Q}/\Q)\rightarrow \GL_n(\overline{\Q_\ell})$ géométrique et irréductible dont les poids de Hodge-Tate sont ceux de $\pi$, de poids de Deligne égal à ${\rm w}(\pi)$ avec la fonction $L$ (de Godement-Jacquet) de $\pi$ égale à la fonction $L$ (d'Artin) de $\rho$.

En réalité, le fait que $\pi_\infty$ ne soit pas un module de Harish-Chandra quelconque, mais bien la composante archimédienne d'une représentation automorphe cuspidale impose que la quantité $w_i+w_{n+1-i}$ est constante lorsque $i$ parcourt $\{1,\cdots,n\}$ (c'est le \emph{lemme de pureté} de Clozel). Il est alors plus naturel de \emph{centrer} les représentations et on dira dans ce texte que $\pi$ est \emph{algébrique} si ses poids $w_1 \geq \cdots \geq w_n$ avec $w_i=-w_{n+1-i}$ sont tous dans $\Z$ ou tous dans $\frac{1}{2}\Z - \Z$. Avec cette convention, le poids motivique de $\pi$, ${\rm w}(\pi)$ est défini comme étant $2w_1$.\newline


Notre seconde restriction concerne le conducteur des représentations automorphes cuspidales. Soit donc $\pi$ une telle représentation. De même que l'on a contraint (avec l'algébricité) la composante archimédienne $\pi_\infty$, nous allons imposer des conditions aux places finies. Soit $p$ un nombre premier, la composante locale $\pi_p$ est une représentation admissible irréductible générique de $\GL_n(\qp)$ et on peut donc définir son conducteur qui se manifeste, d'après la discussion \emph{supra}, de trois façons : invariants par une certaine famille de sous-groupes, équation fonctionnelle, exposant d'Artin du paramètre de Langlands. Le conducteur de $\pi$ est alors défini comme le produit des conducteurs de ses composantes locales $\pi_p$, ce produit étant en fait un produit fini puisque $\pi_p$ est non ramifiée hors d'un nombre fini de places.\newline

Notons que l'on dispose déjà du résultat de finitude suivant dû à Harish-Chandra \cite{HC68} : soit $n\geq 1$, alors \emph{il n'existe qu'un nombre fini de représentations automorphes cuspidales algébriques de $\GL_n$ sur $\Q$ de poids et de conducteur fixés.} Savoir quel est ce nombre, et quelles sont les représentations en question devient une question très difficile dès que $n>2$ car alors $\GL_n(\R)$ n'a plus de séries discrètes et on ne dispose d'aucune \og formule de dimension\fg d'espaces de formes automorphes. C'est donc tout l'intérêt de l'objectif de \emph{classification}. 

Le cas le plus naturel à considérer selon ce prisme est celui du conducteur 1, \ie où $\pi_p$ est non ramifiée \emph{pour tout} $p$ et il a été exploré par Gaëtan Chenevier et ses co-auteurs David Renard \cite{Chen-Ren}, Jean Lannes \cite{Chen-Lannes} et Olivier Taïbi \cite{Chen-Ta}. Ces auteurs parviennent à lister \emph{toutes} les représentations automorphes cuspidales algébriques de $\GL_n$ (bien noter que $n$ est quelconque) de conducteur 1, sous restriction de poids motivique ($\leq 23$).

Notre objectif est de conduire le même travail, mais dans un cas minimalement ramifié, qui est celui du conducteur $p$ (une seule place qui ramifie, et de façon minimale). 
Il faut déjà noter qu'en \og petit poids \fg{}, il existe une amélioration du résultat de finitude de Harish-Chandra, due à Chenevier (\cite{Chen-HM}, Theorem A) : \emph{soit $N \in \Z_{>0}$, alors il n'existe qu'un nombre fini de représentations automorphes cuspidales algébriques de $\GL_n$ (avec $n$ variable) sur $\Q$, de conducteur $N$ et dont les poids sont dans $\left[-\frac{23}{2};\frac{23}{2}\right]\cap \frac{1}{2}\Z$.} Ce résultat ne dit pas davantage que celui de Harish-Chandra quel est ce nombre, et quelles sont les représentations en question : c'est l'objet de notre travail.  \smallskip

On dira qu'une représentation est \emph{régulière} si ses poids sont distincts. \ps

\begin{thmintro}\label{thm_2_w17}
Il existe exactement $10$ représentations automorphes cuspidales algébriques de $\GL_n$ sur $\Q$ de conducteur $2$ et de poids motivique $\leq 17$ : 
\smallskip
\begin{itemize}
\item $6$ pour $\GL_2$ : $\E_7^+,\, \E_9^-,\, \E_{13}^+,\, \E_{13}^-,\, \E_{15}^+,\, \E_{17}^-$ ; \smallskip
\item $4$ pour $\GL_4$ : $\E_{15,5}^-,\, \E_{17,5}^+,\, \E_{17,7}^-,\, \E_{17,9}^+$ ; \smallskip
\item aucune pour $\GL_n$ avec $n\neq 2,4$. \smallskip
\end{itemize}

En particulier, elles sont toutes autoduales 
et régulières.
\end{thmintro}

Il nous faut préciser un peu les notations. Les indices désignent les doubles des poids positifs de la représentation (ainsi les poids de ${\rm E}_{w,v}$ sont $\{\pm\frac{w}{2},\pm\frac{v}{2}\}$), tandis que l'exposant indique le signe du facteur epsilon local en 2 de la représentation (ce que nous appellerons plus brièvement \emph{signe local en $2$}).

Ainsi la représentation $\E_w^\eps$ désigne la\footnote{Étant entendu que $\dim {\rm S}_{w+1}(\Gamma_0(2),\eps)=1$ ; sinon il y a autant de représentations automorphes cuspidales de $\GL_2$ que la dimension de ${\rm S}_{w+1}(\Gamma_0(2),\eps)$, représentations que nous distinguons d'ailleurs par des lettres en exposant.} représentation automorphe cuspidale de $\GL_2$ engendrée par la droite ${\rm S}_{w+1}(\Gamma_0(2),\eps)$ des formes modulaires paraboliques de poids $w+1$ pour le sous-groupe de congruence $\Gamma_0(2) \subset \SL_2(\Z)$, et de signe d'Atkin-Lehner $\eps$ (c'est un fait, commenté au §\ref{En niveau Gamma_0_N}, que le signe local $\eps_2(\pi)$ est le même que le signe d'Atkin-Lehner).\newline

Nous avons également les extensions suivantes. La première concerne les représentations \emph{autoduales}.
\begin{thmintro}\label{thm_2_w19}
Il existe exactement $20$ représentations automorphes cuspidales algébriques \emph{autoduales} de $\GL_n$ sur $\Q$ de conducteur $2$ et de poids motivique $\leq 19$. Outre les $10$ représentations de poids motivique $\leq 17$ ci-dessus, nous avons les représentations suivantes : \smallskip
\begin{itemize}
\item $2$ pour $\GL_2$ :  $\E_{19}^+,\, \E_{19}^-$ ; \smallskip
\item $6$ pour $\GL_4$ :  $\E_{19,3}^+,\,\E_{19,5}^-,\, \E_{19,9}^+,\,\E_{19,9}^-,\,\E_{19,11}^+,\, \E_{19,13}^-$ ; \smallskip
\item $2$ pour $\GL_6$ :  $\E_{19,13,3}^-,\, \E_{19,13,5}^+$ ; \smallskip
\item aucune pour $\GL_n$ avec $n\neq 2,4,6$. \smallskip
\end{itemize}
En particulier, elles sont toutes régulières.
\end{thmintro}

La seconde extension est conjecturale, et concerne les représentations \emph{autoduales régulières}.

\begin{conjintro}\label{thm_2_w21}
Soit $\pi$ une représentation automorphe cuspidale algébrique \emph{autoduale régulière} de $\mathrm{GL}_n$ sur $\mathbb{Q}$ de conducteur $2$ et de poids motivique $\leq 21$.

Alors $n \leq 8$ et $\pi$ appartient à une liste finie explicite dont les éléments sont donnés en Annexe \ref{Tables_de_representations}. Outre les $20$ représentations ci-dessus, on trouve $40$ représentations de poids motivique $21$ exactement, dont on connaît les poids et le signe local en $2$.
\end{conjintro}
La finitude est ici automatique (par le résultat de Harish-Chandra mentionné \emph{supra}, noter cependant l'amélioration de $n\leq 22$ \emph{a priori} à $n\leq 8$), c'est le fait que l'on dispose d'informations précises sur les éléments de cette liste qui importe. \newline

%
%
%

Et on peut enfin étendre dans une autre direction. Notons $b_2=17 \; ; \;  b_3=13\; ; \;  b_5=11 \; ; \; b_7=b_{11}=7 \; ; \; b_{13}=5 \; ; \; b_{17}=3$.
\begin{thmintro}\label{thm_p}
Soit $\pi$ une représentation automorphe cuspidale algébrique de $\mathrm{GL}_n$ sur $\mathbb{Q}$ de conducteur premier $p \leq 17$ et de poids motivique \emph{impair} $\leq b_p$.

Alors $\pi$ appartient à une liste finie explicite dont les éléments sont donnés en Annexe \ref{Tables_de_representations}, en particulier la dimension $n$ est bornée. De plus, $\pi$ est autoduale, régulière, et on connaît son signe local en $p$.
\end{thmintro}

%

Notre travail utilise les deux mêmes techniques que les livre \cite{Chen-Lannes} et article \cite{Chen-Ta} déjà cités.\ps 
\begin{itemize}
\item Une \emph{technique constructive} liée aux travaux de James Arthur qui permet
de démontrer l'existence de représentations automorphes cuspidales autoduales algébriques régulières de $\GL_n$.

En effet, la classification endoscopique d'Arthur, notamment la formule de multiplicité, les met en relation (par une combinatoire inductive et intriquée) avec les représentations automorphes discrètes $\pi$ de groupes classiques (orthogonaux et symplectiques) sur $\Q$ telles que $\pi_\infty$ est une série discrète. Ces dernières sont alors reliées à des objets plus classiques (formes modulaires classiques, de Siegel, fonctions de réseaux) qu'il devient possible de compter (existence de \og formules de dimensions\fg).\ps 
\item Une \emph{technique limitative} liée à la formule explicite de Riemann-Weil dans le cadre présenté par Jean-François Mestre \cite{Mes} qui nous permet d'interdire l'existence de certaines représentations.

Typiquement, on montre qu'une représentation $\pi$ putative de poids et de conducteur fixés n'existe pas en montrant que la fonction $L$ de la paire $\{\pi, \pi'\}$ (pour une représentation $\pi'$ connue bien choisie) ne peut pas exister. \smallskip
\end{itemize}

En poids motivique suffisamment petit, nous arrivons à \og abouter \fg{} exactement ces deux techniques, nous fournissant la liste exhaustive recherchée. \newline

La quasi-totalité de notre travail a porté sur le conducteur 2 et ce, pour deux raisons (outre le fait que c'est le premier nombre premier) :
\begin{itemize}
\item la technique limitative avec la formule explicite ne fonctionne bien qu'en petit poids et en petit conducteur : la \og perte d'efficacité \fg{} est spectaculaire lorsque le conducteur croît comme le laisse deviner l'énoncé du Théorème \ref{thm_p} (à comparer aussi aux énoncés en conducteur 1 de \cite{Chen-Lannes}, \cite{Chen-Ta}) ;
\item il existe deux façons différentes\footnote{Si l'on mesure le conducteur avec le paramètre de Langlands, alors ces deux façons correspondent respectivement au facteur $\W_{\qp}$ et au facteur $\SU(2)$ de $\WD_{\qp}$.} d'être de conducteur $p$ premier et une seule des deux existe quand $p=2$ (c'est l'objet du Lemme \ref{lemme_cond_2_implique_type_I}).
\end{itemize}
\medskip
Nous précisons maintenant la structure de notre travail et les étapes de la démonstration des Théorèmes \ref{thm_2_w17}, \ref{thm_2_w19}, \ref{thm_p}. \newline

Le Chapitre \ref{Chapitre_Représentations_automorphes...} précise de manière plus formelle les objets que l'on étudie. Nous en profitons pour rappeler ensuite au §\ref{Fonctions $L$ et facteurs epsilon} comment calculer les fonctions $L$ et facteurs epsilon des représentations automorphes cuspidales algébriques (et même des \emph{paires} de représentations), ce qui nous sera utile aux Chapitres \ref{La_formule explicite de Riemann-Weil-Mestre} et \ref{Chapitre_calculs}. Ce Chapitre se termine par la présentation de l'alternative symplectique-orthogonale au sens de \cite{Art13} et par la cruciale Proposition \ref{type I sp} qui explique pourquoi nous avons besoin de comprendre les représentations locales du groupe $\SO_{2n+1}$ (et fait donc le lien avec la Première Partie).\smallskip

Le Chapitre \ref{Théorie d'Arthur pour SO-2n+1} présente des rappels de la théorie d'Arthur globale pour les groupes spéciaux orthogonaux impairs dans un cadre \emph{ad hoc}. C'est la mise en place de la \emph{technique constructive} qui relie les représentations automorphes cuspidales de groupes spéciaux orthogonaux bien choisis à des représentations automorphes cuspidales du groupe linéaire (et même \emph{de groupes linéaires} au pluriel), dont l'aboutissement est le Corollaire \ref{cor_762}.
Le Théorème \ref{thm_A} joue un rôle crucial ici, car il nous permet d'expliciter la formule de multiplicité d'Arthur associée à un paramètre global {\it algébrique tempéré et de conducteur sans facteur carré}. Mentionnons que nous utilisons également de manière importante le cas particulier de la conjecture de Ramanujan démontré par Sug-Woo Shin et Ana Caraiani (cité ici comme Théorème \ref{thm_shin_caraiani}).\smallskip

Le Chapitre \ref{Lien avec des objets classiques} \emph{nourrit} cette machinerie en fournissant, à travers des formules de dimensions, les cardinaux des ensembles que le Corollaire \ref{cor_762} met en \oe{}uvre. À noter que ces formules sont\footnote{Il est abusif de noter ici $\SO_{2n+1}$ sans dire de quel groupe il s'agit. Ce sont en fait exactement les groupes présentés au §\ref{Groupes étudiés et leurs représentations} et dont il est question dans tout le Chapitre \ref{Théorie d'Arthur pour SO-2n+1}.} :
\begin{itemize}
\item des dimensions d'espaces de formes modulaires classiques dans le cas $\SO_3\simeq \PGL_2$ ;
\item des dimensions d'espaces de formes modulaires de Siegel dans le cas $\SO_5\simeq {\rm PGSp}_4$ dues à Tomoyoshi Ibukiyama et Hidetaka Kitayama \cite{Ibu-Kita} ;
\item des dimensions d'espaces de \emph{fonctions covariantes} sur des réseaux unimodulaires impairs dues à Chenevier et inédites dans les cas $\SO_7$ et $\SO_9$.
\end{itemize}
\smallskip

Le Chapitre \ref{La_formule explicite de Riemann-Weil-Mestre} développe la \emph{technique limitative} à travers la formule explicite de Riemann-Weil dans le cadre présenté par Mestre \cite{Mes}. Outre une \emph{ingénierie} déjà développée par \cite{Chen-Lannes} et \cite{Chen-Ta}, nous introduisons des outils nouveaux pour prendre en compte le conducteur, ce qui n'est pas sans rajouter une certaine technicité.
\smallskip

Enfin le Chapitre \ref{Chapitre_calculs} détaille, poids motivique par poids motivique, la façon dont les techniques limitative et constructive se rejoignent pour nous fournir les listes exhaustives des Théorèmes \ref{thm_2_w17}, \ref{thm_2_w19}, \ref{thm_p} ci-dessus.
\medskip

Un dernier commentaire s'impose pour préciser le lien entre nos deux Parties.

La technique limitative de la formule explicite fonctionne bien en petit poids motivique, c'est-à-dire qu'elle n'autorise que peu de multi-ensembles de poids comme pouvant être ceux d'une représentation automorphe cuspidale algébrique\footnote{Les multi-ensembles de poids \emph{autorisés} tendent d'ailleurs à être de \emph{vrais} ensembles, \ie les poids sont distincts et les représentations correspondantes sont régulières.}. Plus précisément, et puisqu'une représentation automorphe cuspidale \emph{centrée} et sa contragrédiente ont les mêmes poids, la formule explicite tend (en petit poids motivique toujours) à contraindre les putatives représentations à être autoduales : il n'y a \og pas de place \fg pour deux représentations distinctes, la représentation et sa contragrédiente doivent être les mêmes.


Ainsi, toutes les représentations que nous saurons classifier sont autoduales régulières (voir d'ailleurs les énoncés de nos Théorèmes) symplectiques, ce qui explique le rôle joué par $\SO_{2n+1}$ dans notre travail.

\bigskip\bigskip
Ce travail a été réalisé sous la direction attentive, exigeante et engagée de Gaëtan Chenevier. Je ne saurais trop le remercier de son enthousiasme, de sa disponibilité et c'est un grand honneur pour moi que d'avoir pu bénéficier de son encadrement. Je souhaite également remercier Colette Mœglin de m'avoir mis sur la piste de la démonstration de la (cruciale) Proposition \ref{ser_disc_>4_pas_d'inv}. Anne-Marie Aubert et Ralf Schmidt ont accepté de relire ce travail en détail, je leur sais gré du temps qu'ils y ont consacré et de leurs remarques d'amélioration.




\part{Étude locale}
\chapter{Représentations}\label{Chapitre_Représentations}
\section{Généralités}\label{Généralités}
Dans tout ce paragraphe, $G$ désignera un groupe quelconque et $V$ un $\C$-espace vectoriel. La donnée d'une action de $G$ sur $V$ par automorphismes linéaires est la même que celle d'un morphisme $\pi : G \rightarrow \GL (V)$, on dit alors que $(\pi, V)$ est une représentation linéaire (complexe) du groupe $G$. On peut aussi définir l'algèbre de groupe $\C[G]$ et munir $V$, via $\pi$, d'une structure de $\C[G]$-module (à gauche).

Un sous-espace de $V$ stable par l'action de $G$ définit donc une \emph{sous-représen\-tation} de $V$ et on dit qu'une représentation est \emph{irréductible} si elle n'admet pas d'autres sous-représentations qu'elle-même et la représentation nulle. Il est équivalent de parler de $\C[G]$-module simple.

Étant donnée une famille $(V_i)_{i\in I}$ de représentations de $G$, on peut construire naturellement la représentation somme directe $\bigoplus_{i\in I} V_i$.

Il est faux de dire en général que toute représentation se décompose en somme directe de sous-représentations irréductibles ; une représentation qui vérifie cette propriété sera dite \emph{semi-simple}. De manière équivalente, une représentation est semi-simple si tout sous-espace stable (\ie toute sous-représentation) admet un supplémentaire stable. Cette équivalence provient d'un résultat général concernant les modules semi-simples définis comme somme de modules simples (voir le chapitre XVII de \cite{Lang}). Il existe des groupes dont toutes les représentations sont semi-simples (les groupes finis et plus généralement les groupes compacts, sous l'hypothèse supplémentaire de continuité des représentations).\ps 

Une représentation $(\pi,V)$ est dite de longueur finie, si elle est de longueur finie comme $\C[G]$-module, \ie s'il existe une suite finie de sous-modules (sous-représentations) 
\[
\{0\}=V_0 \subset V_1 \subset \cdots \subset V_{n-1} \subset V_n=V
\]
avec $V_i / V_{i-1}$ $\C[G]$-module simple (\ie $G$-représentation irréductible). On définit alors une \emph{semi-simplifiée} de $V$, notée $V^\mathrm{ss}$ par 
\[
V^\mathrm{ss}=\bigoplus_{i=1}^n V_i/V_{i-1}
\]
qui est bien semi-simple par construction. Le théorème de Jordan-Hölder (voir par exemple \cite{Lang}, III, §8) affirme alors que la classe d'isomorphie de $V^\mathrm{ss}$ ne dépend pas du choix des $V_i$. 
\emph{A fortiori}, l'entier $n$ est bien défini et est appelé la longueur de $V$. Une représentation est semi-simple si, et seulement si elle est isomorphe à sa semi-simplifiée et, dans ce cas, la longueur est le nombre de termes irréductibles qui interviennent dans son écriture comme somme directe (si ce nombre est fini).\ps 

Un morphisme de $(\pi_1,V_1)$ vers $(\pi_2,V_2)$, représentations de $G$, est une application linéaire $\varphi: V_1 \rightarrow V_2$ qui commute à l'action de $G$ \ie telle que $\varphi \circ \pi_1(g)=\pi_2(g) \circ \varphi$ pour tout $g$ dans $G$. L'ensemble de tels morphismes est noté $\Hom_G(V_1,V_2)$ et on parle de $G$-morphismes, ou d'opérateurs d'entrelacement. On peut voir cet espace comme l'ensemble des éléments de $\Hom_\C (V_1,V_2)$ fixes sous l'action de $G$ donnée par $g\cdot \varphi : v_1 \mapsto \pi_2(g)^{-1}(\varphi (\pi_1(g)(v_1)))$. 

En particulier, si on prend $V_2=\C$ muni de la représentation triviale, on définit ainsi la représentation duale $(\pi_1^*,V_1^*)$ de $(\pi_1,V_1)$.


Soit $(\pi,V)$ une représentation de $G$ et soit $(\rho, W)$ une représentation irréductible de $G$ (ou plus exactement une classe d'équivalence pour les opérateurs d'entrelacement bijectifs de représentation irréductible). On définit la composante $\rho$-isotypique de $V$, notée $V^\rho$ comme étant la somme de tous les sous-espaces stables (irréductibles) de $V$ de classe $\rho$. Alors $\Hom_G (W, V) \otimes_\C W$ s'injecte (comme $G$-représentation) dans $V^\rho$ et lui est isomorphe si $\Hom_G (W, V)$ est de dimension finie. Un cas particulier est celui où $\rho$ est la représentation triviale : la composante isotypique associée correspond aux points de $V$ fixes sous l'action de $G$, ce que l'on note plus classiquement $V^G$.

Tout $G$-morphisme préserve les composantes isotypiques. Plus précisément, étant donnés deux représentations $(\pi_1,V_1)$ et $(\pi_2,V_2)$ de $G$, un opérateur d'entrelacement $\varphi: V_1 \rightarrow V_2$, et une représentation irréductible $\rho$ de $G$, on a :
\[
\varphi(V_1^\rho) \subset V_2^\rho.
\] 

Enfin, une représentation semi-simple est somme (canonique) de ses composantes isotypiques.

\section{Groupes localement profinis}\label{Groupes localement profinis}

Soit $F$ un corps local non-archimédien de caractéristique nulle (\ie une extension finie d'un corps $p$-adique). On note $\Ok$ son anneau des entiers et $\pk$ son idéal maximal. On choisit une uniformisante $\varpi$ telle que $\pk=\varpi \Ok$ et la valeur absolue $|\cdot |_F : F \rightarrow \R_+$ est normalisée par $|\varpi|_F=(\#\Ok/\pk)^{-1}$, le corps résiduel $k=\Ok/\pk$ étant bien un corps fini, dont on notera $p$ la caractéristique.


\begin{defi}
Un espace topologique est dit \emph{localement profini} s'il est séparé et si tout point admet une base de voisinages formée par des compacts ouverts.
\end{defi}
\begin{prop}
Un espace topologique est localement profini si, et seulement s'il est séparé, localement compact, totalement discontinu.
\end{prop}
En pratique, on considérera des \emph{groupes} localement profinis, \ie des groupes topologiques dont l'espace sous-jacent est un espace localement profini. Il suffit alors de vérifier que l'élément identité du groupe admet une base de voisinages par des sous-groupes ouverts compacts (et cette condition est nécessaire par le théorème de van Dantzig, voir \cite{HewRoss} 7.7). 
Les sous-groupes fermés (en particulier le centre) et les quotients (par un sous-groupe fermé distingué) d'un groupe localement profini sont des groupes localement profinis.

Le corps $F$ est localement profini ainsi que tous les espaces vectoriels de dimension finie sur $F$. En particulier, pour $n\geq 1$, l'espace des matrices carrées de taille $n$ puis l'ouvert des matrices inversibles $\GL_n(F)$ sont localement profinis.
Ainsi, si $\mathbf{G}$ est un groupe algébrique affine défini sur $F$, on dispose d'une immersion fermée $\mathbf{G} \rightarrow \GL_n$ qui permet de munir $\mathbf{G}(F)$ de la topologie trace de $\GL_n(F)$ et donc d'une structure de groupe localement profini. La topologie ainsi définie ne dépend pas du choix de l'immersion fermée.


On peut maintenant définir la notion de représentation lisse.

\begin{defi}
Soit $G$ un groupe localement profini et soit $(\pi, V)$ une représentation linéaire de $G$ où $V$ est un $\C$-espace vectoriel. On dit que $(\pi, V)$ est \emph{lisse} si pour tout $v \in V$, il existe un sous-groupe ouvert $K \subset G$ (que l'on peut supposer compact) tel que pour tout $ k \in K, \, \pi(k) v=v$ (i.e. $v$ est $K$-invariant).

De façon équivalente, en notant $V ^K$ l'ensemble des vecteurs de $V$ fixés par $K$, $(\pi, V)$ est lisse si 
$$
V= \bigcup_K V^K
$$
où $K$ parcourt l'ensemble des \sgocs.
\end{defi}

\emph{Remarque :} C'est bien une condition de \emph{lissité} puisqu'elle revient à voir que, pour tout $v \in V$, la fonction $g \mapsto \pi(g) v$ est localement constante (\ie continue si l'on munit $V$ de la topologie discrète).\medskip

On voit immédiatement que toute sous-représentation (resp. tout quotient) d'une représentation lisse est lisse.

On se donne pour toute la suite de cette section un groupe localement profini $G$. La catégorie des représentations lisses de $G$, notée ici $\mathrm{Rep} (G)$ est une sous-catégorie pleine de la catégorie des représentations complexes, c'est-à-dire qu'un morphisme entre représentations lisses est simplement une application linéaire qui commute à l'action de $G$.

\begin{defi}\label{defi_Irr_G}
On note $\mathrm{Irr} (G) $ l'ensemble des classes d'équivalence (pour les opérateurs d'entrelacement bijectifs) de représentations \emph{lisses} irréductibles de $G$.
\end{defi}

Si $(\pi, V)$ est une représentation lisse de $G$, on peut, pour chaque sous-groupe compact ouvert $K$, considérer par restriction $(\pi_{|K},V)$ représentation de $K$, qui est alors semi-simple (par lissité et par compacité de $K$) : on dit que $V$ est $K$-semi-simple.

\begin{lemme}\label{5.2.1 de DeB}
Soient $V_1, V_2$ et $V_3$ trois représentations lisses de $G$. Alors la suite 
\begin{equation*}
\begin{tikzcd}
0 \arrow{r} & V_1 \arrow{r} & V_2 \arrow{r} & V_3 \arrow{r} & 0
\end{tikzcd}
\end{equation*}
est exacte si, et seulement si, pour tout sous-groupe compact ouvert $K$ de $G$, la suite
\begin{equation*}
\begin{tikzcd}
0 \arrow{r} & V_1^K \arrow{r} & V_2^K \arrow{r} & V_3^K \arrow{r} & 0
\end{tikzcd}
\end{equation*}
est exacte.
\end{lemme}
\begin{proof}
C'est le Lemme 5.2.1 de \cite{DeB}.
\end{proof}


\begin{defi}\label{defi_rep_lisse_admissible}
Une représentation lisse $(\pi, V)$ de $G$ est dite \emph{admissible} si pour tout sous-groupe compact ouvert $K$, l'espace $V^K$ est de dimension finie.
\end{defi}

On a la caractérisation équivalente suivante :
\begin{prop}\label{323}
Une représentation lisse $(\pi, V)$ est admissible si, et seulement si pour \emph{un} sous-groupe compact ouvert $K_0$, on a :
$$
\forall \rho \in \mathrm{Irr}(K_0), \, \dim V^\rho < +\infty
$$
où $V^\rho$ désigne la composante $\rho$-isotypique.

Cette propriété est alors vérifiée pour \emph{tout} sous-groupe compact ouvert K.
\end{prop}

En particulier, le Lemme \ref{5.2.1 de DeB} nous dit que toute sous-représentation (resp. tout quotient) d'une représentation admissible est admissible.


Les groupes localement profinis que l'on étudie ($F$-points d'un groupe algébrique affine défini sur $F$) possèdent une propriété topologique supplémentaire qui s'avère déterminante : ils sont \emph{séparables}, à savoir qu'ils admettent une base dénombrable d'ouverts. En particulier, l'élément identité possède une base dénombrable de voisinages (par des sous-groupes compacts ouverts si l'on veut). On a alors, pour tout sous-groupe compact ouvert $K$ le fait que l'espace quotient $G/K$ est dénombrable (\cite{Ren}, section II.3.2). Cette propriété implique le lemme technique suivant et sa fameuse conséquence.

\begin{lemme}
Soit $(\pi,V)$ une représentation lisse \emph{irréductible} de $G$. Alors $\dim_\C V$ est au plus dénombrable.
\end{lemme}

\begin{prop}\emph{Lemme de Schur}

Soit $(\pi,V)$ une représentation lisse \emph{irréductible} de $G$. Alors tout opérateur d'entrelacement de $(\pi,V)$ avec elle-même est une homothétie.
\end{prop}

\begin{proof}
(de la Proposition, suivant \cite{BH} p. 21)

Soit $\phi \in (\mathrm{End}_{G}(V)) \setminus \{0\}$. Alors $\Ker(\phi)$ et $\Imm(\phi)$ sont des sous-espaces $G$-invariants de $V$. L'irréductibilité impose $\Ker(\phi)=\{0\}$ et $\Imm(\phi)=V$, autrement dit, tous les éléments non nuls de $\mathrm{End}_{G}(V)$ sont inversibles, \emph{i.e.} $\mathrm{End}_{G}(V)$ est une $\C$-algèbre à division.

Soit $v_0 \in V \setminus \{0\}$, alors par irréductibilité de $V$, on a $V=\sum_g \C \, g \cdot v_0$, ainsi un élément $\phi \in \mathrm{End}_{G}(V)$ est déterminé uniquement par la valeur $\phi(v_0)$. L'espace $\mathrm{End}_{G}(V)$  est donc de dimension dénombrable.
Supposons maintenant qu'il existe un $\phi \in \mathrm{End}_{G}(V)$ qui ne soit pas une homothétie (\emph{i.e.} $\phi \notin \C$), alors $\phi$ est transcendant sur $\C$ et engendre un corps $\C(\phi) \subset \mathrm{End}_{G}(V)$. La famille $\{ (\phi -a)^{-1} | a \in \C \}$ est libre, si bien que la dimension de $\C(\phi)$ est indénombrable, ce qui fournit une contradiction. Finalement $\mathrm{End}_{G}(V)=\C$.
\end{proof}

Donnons tout de suite une conséquence importante du lemme de Schur.
\begin{cor}\label{Caractère central pour représentation locale}
Soit $(\pi,V)$ une représentation lisse irréductible de $G$. Alors le centre $Z(G)$ de $G$ agit sur $V$ via un caractère $\omega_\pi : Z(G) \rightarrow \C^\times$, dit \emph{caractère central}.
\end{cor}
\begin{proof}
Le lemme de Schur nous donne directement l'existence d'un morphisme $\omega_\pi : Z(G) \rightarrow \C^\times$ tel que $z \cdot v=\omega_\pi(z)v$ pour $z \in Z(G)$. Si $K$ est un sous-groupe compact ouvert tel que $V^K \neq 0$ alors $\omega_\pi$ est trivial sur le sous-groupe compact ouvert $K \cap Z(G)$ de $Z(G)$, ainsi $\omega_\pi$ est bien continu.
\end{proof}

Il est faux en général que le dual d'une représentation lisse est lisse, c'est pourquoi nous devons définir la notion de \emph{dual lisse}.

Soit $(\pi, V)$ une représentation lisse de $G$, on a défini au paragraphe \ref{Généralités} sa duale (abstraite) $(\pi^*,V^*)$ 
et on peut définir le \og crochet de dualité \fg{} :
\begin{align*}
V^* \times V &\longrightarrow \mathbb{C} \\
(v^*,v) &\longmapsto \langle v^*,v \rangle
\end{align*}
correspondant à l'évaluation de la forme linéaire $v^*$ en le vecteur $v$.
La représentation $\pi^*$ de $G$ est donc donnée par :
\[
\langle \pi^*(g) v^*, v \rangle=\langle v^* , \pi(g^{-1}) v \rangle
\]

On prend alors le \emph{lissé} de cette représentation, \emph{i.e.} on pose :
\[
V^\vee=\bigcup_K (V^*)^K
\]
où $K$ parcourt les sous-groupes compacts ouverts de $G$. C'est un sous-espace $G$-stable de $V^*$, et en notant $\pi^\vee$ la restriction de $\pi^*$ à $V^\vee$, on obtient la représentation lisse $(\pi^\vee, V^\vee)$ dite \emph{contragrédiente} ou dual lisse de la représentation de départ. Par ailleurs, on a $(V^\vee)^K \simeq (V^K)^*$.
\begin{prop}\label{1.2.10}
Le foncteur contravariant $V \mapsto V^\vee$  de la catégorie des représentations lisses dans elle-même est exact.\ps 

Soit $(\pi, V)$ une représentation lisse de $G$. Alors
\begin{enumerate}
\item $\pi$ est admissible si, et seulement si $\pi^\vee$ est admissible, et dans ce cas, l'application canonique $V \rightarrow (V^\vee)^\vee$ est un isomorphisme de $G$-modules ;
\item $\pi$ est irréductible si, et seulement si $\pi^\vee$ est irréductible.
\end{enumerate}
\end{prop}

\begin{proof}
On renvoie au Lemme 5.1.7 et au paragraphe 5.2. de \cite{DeB}. Le dernier point est facile si l'on suppose que $\pi$ est admissible, mais dans le cas général, il utilise de façon cruciale l'admissibilité des représentations supercuspidales (et la classification à partir de celles-ci) que nous évoquons plus bas (Proposition \ref{admissibilité_sc}).
\end{proof}

\subsection{Mesure de Haar et fonction modulaire}\label{Fonction modulaire}

Les groupes localement profinis sont localement compacts, ils admettent donc une mesure de Haar à gauche (positive et non nulle par définition), unique à multiplication par un scalaire près (\cite{Bbki_Int}, Chapitre VII §1). Ce résultat peut se démontrer dans le cas spécifique des groupes localement profinis avec une preuve plus simple que celle du cas général (voir \cite{Ren} II.3.6).

Pour un groupe localement profini $G$, on notera $\mu_G$ \og la \fg{} mesure de Haar à gauche (il y a un choix de normalisation à faire) qui nous permet de définir la distribution 
\[
\phi \longmapsto \int_G \phi(g) \dd \mu_G(g)
\]
où $\phi$ est localement constante à support compact, si bien que l'intégrale est en réalité une somme finie.

L'invariance à gauche de la mesure de Haar nous dit que 
\begin{equation*}
\int_G \phi(hg) \dd \mu_G(g)= \int_G \phi(g) \dd \mu_G(h^{-1}g)= \int_G \phi(g) \dd \mu_G(g).
\end{equation*}

Étant donné un élément $g \in G$, on peut définir la mesure (translatée à droite par $g$) $r(g)\cdot \mu_G : X \mapsto \mu_G(Xg)$ où $X$ est un borélien de $G$. Cette mesure n'a pas de raison \emph{a priori} d'être égale à la mesure $\mu_G$ mais elle est invariante par translations à gauche et positive, donc c'est une mesure de Haar (à gauche) et donc un multiple scalaire de la mesure de départ. Il existe donc un réel strictement positif $\delta_G(g)$ tel que $r(g)\cdot \mu_G= \delta_G(g) \mu_G$.\ps 

On a donc défini la \emph{fonction modulaire} $\delta_G : G \rightarrow \R_+^*$ dont on voit facilement qu'elle est en fait un morphisme de groupes continu.
Si la fonction modulaire est constamment égale à 1 (par exemple si $G$ est abélien, ou à l'opposé si $G$ est égal à son groupe dérivé fermé), on parle de groupe  \emph{unimodulaire}. Dans ce cas, toute mesure de Haar à gauche est aussi une mesure de Haar à droite.


\begin{prop}
\begin{itemize}
\item Un groupe compact est unimodulaire.
\item Si un groupe est réunion de ses sous-groupes compacts, alors il est unimodulaire.
\item Le groupe des $F$-points d'un groupe algébrique défini et réductif sur $F$ est unimodulaire.
\end{itemize}
\end{prop}
\begin{proof}
On utilise la continuité de $\delta_G$ et le fait que le seul sous-groupe compact de $\R_+^*$ est $\{1\}$.

Pour le troisième point, on renvoie à \cite{Ren} V.5.4.
\end{proof}

Du point de vue des intégrales, si $\phi$ est localement constante à support compact, on a :
\begin{equation*}
\int_G \phi(gh) \dd \mu_G(g)= \int_G \phi(g) \dd \mu_G(gh^{-1})= \delta_G(h)^{-1} \int_G \phi(g) \dd \mu_G(g).
\end{equation*}

%

\section{Coefficients matriciels et classes de représentations}
On se donne pour toute cette section un groupe localement profini séparable $G$ et on s'intéresse aux éléments de $\mathrm{Rep}(G)$, à savoir les représentations complexes lisses. Puisque c'est le seul cas que nous aurons à traiter en pratique, nous supposons dans cette section que $G$ est unimodulaire. En particulier, le sous-groupe fermé $Z(G)$ l'est également et le groupe quotient $G/Z(G)$ aussi.\medskip

Soit $(\pi,V) \in \mathrm{Rep}(G)$ et soit $(\pi^\vee,V^\vee)$ sa contragrédiente. Étant donnés $v \in V$ (un vecteur) et $\lambda \in V^\vee$ (une forme linéaire lisse), on peut définir le \emph{coefficient matriciel} 
\[
\begin{array}{ccccc}
m_{\lambda,v} & : & G &\longrightarrow & \mathbb{C} \\
 & & g & \longmapsto & \lambda(\pi(g)v) \\
\end{array}.
\]

On a, avec le crochet de dualité précédemment introduit, $m_{\lambda,v}(g)=\linebreak\langle \lambda , \pi(g) v \rangle=\langle \pi^\vee (g^{-1}) \lambda, v \rangle$. En particulier, par lissité de $\pi$ et de $\pi^\vee$, on trouve un sous-groupe compact ouvert $K$ tel que $m_{\lambda,v}$ est bi-$K$-invariante.

Il y a autant de coefficients matriciels qu'il y a de vecteurs dans $V$ et de formes linéaires dans $V^\vee$ (à multiplication par un scalaire près). Néanmoins on constate que, pour certaines propriétés, les coefficients matriciels ont un comportement uniforme.\ps 


Soit $(\pi,V)$ une représentation lisse, admettant un caractère central $\omega_\pi$ unitaire (\ie à valeurs dans l'ensemble des nombres complexes de module 1). Soient $v \in V$ et $\lambda \in V^\vee$ et soit $m_{\lambda,v}$ le coefficient matriciel correspondant. Puisqu'il y a un caractère central et que celui-ci est unitaire, la quantité $|m_{\lambda,v}(g)|$ est bien définie sur $G/Z(G)$ (on note abusivement $g$ la classe d'un élément de $G$ dans $G/Z(G)$) et on peut alors dire que le coefficient matriciel est \emph{de carré intégrable modulo le centre} si :
\[
\int_{G/Z(G)} |m_{\lambda,v}(g)|^2 \dd g^* < \infty,
\]
où $\dd g^*$ désigne une mesure de Haar sur $G/Z(G)$.

Cela revient à dire que $m_{\lambda,v}$ est un élément de $\mathrm{L}^2(G,\omega_\pi,\dd g^*)$ l'espace des fonctions lisses de $G$ dans $\C$ pour lesquelles le centre agit par le caractère $\omega_\pi$ et qui sont de carré intégrable sur $G/Z(G)$ pour la mesure $\dd g^*$. Or cet espace, muni de l'action à droite par translations, est une représentation de $G$ : la représentation régulière (c'est en fait une sous-représentation de l'ensemble des fonctions lisses pour la même action). Elle est unitaire pour le produit scalaire ($G$-invariant)
\[
\langle f_1, f_2 \rangle_{\mathrm{L}^2(G,\omega_\pi,\dd g^*)}=\int_{G/Z(G)} \overline{f_1(g)}f_2(g) \dd g^*.
\]

\begin{prop}\label{coeff_mat_L2}
Soit $(\pi,V)$ une représentation lisse, irréductible, dont le caractère central (qui existe d'après le Corollaire \ref{Caractère central pour représentation locale}) est unitaire. Alors si un coefficient matriciel (non nul) est de carré intégrable modulo le centre, tous les coefficients matriciels le sont.
\end{prop}
\begin{proof}
Soient $\lambda_0 \in V^\vee$ et $v_0 \in V$ tel que le coefficient matriciel $m_{\lambda_0, v_0}$ soit non nul et de carré intégrable. En particulier $\lambda_0$ et $v_0$ sont non nuls. Considérons l'ensemble $A$ des $v \in V$ tels que le coefficient matriciel $m_{\lambda_0,v}$ soit de carré intégrable. On voit aisément que $A$ est un sous-espace vectoriel de $V$, non réduit à $\{0\}$ car il contient $v_0$. Pour montrer que $A=V$, il suffit donc, par irréductibilité de $V$, de montrer que $A$ est $G$-stable.

Or, pour $h \in G$, $m_{\lambda_0, h \cdot v}(g)=\lambda_0(g\cdot (h \cdot v))=\lambda_0((gh) \cdot v)=m_{\lambda_0,v}(gh)$. Si $v \in A$, alors (par unimodularité de $G/Z(G)$) $h \cdot v \in A$ également. Les coefficients matriciels $m_{\lambda_0,v}$ sont donc de carré intégrable, quel que soit $v \in V$ et le même raisonnement montre que cela reste vrai en remplaçant $\lambda_0$ par n'importe quelle forme linéaire dans $V^\vee$.
\end{proof}

\begin{defi}\label{déf_série_discrète}
Soit $(\pi,V)$ une représentation lisse irréductible de $G$. On dit que $\pi$ est une \emph{série discrète} ou \emph{de carré intégrable modulo le centre} si son caractère central $\omega_\pi$ est unitaire et si un coefficient matriciel de $\pi$ (et donc chacun d'eux) est de carré intégrable modulo le centre.\ps 

On dit que que $\pi$ est une série \emph{essentiellement} discrète ou \emph{essentiellement} de carré intégrable modulo le centre s'il existe un caractère lisse $\chi$ tel que $\chi \otimes \pi$ soit une série discrète.
\end{defi}

Une situation où la condition d'intégrabilité est naturellement remplie est quand le coefficient matriciel, vu comme fonction sur $G$ est à support compact modulo le centre, \ie quand l'image du support dans $G/Z(G)$ est compacte.

\begin{defi}\label{défi_sc}
Soit $(\pi,V)$ une représentation lisse irréductible de $G$. On dit que $\pi$ est \emph{supercuspidale} si tous ses coefficients matriciels sont à support compact modulo le centre.
\end{defi}

\begin{prop}\label{admissibilité_sc}
\begin{enumerate}
\item Une représentation supercuspidale est admissible.
\item Si $(\pi,V)$ est supposée lisse irréductible et qu'\emph{un} coefficient matriciel (non nul) est à support compact modulo le centre, alors tous le sont et $(\pi,V)$ est supercuspidale.
\end{enumerate}
\end{prop}
\begin{proof} (suivant \cite{BH} p. 71)

Raisonnons par l'absurde en supposant que $(\pi,V)$ n'est pas admissible. On trouve alors un sous-groupe compact ouvert $K$ tel que $V^K$ soit de dimension infinie, et en fait de dimension infinie dénombrable à cause de l'hypothèse de séparabilité. Par conséquent $(V^K)^*=\Hom_\C(V^K,\C)$, qui est isomorphe à $(V^\vee)^K$, est de dimension infinie non dénombrable.

Soit $v \in V^K$ non nul et soit l'application
\[
\begin{array}{ccccc}
\Gamma_v & : & (V^\vee)^K &\longrightarrow & C^\infty(G) \\
 & & \lambda & \longmapsto & m_{\lambda,v} \\
\end{array}
\]
où $C^\infty(G)$ désigne l'ensemble des fonctions lisses de $G$ dans $\C$. L'application $\Gamma_v$ est $G$-équivariante (en munissant $C^\infty(G)$ de l'action par translations à droite).

La représentation $V$, irréductible, est $G$-engendrée par $v$, ce qui nous assure que $\Gamma_v$ est injective. Son image est composée de fonctions $f$ vérifiant :
\[
f(zkgk')=\omega_\pi(z)f(g), \; \forall g\in G, \forall z \in Z, \forall k,k' \in K.
\]
Ainsi, sur une double classe $ZKgK$, $f$ est soit identiquement nulle, soit ne s'annule jamais. Or on a la décomposition $G=\coprod _i Kg_iK$ où l'ensemble d'indices $I$ est dénombrable (c'est la séparabilité de $G$), de laquelle on déduit $Z\backslash G=\bigcup_i Z\backslash ZKg_iK$ (l'union n'est plus disjointe). En tous les cas, cette décomposition liée à la compacité modulo le centre du support des coefficients matriciels nous dit que chaque élément de l'image de $\Gamma_v$ est à support dans un nombre fini de doubles classes $ZKg_iK$. L'image de $\Gamma_v$ est donc un sous-espace de l'espace engendré par les fonctions caractéristiques de ces doubles classes, et est en particulier de dimension au plus dénombrable, ce qui fournit une contradiction.

La démonstration du deuxième point est similaire à celle de la Proposition \ref{coeff_mat_L2}, utilisant l'irréductibilité de $V$.
\end{proof}

\emph{Remarque 1 :} Une représentation supercuspidale est ici toujours supposée irréductible (point sur lequel la littérature diverge).\medskip

\emph{Remarque 2 :} Nous verrons plus loin (§ \ref{Lien avec les représentations supercuspidales}) une caractérisation très différente (mais équivalente) de la supercuspidalité. Certains auteurs parlent à propos de la propriété que nous venons d'introduire de $\gamma$-cuspidalité.\medskip

On déduit immédiatement qu'une représentation supercuspidale est essentiellement série discrète : il suffit de s'occuper du caractère central, il existe, et en tordant la représentation par un caractère convenable, on peut supposer qu'il est unitaire.\ps 

Si, inversement, on relâche légèrement la condition d'intégrabilité, on obtient les représentations tempérées.

\begin{defi}\label{déf_rep_tempérée}
Soit $(\pi,V)$ une représentation lisse irréductible de $G$. On dit que $\pi$ est \emph{tempérée} si son caractère central $\omega_\pi$ est unitaire et si les modules des coefficients matriciels de $\pi$ appartiennent à ${\rm L}^{2+\eps}(G/Z(G))$ pour tout $\eps >0$ (comme précédemment, il suffit en fait de vérifier cette propriété pour \emph{un} coefficient matriciel non nul).

On dit que $\pi$ est \emph{essentiellement} tempérée s'il existe un caractère lisse $\chi$ tel que $\chi \otimes \pi$ soit tempérée.
\end{defi}

\begin{prop}
Soit $(\pi, V)$ une série discrète de $G$. Alors $(\pi, V)$ est unitarisable et ses coefficients matriciels sont bornés.
\end{prop}
La proposition implique que les coefficients matriciels d'une série discrète, qui sont dans $\mathrm{L}^{2}(G/Z(G))$ par hypothèse, sont \emph{a fortiori} dans $\mathrm{L}^{2+\eps}(G/Z(G))$ pour tout $\eps >0$. Donc une série discrète (resp. essentiellement discrète) est en particulier une représentation tempérée (resp. essentiellement tempérée).

\begin{proof}
Soit $\lambda \in V^\vee$ tel qu'on trouve un $v_0 \in V$ pour lequel le coefficient matriciel $m_{\lambda, v_0}$ soit non nul. On pose alors :
\[
(v_1|v_2)_V=\int_{G/Z(G)} \overline{m_{\lambda,v_1}(g)}m_{\lambda,v_2}(g) \mathrm{d}g^*,
\]
forme sesquilinéaire hermitienne positive $G$-invariante. Vérifions la définition : soit donc $w \in V$ tel que $(w|w)_V =0$, \ie $|m_{\lambda,w}|$ est nul dans $\mathrm{L}^{2}(G/Z(G))$. Alors, comme le centre agit par un caractère et que le coefficient matriciel est une application continue (lisse), $m_{\lambda,w}$ est en fait nul comme fonction de $G$ dans $\C$.
Si l'on fait agir $G$ sur $w$ alors on obtient que pour tout $h$ dans $G$, le coefficient matriciel $m_{\lambda,h\cdot w}$ est encore nul. Or $V$ est supposée irréductible et donc, si $w$ est non nul, les translatés de $w$ sous $G$ engendrent $V$, si bien que c'est pour tout $v$ dans $V$ que le coefficient matriciel $m_{\lambda,v}$ est nul. Ceci contredit l'existence du $v_0$ dans le début de la preuve, et donc on a en fait $w=0$, assurant que la forme hermitienne est définie et qu'il s'agit d'un produit scalaire hermitien.


La représentation $(\pi,V)$ est alors unitaire pour ce produit scalaire hermitien et l'injection :
\[
\begin{array}{ccccc}
\varphi & : & V &\longhookrightarrow & V^* \\
 & & v & \longmapsto & (\cdot|v) \\
\end{array}
\]
est en fait une bijection entre $V$ et $V^\vee$ (les vecteurs lisses sont envoyés sur des vecteurs lisses).

Considérons maintenant $\lambda \in V^\vee$ et $v \in V$. Alors $\lambda$ est \og représentable \fg{} par un vecteur $x_\lambda$ de $v$ (c'est la surjectivité de $\varphi$), si bien que $m_{\lambda,v}(g)=\lambda(g\cdot v)=(x_\lambda|g\cdot v)$ et que le coefficient matriciel est donc borné en module par $||x_\lambda|| \, ||g\cdot v||$, c'est-à-dire (par unitarité) par $||x_\lambda|| \, ||v||$.
\end{proof}


\begin{cor}
Les représentations tempérées irréductibles sont unitarisables.
\end{cor}

\begin{proof}
On fait intervenir l'induction parabolique définie au paragraphe suivant.
On utilise alors le fait que l'induction parabolique normalisée préserve l'unitarité (Proposition \ref{prop_induction_normalisée}) et la caractérisation des représentations tempérées comme sous-quotients (sous-représentations si l'on veut) des induites de série discrète (\cite{waldspurger_2003} Proposition III.4.1).
\end{proof}

\section{Notations pour les groupes réductifs}\label{Notations pour les groupes reductifs}
Soit $\mathbf{G}$ un groupe algébrique défini et réductif sur le corps $F$. Par hypothèse, les groupes réductifs sont ici supposés connexes. 

Il sera parfois utile de fixer un sous-groupe (algébrique) parabolique minimal $\bm{\mathrm{P}_0}$ et de considérer seulement les sous-groupes (algébriques) paraboliques $\mathbf{P}$ le contenant, on parlera de sous-groupes paraboliques \emph{standard}. Ce choix, non canonique \emph{a priori} ne porte pas à conséquence, puisque tous les sous-groupes paraboliques minimaux sont conjugués sous $\mathbf{G}(F)$. Cela permettra néanmoins de donner des énoncés plus facilement manipulables.\ps 

En pratique, nous considérerons des groupes $\mathbf{G}$ déployés sur $F$ pour lesquels le sous-groupe parabolique minimal sera un sous-groupe de Borel, noté $\mathbf{B}$. Ayant fixé un sous-groupe parabolique minimal, nous fixons également (en tant que de besoin) un tore maximal $\mathbf{T}$, contenu dans $\mathbf{B}$ -- tous les tores maximaux contenus dans $\mathbf{B}$ sont conjugués sous $\mathbf{B}(F)$.

Si $\mathbf{P}$ est un sous-groupe parabolique, il admet une factorisation du type  $\mathbf{P}=\mathbf{M}\mathbf{N}$ où $\mathbf{N}$ est \emph{le} radical unipotent de $\mathbf{P}$ et $\mathbf{M}$ \emph{un} sous-groupe de Levi. Si l'on travaille dans le cadre \og standard \fg{} (et en supposant $\mathbf{G}$ déployé), alors on peut imposer que le sous-groupe de Levi contienne le tore maximal $\mathbf{T}$ (on parle de sous-groupe de Levi \emph{standard}), si bien que $\mathbf{M}$ est alors canoniquement associé à $\mathbf{P}$.



Ces notations servent pour toute la suite du chapitre et pour le chapitre suivant, en particulier nous noterons en maigre italique les $F$-points des groupes algébriques considérés. Nous commettrons par ailleurs l'abus de langage consistant à dire que $B=\mathbf{B}(F)$ est un sous-groupe de Borel de $G=\mathbf{G}(F)$, de même pour tous les sous-groupes paraboliques et les sous-groupes de Levi (en particulier le tore maximal).

\section{Induction parabolique et foncteur de Jacquet}
%

\subsection{Induction parabolique}
Soit donc $P=MN$ un sous-groupe parabolique de $G$ et soit $(\sigma, W)$ une représentation lisse de $M$. On peut étendre cette représentation en une représentation de $P$ (en imposant que l'action de $N$ soit triviale) puis construire l'espace :
\[
\Ind_P^G \sigma=\left\lbrace f \in {\rm C}^\infty (G,W) \middle| f(pg)= \sigma(p) f(g), \forall p \in P, \forall g \in G \right\rbrace,
\]
où ${\rm C}^\infty (G,W)$ désigne l'ensemble des fonctions lisses de $G$ dans $W$, \ie invariantes par translation pour un certain sous-groupe ouvert (compact si l'on veut). Puisque $G/P$ est compact, il est équivalent de parler de fonctions localement constantes.
L'action de $G$ par translations à droite fait de $\Ind_P^G \sigma$ une $G$-représentation, dite représentation induite. On énonce désormais sans démonstration certaines propriétés de l'induction pour lesquelles on renvoie à \cite{Ren} III.2.\ps 

La première est que $\Ind_P^G$, ainsi défini, est un foncteur de $\mathrm{Rep}(M)$ dans $\mathrm{Rep}(G)$, exact. Par ailleurs, on a une propriété naturelle de transitivité : étant donnés $P_1=M_1N_1$ et $P_2=M_2N_2$ deux sous-groupes paraboliques avec $P_1 \subset P_2$ et $M_1 \subset M_2$, et une représentation lisse $\sigma$ de $M_1$, on a un isomorphisme naturel :
\[
\Ind_{P_1}^G \sigma \simeq \Ind_{P_2}^G(\Ind_{P_1\cap M_2}^{M_2} \sigma).
\]

De plus, si $\sigma$ est admissible, $\Ind_P^G \sigma$ l'est également (ceci est lié au fait qu'on a des sous-groupes paraboliques et que l'espace $G/P$ est compact). Enfin, si $\sigma$ est de longueur finie (comme $M$-représentation), $\Ind_P^G \sigma$ l'est également (comme $G$-représentation) -- par exactitude, il suffit de le montrer quand $\sigma$ est irréductible.\ps 



On définit également (en reprenant les notations précédentes) le foncteur d'induction \emph{normalisée} par : 
\[
\ii_P^G \sigma=\left\lbrace f \in {\rm C}^\infty (G,W) \middle| f(pg)= \delta_P(p)^{1/2}\sigma(p) f(g), \forall p \in P, \forall g \in G \right\rbrace,
\]
où $\delta_P$ désigne la fonction modulaire du groupe $P$ définie au paragraphe \ref{Fonction modulaire}. On a donc $\ii_P^G \sigma=\Ind_P^G (\sigma \otimes \delta_P^{1/2})$. Le foncteur d'induction normalisée est encore transitif (pour des rasions de compatibilité de fonctions modulaires) ; avec les notations ci-dessus, on a un isomorphisme naturel :
\begin{equation*}
\ii_{P_1}^G \sigma \simeq \ii_{P_2}^G(\ii_{P_1\cap M_2}^{M_2} \sigma).
\end{equation*}
L'intérêt de cette normalisation est que le foncteur $\ii_P^G$ a de meilleures propriétés.
\begin{prop}\label{prop_induction_normalisée}
Soit $\sigma$ une représentation lisse de $M$, alors
\begin{itemize}
\item si $\sigma$ est unitaire, $\ii_P^G \sigma$ l'est également ;
\item $(\ii_P^G \sigma)^\vee \simeq \ii_P^G (\sigma^\vee).$
\end{itemize}
\end{prop}
\begin{proof}
On renvoie aux Propositions 3.1.2, 3.1.3 et 3.1.4 de \cite{Cass-notes}.
\end{proof}

\subsection{Foncteur de Jacquet}
Donnons-nous à nouveau $P=MN$ un sous-groupe parabolique de $G$ et $(\pi,V)$ une représentation lisse de $G$. On peut alors considérer l'espace 
\[
V(N)=\Vect_\C \{\pi(n)v-v \,|\, v \in V, n \in N\}
\]
et on définit $V_N=V/V(N)$ l'espace des coïnvariants sous $N$ : c'est le plus grand quotient de $V$ sur lequel $N$ agit trivialement.
Puisque $N$ est distingué dans $P$, $V(N)$ est laissé stable par $P$ et $V_N$ est donc une représentation de $P$ sur laquelle $N$ agit trivialement, soit finalement une représentation de $P/N =M$, notée $(\pi_N,V_N)$.

On a ainsi défini un foncteur de la catégorie $\mathrm{Rep}(G)$ dans la catégorie $\mathrm{Rep}(M)$, qui est exact, on parle du \emph{foncteur de Jacquet}, que l'on note temporairement $\mathcal{F}$.

Soient $(\sigma,W)$ une représentation lisse de $M$ et $(\pi,V)$ une représentation lisse de $G$. On a les égalités :
\begin{equation*}
\Hom_G(\pi,\Ind_P^G\sigma)=\Hom_P(\mathrm{Res}_P^G \pi,\sigma)=\Hom_M(\mathcal{F}(\pi),\sigma).
\end{equation*}
La première égalité est la réciprocité de Frobenius \og classique \fg pour l'induction, où l'on voit -- de même que ci-dessus -- $\sigma$ comme une représentation de $P$ grâce à la factorisation $P=MN$ (l'action de $N$ étant triviale). La deuxième égalité vient justement de ce que l'action de $N$ est triviale : on peut donc prendre les coïnvariants et obtenir un morphisme entre $M$-représentations.\ps 

Le foncteur $\mathcal{F}$ de $\mathrm{Rep}(G)$ dans $\mathrm{Rep}(M)$ est ainsi adjoint à gauche du foncteur $\Ind_P^G$. Enfin, si $\pi$ est admissible, $\mathcal{F}(\pi)$ l'est également (\cite{Cass-notes} Theorem 3.3.1); si $\pi$ est de plus de longueur finie, alors $\mathcal{F}(\pi)$ l'est également (\cite{Cass-notes} Corollary 6.3.8).


On peut également définir un foncteur de Jacquet normalisé (adjoint à gauche de l'induction parabolique normalisée) : $\rr_P^G \pi=\mathcal{F}(\pi) \otimes \delta_P^{-1/2}$.
Ce foncteur préserve bien sûr l'admissibilité et la longueur finie et on a :

\begin{equation}
\label{réciprocité_de_Frobenius_foncteur_Jacquet}
\Hom_G(\pi,\ii_P^G\sigma)=\Hom_M(\rr_P^G \pi,\sigma).
\end{equation}

Par adjonction, le foncteur de Jacquet normalisé est aussi transitif. Étant donnés $P_1=M_1N_1$ et $P_2=M_2N_2$ deux sous-groupes paraboliques avec $P_1 \subset P_2$ et $M_1 \subset M_2$, et une représentation lisse $\pi$ de $G$, on a un isomorphisme naturel :
\begin{equation*}
\rr_{P_1}^G \pi \simeq \rr_{P_1\cap M_2}^{M_2}(\rr_{P_2}^{G} \pi).
\end{equation*}

\subsection{Lien avec les représentations supercuspidales}\label{Lien avec les représentations supercuspidales}
Nous avons défini en \ref{défi_sc} les représentations supercuspidales par les propriétés de leurs coefficients matriciels. Il existe une autre caractérisation par le fait de n'apparaître dans aucune induite parabolique et d'être en quelque sorte la brique la plus élémentaire possible en termes de représentations.
\begin{thm}
Soit $(\pi, V)$ une représentation lisse irréductible de $G$. Alors les propriétés suivantes sont équivalentes :
\begin{enumerate}
\item $(\pi,V)$ est supercuspidale (au sens de la Définition \ref{défi_sc}) ;

\item pour tout sous-groupe de Levi strict $M$ et pour toute représentation $\sigma$ de $M$, $\Hom_G(\pi,\ii_P^G \sigma)=0$ ;
\item pour tout sous-groupe parabolique strict $P$, on a $\rr_P^G \pi =0$.
\item pour tout sous-groupe de Levi \emph{standard} strict $M$ et pour toute représentation $\sigma$ de $M$, $\Hom_G(\pi,\ii_P^G \sigma)=0$ ;
\item pour tout sous-groupe parabolique \emph{standard} strict $P$, on a $\rr_P^G \pi =0$.
\end{enumerate}
\end{thm}
\begin{proof}
L'équivalence entre 2. et 4. d'une part et entre 3. et 5. d'autre part vient du fait que l'on peut modulo conjugaison se ramener à considérer des sous-groupes standard.
L'équivalence de 4. et 5. provient de \eqref{réciprocité_de_Frobenius_foncteur_Jacquet}, pour l'équivalence avec 1. on renvoie à \cite{Ren} VI.2.1.
\end{proof}


\newpage
\chapter{Paramétrisation de Langlands}\label{Chapitre_Paramétrisation_Lgl}

\section{Le groupe de Weil}
\label{Groupe de Weil}
Comme précédemment, $F$ est un corps local non-archimédien de caractéristique nulle, $k$ son corps résiduel (fini), de caractéristique $p$. Afin d'éviter toute ambiguïté, nous écrirons dans cette partie $k_F$ au lieu de $k$ quand cela paraîtra nécessaire.

Il s'agit de rappeler ici comment on définit le groupe de Weil de $F$, substitut du groupe de Galois absolu $\Gal(\overline{F}/F)$, où $\overline{F}$ désigne une clôture algébrique de $F$. Nous suivons \cite{Tate-Corvallis} et \cite{Kna}.\ps 

Soit donc $K$ une extension galoisienne finie de $F$. On désigne par $\mathcal{O}_K$ l'anneau des entiers de $K$ et par $\mathfrak{p}_K$ son unique idéal maximal. Le corps résiduel $k_K=\mathcal{O}_K/\mathfrak{p}_K$ est une extension finie de $k_F$. Un élément du groupe de Galois de $K/F$ stabilise $\Okk$ et $\pkk$ et induit donc un automorphisme du quotient $k_K$ qui fixe point par point $k_F$.
%

On obtient ainsi un morphisme $\Gal(K/F) \rightarrow \Gal(k_K/k_F)$ dont on montre qu'il est surjectif, le groupe d'arrivée étant cyclique, engendré par le morphisme de Frobenius $x \mapsto x^{q_F}$ où $q_F=\# k_F$.

En prenant la limite projective sur $K$, on obtient un morphisme surjectif $\Gal( \overline{F} /F) \rightarrow \Gal (\overline{k}_F /k_F)$ dont le noyau est, par définition, le \emph{groupe d'inertie} de $F$, noté $I_F$.

\begin{prop}\label{galois z chapeau}
On a : $\Gal (\overline{k}_F /k_F) \simeq \widehat{\mathbb{Z}} \simeq \prod_p \mathbb{Z}_p$.
\end{prop}
\begin{proof}
Pour chaque entier naturel non nul $n$, il existe une extension de $k_F$ de degré $n$, unique à isomorphisme près. Elle est de groupe de Galois cyclique $\mathbb{Z}/n \mathbb{Z}$, et le corps $\overline{k}_F$ est l'union de ces extensions finies, qui forment un système projectif pour la divisibilité du degré.

On a donc :
\[
\Gal (\overline{k}_F /k_F) \simeq \varprojlim \mathbb{Z}/n \mathbb{Z} = \widehat{\mathbb{Z}} \simeq \prod_p \mathbb{Z}_p.
\]
\end{proof}
On peut donc écrire la suite exacte
\begin{equation}\label{suite exacte gal}
\begin{tikzcd}
1 \arrow{r} & I_F \arrow{r} & \Gal(\overline{F}/F) \arrow{r} & \widehat{\mathbb{Z}} \arrow{r} & 0.
\end{tikzcd}
\end{equation}

Le groupe $\widehat{\mathbb{Z}}$ contient $\mathbb{Z}$ comme sous-groupe dense, et par l'isomorphisme de la Proposition \ref{galois z chapeau}, $\mathbb{Z}$ correspond au groupe engendré par le morphisme de Frobenius \emph{arithmétique}, noté $\sigma_F$, qui envoie $x \in \overline{k}_F$ sur $x^{q_F}$. Nous choisissons de faire correspondre le morphisme de Frobenius arithmétique à l'élément $-1$ de $\Z$ et nous  notons $\mathrm{Fr}$ un élément dont l'image est $\sigma_F^{-1}$ (\emph{i.e.} qui correspond à l'élément $+1$ de $\Z$).\ps

On peut alors définir le groupe de Weil (abstrait) $\W_F$ de $F$ comme l'image inverse de $\mathbb{Z}$ dans le diagramme \eqref{suite exacte gal}, \emph{i.e.} $\W_F$ est engendré par le morphisme $\mathrm{Fr}$ et par le groupe d'inertie $I_F$.

On a donc la suite exacte 
\begin{equation*}
\begin{tikzcd}
1 \arrow{r} & I_F \arrow{r} & \W_F \arrow{r} & \mathbb{Z} \arrow{r} & 0.
\end{tikzcd}
\end{equation*}
On munit alors $\W_F$ de l'unique topologie de groupe localement compact telle que
\begin{itemize}
\item la projection $\W_F \rightarrow \langle \sigma_F \rangle \simeq \mathbb{Z}$ soit continue ($\mathbb{Z}$ étant muni de la topologie discrète),
\item la topologie induite sur $I_F$ est la topologie profinie induite par la topologie de $\Gal (\overline{F} /F)$.\ps 
\end{itemize}

On remarque que cette topologie \emph{n'est pas} celle induite par l'inclusion $\W_F \subset \Gal (\overline{F} /F)$, mais que l'injection topologique $\W_F \hookrightarrow \Gal (\overline{F} /F)$ est encore continue et d'image dense.\medskip 



La théorie du corps de classes local (nous renvoyons à \cite{Se} Chapitre XIII) nous permet de définir, dans le langage de \cite{Tate-Corvallis}, l’isomorphisme de groupes topologiques : 
\begin{equation}\label{application d'Artin-groupe de Weil}
\mathrm{Art}_F:F^\times \xlongrightarrow{\sim} \W_F^\ab.
\end{equation}

Cet isomorphisme envoie (surjectivement) les éléments de $\mathcal{O}^\times$ sur l'image du groupe d'inertie $I_F$ dans l'abélianisé et envoie une uniformisante sur un \og Frobenius géométrique \fg{} \emph{i.e.} sur un élément dont l'image dans $\Gal (\overline{k}_F /k_F)$ est $\sigma_F ^{-1}$. Cela permet en outre de définir la \emph{norme} d'un élément de $\W_F$ comme étant la norme ($p$-adique) de l'élément de $F^\times$ associé par \eqref{application d'Artin-groupe de Weil}, que l'on notera de la même façon.\ps

Par construction de $\W_F$, on a l'injection $\W_F \hookrightarrow \Gal(\overline{F} / F)$, d'image dense, ce qui nous permet de voir la catégorie des représentations de $\Gal(\overline{F} / F)$ (complexes, continues, de dimension finie) comme une sous-catégorie pleine de la catégorie des représentations de $\W_F$ (complexes, continues, de dimension finie). Une représentation de cette sous-catégorie est dite \emph{de type galoisien}. Nous avons la proposition suivante :
\begin{prop}\label{rep_WF_rep_type_galoisien}
\begin{enumerate}
\item Soit $(\rho,V)$ une représentation de $\W_F$. Alors $\rho$ est de type galoisien si, et seulement si l'image $\rho(\W_F)$ est finie.
\item Une représentation de type galoisien de $\W_F$ est irréductible si, et seulement si elle l'est comme représentation de $\Gal(\overline{F} / F)$.
\item Toute représentation \emph{irréductible} de $\W_F$ est de la forme $\sigma \otimes |\cdot |^s$ avec $\sigma$ de type galoisien, et $s \in \C$.
\item Une représentation \emph{irréductible} $\rho$ de $\W_F$ est de type galoisien si, et seulement si $\det \circ \rho (\W_F)$ est un sous-groupe fini de $\C ^\times$.
\end{enumerate}
\end{prop}
\begin{proof}
On démontre le point 3. les autres s'en déduisent aisément.

Soit $(\rho,V)$ une représentation irréductible de $\W_F$.
La continuité de $\rho$ implique que le noyau $\Ker \rho$ contient un sous-groupe ouvert de $I_{F}$. En particulier, l'image $\rho(I_{F})$ est un sous-groupe fini de $\GL(V)$. Comme $I_{F}$ est distingué dans $\W_{F}$, ce sous-groupe est normalisé par $\rho(\W_{F})$.

Notons maintenant $\mathrm{Fr}$ un morphisme de Frobenius de $\W_{F}$. Alors, en particulier, $\rho(\mathrm{Fr})$ induit une permutation de l'ensemble fini $\rho(I_{F})$, si bien qu'en notant $m=\# \rho(I_{F})$, on a que $\rho(\mathrm{Fr})^{m!}$ induit la permutation identité, \emph{i.e.} commute à tous les éléments de $\rho(I_{F})$. Or $\W_{F}$ est engendré par $I_{F}$ et par $\mathrm{Fr}$, si bien que $\rho(\mathrm{Fr})^{m!}$ commute à tout $\rho(\W_{F})$.

Par le lemme de Schur, on a donc que $\rho(\mathrm{Fr})^{m!}$ est une homothétie, de rapport $\lambda$ disons. On peut alors définir un caractère $\chi$ de $\W_{F}$ par $\chi_{|I_{F}}=1$ et $\chi(\mathrm{Fr})=\frac{1}{\alpha}$ où $\alpha$ est une racine $m!$-ème de $\lambda$. Le caractère $\chi$ est non ramifié et peut donc bien s'écrire comme une puissance (complexe) de la norme : $\chi=|\cdot|^s$ pour un certain $s \in \C$. On voit alors de suite que $\rho \otimes \chi$ est d'image finie et de type galoisien.
\end{proof}

\section{Paramètres de Langlands}\label{Paramètres de Langlands}

Nous reprenons les notations du paragraphe \ref{Notations pour les groupes reductifs}. Soit donc $G$ le groupe des $F$-points d'un groupe algébrique $\mathbf{G}$ défini et réductif sur $F$. Nous renvoyons à \cite{Springer-Corvallis} et à \cite{Borel-Corvallis} pour la définition en toute généralité du $L$-groupe associé à $G$.

On se place ici dans le cas où $\mathbf{G}$ est déployé sur $F$. On a alors simplement $\tensor*[^L]{G}{}=\bm{\widehat{\mathrm{G}}}(\C)$ où $\bm{\widehat{\mathrm{G}}}(\C)$ est \emph{un} groupe complexe réductif, dont la donnée radicielle basée est duale de celle de $\mathbf{G}$. Ce groupe dual n'est bien défini qu'à un automorphisme intérieur près (d'où l'article indéfini), on remarque néanmoins que, dans la suite, on considère des morphismes \emph{à conjugaison près} si bien que cette ambiguïté (ou ce choix d'\emph{un} groupe dual) ne portera pas à conséquence. On notera $\widehat{G}=\bm{\widehat{\mathrm{G}}}(\C)$ dans toute la suite.\medskip

Un morphisme 
\[
\varphi : \W_{F} \times \SU (2) \longrightarrow \widehat{G}
\]
est dit admissible si :
\begin{enumerate}
\item $\varphi$ est continu,
\item $\varphi(\W_F)$ est constitué d'éléments semi-simples.
\end{enumerate}

\begin{defi}\label{defi_param_Lgl_p_adique}
Un \emph{paramètre de Langlands} de $G$ est une classe d'équivalence, pour la $\widehat{G}$-conjugaison, de morphismes admissibles. On note $\Phi(G)$ l'ensemble des paramètres de Langlands de $G$.
\end{defi}
Lorsque cela ne portera pas à conséquence, on commettra l'abus de langage consistant à désigner comme un paramètre de Langlands un représentant de la classe qu'il constitue. En particulier, lorsqu'il s'agira d'étudier des propriétés des morphismes admissibles préservées par $\widehat{G}$-conjugaison, on se permettra d'attribuer cette propriété au paramètre de Langlands lui-même.

\begin{lemme}\label{Frob-semi-simple}
Soit $\rho : \widehat{G} \rightarrow \GL(V)$ une représentation fidèle algébrique (où $V$ est un $\C$-espace vectoriel de dimension finie) de $\widehat{G}$.

Alors on peut remplacer la condition $2.$ ci-dessus par $2.'$ \og $\rho \circ \varphi_{|\W_{F}}$ est semi-simple comme représentation de $\W_{F}$ \fg .
\end{lemme}

\begin{proof}

Soit $\varphi$ un morphisme continu de $\W_{F} \times \SU (2)$ vers $\widehat{G}$.

On suppose d'abord que $\rho \circ \varphi_{|\W_{F}}$ est semi-simple et on veut montrer que $\varphi(\W_F)$ est constitué d'éléments semi-simples.

Un élément de $\widehat{G}$ est semi-simple si, et seulement s'il l'est dans une représentation fidèle (et alors il l'est dans toutes). Il est donc équivalent de démontrer que $(\rho \circ \varphi)(\W_F)$ est constitué d'éléments semi-simples (de $\GL(V)$).
D'après la Proposition \ref{rep_WF_rep_type_galoisien}, $\rho \circ \varphi=\bigoplus_i \sigma_i \otimes |\cdot |^{s_i}$ comme représentation de $\W_{F}$, où les $\sigma_i$ sont des représentations de type galoisien 
et les $s_i$ des nombres complexes. La torsion par une puissance de la norme ne jouant aucun rôle quant à la semi-simplicité, il suffit donc de montrer que, pour tout $i$, l'image de $\W_{F}$ par $\sigma_i$ est constituée d'éléments semi-simples. Chaque $\sigma_i$ se factorise, par continuité, en une représentation de $\Gal(K_i/F)$ où $K_i$ est une extension finie de $F$. Il s'agit ainsi d'une représentation d'un groupe fini, dans un espace vectoriel complexe de dimension finie, les éléments de $\sigma_i(\W_{F})$ sont donc semi-simples et, partant, ceux de $(\rho \circ \varphi)(\W_F)$ également.\ps 

Réciproquement, on suppose que l'image de $\W_{F}$ par $\varphi$ est constituée d'éléments semi-simples et on veut montrer que $\rho \circ \varphi_{|\W_{F}}$ est une représentation semi-simple de $\W_F$.

L'image de $\W_F$ par $\varphi'=\rho \circ \varphi$ est alors également constituée d'éléments semi-simples. Par continuité, $\Ker \varphi'$ contient un sous-groupe ouvert de $I_{F}$, en particulier l'image $\varphi'(I_{F})$ est finie.
En notant encore $\Fr$ un morphisme de Frobenius, on trouve, comme dans la preuve de la Proposition \ref{rep_WF_rep_type_galoisien},
une puissance du Frobenius $\Fr ^\alpha$ telle que $\varphi'(\Fr^\alpha)$ commute à tous les éléments de $\varphi'(\W_{F})$.
Par hypothèse, l'élément $\varphi'(\Fr ^\alpha)$ est semi-simple, notons $E_\lambda$ ses espaces propres. Par commutation, chacun de ces espaces propres est laissé stable par tous les éléments de $\varphi'(\W_{F})$, on a donc décomposé notre représentation $\varphi'$ en somme directe de sous-représentations, reste à voir que chacune d'entre elles est semi-simple.

Considérons donc un espace propre $E_\lambda$ associé à la valeur propre $\lambda$, quitte à tordre par une puissance de la norme, on peut supposer que $\lambda=1$. Ainsi $E_1$ est préservé par $\Fr^\alpha$ et, par continuité, par un sous-groupe ouvert de $I_{F}$ et donc finalement par un sous-groupe d'indice fini de $\W_{F}$. Ainsi $E_1$ est la représentation (complexe de dimension finie) d'un groupe fini et est donc complètement réductible.
\end{proof}

Par ailleurs, en gardant les notations du Lemme, puisque le groupe $\SU(2)$ est compact, la représentation continue et de dimension finie $\rho \circ \varphi_{|\SU(2)}$ est unitarisable et donc complètement réductible. Le produit $\W_{F} \times \SU (2)$ étant direct, il est donc équivalent d'avoir $2.''$ \og $\rho \circ \varphi$ est semi-simple \fg. \medskip 

\emph{Remarque :} On pourrait dans l'autre sens chercher à affaiblir la condition $2.$ et remarquer qu'il suffit en fait de savoir que $\varphi(\Fr)$ est semi-simple pour $\Fr$ un morphisme de Frobenius : puisque le groupe de Weil est engendré par le sous-groupe d'inertie (compact) et par $\Fr$, le même argument de commutation d'une puissance suffisamment élevée du $\Fr$ que dans la preuve du Lemme \ref{Frob-semi-simple} permet de conclure.

\section{Correspondance de Langlands locale}\label{Énoncé de la correspondance}
\subsection{Compatibilités}
\begin{conj} \emph{Correspondance de Langlands locale}

Soit $G$ le groupe des $F$-points d'un groupe algébrique $\mathbf{G}$ défini et réductif sur $F$, où $F$ est un corps local non archimédien.
Il existe une application \og naturelle \fg{}  :
\[
\Ll : \mathrm{Irr}(G) \longrightarrow \Phi(G)
\]
surjective à fibres finies.
\end{conj}
Tout l'enjeu est évidemment de préciser l'adjectif \og naturel \fg{} et d'imposer un certain nombre de compatibilités afin de \og rigidifier \fg la correspondance. Nous allons énoncer certaines de ces compatibilités attendues, qui sont d'ailleurs bel et bien vérifiées pour les cas (démontrés) de la correspondance que nous utiliserons.\ps 

Étant donnée une représentation irréductible $\pi$, on parlera de son paramètre de Langlands $\Ll(\pi)$ et, étant donné un paramètre de Langlands $\varphi$, on parlera du paquet de Langlands associé $\Pi_\varphi=\Ll^{-1}(\varphi)$.

\begin{defi}

\begin{enumerate}
\item On dit qu'un paramètre de Langlands de $G$ est \emph{tempéré} si son image dans $\widehat{G}$ est bornée. On note $\Phi_\mathrm{temp}(G)$ le sous-ensemble de $\Phi(G)$ constitué de tels paramètres (modulo conjugaison).
\item On dit qu'un paramètre de Langlands de $G$ est \emph{discret} si son image n'est contenue dans aucun sous-groupe parabolique strict (ou de manière équivalente, dans aucun sous-groupe de Levi strict) de $\widehat{G}$. On note $\Phi_\mathrm{disc}(G)$ le sous-ensemble de $\Phi(G)$ constitué de tels paramètres (modulo conjugaison).
\end{enumerate}
\end{defi}

Alors pour une représentation irréductible $\pi$ de $G$, on veut :
\begin{align}
\pi \in \mathrm{Irr}_\mathrm{disc} (G) &\Longleftrightarrow \Ll(\pi) \in \Phi_{\mathrm{disc}}(G), \label{rep_disc_param_disc}\\ 
\pi \in \mathrm{Irr}_\mathrm{temp} (G) &\Longleftrightarrow \Ll(\pi) \in \Phi_{\mathrm{temp}}(G), \label{rep_temp_param_temp}
\end{align}
où $\mathrm{Irr}_\mathrm{disc}(F)$ indique l'ensemble des séries discrètes de $G$ introduites en \ref{déf_série_discrète}, et $\mathrm{Irr}_\mathrm{temp} (G)$ l'ensemble des représentations irréductibles tempérées de $G$, introduites en \ref{déf_rep_tempérée}. Ainsi, l'ensemble des séries discrètes (resp. des représentations irréductibles tempérées) de $G$ est partitionné en paquets de Langlands discrets (resp. tempérés), \ie correspondant à un paramètre discret (resp. tempéré) de $G$.

\begin{defi}
On dit qu'un paramètre de Langlands de $G$ est \emph{non ramifié} s'il est trivial sur le sous-groupe $I_F \times \SU(2)$, où $I_F$ désigne le sous-groupe d'inertie du groupe de Weil $\W_F$. On note $\Phi_\mathrm{nr}(G)$ le sous-ensemble de $\Phi(G)$ constitué de tels paramètres (modulo conjugaison).
\end{defi}
On remarque tout de suite que l'application qui à un tel paramètre $\varphi$ associe la classe de $\widehat{G}$-conjugaison de $\varphi(\Fr)$, où $\Fr$ est un morphisme de Frobenius, réalise une bijection entre $\Phi_\mathrm{nr}(G)$ et l'ensemble des classes de conjugaison d'éléments semi-simples de $\widehat{G}$.\ps

Pour définir les représentations non ramifiées, il faut définir une classe particulière de sous-groupes compacts ouverts, dits \emph{hyperspéciaux}. De tels sous-groupes n'existent pas toujours et quand ils existent, ils ne forment pas toujours une seule classe de conjugaison. Nous nous restreignons au cas des groupes dont le centre (schématique) est un tore déployé, pour lequel on a bien des sous-groupes hyperspéciaux, tous conjugués (nous renvoyons à \cite{Tits-Corvallis}, §3.8).\ps 

Soit donc $\mathbf{G}$ un groupe algébrique défini et réductif sur $F$, dont le centre schématique est un tore déployé (éventuellement trivial). Alors il existe un modèle réductif sur $\Ok$ de $\mathbf{G}$, \ie un schéma en groupes $\bm{\mathcal{G}}$ défini et réductif sur $\Ok$ et un isomorphisme $\bm{\mathcal{G}} \times_\Ok F \simeq \mathbf{G}$. Pour chaque modèle $\bm{\mathcal{G}}$, on peut considérer le groupe $K_0=\bm{\mathcal{G}} (\Ok)$, sous-groupe compact (maximal) de $G=\mathbf{G} (F)=\bm{\mathcal{G}} (F)$, hyperspécial par définition.


On dira d'une représentation $(\pi,V)$ de $G$ qu'elle est non ramifiée (ou $K_0$-sphérique) si l'espace $V^{K_0}$ des $K_0$-invariants est non nul (ce qui ne dépend que de la classe de conjugaison de $K_0$). On note $\mathrm{Irr}_\mathrm{nr}(G)$ l'ensemble des représentations (lisses) irréductibles non ramifiées, et on veut :
%
%
\begin{align}\label{LLC_resp_nr}
\pi \in \mathrm{Irr}_\mathrm{nr} (G) &\Longleftrightarrow \Ll(\pi) \in \Phi_{\mathrm{nr}}(G).
\end{align}

\subsection{Énoncés}
La conjecture de Langlands locale est un théorème pour le groupe $\GL_n$ avec trois preuves indépendantes (Harris et Taylor, Henniart, Scholze). On remarque que la théorie du corps de classes, notamment dans la formulation de l'équation \eqref{application d'Artin-groupe de Weil} nous donne une bijection \og naturelle \fg entre $\mathrm{Irr}(\GL_1)$ et $\Phi(\GL_1)$, c'est celle qui est retenue.

\begin{thm} \emph{Correspondance de Langlands locale pour $\GL_n$ (\cite{HT},\linebreak\cite{Henniart2000})}\ps  \label{LLC_GLn}

Il existe une unique collection de bijections :
\[
\Ll_n : \mathrm{Irr}(\GL_n) \longrightarrow \Phi(\GL_n)
\]
compatible à la théorie du corps de classes, à la torsion par un caractère, à la dualité, au caractère central et qui \og respecte les fonctions $L$ et les facteurs epsilon\footnote{Il faut bien considérer ici, contrairement à ce que nous ferons au paragraphe \ref{Fonctions $L$ et facteurs epsilon}, \emph{tout} le facteur epsilon, y compris ce qui correspond au conducteur.} de paires \fg .
\end{thm}
On remarque que, dans le cas de $\GL_n$, les paquets de Langlands sont des singletons.\medskip

Le seul autre cas dont on aura besoin est celui où $G$ est le groupe spécial orthogonal impair déployé sur $F$ (défini au § \ref{Notations}). On a $G=\SO_{2n+1}(F)$ et $\widehat{G}=\Sp_{2n}(\C)$. On a alors bien la surjection à fibres finies recherchée, mais il y a une façon canonique d'indexer les éléments d'un paquet de Langlands (et donc de connaître le cardinal d'un paquet) pour laquelle nous devons introduire quelques notations.

Au paramètre $\varphi$, on peut associer son centralisateur dans $\widehat{G}$, noté $\Cent(\varphi)$, qui est par défnition le sous-groupe des éléments de $\widehat{G}$ qui commutent à l'image de $\varphi$. C'est un sous-groupe algébrique de $\widehat{G}$, qui contient le centre de $\widehat{G}$ et dont on notera $\Cent(\varphi)^0$ la composante neutre.
On pose
\[
\mathcal{S}_\varphi=\Cent(\varphi)/\Cent(\varphi))^0Z(\widehat{G}).
\]
C'est un groupe abélien fini -- on peut remarquer que c'est le groupe des composantes connexes de $\Cent(\varphi)/Z(\widehat{G})$.

%
%

\begin{thm} \label{LLC_SO}\emph{Correspondance de Langlands locale pour $\SO_{2n+1}$ (\cite{Art13}, Theorem 2.2.1 avec les résultats de \cite{moeglin2011multiplicite})} \ps 

Soit $G$ \emph{le} groupe (\cf § \ref{Notations}) spécial orthogonal impair déployé sur $F$.
Il existe une surjection \emph{canonique}
\[
\Ll : \mathrm{Irr}(G) \longrightarrow \Phi(G)
\]
telle que les éléments de $\Pi_\varphi=\Ll^{-1}(\varphi)$ sont indexés par $\widehat{\mathcal{S}_\varphi}$ où $\widehat{\mathcal{S}_\varphi}$ désigne l'ensemble des caractères linéaire du groupe (abélien fini) $\mathcal{S}_\varphi$. Cette indexation est déterminée par des identités endoscopiques (et endoscopiques tordues). En particulier, le paquet $\Pi_\varphi$ ne contient qu'un nombre fini d'éléments.
\end{thm}

\section{Paramètres discrets}\label{Paramètres discrets}
Commençons par le résultat général suivant :
\begin{prop}
Soit $H$ le groupe des $\C$-points d'un groupe réductif complexe.
Soit $\varphi : \W_F \times \SU(2) \rightarrow H$ un morphisme dont l'image est constituée d'éléments semi-simples. Alors l'image de $\varphi$ n'est contenue dans aucun sous-groupe de Levi strict si, et seulement si $\Cent(\varphi)/Z(H)$ est fini, où $\Cent(\varphi)$ désigne le centralisateur dans $H$ de l'image de $\varphi$.
\end{prop}
\begin{proof}

Supposons d'abord que l'image de $\varphi$ soit incluse dans un sous-groupe de Levi strict $L$ de $H$. Alors $\Cent(\varphi)$ contient $Z(L)$. Or, la composante neutre du centre de $L$ est un tore plus gros que la composante neutre du centre de $H$, si bien que $Z(L)^0/Z(H)^0$ puis $Z(L)/Z(H)$ sont infinis et $\Cent(\varphi)/Z(H)$ également. \ps 

Réciproquement, supposons $\Cent(\varphi)/Z(H)$ infini. Pour des raisons générales, $\Cent(\varphi)^0$ est d'indice fini dans $\Cent(\varphi)$ ; de même $Z(H)^0$ est d'indice fini dans $Z(H)$, si bien que l'on a également $\Cent(\varphi)^0/Z(H)^0$ infini.

Le Lemme \ref{Frob-semi-simple} (combiné à la Remarque qui le suit) nous dit que $\rho \circ \varphi$ est une représentation semi-simple de $\W_F \times \SU(2)$ pour toute représentation fidèle $\rho$ de $H$ dans un espace vectoriel complexe de dimension finie. Cela signifie en particulier que l'adhérence (pour la topologie de Zariski) de $\varphi(\W_F \times \SU(2))$ dans $H$ est de composante neutre réductive, puis que $\Cent(\varphi)^0$ l'est également d'après des résultats classiques de Steinberg.

Le groupe réductif $\Cent(\varphi)^0$ contient donc un tore maximal $T$ et on a la série d'inclusions $Z(H)^0 \subset T \subset \Cent(\varphi)^0$. Ainsi l'image de $\varphi$ centralise le tore $T$, et est donc incluse dans un sous-groupe de Levi, strict si $T$ est plus gros que $Z(H)^0$. Supposons donc par l'absurde que $Z(H)^0=T$. Alors, 
puisqu'un élément semi-simple d'un groupe est toujours conjugué à un élément du tore maximal, on a que 
tout élément semi-simple de $\Cent(\varphi)^0$ est conjugué à un élément de 
$T$ et donc de $Z(H)^0$. Ainsi les éléments semi-simples de $\Cent(\varphi)^0$ sont exactement les éléments de $Z(H)^0$. Or ceux-là sont Zariski-denses, et on obtient $\Cent(\varphi)^0=Z(H)^0$, une contradiction. Donc l'image de $\varphi$ est incluse dans un sous-groupe de Levi strict.
\end{proof}

En appliquant ce résultat aux paramètres de Langlands de $G$, on obtient qu'un paramètre $\varphi$ est discret si, et seulement si $\Cent(\varphi)/Z(\widehat{G})$ est fini, ou encore dans le cas où $Z(\widehat{G})$ est fini, si, et seulement si $\Cent(\varphi)$ est fini.\medskip 

Nous poursuivons par un lemme technique important.

Soit $\Gamma$ un groupe quelconque et soit $V$ un $\C$-espace vectoriel de dimension finie, muni d'une forme bilinéaire symétrique (resp. alternée) non dégénérée $f$ telle que $V \simeq V^*$ (pour l'application naturelle induite par $f$). On suppose que $V$ est également muni d'une action semi-simple de $\Gamma$ compatible à la forme $f$ (\ie cette dernière est $\Gamma$-invariante) : on a alors une représentation de $\Gamma$ dans le groupe orthogonal (resp. symplectique) de $V$ pour la forme $f$.

On peut, par semi-simplicité, décomposer $V$ en somme de composantes $\Gamma$-isotypiques, \ie :
\[
V=M_1 \oplus \cdots \oplus M_r
\]
avec, pour $1\leq i \leq r$, $M_i$ qui est $V_i$-isotypique où $V_i$ est irréductible.
\begin{lemme}\label{invariants-non-dégénérés}
Si $V_i \simeq V_i^*$, alors $M_i$ est non dégénéré pour $f$.

En particulier, si on considère la représentation triviale (autoduale), dont la composante isotypique est formée des invariants sous l'action de $\Gamma$, on a $V^\Gamma$ non dégénéré.
\end{lemme}
\begin{proof}
On a par hypothèse que le morphisme canonique
\[
\begin{array}{ccccc}
\zeta & : & V &\longrightarrow & V^* \\
 & & v & \longmapsto & f(v,\cdot) \\
\end{array}
\]
est bijectif et $\Oo(V)$-équivariant (resp. $\Sp(V)$-équivariant). La restriction de $\zeta$ à chaque $M_i$, notée $\zeta_i$, est donc injective, $\zeta_i : M_i \hookrightarrow V^*$. On peut donc considérer, par restriction et projection :
\[
\begin{array}{ccccc}
\zeta_{i,j} & : & M_i &\longrightarrow & M_j^* \\
 & & v & \longmapsto & f(v,\cdot) \\
\end{array}.
\]

Les morphismes $\zeta,\zeta_i$ et $\zeta_{i,j}$ sont $\Gamma$-équivariants, et donc, comme rappelé au paragraphe \ref{Généralités}, préservent les composantes isotypiques. Ainsi l'image de $\zeta_{i,j}$ est à la fois $V_i$-isotypique et $V_j^*$-isotypique et donc est nulle si $V_i \not\simeq V_j^*$, ce qui revient à dire qu'alors $M_i \perp M_j$.\ps 

On a $\zeta_i=\bigoplus_j \zeta_{i,j}$ injectif et tous les $\zeta_{i,j}$ sont nuls sauf pour l'unique $j_0$ (il en faut un) tel que $V_i \simeq V_{j_0}^*$ et alors $\zeta_{i,j_0}$ est injectif.

Si $j_0=i$ (\ie si $V_i \simeq V_i^*$), alors l'injectivité de $\zeta_{i,j_0}$ nous dit exactement que l'espace $V_i$-isotypique $M_i$ est non dégénéré, ce qui achève la preuve (on note d'ailleurs que, pour des raisons de dimension, $\zeta_{i,i}$ est en fait un isomorphisme).

Poursuivons néanmoins l'analyse : dans le cas où $j_0 \neq i$, on a par bidualité $V_i \simeq V_{j_0}^* \Rightarrow V_{j_0} \simeq V_i^*$ si bien que les deux injections $\zeta_{i,j_0} : M_i \hookrightarrow M_{j_0}^*$ et $\zeta_{j_0,i} : M_{j_0} \hookrightarrow M_i^*$ sont en fait des isomorphismes pour des raisons de dimension. Finalement, on peut écrire
\[
V=\overset{\perp}{\bigoplus_{j \in I_1}} M_j \overset{\perp}{\oplus} \overset{\perp}{\bigoplus_{j \in I_2}} (M_j \oplus M_j^*),
\]
où $I_1$ désigne l'ensemble des indices $j$ tels que $V_j \simeq V_j^*$ et $I_2$ un choix de la moitié des indices dans son complémentaire tel que $\{V_j,V_j^* |\,j \in I_2\}$ parcourt bien l'ensemble des $V_j$ non isomorphes à leur dual.

De plus chaque terme de la somme orthogonale est non dégénéré.
\end{proof}

\begin{lemme}\label{prod_tens_fb}
Soient $A$ et $B$ deux $\C$-espaces vectoriels de dimension finie, munis chacun d'une forme bilinéaire (symétrique ou alternée) non dégénérée. Alors on peut munir canoniquement $A \otimes_\C B$ d'une forme bilinéaire non dégénérée dont la nature est donnée selon la \og règle des signes \fg{} (une forme alternée correspondant au signe $-$).
\end{lemme}
\begin{proof}
Notons $f_A$ et $f_B$ les formes bilinéaires sur $A$ et $B$ respectivement. Alors on définit 
\[
(f_A \otimes f_B)(a\otimes b,a' \otimes b')=f_A (a,a')f_B(b,b')
\]
sur les tenseurs purs, ce qu'on étend par bilinéarité pour obtenir une forme sur $A\otimes_\C B$. Vérifions la non-dégénérescence. Si $f_A$ est symétrique, on peut trouver une base de $A$ qui soit orthogonale (même orthonormée) pour $f_A$ ; et si $f_A$ est alternée, on peut trouver une base de $A$ qui soit symplectique (dite base de Darboux) pour $f_A$. En tous les cas, si l'on note $(a_1, \cdots,a_n)$ une telle base, chaque vecteur de base est orthogonal à tous les vecteurs de la base sauf un. Choisissons de même $(b_1, \cdots,b_m)$ une base de $B$, adaptée à $f_B$.\ps 

Soit maintenant un élément $x$ non nul de $A\otimes_\C B$, que l'on peut écrire $x=\sum_{i,j} \lambda_{ij} a_i \otimes b_j$. Par non-nullité, on trouve un couple $\{i_0,j_0\}$ tel que $\lambda_{i_0,j_0} \neq 0$. Par ce qui précède, il existe un indice $i'_0$ tel que $f_A(a_{i},a_{i'_0})=\delta_{i,i_0}$ où $\delta$ désigne le symbole de Kronecker. De même pour $f_B$ et un indice $j'_0$. Alors $(f_A \otimes f_B)(x,a_{i'_0}\otimes b_{j'_0})=\lambda_{i_0,j_0}\neq 0$ si bien que $\Ker(f_A \otimes f_B)$ est réduit à $\{0\}$ et que la forme est non dégénérée.

Puisqu'on travaille en caractéristique différente de 2,\og alternée \fg{} équivaut à \og antisymétrique \fg et on voit alors immédiatement que si $f_A$ et $f_B$ sont toutes deux symétriques (resp. toutes deux alternées), alors $f_A \otimes f_B$ est symétrique ; et que si l'une est alternée et l'autre symétrique, alors $f_A \otimes f_B$ est alternée.
\end{proof}

\begin{prop}\label{decomposition_par_discret_sp}
Soit $V$ un $\C$-espace vectoriel de dimension finie muni d'une forme alternée non dégénérée $f$. Soit $\Gamma$ un groupe quelconque et soit $\varphi : \Gamma \rightarrow \Sp (V)$ une représentation \emph{semi-simple} de $\Gamma$.

Alors le centralisateur de $\varphi$ est fini si, et seulement si $V \simeq V_1 \perp \cdots \perp V_r$ avec les $V_i$ irréductibles pour l'action de $\Gamma$, deux à deux non isomorphes, et $f_{|V_i}$ non dégénérée.
\end{prop}
\begin{proof}
Supposons que $V$ admette l'écriture mentionnée. Alors un élément du centralisateur de $\varphi$ (\ie un élément de $\Sp(V)$ qui commute à $\varphi(g)$ pour tout $g \in \Gamma$) donne un opérateur d'entrelacement entre $V$ et lui-même et, du fait que les $V_i$ sont irréductibles et deux à deux non isomorphes, un opérateur d'entrelacement entre chaque $V_i$ et lui-même.

Or, le lemme de Schur nous dit qu'un tel opérateur est une homothétie sur $V_i$, celle-ci devant respecter la structure symplectique, c'est finalement $\pm \mathrm{id}_{V_i}$. On a donc $\mathrm{Cent}(\varphi)=\bigoplus_i \{\pm \mathrm{id}_{V_i}\}$, qui est de cardinal $2^r$.

Réciproquement, partons d'un sous-espace $U$ de $V$, stable par $\Gamma$ et irréductible comme $\Gamma$-module. Alors $f_{|U}$ a un noyau stable par $\Gamma$, à savoir $\{0\}$ (auquel cas elle est non dégénérée) ou $U$ (auquel cas elle est nulle). \ps 

Supposons par l'absurde être dans le deuxième cas, alors $U$ est totalement isotrope, \ie $U \subset U^\perp$. L'orthogonal d'un espace $\Gamma$-stable étant $\Gamma$-stable, on peut utiliser la semi-simplicité du $\C[\Gamma]$-module $V$ pour écrire $U^\perp=W\oplus U$ où $W$ est $\Gamma$-stable. On remarque que $\Ker f_{|U^\perp}=(U^\perp)^\perp \cap U^\perp=U$ donc $f$ est non-dégénérée sur $W \simeq U^\perp / \Ker f_{|U^\perp}$.

On peut alors écrire $V=W \oplus W^\perp$ avec $W^\perp$ non dégénéré et $\Gamma$-stable. On a $W \subset U^\perp$ donc $U \subset W^\perp$. Par semi-simplicité, on peut de nouveau écrire $W^\perp=U \oplus U'$ avec $U'$ sous-espace $\Gamma$-stable. Écrivons désormais $V'$ pour $W^\perp$ : on s'est ramené à un espace non dégénéré $V'$ qui admet comme facteur direct un espace totalement isotrope (en fait maximal). Si $U'$ est lui aussi totalement isotrope, on constate que le centralisateur de $\varphi$ contient les éléments de la forme ${\rm id}_W \oplus \lambda {\rm id}_U \oplus \lambda^{-1} {\rm id}_{U'}$, où $\lambda$ appartient à $\C^\times$ ; il ne saurait donc être fini.
Il s'agit maintenant de montrer que l'on peut toujours se ramener à ce cas. \ps 

On a $U^\perp \cap W^\perp=U$ donc, dans $V'$, $U^\perp=U$ ce qui donne $U' \simeq V' / U \simeq V' / U^\perp \simeq U^*$. Ainsi $U'$ est irréductible et $f_{|U'}$ est de nouveau soit nulle, soit non dégénérée. Dans le premier cas, on peut conclure directement d'après ce qui précède ; supposons donc être dans le second cas. L'espace $U'$ est non dégénéré donc l'application $U' \rightarrow (U')^*$ induite par $f$ est injective et donc un isomorphisme. Finalement (par bidualité) on obtient $U \simeq U'$, ce qui signifie que $V'$ est isotypique et contient deux copies de $U$.

On peut alors écrire canoniquement $V'=U \otimes_\C \Hom_\Gamma (U,V')$ comme $\C$-espaces vectoriels, mais aussi comme $\Gamma$-représentations : il suffit de faire agir $\Gamma$ trivialement sur $\Hom_\Gamma (U,V')$. 
Si l'on note $U''$ ce dernier espace, alors il est isomorphe à $\C^2$ comme espace vectoriel et on a $U''=\Hom_\Gamma (U,V') =(\Hom(U,V'))^\Gamma  \simeq (V' \otimes_{\C} U^*)^\Gamma $.  
Or, les espaces $V'$ et $U  \simeq U^*$ étant chacun munis d'une forme bilinéaire alternée non dégénérée, on peut, d'après le Lemme \ref{prod_tens_fb}, munir $V' \otimes_{\C} U^*$ d'une forme bilinéaire symétrique non dégénérée. Et, par le Lemme \ref{invariants-non-dégénérés}, la forme reste non dégénérée sur $(V' \otimes_{\C} U^*)^\Gamma  \simeq U''$.

Finalement, $U''$ est un $\C$-espace vectoriel de dimension 2 muni d'une forme bilinéaire symétrique non dégénérée, on peut donc écrire $U''=\C e_1 \oplus \C e_2$ avec $e_1$ et $e_2$ isotropes. Cette décomposition étant $\Gamma$-stable (l'action est triviale), on peut finalement écrire $V'=(U \otimes_\C \C e_1) \oplus (U \otimes_\C \C e_2)$, somme de deux lagrangiens $\Gamma$-stables, comme souhaité.

Ceci achève de démontrer que \textbf{ sur tout sous-espace stable irréductible, $f$ est non dégénérée}.\ps 

Utilisant ce résultat, on peut écrire par récurrence $V=V_1 \perp \cdots \perp V_r$ avec les $V_i$ irréductibles (là où la semi-simplicité nous donnait une somme directe non orthogonale \emph{a priori} de représentations irréductibles). Soit alors $U$ une représentation irréductible de $\Gamma$ intervenant dans $V$. La composante isotypique $C(U)$ est non dégénérée et on peut comme précédemment écrire $C(U)=U \otimes_{\C[\Gamma]} \Hom_\Gamma(U,V)$ et si ce deuxième facteur est de dimension $>1$ alors il contient une droite isotrope et on trouve dans $C(U)$ (et donc dans $V$) une copie de $U$ qui est dégénérée, ce qui fournit une contradiction.

Ainsi, les $V_i$ sont deux à deux non isomorphes.
\end{proof}

\begin{lemme}\label{sp_tens_orth}
Soit $V$ un espace symplectique muni d'une action linéaire irréductible de $\Gamma$ compatible à la structure symplectique. On suppose que $\Gamma=\Gamma_1 \times \Gamma_2$. Alors $V=V_1 \otimes_\C V_2$ avec $V_i$ irréductible sous l'action de $\Gamma_i$. Mieux, on peut munir chaque $V_i$ d'une structure symplectique (ou orthogonale) compatible à l'action de $\Gamma_i$, les structures sur $V_1$ et $V_2$ étant \og de signe opposé \fg{} (\ie l'une est symplectique, l'autre est orthogonale).
\end{lemme}

\begin{proof}
L'écriture $V=V_1 \otimes_\C V_2$ pour une représentation irréductible d'un groupe produit est standard et unique à isomorphisme près. Puisqu'on a $V \simeq V^* \simeq V_1^* \otimes_\C V_2^*$, on en déduit $V_1 \simeq V_1^*$ et $V_2 \simeq V_2^*$, d'où :
\[
\Hom(V_i,V_i) \simeq V_i \otimes_\C V_i^* \simeq V_i^* \otimes_\C V_i^* \simeq \mathrm{Sym}^2 V_i^* \oplus \Lambda^2 V_i^*,
\]
pour chaque $i$. Ces isomorphismes sont $\Gamma_i$-équivariants, on peut donc considérer les $\Gamma_i$-invariants et on obtient :
\[
\Hom_{\Gamma_i}(V_i,V_i) \simeq (\mathrm{Sym}^2 V_i^*)^{\Gamma
_i} \oplus (\Lambda^2 V_i^*)^{\Gamma_i}.
\]
Or $V_i$ est irréductible donc le lemme de Schur nous dit que $\Hom_{\Gamma_i}(V_i,V_i)$ est de dimension 1. L'un exactement des deux termes de droite est non nul (et de dimension 1), nous fournissant une forme bilinéaire symétrique ou alternée $\Gamma_i$-invariante, non dégénérée et unique à un scalaire près.

La forme sur $V_1 \otimes_\C V_2$ construite par le Lemme \ref{prod_tens_fb} est, à un scalaire près, celle de départ. On doit donc avoir sur $V_1$ et $V_2$ des structures \og de signe opposé \fg{}.
\end{proof}

On veut bien évidemment appliquer les résultats précédents au cas d'un paramètre de Langlands discret pour le groupe $G=\SO_{2n+}(F)$ (étudié au Chapitre \ref{Le groupe paramodulaire}), \ie un morphisme $\varphi : \W_F \times \SU(2) \rightarrow \Sp_{2n}(\C)$. On fait alors apparaître, par la Proposition \ref{decomposition_par_discret_sp} des $V_i$ représentations irréductibles autoduales de $\W_F \times \SU(2)$, que l'on peut en fait écrire $V_i=\xi_i\otimes U_{a_i}$ où $\xi_i$ est une représentation irréductible autoduale de $\W_F$ et $U_{a_i}$ est \emph{la} représentation irréductible (autoduale) de $\SU(2)$ de dimension $a_i$.

Le Lemme \ref{sp_tens_orth} nous dit alors que $\xi_i$ est symplectique quand $U_{a_i}$ est orthogonale, et réciproquement. Or $U_{a_i}=\mathrm{Sym}^{a_i-1}\C^2$ est orthogonale quand $a_i$ est impair (et symplectique quand $a_i$ est pair).\medskip

Finalement, un paramètre discret est la donnée de $\varphi=\bigoplus_i \xi_i \otimes U_{a_i}$ avec $\sum \dim(\xi_i) a_i=2n$, les couples $(\xi_i,a_i)$ deux à deux distincts, $\xi_i$ autoduale symplectique si $a_i$ est impair et autoduale orthogonale si $a_i$ est pair.

\newpage
\chapter{Le groupe paramodulaire}\label{Le groupe paramodulaire}
\section{Rappels sur les formes quadratiques}\label{Rappels_sur_les_FQ}
\begin{prop-def}\label{Prop_def_FQ}
Soit $A$ un anneau commutatif unitaire et soit $M$ un $A$-module libre.

On dit que $q : M \rightarrow A$ est une forme quadratique sur $M$ si :
\begin{itemize} 
\item pour tout $\lambda \in A$ et tout $m \in M$, on a $q(\lambda x)=\lambda^2 q(x)$ ;
\item l'application $(m,n) \mapsto q(m+n)-q(m)-q(n)$ est $A$-bilinéaire (automatiquement symétrique).\ps 
\end{itemize}

On dit alors que $(M,q)$ est un $A$-module quadratique et on note $<\cdot,\cdot>$ la forme bilinéaire associée. On remarque que $2q(x)=<x,x>$.

Le noyau de $q$, noté $\Ker q$, est par définition le noyau de la forme bilinéaire symétrique $<\cdot,\cdot>$.

Si $\mathcal{F}=(x_1,\cdots,x_n)$ est une famille de vecteurs de $M$, la matrice de Gram de $\mathcal{F}$ est $(<x_i,x_j>)_{ij}$ que l'on note ${\rm Gram}(\mathcal{F})$. C'est une matrice symétrique à diagonale dans $2\Ok$.

Le déterminant du module quadratique $(M,q)$ est la classe, dans $A/(A^\times)^2$ du déterminant de ${\rm Gram}(\mathcal{F})$ où $\mathcal{F}$ est une $A$-base de $M$.

On dit que le module quadratique $(M,q)$ est \emph{non dégénéré} si $\det (M,q)$ est un inversible de $A$ (modulo $(A^\times)^2$).

Si ${\rm rg}_A M$ est impair, alors on peut définir le \og demi-déterminant \fg{} $\frac{1}{2}\det$ de $(M,q)$ à valeurs dans $A/(A^\times)^2$. On dit alors que $(M,q)$ est \emph{régulier} si $\frac{1}{2}\det (M,q) \in A^\times$ (modulo $(A^\times)^2$).
\end{prop-def}
\begin{proof}
Il s'agit simplement de vérifier que le déterminant d'un module quadratique ne dépend pas du choix d'une $A$-base. Or, si $\mathcal{B}_1$ et $\mathcal{B}_2$ sont deux $A$-bases de $M$, on a ${\rm Gram}(\mathcal{B}_2)={}^t P \,{\rm Gram}(\mathcal{B}_1) P$ où $P$ est la matrice de passage de $\mathcal{B}_1$ à $\mathcal{B}_2$. Donc $\det {\rm Gram}(\mathcal{B}_2)=(\det P)^2 \det {\rm Gram}(\mathcal{B}_1)$ et on a bien sûr $\det P \in A^\times$.

Pour le dernier point, on renvoie à \cite{Conrad}, Proposition C.1.4. Il est à noter que le demi-déterminant a toutes les \og bonnes propriétés \fg du déterminant (multiplicativité, compatibilité à l'extension des scalaires, etc.) et qu'il vérifie, bien sûr, $2\cdot\frac{1}{2}\det=\det$.
\end{proof}
Par abus, on parlera parfois du déterminant d'une forme quadratique ou de forme quadratique non dégénérée.\medskip

\emph{Remarque 1 :} Quand 2 n'est pas diviseur de zéro dans $A$, il est équivalent de se donner une forme quadratique et une forme bilinéaire symétrique paire (\ie à valeurs dans $2A$).\medskip

\emph{Remarque 2 :} Quand 2 est inversible, la notion de régularité (en dimension impaire) est équivalente à la notion de non dégénérescence.
Quand 2 n'est pas inversible, la notion de régularité correspond à une forme \og minimale \fg{} de non dégénérescence.

\begin{prop-def}\label{Prop-Def_gpe_orthogonal}
Soit $(M,q)$ un $A$-module quadratique. Le groupe orthogonal $\Oo(M,q)$ est le sous-groupe de $\GL_A(M)$ des éléments laissant la forme $q$ invariante : $\Oo(M,q)=\{g \in \GL_A(M) \, | \, q \circ g=q\}$.

En particulier, pour tout $g \in \Oo(M,q)$ et pour tous $x,y \in M$, on a $<g(x),g(y)>=<x,y>$.

On note également $\SO(M,q)$ le sous-groupe des éléments de $\Oo(M,q)$ de déterminant $1$\footnote{On n'a pas en général, selon cette définition, que $\SO(M,q)$ est d'indice deux dans $\Oo(M,q)$, cela sera néanmoins le cas si $A$ est intègre et $2\neq 0$.}.\ps  

Pour chaque $A$-algèbre $B$ commutative, il existe une et une seule forme quadratique sur $M \otimes_A B$, notée $q\otimes B$ telle que $(q\otimes B)(m\otimes 1)=q(m)$ pour tout $m \in M$.

On dispose ainsi d'un foncteur en groupes $B \mapsto \Oo(M\otimes_A B,q\otimes B)$, représentable par un schéma en groupes affine que l'on note $\bm{\Oo_M}$.

On dispose d'un morphisme de schémas en groupes $\bm{\mathrm{O}_{M}} \rightarrow \mathbb{G}_m$ donné par le déterminant, et dont le noyau est noté $\bm{\mathrm{SO}_{M}}$. 

Quand $A$ est de caractéristique $2$, on dispose d'un morphisme $\bm{\mathrm{O}_{M}} \rightarrow \underline{\{\pm 1\}}$ donné par le déterminant de Dickson-Dieudonné (noté $\det_{\rm DD}$), où $ \underline{\{\pm 1\}}$ désigne le schéma en groupes constant sur $A$.


\end{prop-def}
\begin{proof}
Pour l'extension des scalaires, on renvoie à \cite{Bbki_Alg}, Chapitre 9, §3 n\textsuperscript{o}4 , Proposition 3.

Pour la représentabilité du foncteur, on renvoie à \cite{Conrad}, §C.1. 


Le déterminant de Dickson-Dieudonné est défini au Chapitre II, §10 de \cite{Dieudonne}, dans le cas où $A$ est un corps. La définition schématique se trouve dans \cite{Conrad}, §C.2. Pour unifier la présentation, nous choisissons de prendre $\det_{\rm DD}$ à valeurs dans $\underline{\{\pm 1\}}$, plutôt que dans $\underline{\Z/2\Z}$ comme cela est fait \emph{loc. cit.}
\end{proof}

\begin{lemme}\label{lemme_det_DD_transpo}
Soit $(W,r)$ un espace quadratique déployé sur un corps de caractéristique $2$ (donc de dimension paire).
Soit $H$ un plan hyperbolique de $W$. Alors on a $W=H\overset{\perp}{\oplus}H^\perp$.

Soit $g \in \Oo(W,r)$ qui stabilise $H$. Alors $g$ stabilise $H^\perp$ et $\det_{\rm DD}(g)=\det_{\rm DD}(g_{|H})\det_{\rm DD}(g_{|H^\perp})$.

Le plan hyperbolique $H$ contient exactement deux droites isotropes, l'ensemble de ces deux droites est donc laissé stable par $g_{|H}$. On a alors $\det_{\rm DD}(g_{|H})=1$ si $g_{|H}$ stabilise chacune des deux droites et $\det_{\rm DD}(g_{|H})=-1$ si $g_{|H}$ les échange.
\end{lemme}
\begin{proof}
Ces propriétés sont classiques et découlent de manière élémentaire des définitions de \cite{Dieudonne}, Chapitre II, §10.
\end{proof}

\begin{lemme}\label{lemme_det_DD_parabolq_Siegel}
Soit $(W,r)$ un espace quadratique déployé sur un corps de caractéristique $2$ (donc de dimension paire).

Si $X$ est un sous-espace totalement isotrope maximal de $W$, alors les éléments de $\Oo(W,r)$ qui stabilisent $X$ sont de déterminant de Dickson-Dieudonné égal à $1$.
\end{lemme}
\begin{proof}
Notons $k$ le corps de base. On peut considérer l'espace quadratique $(W\otimes_k \bar{k},r\otimes_k \bar{k})$, où $\bar{k}$ est une clôture algébrique de $k$. Alors on peut définir le déterminant de Dickson-Dieudonné pour $\Oo(W\otimes_k \bar{k},r\otimes_k \bar{k})$ et, si l'on a un élément de $\Oo(W,r)$, il est équivalent de calculer son déterminant de Dickson-Dieudonné sur $k$ ou sur $\bar{k}$ (plus exactement dans $\Oo(W,r)$ ou dans $\Oo(W\otimes_k \bar{k},r\otimes_k \bar{k})$).

Travaillons donc sur $\bar{k}$, le déterminant de Dickson-Dieudonné est polynomial (au moins dans sa version à valeurs dans $\{x^2=x\}$, mais cela ne change rien à l'argument) donc est constant sur les composantes connexes de $\Oo(W\otimes_k \bar{k},r\otimes_k \bar{k})$. Il suffit alors de remarquer que les éléments considérés sont les éléments d'\emph{un} sous-groupe parabolique de Siegel (dépendant du choix du sous-espace totalement isotrope maximal), manifestement connexe, ce qui conclut.
\end{proof}
\section{Notations}\label{Notations}
On reprend les notations du paragraphe \ref{Groupes localement profinis}. Soit $F$ un corps local non-archimédien de caractéristique nulle (\ie une extension finie d'un corps $p$-adique). On note $\Ok$ son anneau des entiers et $\pk$ son idéal maximal. On choisit une uniformisante $\varpi$ telle que $\pk=\varpi \Ok$ et la valeur absolue $|\cdot |_F : F \rightarrow \R_+$ est normalisée par $|\varpi|_F=(\#\Ok/\pk)^{-1}$, le corps résiduel $k=\Ok/\pk$ étant bien un corps fini, dont on note $p$ la caractéristique.\ps 

Étant donné un $F$-espace vectoriel $V$ muni d'une forme quadratique $q$, on peut considérer le groupe algébrique sur $F$ des automorphismes de $V$ qui préservent la forme quadratique $q$, noté $\bm{\mathrm{O}_{V}}$, et son sous-groupe algébrique $\bm{\mathrm{SO}_{V}}$ (c'est la Proposition-Définition \ref{Prop-Def_gpe_orthogonal}). Ce dernier groupe algébrique est réductif sur $F$ (et même semi-simple si $\dim V >2$).
On a alors $\Oo(V,q)=\bm{\mathrm{O}_{V}}(F)$ et $\SO(V,q)=\bm{\mathrm{SO}_{V}}(F)$.


Nous nous intéressons ici au $F$-espace vectoriel $V_n$ de dimension $2n+1$ muni de la base $\mathcal{B}_0=(e_1,\cdots,e_n,v_0,f_n,\cdots,f_1)$ et de la forme bilinéaire symétrique définie par :
\begin{align*}
<e_i, f_j> & =\delta_{ij} ,\\
<v_0, v_0> & =2 ,\\
<e_i, e_j> &=<f_i,f_j>=0 ,\\
<e_i, v_0> &=<f_i,v_0>=0.
\end{align*}

On peut alors considérer la forme quadratique associée $q$ définie par $2q(x)=<x,x>$ et on a ainsi muni $V_n$ d'une structure d'espace quadratique sur $F$.\medskip

\emph{Remarque :} Dans le cas $2n+1=1$, dont nous aurons besoin pour la suite, le groupe $\bm{\mathrm{SO}_{V_0}}$ est trivial.\medskip

On observe que l'on a naturellement une décomposition
\begin{equation}\label{dec_canonique_V_n}
V_n=(X \oplus Y) \overset{\perp}{\oplus} F v_0,
\end{equation}
où $X=\overset{\perp}{\bigoplus} F e_i$ et $Y=\overset{\perp}{\bigoplus} F f_i$ sont deux \emph{lagrangiens} (sous-espaces totalement isotropes maximaux) transverses. En particulier, on note que $q$ est d'indice de Witt maximal $n=\lfloor \frac{2n+1}{2} \rfloor$.\medskip

Nous allons maintenant expliquer à quoi correspondent concrètement les sous-groupes présentés au paragraphe \ref{Notations pour les groupes reductifs}.


Le drapeau complet associé au choix de la base $\mathcal{B}_0$ nous permet de définir le sous-groupe (algébrique) de Borel standard, noté $\mathbf{B}$, des éléments de $\bm{\mathrm{SO}_{V_n}}$ qui préservent ce drapeau. Ce sont donc les éléments dont l'écriture matricielle dans la base $\mathcal{B}_0$ est triangulaire supérieure. Il est en fait suffisant qu'ils préservent le drapeau :
\[
\Vect(e_1)\subset \Vect(e_1,e_2) \subset \cdots \subset \Vect(e_1,\cdots,e_n).
\]

On notera $B=\mathbf{B}(F)$ sous-groupe de Borel standard de $G=\bm{\mathrm{SO}_{V_n}}(F)=\SO(V_n,q)$. Le sous-groupe $B$ contient le tore (maximal) de rang $n$ :
\[
T=\mathbf{T}(F)=
\left\{
\begin{pmatrix}
\lambda_1 & & & & & & \\
& \ddots & & & & & \\
& & \lambda_n & & & & \\
& & & 1 & & & \\
& & & & \lambda_n^{-1} & & \\
& & & & & \ddots & \\
& & & & & & \lambda_1^{-1}
\end{pmatrix},
(\lambda_1,\cdots, \lambda_n) \in (F^\times)^n\right\},
\]
si bien que le groupe $G$ est déployé sur $F$.
(On remarque que ceci utilise uniquement le fait que $q$ est non dégénérée et d'indice de Witt maximal, et pas sa forme particulière.)

On peut, plus généralement, définir des sous-groupes paraboliques standard $P$, contenant $B$ et qui préservent un drapeau compatible avec la base $\mathcal{B}_0$, il leur correspond des matrices triangulaires supérieures par blocs. Chaque sous-groupe parabolique admet une factorisation $P=MN$, où $N$ est le radical unipotent et $M$ est un sous-groupe de Levi, uniquement déterminé si on impose $T \subset M$ (ce que nous faisons dans la suite, on parlera de sous-groupe de Levi standard).\ps

Soit donc $\mathcal{P} :n=n_1+\cdots+n_k+\tilde{n}$ une partition de l'entier $n$ (avec éventuellement $\tilde{n}=0$). Le sous-groupe de Levi (standard) $M_\mathcal{P}$ associé à cette partition est constitué des éléments de $G$ qui préservent chacun des sous-espaces $\Vect(e_{n_1+\cdots+n_i+1},\cdots, e_{n_1+\cdots+n_{i+1}})$ pour $i \in \{1,\cdots,k\}$ (en posant $n_0=0$) ainsi que le sous-espace $V_{\tilde{n}}=\Vect(e_{n_1+\cdots+n_k+1},\cdots,e_n,v_0,f_n,\cdots,f_{n_1+\cdots+n_k+1})$.

Le respect des relations d'orthogonalité impose que les matrices des éléments de $M_\mathcal{P}$ dans la base $\mathcal{B}_0$ soient de la forme :

\[
\begin{pmatrix}
C_1 & & & & & & \\
& \ddots & & & & & \\
& & C_k & & & & \\
& & & D & & & \\
& & & & {}^\tau \! C_k^{-1} & & \\
& & & & & \ddots & \\
& & & & & & {}^\tau \! C_1^{-1}
\end{pmatrix}
\]
avec $C_i \in \GL_{n_i}(F)$, $D \in \SO(V_{\tilde{n}},\tilde{q})$ où $\tilde{q}$ est la restriction de $q$ à $V_{\tilde{n}}$, et où le pré-exposant ${}^\tau$ désigne la transposition par rapport à la seconde diagonale\footnote{${}^\tau \! (a_{ij})_{1\leq i,j \leq m}=(a_{m+1-j \; m+1-i})_{1\leq i,j \leq m}$} (qui commute bien à l'inversion, d'où l'absence de parenthèses). On a donc un isomorphisme canonique :
\begin{equation}\label{Levi_SO_temp}
M_{\mathcal{P}}\simeq \GL_{n_1}(F) \times \cdots \times \GL_{n_k}(F) \times \SO(V_{\tilde{n}},\tilde{q})
\end{equation}

On remarque que le tore maximal $T$ correspond au cas où $k=n$ et tous les $n_i$ sont égaux à 1, le groupe spécial orthogonal résiduel étant alors trivial ($\tilde{n}=0$).

%

\begin{defi}
Soit $\mathcal{B}=(\eps_1,\cdots,\eps_n,\upsilon,\varphi_n,\cdots,\varphi_1)$ une base de $V_n$. On dira que $\mathcal{B}$ est \emph{adaptée} si on a les relations suivantes :
\begin{align*}
<\eps_i, \varphi_j> & =\delta_{ij} ,\\
<\upsilon, \upsilon> & =2 ,\\
<\eps_i, \eps_j> &=<\varphi_i,\varphi_j>=0 ,\\
<\eps_i, \upsilon> &=<\varphi_i,\upsilon>=0.
\end{align*}
\end{defi}

La base $\mathcal{B}_0$ est alors une base adaptée et on remarque que le groupe orthogonal $\Oo(V_n,q)$ permute simplement transitivement les bases adaptées.

Si, partant d'une base adaptée quelconque, on réalise les mêmes constructions de sous-groupes standard que précédemment, on obtient donc des groupes qui sont conjugués (par $\Oo(V_n,q)$) à leurs équivalents pour le choix particulier de $\mathcal{B}_0$. Or, puisque l'on a $\Oo(V_n,q) \simeq \{\pm 1\} \times \SO(V_n,q)$ (nous sommes en dimension impaire), lesdits groupes sont en fait conjugués par $\SO(V_n,q)$.
\\


\emph{Remarque 1 :}
On peut se demander en quoi le groupe $G$ dépend de la forme particulière de la forme quadratique $q$. Soit donc une forme quadratique non dégénérée $r$ sur $F^{2n+1}$, on peut considérer le groupe $\SO(F^{2n+1},r)$. S'il est déployé, \ie s'il contient un tore $T_r$ déployé sur $F$ de rang $n$, alors on peut exhiber $n$ droites isotropes, contenues dans $n$ plans hyperboliques en somme directe orthogonale, le tout en somme directe orthogonale avec une droite non isotrope $Fv$. En particulier, la forme $r$ est nécessairement d'indice de Witt maximal, à savoir $n$, et son discriminant est donné par la classe de $(-1)^n r(v)$ dans $F^\times/(F^\times)^2$. 

Nous avons donc montré qu'à une forme quadratique (non dégénérée) d'indice de Witt maximal, on associait un groupe spécial orthogonal déployé, et que réciproquement, si le groupe spécial orthogonal d'une forme quadratique (non dégénérée) était déployé, alors celle-ci était d'indice de Witt maximal. De plus, la caractérisation par le discriminant nous dit qu'on a autant de formes quadratiques d'indice de Witt maximal inéquivalentes que d'éléments de $F^\times/(F^\times)^2$. 

En notant $\alpha=r(v)$, on a $\mathrm{disc}(r)=\mathrm{disc}(\alpha q)$ dans $F^\times/(F^\times)^2$ et, puisqu'il s'agit de deux formes quadratiques sur $F^{2n+1}$ d'indice de Witt maximal, on a $r \sim \alpha q$. Or, le groupe orthogonal est inchangé lorsqu'on multiplie la forme par un scalaire non nul et deux formes équivalentes ont des groupes orthogonaux canoniquement isomorphes. On a donc $\SO(F^{2n+1},r) \simeq \SO(F^{2n+1},\alpha q)= \SO(F^{2n+1},q)$, si bien qu'on a, à isomorphisme près, \emph{un seul} groupe spécial orthogonal déployé.

\paragraph{Notation}

On parlera dans la suite \emph{du} groupe spécial orthogonal déployé de dimension $2n+1$ sur $F$ et l'on notera, par raccourci, $\SO_{2n+1}$ pour $\bm{\SO_{V_n}}$ (en particulier $\SO_{2n+1}(F)$ désigne $\SO(V_n,q)$).

On peut alors réécrire l'isomorphisme \eqref{Levi_SO_temp} en :
\begin{equation}\label{Levi_SO}
M_{\mathcal{P}}\simeq \GL_{n_1}(F) \times \cdots \times \GL_{n_k}(F) \times \SO_{2\tilde{n}+1}(F).
\end{equation} \ps 

\emph{Remarque 2 :} On n'a pas utilisé la structure de corps local $p$-adique de $F$.


\section{Le groupe ${\rm K}_0$}\label{Le groupe K_0}

On se donne une base adaptée de $V_n$ (par exemple $\mathcal{B}_0$) et on considère le $\Ok$-module qu'elle engendre $L_0$ (qui est un \emph{réseau} de $V_n$). 

On peut alors considérer le schéma en groupes orthogonaux  $\bm{\Oo_{L_0}}$ associé au $\Ok$-module quadratique $(L_0,q_{|L_0})$ et le morphisme de schémas en groupes $\det : \bm{\Oo_{L_0}} \rightarrow \bm{\mu_2}$ dont le noyau est, par définition, $\bm{\SO_{L_0}}$. Ce schéma en groupes est adjoint, réductif (et même semi-simple si $n\geq 1$) sur $\Ok$ (\cite{Conrad}, Proposition C.3.10) et c'est un modèle entier de $\bm{\SO_{V_n}}$. On a donc que $\K_0=\bm{\SO_{L_0}}(\Ok)$ est un sous-groupe compact (maximal) hyperspécial de $G=\bm{\SO_{L_0}}(F)=\bm{\SO_{V_n}}(F)$.

\paragraph{Notations}
Pour le choix de la base adaptée $\mathcal{B}_0$, on note encore, par abus de notation, $\SO_{2n+1}$ le schéma en groupes $\bm{\SO_{L_0}}$ sur $\Ok$.

On définit également le schéma en groupes $\bm{\mathcal{B}}$ sur $\Ok$, sous-schéma en groupes de $\bm{\SO_{L_0}}$ préservant les $n$ $\Ok$-modules $\sum_{1 \leq i \leq k} \Ok e_i$, pour $k\in\{1,\cdots,n\}$. On remarque que l'on a $\bm{\mathcal{B}} \times_\Ok F \simeq \mathbf{B}$.

On définit enfin le schéma en groupes $\bm{\mathcal{T}}$ sur $\Ok$, sous-schéma en groupes de $\bm{\SO_{L_0}}$ (et de $\bm{\mathcal{B}}$) préservant les $n$ $\Ok$-modules $\Ok e_k$, pour $k\in\{1,\cdots,n\}$. On remarque que l'on a $\bm{\mathcal{T}} \times_\Ok F \simeq \mathbf{T}$. \newline


La construction de tels groupes hyperspéciaux dépend du choix d'une base adaptée, ils sont donc naturellement conjugués par $G$. 

On dispose de l'application naturelle de réduction modulo $\pk$,  $\pi_0 : \bm{\mathrm{SO}_{L_0}}(\Ok) \rightarrow \bm{\mathrm{SO}_{L_0}}(\Ok/\pk)$. 

\begin{defi}\label{defi Iwahori}
On définit le sous-groupe d'Iwahori (standard) $\mathrm{I}$ comme l'ensemble des éléments de $\K_0$ dont l'image par $\pi_0$ est dans $\bm{\mathcal{B}}(k)$.
\end{defi}

On mentionne enfin une propriété concernant l'intersection avec les sous-groupes de Levi.

\begin{prop}\label{Hypersp_inter_Levi}
Soit $M$ un sous-groupe de Levi standard de $\SO_{2n+1}(F)$, on a par \eqref{Levi_SO} :
\[
M \simeq \GL_{n_1}(F) \times \cdots \times \GL_{n_k}(F) \times \SO_{2\tilde{n}+1}(F),
\]
avec $\tilde{n}=n-(n_1+\cdots + n_k)$.\ps 


Alors $M \cap \K_0=\GL_{n_1}(\Ok) \times \cdots \times \GL_{n_k}(\Ok) \times \SO_{2\tilde{n}+1}(\Ok)$.
\end{prop}
\begin{proof}
C'est immédiat.
\end{proof}\ps 

\emph{Remarque :}
On observe que $\SO_{2\tilde{n}+1}(\Ok)$ est un sous-groupe compact maximal hyperspécial de $\SO_{2\tilde{n}+1}(F)$ (et c'est le seul à conjugaison près). L'intersection avec un sous-groupe de Levi fait donc apparaître un $\K_0$ \og plus petit \fg{} et préserve ainsi le type de sous-groupe, propriété qui s'avérera précieuse dans la suite.

\section{Le groupe $\mathrm{J}$}\label{Le_groupe_J}



La définition du groupe $\J$ (et du groupe $\J^+$ au paragraphe suivant) se trouve dans un article de Benedict Gross (\cite{Gross}). Elle généralise le cas de sous-groupes compacts ouverts canoniques pour $\SO_3\simeq \PGL_2$ et pour $\SO_5 \simeq {\rm PGSp}_4$ et Gross en fait remonter l'idée à des travaux d'Armand Brumer \cite{Brumer}.

Il sera plus naturel pour nous de définir d'abord le groupe $\J$ puis (au paragraphe suivant) un sous-groupe $\J^+$ d'indice 2. Dans le cas classique $\SO_3\simeq \PGL_2$, comme dans le cas $\SO_5 \simeq {\rm PGSp}_4$ étudié par Roberts et Schmidt dans \cite{RS-art} et \cite{RS_book}, c'est plutôt le sous-groupe d'indice 2 qui est défini d'abord (\cf début du §\ref{J^+}). Il existe alors un élément canonique de carré 1 normalisant le groupe en question, dit involution d'Atkin-Lehner, qui nous permet de définir l'analogue du groupe $\J$. (Les \og traductions \fg{} se trouvent dans \cite{Tsai-phd}, Chapter 6.)

Le groupe $\J$ est d'ailleurs noté $J_0(\varpi)$ dans \cite{Gross} et $\J(\pk)$ dans \cite{Tsai-phd} et ces auteurs définissent plus généralement une famille de sous-groupes notés $J_0(\varpi^m)$ (resp. $\J(\pk^m)$) pour $m$ entier naturel. \newline

Pour fixer les idées, commençons par raisonner avec la base $\mathcal{B}_0$, qui nous avait permis de définir le réseau $L_0$.

\begin{defi}
Considérons le réseau $L$ de $V_n$ défini par
\begin{equation}\label{reseau_L}
\begin{array}{rl}
\vspace{3 pt}
L &=\left(\bigoplus(\Ok e_i \oplus \pk f_i)\right)\oplus \pk v_0 \\ \vspace{3 pt}
  &=X_\Ok \oplus \varpi Y_\Ok \oplus \pk v_0 \\ 
  &=X_\Ok + \varpi L_0,\end{array}
\end{equation}
où $X_\Ok=\overset{\perp}{\bigoplus} \Ok e_i$ et $Y_\Ok=\overset{\perp}{\bigoplus} \Ok f_i$.\ps 

On pose $\mathcal{B}'=(e_1,\cdots,e_n, v'_0, f'_n,\cdots, f'_1)$, où $v'_0=\varpi v_0$ et $f'_i=\varpi f_i$ : c'est une $\Ok$-base du réseau $L$ (mais ce n'est pas une $F$-base adaptée de $V_n$).

On note $\mathrm{J}$ (ou $\J_{2n+1}$ si on veut spécifier la dimension de l'espace ambiant) le stabilisateur de $L$, c'est un sous-groupe compact ouvert de $G$, dit sous-groupe \emph{épiparamodulaire}. En particulier, les éléments de $\J$, vus dans la base $\mathcal{B}'$, sont à coefficients dans $\Ok$.
\end{defi}

De même que pour $\K_0$ ci-dessus, la définition de $\J$ dépend du choix d'une base adaptée, mais la classe de conjugaison du groupe $\J$ est, elle, canoniquement définie. Lorsque ce sera nécessaire, on raisonnera sur \og le \fg{} groupe $\J$ particulier associé au choix de la base $\mathcal{B}_0$.

\begin{prop}\label{prop_Iwahori_inclus_J}
Le sous-groupe d'Iwahori $\mathrm{I}$ est inclus dans $\mathrm{J}$.
\end{prop}
\begin{proof}
Par définition, on a $\mathrm{I}=\{g \in \mathrm{K}_0 \, | \, \forall i \in \{1,\cdots,n\}, g(e_i) \in \sum_{j\leq i} \Ok e_j \mod \varpi L_0\}$. On a alors immédiatement, pour $g \in \mathrm{I}, g(X_\Ok) \subset X_\Ok + \varpi L_0$ et, bien sûr, $g(\varpi L_0) \subset \varpi L_0$, donc $g$ préserve $L=X_\Ok + \varpi L_0$ (\cf \eqref{reseau_L}) et on a ainsi l'inclusion $\mathrm{I} \subset \mathrm{J}$. 
\end{proof}

\begin{prop}\label{Param_inter_Levi}
Soit $M$ un sous-groupe de Levi standard de $\SO_{2n+1}(F)$, on a par \eqref{Levi_SO} :
\[
M \simeq \GL_{n_1}(F) \times \cdots \times \GL_{n_k}(F) \times \SO_{2\tilde{n}+1}(F),
\]
avec $\tilde{n}=n-(n_1+\cdots + n_k)$.\ps 

Alors $M \cap \J=\GL_{n_1}(\Ok) \times \cdots \times \GL_{n_k}(\Ok) \times \J_{2\tilde{n}+1}$ où $\J_{2\tilde{n}+1}$ correspond au sous-groupe épiparamodulaire de $\SO_{2\tilde{n}+1}(F)$. 
\end{prop}
\begin{proof}
Commençons par préciser un peu les notations. On note $n'=n_1+\cdots+n_k$. Alors 
notons $\tilde{L_0}$ le $\Ok$-module libre engendré par $e_{n'+1},\cdots,e_n,v_0,f_n,\linebreak \cdots,f_{n'+1}$. C'est un $\Ok$-module quadratique et un réseau de $\tilde{L_0} \otimes F \simeq F^{2\tilde{n}+1}$, on peut alors considérer le schéma en groupes $\bm{\mathrm{SO}_{\tilde{L_0}}}$, défini sur $\Ok$. Le groupe de ses $F$-points est déployé et on le note, conformément  aux conventions précédentes, $\SO_{2\tilde{n}+1}(F)$.

On peut, comme ci-dessus, définir le sous-réseau $\tilde{L}$ de $\tilde{L_0}$ et regarder son stabilisateur, que l'on note $\J_{\tilde{n}}$.

Rappelons la forme que prennent les éléments du sous-groupe de Levi $M$ {\bfseries dans la base $\mathcal{B}_0$ }:

\[
\begin{pmatrix}
C_1 & & & & & & \\
& \ddots & & & & & \\
& & C_k & & & & \\
& & & D & & & \\
& & & & {}^\tau \! C_k^{-1} & & \\
& & & & & \ddots & \\
& & & & & & {}^\tau \! C_1^{-1}
\end{pmatrix}
\]
avec $C_i \in \GL_{n_i}(F)$ et $D \in \SO_{2\tilde{n}+1}(F)$ où
le pré-exposant ${}^\tau$ désigne la transposition par rapport à la seconde diagonale (qui commute bien à l'inversion, d'où l'absence de parenthèses). La présence simultanée de $C_i$ et ${}^\tau \! C_i^{-1}$ vient du respect des relations d'orthogonalité.

Si l'on s'intéresse maintenant à $M\cap \J$, il s'agit de voir quels sont les éléments de cette forme qui préservent le réseau $L$. {\bfseries Dans la base $\mathcal{B}'$}, on a l'écriture
\[
\begin{pmatrix}
C_1 & & & & & & \\
& \ddots & & & & & \\
& & C_k & & & & \\
& & & D' & & & \\
& & & & {}^\tau \! C_k^{-1} & & \\
& & & & & \ddots & \\
& & & & & & {}^\tau \! C_1^{-1}
\end{pmatrix}.
\]

\noindent (Seul le bloc matriciel central est modifié.) La contrainte sur $C_i$ est donc d'être à coefficients entiers avec ${}^\tau \! C_i^{-1}$ également à coefficients entiers
, \ie $C_i \in \GL_{n_i}(\Ok)$. 
Quant à $D$ (ou $D'$), elle doit préserver le réseau $\tilde{L}$ et en fait le stabiliser puisque la surjectivité générale vers $L$ et la forme de la matrice impliquent que $D$ envoie $\tilde{L}$ surjectivement sur $\tilde{L}$. C'est donc bien un élément de $\J_{2\tilde{n}+1}$.

Réciproquement, étant données des matrices dans les ensembles considérés, la construction de la matrice par blocs correspondante nous donne bien un élément de $M \cap \J$, si bien qu'on a effectivement  $M \cap \J=\GL_{n_1}(\Ok) \times \cdots \times \GL_{n_k}(\Ok) \times \J_{2\tilde{n}+1}$.
\end{proof}

\section{Le groupe $\mathrm{J}^+$}\label{J^+}
Il s'agit maintenant de définir un sous-groupe de $\J$, d'indice 2, que nous noterons $\J^+$. Il est noté $K(\varpi)$ dans \cite{Gross} et $\K(\pk)$ dans \cite{Tsai-phd}. Ce sous-groupe généralise le sous-groupe de congruence $\Gamma_0(\pk)$ pour $\SO_3\simeq\PGL_2$ et le sous-groupe paramodulaire étudié par Roberts et Schmidt dans \cite{RS-art} et \cite{RS_book}, noté $\K(\pk)$, pour $\SO_5\simeq{\rm PGSp_4}$. 

Il faut d'ailleurs noter que ces auteurs étudient en fait le cas plus général d'une famille de sous-groupes $\K(\pk^m)$ pour $m$ entier naturel (à mettre en parallèle avec la famille des $\Gamma_0(\pk^m)$). \newline

Un élément $g$ de $\mathrm{J}$ préserve le réseau $L$. Sa matrice dans la base $\mathcal{B}'$ est à coefficients dans $\Ok$, si bien qu'on peut réduire cette matrice modulo $\pk$ pour obtenir une matrice de $k$. 
L'élément $g$ préserve la forme quadratique $q$ (définie sur $L_0 \supset L$) donc également n'importe lequel de ses multiples scalaires, en particulier $q'=\frac{q}{\varpi}$, qui vérifie $q'(L) \subset \Ok$. 

L'élément $g$, vu comme automorphisme linéaire de $V'_k=\Vect_k(\mathcal{B}')$ préserve donc $q'\otimes k$, on a ainsi un élément de $\mathrm{O}(V'_k,q'\otimes k)$ et même de $\mathrm{SO}(V'_k,q'\otimes k)$ puisque le déterminant \og passe à la réduction modulo $\pk$ \fg. On a donc défini :

\begin{equation}\label{pi_J_vers_SO}
\pi : \J \longrightarrow \mathrm{SO}(V'_k,q'\otimes k).
\end{equation}

Il faut cependant noter que la forme $q' \otimes k$ est dégénérée sur $V'_k$. En effet le vecteur $v'_0$ qui était déjà orthogonal à tous les autres vecteurs de $\mathcal{B'}$ avant réduction, l'est également à lui-même après réduction (on a $q(v'_0)= \varpi^2$ donc $q'(v'_0)=\varpi$ et $(q'\otimes k)(v'_0)=0$). Plus précisément, on a $\Ker (q'\otimes k)=k v'_0$, puisque $q'\otimes k$ est non dégénérée sur l'espace quotient $V'_k /k v'_0$ qui est (isomorphe à) l'espace engendré sur $k$ par les $e_i$ et les $f'_i$. On notera $\bar{q'}$ la forme sur l'espace quotient $V'_k/kv'_0$ et $\bar{g}$ l'endomorphisme correspondant, qui est donc un élément de $\mathrm{O}(V'_k/kv'_0,\bar{q'})$. On a ainsi défini :

\begin{equation}\label{J_vers_O_pair}
\begin{array}{ccccc}
\bar{\pi} & : & \J & \longrightarrow & \mathrm{O}(V'_k/kv'_0,\bar{q'}) \\
 & & g & \longmapsto & \bar{g} \\
\end{array}
\end{equation}
comme composée du morphisme $\pi$ défini ci-dessus suivi de la réduction modulo le noyau de $q'\otimes k$.

On définit alors $d : \mathrm{O}(V'_k/kv'_0,\bar{q'}) \rightarrow \{\pm 1\}$ par $d=\det$ quand ${\rm car}(k) \neq 2$ et par $d=\det_{\rm DD}$ quand ${\rm car}(k)=2$.

\begin{prop-def}\label{prop-def_J+}
On note $\alpha : \J_{2n+1} \rightarrow \{\pm1\}$ le morphisme obtenu en composant $\bar{\pi}$ et $d$.
On note $\J^+_{2n+1}$ (ou $\J^+$) le noyau de ce morphisme, c'est un sous-groupe compact ouvert de $G$, dit sous-groupe \emph{paramodulaire}.
Le morphisme $\alpha$ est surjectif si $2n+1>1$. 
\end{prop-def}

%
\begin{proof}
Il s'agit donc de voir que le morphisme $\alpha$ est surjectif (si $2n+1>1$), quelle que soit la caractéristique. Considérons l'élément $u$ qui échange $e_1$ et $f'_1$, envoie $v'_0$ sur $-v'_0$ et laisse les autres vecteurs de base inchangés. On a $u \in G$, $u$ préserve $L$ et, après réductions, il s'agit d'évaluer le déterminant (resp. le déterminant de Dickson-Dieudonné, \cf Lemme \ref{lemme_det_DD_transpo})
d'un élément qui échange les deux droites isotropes d'un plan hyperbolique (et fixe l'orthogonal dudit plan hyperbolique), qui vaut dans tous les cas $-1$.
\end{proof}\ps


\emph{Remarque 1 :} Si $2n+1=1$, les étapes ci-dessus imposent $\J^+_1=\J_1=\{1\}$.\medskip

\emph{Remarque 2 :} Le noyau de $q'\otimes k$ est préservé par son groupe orthogonal, si bien que $\pi(g)$ agit sur $v'_0$ par un scalaire $\beta \in k^\times$. Le déterminant de $\pi(g)$ est donc donné par le produit de $\beta$ et du déterminant \emph{naïf} de $\bar{g}$. Or, $\pi(g)$ est dans le groupe spécial orthogonal si bien que $\beta=\det (\bar{g})^{-1}=\det (\bar{g})$. En particulier, $\beta \in \{ \pm 1 \}$ et, {\bfseries en caractéristique différente de $2$}, on peut déterminer l'appartenance à $\mathrm{J}^+$ en regardant simplement ce coefficient.\medskip

\emph{Remarque 3 :} Sur le corps fini $k$, il n'y a que deux classes d'équivalence de formes quadratiques non dégénérées qui correspondent au même groupe orthogonal en dimension impaire (les deux formes quadratiques sont d'indice de Witt maximal, on peut employer les mêmes arguments qu'au paragraphe \ref{Notations}) et à deux groupes orthogonaux non isomorphes en dimension paire $2n$. En effet, dans ce cas, les deux formes quadratiques inéquivalentes sont respectivement d'indice de Witt $n$ et $n-1$ et on aura donc un groupe orthogonal déployé (dans le cas de l'indice de Witt maximal) et l'autre non. Dans le cas considéré ici, on sait donc, puisque la forme $\bar{q'}$ est d'indice de Witt $n$, à quel groupe orthogonal on a affaire, ce groupe est noté $\mathrm{O}_{2n}^+(k)$.

On peut résumer la situation par le diagramme suivant :

\begin{equation*}
\begin{tikzcd}
\bar{\pi} : \J \arrow[r,"\pi"] \arrow[drr,"\alpha"] & \mathrm{SO}(V'_k,q'\otimes k)  \arrow{r} & (\{\pm 1\} \times \mathrm{O}_{2n}^+(k))^{\mathrm{det}=1} \arrow[d,dashrightarrow,"1\times d"] \\
 & & \{\pm 1\}
\end{tikzcd}.
\end{equation*}

On a \emph{de facto} légèrement redéfini $\bar{\pi}$ puisque le groupe d'arrivée n'est pas tout à fait le même, quoique isomorphe.
\begin{prop}\label{prop_I_inclus_J+}
Le sous-groupe d'Iwahori $\mathrm{I}$ défini en \ref{defi Iwahori} est non seulement inclus dans $\mathrm{J}$, mais plus précisément est inclus dans $\mathrm{J}^+$.

Cette inclusion est une égalité si $2n+1=3$ (et si $2n+1=1$), et est stricte sinon.
\end{prop}
\begin{proof}
Soit $g$ un élément de ${\rm I}$. En le réduisant modulo $\pk$, on obtient un élément triangulaire supérieur avec les conditions d'orthogonalité qui imposent que les coefficients diagonaux (réduits) soient de la forme $(\alpha_1, \cdots \alpha_n, \beta, \alpha_n^{-1}, \cdots,\linebreak \alpha_1^{-1})$. La condition de déterminant 1 impose donc $\beta=1$, ce qui suffit à conclure que $g \in \J^+$ si ${\rm car}(k) \neq 2$ d'après la Remarque 2. En caractéristique 2, on a encore $g\in \J^+$ en utilisant le Lemme \ref{lemme_det_DD_parabolq_Siegel}.\ps 


Pour le second point, si $2n+1>3$, considérons l'élément $g$ qui échange $e_1$ et $e_2$, ainsi que $f_1$ et $f_2$ en laissant tous les autres vecteurs de la base $\mathcal{B}_0$ inchangés. Alors $g$ préserve les réseaux $L_0$ et $L$, si bien que c'est un élément de $\J$ (et de $\K_0$) dont on voit immédiatement (c'est un produit d'éléments qui échangent les deux droites isotropes d'un plan hyperbolique et fixent l'orthogonal dudit plan hyperbolique) qu'il est dans $\J^+$. Or $g(e_1) \not\in \Ok e_1 + \pk L_0$, si bien que $g \not\in \mathrm{I}$.\ps 

Dans le cas $2n+1=1$, tous les groupes considérés sont triviaux et il n'y a rien à dire. Dans le cas $2n+1=3$, il faut voir que tout élément $g$ de $\J^+$ est dans $\mathrm{I}$. On notera $e$ pour $e_1$ et $f$ pour $f_1$. Puisque $g$ préserve $L$, on a $g(e) \in \Ok e + \pk v_0 + \pk f$ si bien que la condition qui définit le sous-groupe d'Iwahori $\mathrm{I}$ est vérifiée ; il reste simplement à voir que $g$ appartient à $\K_0$,
\ie $g(v'_0) \in \varpi L_0$ et $g(f') \in \varpi L_0$.

Après réduction modulo $\pk$, $k v'_0$ est le noyau de $q' \otimes k$ et est donc préservé par $\pi(g)$ (\cf \eqref{pi_J_vers_SO}), soit $g(v'_0) \in v'_0+\varpi L_0 \subset \varpi L_0$.

Par ailleurs, le fait que $g$ soit dans $\J^+$ signifie que $\bar{\pi}(g)$ (\cf \eqref{J_vers_O_pair}) est de déterminant (de Dickson-Dieudonné, le cas échéant) égal à 1, \ie qu'il stabilise chacune des deux droites isotropes du plan hyperbolique $ke\oplus kf'$. Cela signifie exactement que $g(f') \in \varpi L_0$ et achève la démonstration.
%
%
%
\end{proof}

\begin{prop}\label{Param+_inter_Levi}
Soit $M$ un sous-groupe de Levi standard de $\SO_{2n+1}(F)$, on a par \eqref{Levi_SO} :
\[
M \simeq \GL_{n_1}(F) \times \cdots \times \GL_{n_k}(F) \times \SO_{2\tilde{n}+1}(F),
\]
avec $\tilde{n}=n-(n_1+\cdots + n_k)$.\ps 

Alors $M \cap \J^+=\GL_{n_1}(\Ok) \times \cdots \times \GL_{n_k}(\Ok) \times \J_{2\tilde{n}+1}^+$ où $\J^+_{2\tilde{n}+1}$ correspond au sous-groupe paramodulaire de $\SO_{2\tilde{n}+1}(F)$.
\end{prop}
\begin{proof}
En reprenant les notations de la preuve de la Proposition \ref{Param_inter_Levi}, le déterminant à calculer dans $V'_k/kv'_0$ ne va pas faire intervenir les matrices $C_i$ car $\det({}^\tau \! C_1^{-1})=\det(C_i)^{-1}$, égalité qui demeure après réduction modulo $\pk$. Le calcul ne dépend donc que de la matrice $D$ et un élément diagonal par blocs est dans $\J^+$ si, et seulement si $D$ est dans $\J^+_{2\tilde{n}+1}$.
\end{proof}

\begin{lemme}\label{lemme_BcapJ}
On a :
\[
B \cap \J \subset \J^+.
\]
De plus, l'image par la projection $B \twoheadrightarrow B/N \simeq T$ envoie $B \cap \J$ sur $\bm{\mathcal{T}}(\Ok)$. 
\end{lemme}

\begin{proof}
Considérons un élément $m \in B \cap \J$. Alors dans la base $\mathcal{B}'$, on a :
\[
m=\begin{pmatrix}
\lambda_1 &  &* &* &* &* &* \\
& \ddots & &* &* &* &* \\
& & \lambda_n & &* &* &* \\
& & & 1 & &* &* \\
& & & & \lambda_n ^{-1} & &* \\
& & & & & \ddots & \\
& & & & & & \lambda_1^{-1}
\end{pmatrix}
\]
avec les $\lambda_i$ et les $\lambda_i^{-1}$ dans $\Ok$. Ainsi, les coefficients diagonaux sont dans $\Ok^\times$, et ceux-ci sont inchangés dans la base $\mathcal{B}_0$, ce qui prouve le deuxième point.

Pour le premier point, il faut réduire modulo $\pk$ {\it et} modulo le noyau de la forme désormais dégénérée sur $k$ (ce noyau est canonique et est la droite vectorielle engendrée par $v_0$). On obtient alors un élément de $\mathrm{O}_{2n}^+(k)$ dont il s'agit de connaître le déterminant (resp. le déterminant de Dickson-Dieudonné) pour savoir s'il est dans $\J^+$ ou non. Or l'élément réduit en question est :
\[
\begin{pmatrix}
\overline{\lambda_1} &  &* &* &* &*  \\
& \ddots & &* &* &*  \\
& & \overline{\lambda_n} & &* &*  \\
& & &  \overline{\lambda_n ^{-1}} & &* \\
& & &  & \ddots & \\
& & &  & & \overline{\lambda_1^{-1}}
\end{pmatrix},
\]
 (où la barre désigne la réduction dans $k$), qui est trivialement de déterminant 1 et, si ${\rm car}(k)=2$, de déterminant de Dickson-Dieudonné égal à 1, d'après le Lemme \ref{lemme_det_DD_parabolq_Siegel}.
\end{proof}

\section{Norme spinorielle}\label{Norme spinorielle}

Nous allons voir ici une autre caractérisation du sous-groupe $\mathrm{J}^+$ de $\mathrm{J}$ grâce à la norme spinorielle. Cette caractérisation est, à notre connaissance, originale.\medskip

Considérons d'abord un espace quadratique $(V,q)$ non dégénéré sur $F$, au sens de la Proposition-Définition \ref{Prop_def_FQ}.

Le groupe $\SO (V,q)$ est un groupe spécial orthogonal sur le corps $F$, on peut définir le morphisme \og norme spinorielle \fg{} $\nu : \SO (V,q) \rightarrow F^\times / (F^\times)^2$. Ce morphisme peut être défini à partir de l'algèbre de Clifford de $V$ et du morphisme $\mathrm{GSpin}(V,q) \rightarrow \SO(V,q)$.  On peut aussi définir directement $\nu$ sur les réflexions de $\Oo (V,q)$ en posant $\nu(\tau_x)=q(x)$ où $\tau_x$ désigne la réflexion par rapport au vecteur $x$ (ou à la droite $Fx$), nécessairement non isotrope (sans quoi la réflexion n'est pas définie). On a $\tau_{\lambda x}=\tau_x$ pour $\lambda \in F^\times$ et $q(\lambda x)=\lambda^2 q(x)$ ; $\nu(\tau_x)$ est donc bien défini dans $F^\times / (F^\times)^2$. Tout élément de $\SO (V,q)$ est produit (d'un nombre pair) de réflexions et on peut ainsi calculer $\nu$ sur $\SO (V,q)$, le résultat ne dépendant pas de la décomposition comme produit de réflexions (\cite{OMeara}, §55).\medskip


Étant donné un caractère $\chi$ de $F^\times$ trivial sur $(F^\times)^2$ (on parlera de caractère quadratique), on peut donc considérer le caractère $\chi \circ \nu$ de $\SO (V,q)$ (et de ses sous-groupes). Nous commettrons parfois l'abus de notation consistant à désigner encore ce caractère par $\chi$.

Deux caractères quadratiques de $F^\times$ vont être d'une importance particulière pour la suite : le caractère trivial (d'ordre exactement 1) et le caractère $\eta$, unique caractère non ramifié d'ordre {\it exactement} 2, \ie défini par $\eta_{|\Ok^\times}=1$ et $\eta(\varpi)=-1$. Puisque deux uniformisantes diffèrent par un élément de $\Ok^\times$, on voit en particulier que la définition de $\eta$ ne dépend pas du choix particulier de $\varpi$.

Les caractères $\1$ et $\eta$ sont les seuls caractères quadratiques non ramifiés de $F^\times$.\medskip


On se restreint désormais au seul $F$-espace quadratique $(V_n,q)$ introduit au paragraphe \ref{Notations} et à son groupe spécial orthogonal $\SO (V_n,q)$, noté $\SO_{2n+1}(F)$.

Nous allons démontrer le théorème suivant.

\begin{thm}\label{thm_J+_noyau_norme_spin}
Le morphisme $\eta \circ \nu :\SO_{2n+1}(F) \rightarrow \{\pm 1 \}$, restreint à $\J$ est surjectif quand $2n+1>1$. Son noyau est exactement $\J^+$.
\end{thm}

Nous supposons dans toute la suite que $2n+1>1$. Commençons par démontrer la factorisation suivante.

\begin{prop}
On a la factorisation $\mathrm{J}=(\mathrm{K}_0 \cap \mathrm{J}) \cdot (\mathrm{Stab}_\mathrm{J} \; \pk v_0)$.
\end{prop}
\begin{proof}
Il faut en premier lieu caractériser les éléments de $\mathrm{K}_0 \cap \mathrm{J}$ parmi les éléments de $\mathrm{J}$. Raisonnons à l'aide des décompositions $L_0=X_\Ok \oplus Y_\Ok \oplus \Ok v_0$ et $L=X_\Ok \oplus \pk Y_\Ok \oplus \pk v_0$ des réseaux qui servent à définir $\mathrm{K}_0$ et $\mathrm{J}$.\ps 

Soit $g \in \mathrm{K}_0 \cap \mathrm{J}$. L'élément $g$ préserve le réseau $L$ ; traduisons le fait qu'il préserve également le réseau $L_0$ (ou de façon équivalente le réseau $\pk L_0$) par rapport à chacun des termes de la somme directe définissant $L$ :\ps 
\begin{itemize}
\item $g$ envoie $X_\Ok$ dans $L_0$ (ne dit rien de plus puisque $L \subset L_0$) ;\ps 
\item $g$ envoie $\pk Y_\Ok$ dans $\pk L_0=\pk X_\Ok \oplus \pk Y_\Ok \oplus \pk v_0$ ; \ps 
\item $g$ envoie $\pk v_0$ dans $\pk L_0$ (ne dit rien de plus puisque $\pi(g)$ préserve le noyau $k v'_0$ de la forme $q' \otimes k$ et donc $g$ envoie $\pk v_0$ dans $\pk v_0 + \pk L \subset \pk L_0$).\ps 
\end{itemize}

On voit alors qu'une condition nécessaire et suffisante pour qu'un élément de $\J$ appartienne également à $\K_0$ est le fait d'envoyer $\pk Y_\Ok$ dans $\pk L_0$. Ceci se traduit de façon agréable à l'aide du morphisme $\pi$. L'élément $\pi(g)$ est une isométrie de l'espace quadratique $V'_k$, regardons sa matrice dans la base $(v'_0,f'_n,\cdots,f'_1,e_1,\linebreak \cdots, e_n)$ (c'est la base $\mathcal{B}'$ à une permutation circulaire près) : 
\[
\left(\begin{array}{c|c|c}
* & *  & * \\\hline
0 & C & * \\\hline
0 & 0 & D
\end{array}\right)
.\]
Le respect des relations d'orthogonalité entre l'image de $e_i$ et celle de $f'_i$ pour $i \in \{1,\cdots n\}$ impose d'ailleurs que $D={}^\tau \! C^{-1}$ où ${}^\tau$ désigne la transposition par rapport à la seconde diagonale (qui commute bien à l'inversion, d'où l'absence de parenthèses). On a ainsi $\bar{\pi}(g)$ (qui correspond à la matrice sur $V'_k/kv'_0$ et donc aux quatre blocs en bas à droite) égal à
$\left(\begin{array}{c|c}
C & * \\\hline
0 & {}^\tau \! C^{-1}
\end{array}\right),$
matrice de déterminant 1 (et, le cas échéant, de déterminant de Dickson-Dieudonné égal à 1 d'après le Lemme \ref{lemme_det_DD_parabolq_Siegel}) donc l'élément $g$ est dans $\J^+$. 

On a finalement prouvé que $\J \cap \K_0$ est l'image inverse par le morphisme $\bar{\pi}$ de \eqref{J_vers_O_pair} des matrices triangulaires supérieures par blocs dans la $k$-base \linebreak $(\bar{f'_n},\cdots, \bar{f'_1}, \bar{e_1},\cdots,\bar{e_n})$, 
qui sont de la forme
$\left(\begin{array}{c|c}
* & * \\\hline
0 & *
\end{array}\right)$ ou même, plus précisément 
$\left(\begin{array}{c|c}
C & * \\\hline
0 & {}^\tau \! C^{-1}
\end{array}\right)$. De plus, on a  $\J \cap \K_0 \subset \J^+$.\ps 

Pour obtenir la factorisation souhaitée, il suffit donc de voir qu'en modifiant un élément de $\J$ par un élément de $S=\mathrm{Stab}_\mathrm{J} \; \pk v_0$, on peut imposer que son image par $\bar{\pi}$ ait la forme
$\left(\begin{array}{c|c}
* & * \\\hline
0 & *
\end{array}\right)$.

On va faire mieux : on va imposer qu'il ait la forme
$\left(\begin{array}{c|c}
I_n & 0 \\\hline
0 & I_n
\end{array}\right)$, ce pour quoi il suffit de voir que $\J$ et $S$ ont la même image par $\bar{\pi}$. Et on montre ce dernier résultat en montrant que $\bar{\pi}_{|S}$ est surjectif sur $\Oo^+_{2n}(k)$.\ps 

Étudions de plus près le groupe $S$. Un élément de $S$ :\ps 
\begin{itemize}
\item agit sur $\varpi v_0$ par un scalaire, nécessairement $\pm 1$ par préservation de la norme ; \ps 
\item préserve l'orthogonal de $\pk v_0$ à savoir $X_\Ok \oplus \pk Y_\Ok$. \ps 
\end{itemize}
On a donc $S=(\{\pm 1 \} \times \Oo(X_\Ok \oplus \pk Y_\Ok,q'))^{\mathrm{det}=1}$. \medskip 

Si on considère la forme $q'$, elle est non dégénérée sur $(X_\Ok \oplus \pk Y_\Ok) \otimes k$  et donc par lissité (\cite{Conrad}, Theorem C.1.5), l'application de réduction $\Oo(X_\Ok \oplus \pk Y_\Ok,q') \rightarrow \Oo((X_\Ok \oplus \pk Y_\Ok)\otimes k,q' \otimes k)$ est surjective, le groupe d'arrivée s'identifiant d'ailleurs à $\Oo^+_{2n}(k)$ par la Remarque 3 du paragraphe \ref{J^+}. \medskip

Récapitulons. Si l'on considère $g \in \J$, on a $\bar{\pi}(g) \in \Oo^+_{2n}(k)$. Or, par la surjectivité depuis $S$ précédemment mentionnée (qui a pour conséquence que $\bar{\pi}$ est bien surjective au départ de $\J$), on peut trouver $s \in S$ tel que $\bar{\pi}(s)=\bar{\pi}(g)$, soit $gs^{-1} \in \Ker \bar{\pi} \subset (\K_0 \cap \J)$ et $g \in (\K_0 \cap \J) \cdot S$.

On a donc bien la factorisation recherchée (on a même mieux : $\J= (\Ker \bar{\pi})\cdot S$).
\end{proof}

\begin{cor}\label{facto_J+}
On a la factorisation $\mathrm{J}^+=(\mathrm{K}_0 \cap \mathrm{J}) \cdot (\mathrm{Fix}_\mathrm{J} \; \varpi v_0)$.
\end{cor}
\begin{proof}
On pose $S^+=\mathrm{Fix}_\mathrm{J}\; \varpi v_0$ et on a $S^+=\J^+ \cap \mathrm{Stab}_\mathrm{J} \; \pk v_0 \simeq \SO(X_\Ok \oplus \pk Y_\Ok,q')$. Le même raisonnement que précédemment nous montre que $\bar{\pi}_{|S^+}$ est surjectif sur $\SO^+_{2n}(k)$ (et donc $\bar{\pi}$ est surjectif depuis $\J^+$ \emph{a fortiori}). On conclut comme ci-dessus, avec la même amélioration $\J^+= (\Ker \bar{\pi})\cdot S^+$.
\end{proof}

Évaluons maintenant la norme spinorielle grâce à ces factorisations.
Le $\Ok$-module quadratique $(L_0,q)$ est régulier (au sens de la Proposition-Définition \ref{Prop_def_FQ}) et donc, par \cite{Knu} (6.2.3 p.231), 
$\nu(\K_0) \subset \Ok^\times / (\Ok^\times)^2$. En particulier, le caractère $\eta$ introduit plus haut étant non ramifié, on a $(\eta \circ \nu) (\K_0) =1$.

On a $S^+=\{1\}\times \SO(X_\Ok \oplus \pk Y_\Ok,q')$. Or le $\Ok$-module quadratique $(X_\Ok \oplus \pk Y_\Ok,q')$ est régulier, si bien que, par le même argument\footnote{Il faut remarquer que $q'$ et $q$ diffèrent d'un scalaire de $F^\times$ et que, puisque les éléments considérés sont produits d'un nombre pair de réflexions, les normes spinorielles calculées avec $q$ ou avec $q'$ sont bien égales modulo $(F^\times)^2$.}, $\nu(S^+) \subset \Ok^\times / (\Ok^\times)^2$.\ps

%

Par le Corollaire \ref{facto_J+}, on a donc $\eta \circ \nu$ trivial sur $\J^+$. Le morphisme $\eta \circ \nu$, d'ordre au plus 2 sur $\J$, est trivial sur un sous-groupe d'indice 2 : soit il est trivial sur $\J$ tout entier, soit il est non trivial et son noyau est exactement ce sous-groupe d'indice 2 (à savoir $\J^+$). Pour conclure, il suffit donc d'exhiber un élément de $\J$ dont l'image par $\eta \circ \nu$ vaut $-1$.\ps 

Soit $u$ l'élément qui échange $e_1$ et $\varpi f_1$, agit par $-1$ sur la droite $F v_0$ et laisse les autres $e_i$ et $f_i$ inchangés (c'est le même élément $u$ qui a été introduit au paragraphe \ref{J^+}). Alors $u$ est bien un élément de $\SO_{2n+1}(F)$, il préserve $L$, c'est donc un élément de $\J$. On peut écrire $u$ comme la composée des réflexions $\tau_{v_0}$ et $\tau_{e_1-\varpi f_1}$ (peu importe l'ordre puisque les réflexions commutent). On a donc $\nu(u)=q(v_0).q(e_1- \varpi f_1)=1.(-\varpi)=-\varpi$ et $\eta(-\varpi)=\eta(-1).\eta(\varpi)=1.(-1)=-1$ puisque $-1 \in \Ok^\times$.

Ceci achève la preuve du théorème et nous avons donc bien $\J^+=\Ker (\eta \circ \nu)$.

\begin{prop}\label{prop_norme_spin_Levi}
Soit $M$ un sous-groupe de Levi standard de $\SO_{2n+1}(F)$, on a par \eqref{Levi_SO} :
\[
M \simeq \GL_{n_1}(F) \times \cdots \times \GL_{n_k}(F) \times \SO_{2\tilde{n}+1}(F),
\]
avec $\tilde{n}=n-(n_1+\cdots + n_k)$.\ps 

Un élément de $M$ correspond à une matrice diagonale par blocs : $X=\mathrm{diag}(C_1,\cdots,C_k,D,{}^\tau \! C_k^{-1},\cdots,{}^\tau \! C_1^{-1})$. Alors $\nu(X)= \prod_i \overline{\det(C_i)} \nu(D)$, où l'on note encore $\nu$ la norme spinorielle sur $\SO_{2\tilde{n}+1}(F)$, et où la barre désigne la réduction modulo $(F^\times)^2$.
\end{prop}
\begin{proof}
Au sous-groupe de Levi $M$ correspond le sous-espace quadratique :
\[
(V_1 \oplus V'_1) \overset{\perp}{\oplus} \cdots \overset{\perp}{\oplus} (V_k \oplus V'_k) \overset{\perp}{\oplus} V_{\tilde{n}}
\]
avec $V_i$ et $V'_i$ sous-espaces totalement isotropes transverses de dimension $n_i$ avec $q$ non dégénérée sur $(V_i \oplus V'_i)$ ; et $V_{\tilde{n}}$ espace quadratique de dimension $2\tilde{n}+1$ sur lequel $q$ est d'indice de Witt maximal $\tilde{n}$.

La norme spinorielle d'un élément de $M$ est alors le produit des normes spinorielles définies sur chaque terme de la somme orthogonale (\cite{OMeara}, 55:4). Il suffit donc de comprendre la norme spinorielle d'un élément $\mathrm{diag}(C,{}^\tau \! C^{-1})$ où $C \in \GL_m(F)$ pour pouvoir conclure.\medskip

Or $\GL_m(F)$ est engendré par les transvections et les dilatations, il suffit donc de savoir calculer la norme spinorielle de chaque élément correspondant. Notons $(e_1,\cdots,e_m,f_m,\cdots,f_1)$ une base de l'espace quadratique $W \oplus W'$ de dimension $2m$ associé avec $<e_i,f_j>=\delta_{ij}$ et $<e_i,e_j>=<f_i,f_j>=0$. 

Considérons pour $\lambda \in F- \{0,1\}$ la dilatation $C$ d'hyperplan $\Vect_F(e_2,\cdots,\linebreak e_m)$, de droite $Fe_1$ et de rapport $\lambda$. (Considérer uniquement ces dilatations particulières suffit pour la suite.) Alors l'élément $\mathrm{diag}(C,{}^\tau \! C^{-1})$ de $\SO(W \oplus W')$ peut s'écrire comme la composée des réflexions $\tau_{e_1-\lambda f_1}$ suivie de $\tau_{e_1-f_1}$. En particulier, sa norme spinorielle est égale à $q(e_1-f_1)q(e_1-\lambda f_1)=(-1)(-\lambda)=\lambda$ dans $F^\times/(F^\times)^2$.\ps 

Soient maintenant $i \neq j$ deux indices et considérons la transvection $T(x)=I_m +xE_{i,j}$, où $x \in F$. 
En particulier $T(x)=T(\frac{x}{2})^2$ et $\mathrm{diag}(T(x),{}^\tau \! T(x)^{-1})=\mathrm{diag}(T(\frac{x}{2}),{}^\tau \! T(\frac{x}{2})^{-1})^2$ si bien que sa norme spinorielle est triviale dans $F^\times/(F^\times)^2$.

Le morphisme de $\GL_m(F)$ vers $F^\times/(F^\times)^2$ est donc le déterminant (réduit modulo $(F^\times)^2$).
\end{proof}

\emph{Remarque : } Cette Proposition permet de redémontrer la Proposition \ref{Param+_inter_Levi}. 
Conservons les notations ci-dessus et considérons un élément $X$ de $M \cap \J^+$.
Alors le fait que $X$ préserve le réseau $L$ et sa forme diagonale par blocs impliquent que chaque $C_i$ (pour $1\leq i \leq k$) est à coefficients entiers. On donc $\det(C_i) \in \Ok^\times$, et donc $\eta(\overline{\det(C_i)})=1$. Or, puisque $X$ est dans $\J^+$, on a $\eta\circ \nu(X)=1$, d'où l'on tire $\eta\circ\nu(D)=1$, \ie $D \in \J_{2\tilde{n}+1}^+$.

\section{Factorisations d'Iwasawa}\label{Factorisations d'Iwasawa}
Ces factorisations sont automatiques si l'on définit les groupes $\mathrm{K}_0$ et $\mathrm{J}$ par rapport à l'immeuble de Bruhat-Tits de $\SO_{2n+1}(F)$ (c'est ce qui est fait dans \cite{Tsai-phd}). Nous choisissons néanmoins de les démontrer de façon géométrique en utilisant les réseaux que ces groupes stabilisent.

On commence par le lemme suivant.
\begin{lemme}\label{lemme_decomposition_conjugaison}

Soient $G$ un groupe, ainsi que $H$ et $K$ des sous-groupes de $G$ vérifiant $G=HK$. Alors on a $G=(g_1Hg_1^{-1})(g_2Kg_2^{-1})$ pour tous $g_1,g_2$ dans $G$.

\end{lemme}
\begin{proof}
Si on a $G=HK$, on a aussi, en considérant une décomposition pour l'inverse, $G=KH$. 

Soit donc $h \in H$, on a pour un élément générique $g \in G$ :
\begin{align*}
hgh^{-1}& =h_1 k_1 \\
g		& =(h^{-1}h_1h) (h^{-1}k_1h )
\end{align*}
et on a $h^{-1}Hh=H$ d'où $G=H (h^{-1}Kh)$, et ce pour tout $h \in H$.\ps 

Or, on a $K=k^{-1}Kk$, pour tout $k \in K$ et donc $G=H (kh)^{-1}K(kh)$ si bien qu'on peut conjuguer le groupe $K$ par n'importe quel élément de $KH=G$.

La décomposition est donc valable en remplaçant le second facteur par n'importe quel conjugué, et, en utilisant le fait qu'on peut utiliser la décomposition \og symétrique \fg{}, également en remplaçant le premier facteur par n'importe quel conjugué.
\end{proof}

Il suffit donc dans la suite de raisonner avec le sous-groupe de Borel standard, défini à partir du choix particulier de la base adaptée $\mathcal{B}_0$. Les groupes $\mathrm{K}_0$ et $\mathrm{J}$ sont également définis en fonction de cette base. Puisque le choix d'une autre base adaptée correspond à la conjugaison par un élément de $\SO_{2n+1}(F)$, les factorisations démontrées ne dépendent pas de ce choix. 


\begin{defi}\label{Defi_reseau_type_a}
Soit $(V,q)$ un $F$-espace vectoriel quadratique de dimension $2n+1$, où $q$ est d'indice de Witt maximal $n$. Fixons $a$ dans $\Ok-\{0\}$. Un réseau $L$ de $V$ sera dit \emph{de type $a$} s'il possède une $\Ok$-base
\og a-adaptée \fg $(e_1,\cdots, e_n, v_0, \linebreak f_n, \cdots, f_1)$ avec
\[
\left\{\begin{array}{ll}
<e_i,e_j>&=<e_i,v_0> = 0, \\
<f_i,f_j>&= <f_i,v_0> = 0, \\
<e_i,f_j>&= \delta_{ij}, \\
<v_0,v_0>&= 2a.
\end{array}
\right. 
\]
\end{defi}

En particulier, un réseau $L$ de type $a$ est un réseau entier pair. De plus, si $2a$ est non inversible dans $\Ok$, le noyau de $L/2aL$ est un $\Ok/2a \Ok$-module libre de rang 1, engendré par $v_0$.

On note $\theta : L \rightarrow L^*$ l'application canonique $x \mapsto <x,\cdot>$.

\begin{lemme}\label{Lemme_1}
Soit $L$ un réseau de type $a$ dans $(V,q)$. Alors l'image de $\theta : L \rightarrow L^*$ est le sous-$\Ok$-module des formes $\Ok$-linéaires $\phi : L \rightarrow \Ok$ telles que $\phi(v_0) \equiv 0 \mod 2a$.
\end{lemme}
\begin{proof}
Soit $(e_1^*,\cdots,e_n^*,v_0^*,f_n^*,\cdots,f_1^*)$ la base duale d'une base $a$-adap\-tée de $L$ (on reprend les notations de la Définition \ref{Defi_reseau_type_a}).

L'image de $\theta$ est engendrée par les $\theta(e_i)=f_i^*$, les $\theta(f_i)=e_i^*$ pour $i \in \{1,\cdots,n\}$, et par $\theta(v_0) = 2a v_0^*$, ce qui conclut.
\end{proof}

\begin{lemme}\label{Lemme_2}
Soit $L$ un réseau de type $a$ dans $(V,q)$, avec $\val_\pk (a)=0$ ou $1$. Soit $W$ un sous-espace isotrope de $V$, et $C = W \cap L$. Alors $\theta$ induit une surjection de $L$ sur $C^*$.
\end{lemme}

\begin{proof}
Par construction, $C$ est saturé dans $L$, donc facteur direct comme $\Ok$-module.

Supposons d'abord que $2a$ est inversible dans $\Ok$. Dans ce cas, $\theta$ est bijective d'après le Lemme \ref{Lemme_1}. On conclut par le fait que toute forme linéaire sur $C$ s'étend en une forme linéaire sur $L$ (par 0 sur un supplémentaire de $C$).\ps 

Supposons désormais $2a$ non inversible dans $\Ok$ et fixons une base $a$-adaptée de $L$. On dispose alors d'un vecteur $v_0$ tel que $<v_0,v_0> = 2a$. En particulier $q(v_0) \neq 0$ et $\Ok v_0 \cap C =\{0\}$.

Montrons que $C \oplus \Ok v_0$ est saturé dans $L$. 
Supposons $\varpi^{-1} (c + x v_0) \in L$; on veut montrer $c \in \varpi C$ et $x \in \pk$. Si $x \in \pk$, alors $\varpi^{-1} c \in L \cap C = C$, donc $c \in \varpi C$ : on peut donc supposer $x \in \Ok^\times$ et même $x=1$. Écrivons alors $c = v_0 + \varpi z$ avec $c \in C$ et $z \in L$.
En appliquant $q$, on trouve :
\[
0 = a + \varpi^2 q(z) + \varpi <v_0,z>,
\]
d'où l'on déduit $\varpi | a$ puis, en utilisant le fait que $<v_0,z> \in 2 a \Ok$, $\varpi^2 |a$, ce qui nous fournit une contradiction.

Puisque $C \oplus \Ok v_0$ est bien saturé dans $L$, donc facteur direct dans $L$, toute forme linéaire sur $C$ peut se prolonger en une forme linéaire sur $L$ nulle sur $v_0$. On conclut par le Lemme \ref{Lemme_1}.
\end{proof}

{\bfseries On suppose désormais\footnote{Un réseau de type $\varpi^m$ de $(V,q)$ pour $m$ entier naturel permettrait de définir le groupe $\J(\pk^m)$ selon les notations de \cite{Tsai-phd}. Le fait que la factorisation d'Iwasawa \emph{échoue op. cit.} pour $m>1$ est à rapprocher du fait que le Lemme \ref{Lemme_2} utilise de façon cruciale $m\leq 1$.} $\val_\pk (a)=0$ ou $1$.}

\begin{cor}\label{Corollaire}
Soient $L$ un réseau de type $a$ de $(V,q)$, $W$ un sous-espace isotrope de dimension $r$ de $V$, $C = L \cap W$, et $(\eps_1,\cdots,\eps_r)$ une $\Ok$-base de $C$. 

Alors, il existe une décomposition de $\Ok$-modules $L = (C \oplus D) \overset{\perp}{\oplus} L'$ et $(\varphi_1,\cdots,\varphi_r)$ une $\Ok$-base de $D$, telle que $<\eps_i,\varphi_j> = \delta_{ij}$.
\end{cor}
\begin{proof}
On considère la base duale des $\eps_i^* : C \rightarrow \Ok$ pour $i \in \{1,\cdots,r\}$. D'après le Lemme \ref{Lemme_2}, il existe $\varphi_1,\cdots,\varphi_r$ dans $L$ tels que $\theta(\varphi_j)_{|C}=\eps_j^*$, soit encore $<\eps_i,\varphi_j>=\delta_{ij}$.

Pour $X=(x_{ij}) \in {\rm M}_r(\Ok)$, on peut considérer la famille des $\varphi'_i=\varphi_i+\sum_k x_{ik} e_k$. On a alors ${\rm Gram}(\varphi'_1,\cdots,\varphi'_r)={\rm Gram}(\varphi_1,\cdots,\varphi_r)+X+{}^t X$. Or la matrice de Gram de la base $(\varphi_1,\cdots,\varphi_r)$ -- et de n'importe quelle famille en fait, \cf Proposition-Définition \ref{Prop_def_FQ} -- est symétrique, avec des éléments diagonaux dans $2\Ok$. On peut donc trouver $X \in {\rm M}_r(\Ok)$ telle que $X+{}^t X=-{\rm Gram}(\varphi_1,\cdots,\varphi_r)$.

On a alors ${\rm Gram}(\varphi'_1,\cdots,\varphi'_r)=0$ et toujours $<\eps_i,\varphi'_j>=\delta_{ij}$ : le $\Ok$-module $D=\Ok \varphi'_1 \oplus \cdots \oplus \Ok \varphi'_r$ est totalement isotrope et le sous-module $C\oplus D$ de $L$ est hyperbolique. En particulier, en notant $L'=(C  \oplus D)^\perp$, on a bien $L=(C \oplus D) \overset{\perp}{\oplus} L'$.
\end{proof}

\begin{prop}\label{Proposition}
Soient $L_1$ et $L_2$ deux réseaux de type $a$ de $(V,q)$. Soit $W$ un sous-espace totalement isotrope de $V$ et soit $\{0\} \subset W_1 \subset \cdots \subset W_n=W$ un drapeau complet de $W$.

Alors, il existe $g \in \SO(V,q)$ tel que $g(L_1)=L_2$ et, pour tout $i \in \{1,\cdots,n\}, \linebreak g(W_i)=W_i$.
\end{prop}
\begin{proof}
Posons, pour $j\in\{1,2\}$, $C_j=L_j \cap W$. D'après le Corollaire \ref{Corollaire}, on peut écrire
\[
L_j=(C_j+D_j) \overset{\perp}{\oplus} \Ok v_j,
\]
avec $D_j$ totalement isotrope et de telle sorte que $\theta$ induise des isomorphismes $\theta_1 : C_1 \overset{\sim}{\rightarrow} D_1^*$ et $\theta_2 : C_2 \overset{\sim}{\rightarrow} D_2^*$. \ps 

En considérant le déterminant des $\Ok$-modules quadratiques $(L_j,q)$, on trouve $<v_j,v_j>=2a \alpha_j^2$ avec $\alpha_j \in \Ok^\times$.
On peut donc, quitte à remplacer $v_j$ par $\alpha_j^{-1}v_j$, supposer que $<v_j,v_j>=2a$.

Choisissons une $\Ok$-base $(e_1,\cdots,e_n)$ de $C_1$ telle que $W_i \cap C_1=\Ok e_1 \oplus \cdots \Ok e_i$ pour tout $i \in \{1,\cdots,n\}$. De même, choisissons une $\Ok$-base $(e'_1,\cdots,e'_n)$ de $C_2$ telle que $W_i \cap C_2=\Ok e'_1 \oplus \cdots \Ok e'_i$ pour tout $i \in \{1,\cdots,n\}$.

Soit $h : C_1 \rightarrow C_2$ l'application $\Ok$-linéaire envoyant $e_i$ sur $e'_i$ pour tout $i \in \{1,\cdots,n\}$. C'est, par définition, un isomorphisme $\Ok$-linéaire. La composée
\begin{equation*}
\begin{tikzcd}
D_1 \arrow[r,"{}^t \theta_1"] & C_1^* \arrow[r,"{}^t h^{-1}"] & C_2^* \arrow[r,"{}^t \theta_2^{-1}"] & D_2
\end{tikzcd}
\end{equation*}
nous donne un isomorphisme $\Ok$-linéaire $h' : D_1 \overset{\sim}{\rightarrow} D_2$.

On en déduit un isomorphisme $\Ok$-linéaire $g : L_1 \rightarrow L_2$ envoyant $C_1$ sur $C_2$ via $h$, $D_1$ sur $D_2$ via $h'$ et $v_1$ sur $v_2$. Par construction de $h'$, $g$ est une isométrie, \ie un élément de $\Oo(V,q)$ (et même de $\SO(V,q)$ quitte à remplacer $v_2$ par $-v_2$), qui préserve manifestement $W_i$ pour tout $i \in \{1,\cdots,n\}$.
\end{proof}

\begin{cor}\label{Factorisation G=BK_0}
Soit $B$ un sous-groupe de Borel de $\SO_{2n+1}(F)$ et $\K_0$ un sous-groupe compact hyperspécial de $\SO_{2n+1}(F)$ tel que défini au paragraphe \ref{Le groupe K_0}.

Alors, on a $\SO_{2n+1}(F)=B\K_0$.
\end{cor}
\begin{proof}
Il suffit de considérer le réseau $L_0$ défini \emph{loc. cit.} (qui est un réseau de type 1 de $(V_n,q)$) et le sous-espace totalement isotrope $X$ de \eqref{dec_canonique_V_n} avec son drapeau complet canonique. Alors, si $g$ est un élément de $\SO_{2n+1}(F)$, $g(L_0)$ est également un réseau de type 1 et donc, par la Proposition \ref{Proposition}, on trouve $b \in \SO_{2n+1}(F)$ tel que $b(L_0)=g(L_0)$ avec de plus $b$ qui préserve le drapeau
\[
\Vect(e_1)\subset \Vect(e_1,e_2) \subset \cdots \subset \Vect(e_1,\cdots,e_n),
\]
c'est-à-dire que $b$ est dans le sous-groupe de Borel standard $B$ de $\SO_{2n+1}(F)$. Donc $b^{-1}g$ stabilise le réseau $L_0$, \ie $b^{-1}g\in\K_0$ et finalement $g \in B\mathrm{K}_0$.
\end{proof}

\begin{cor}\label{Factorisation G=BJ}
Soit $B$ un sous-groupe de Borel de $\SO_{2n+1}(F)$ et $\J$ un sous-groupe compact épiparamodulaire de $\SO_{2n+1}(F)$ tel que défini au paragraphe \ref{Le_groupe_J}.

Alors, on a $\SO_{2n+1}(F)=B\J$.
\end{cor}
\begin{proof}
Il faut maintenant considérer l'espace quadratique $(V_n,q')$ où $q'$ est la forme quadratique $\frac{q}{\varpi}$ introduite au paragraphe \ref{J^+}. Le réseau $L$ défini \emph{loc. cit.} est alors un réseau de type $\varpi$ de $(V_n,q')$. On est donc encore redevable de la Proposition \ref{Proposition} et la suite du raisonnement est la même que pour le Corollaire précédent.
\end{proof}

Remarquons enfin que l'on ne peut pas espérer une factorisation avec le sous-groupe paramodulaire.
\begin{prop}\label{prop_G_neq_BJ+}
Si $2n+1>1$, alors $B \J^+ \subsetneq \SO_{2n+1}(F)$.
\end{prop}
\begin{proof}
Dans le cas $2n+1=1$, tous les groupes sont triviaux. Si $2n+1>1$, alors il existe des éléments de $\J$ qui ne sont pas dans $\J^+$, comme l'élément $u$ introduit au paragraphe \ref{J^+}. Soit donc $j^-$ un tel élément et supposons disposer d'une factorisation $\SO_{2n+1}(F)=B \J^+$. On a alors $j^-=bj^+$, soit $b=j^-(j^+)^{-1}$. Or, d'après le Lemme \ref{lemme_BcapJ}, on a $B \cap \J \subset \J^+$ donc $j^-(j^+)^{-1} \in \J^+$ et finalement $j^- \in \J^+$, contradiction.
\end{proof}

\chapter{Représentations avec des invariants paramodulaires}\label{Représentations avec des invariants paramodulaires}
L'objectif de ce Chapitre est de classifier, pour $n$ quelconque, toutes les représentations du groupe $\SO_{2n+1}(F)$ (tel qu'étudié au Chapitre \ref{Le groupe paramodulaire}) ayant des invariants paramodulaires.

\begin{defi}
Soit $(\pi,V)$ une représentation de $\SO_{2n+1}(F)$. On pose :
\begin{align*}
\pi^{(\J,+)}&=\{v \in V \,|\, \pi(j)(v)=v, \forall j \in \J\}, \\
\pi^{(\J,-)}&=\{v \in V \,|\, \pi(j)(v)=\eta\circ \nu(j)v, \forall j \in \J\}.
\end{align*}
On parlera respectivement de $(\J,+)$-invariants (ou de $\J$-invariants) et de $(\J,-)$-variants.
\end{defi}

\begin{lemme}\label{lemme_decomposition_inv_param}
Soit $(\pi,V)$ une représentation de $\SO_{2n+1}(F)$. On dénote par $\pi^{\J^+}$ l'espace (éventuellement nul) de ses invariants sous l'action du sous-groupe $\J^+$.
Alors, on a la décomposition :
\[
\pi^{\J^+}=\pi^{(\J,+)} \oplus \pi^{(\J,-)}.
\]


De plus, en posant $\pi'=\eta \otimes \pi$, 
l'application $v \mapsto 1 \otimes v$ induit des isomorphismes (de $\J$-représentations) $\pi^{(\J,+)} \simeq (\pi')^{(\J,-)}$ et $\pi^{(\J,-)} \simeq (\pi')^{(\J,+)}$.

\end{lemme}
\begin{proof}
L'action de $\J$ sur $\pi^{\J^+}$ se décompose selon les deux caractères de $\J/\J^+\simeq \Z/2\Z$. Au caractère trivial correspond l'espace $\pi^{(\J,+)}=\pi^\J$. 
Le Théorème \ref{thm_J+_noyau_norme_spin} identifie le caractère non trivial comme étant $\eta \circ \nu$ restreint à $\J$, si bien qu'il lui correspond l'espace $\pi^{(\J,-)}$.

Le dernier point découle de la définition de la torsion (étant entendu que tordre par $\eta$, c'est tordre par $\eta \circ \nu$, fait déjà mentionné au début du paragraphe \ref{Norme spinorielle}).
\end{proof}

{\bfseries Ainsi, disposant d'une représentation ayant des invariants (non triviaux) par $\J^+$, on pourra supposer, quitte à tordre par $\eta$, qu'elle a des invariants par $\J$.}

Achevons ces préliminaires par un lemme facile mais bien utile.
\begin{lemme}\label{lemme_inv_induite_induisante}
Soit $P=MN$ un sous-groupe parabolique standard de $G$. Soit $(\sigma,W)$ une représentation lisse de $M=\GL_{n_1}(F) \times \cdots \times \GL_{n_k}(F) \times \SO_{2\tilde{n}+1}(F)$. Alors on a un isomorphisme linéaire :
\begin{align*}
(\ii_P^G \sigma)^\J & \simeq \sigma^{\J \cap M}.
\end{align*}

En particulier, si $\sigma$ est de la forme $\sigma_1 \boxtimes \cdots \boxtimes \sigma_k \boxtimes \tau$, où $\sigma_i$ est une représentation lisse de $\GL_{n_i}(F)$ pour tout $i$ et $\tau$ une représentation lisse de $\SO_{2\tilde{n}+1}(F)$, alors :
\begin{align*}
(\ii_P^G \sigma)^\J &\simeq \sigma_1^{\GL_{n_1}(\Ok)} \otimes \cdots \otimes \sigma_k^{\GL_{n_k}(\Ok)} \otimes \tau^{\J_{2\tilde{n}+1}}.
\end{align*}

\end{lemme}
\begin{proof}
Il suffit de considérer l'application $\alpha : f \mapsto f(1)$ de $(\ii_P^G \sigma)^\J$ vers $\sigma^{\J \cap M}$ -- on a bien $\sigma(j)f(1)=f(j)=f(1)$. La factorisation $G=P\J$ (qui découle de la décomposition $G=B\J$ du Corollaire \ref{Factorisation G=BJ}) montre que si $f(1)=0$, alors $f(g)=f(pj)=\sigma(p) f(1)=0$ pour tout $g \in G$, si bien que $\alpha$ est injective.\ps 

Pour voir que $\alpha$ est surjective, on se donne $v \in \sigma^{\J \cap M}$ et on définit $f_v$ qui à $g=pj$ associe $\sigma(p) v$. Il faut déjà voir que $f_v$ est bien définie. Soit donc $g=p_1j_1=p_2j_2$ un élément de $G$, 
alors $\sigma(p_2^{-1}p_1)v=\sigma(j_2j_1^{-1})v=v$ par hypothèse sur $v$ et donc la valeur de $f_v$ en $g$ ne dépend pas de la décomposition. La fonction $f_v$ est bien lisse (elle est invariante par translations à droite par le sous-groupe compact ouvert $\J$) et on a $f_v(pgj)=\sigma(p)f_v(g)$ pour tous $p\in P, g\in G, j \in \J$ et donc $f_v \in (\ii_P^G \sigma)^\J$, ce qui conclut.\ps

Le deuxième isomorphisme est une conséquence de la Proposition \ref{Param_inter_Levi}.
%
%
\end{proof}

\section{Foncteur de Jacquet des séries discrètes}
Colette Mœglin et Marko Tadi\'c ont classifié les séries discrètes des groupes classiques $p$-adiques (\cite{Moeg_ser_dis} \cite{Moeg-Tad}), ce qui fait intervenir le calcul des foncteurs de Jacquet de ces séries discrètes par rapport aux sous-groupes paraboliques standard. Ces classifications ont été faites sans faire appel à la correspondance de Langlands locale pour les groupes classiques, démontrée par James Arthur (\cite{Art13}).
Il sera commode de suivre la belle exposition de ces résultats par Bin Xu (\cite{BinXu}) et par Hiraku Atobe (\cite{Atobe}).\medskip

On travaille ici uniquement dans le cas où $G$ est le groupe déployé $\SO_{2n+1}(F)$ étudié au Chapitre \ref{Le groupe paramodulaire} et on s'intéresse aux paramètres de Langlands correspondants (à valeurs dans $\widehat{G}=\Sp_{2n}(\C)$ donc).

Soit donc $\varphi$ un tel paramètre \emph{discret}. D'après la fin du paragraphe \ref{Paramètres discrets}, on a donc $\varphi=\bigoplus_i \xi_i \otimes U_{a_i}$ avec $\sum \dim(\xi_i) a_i=2n$, les couples $(\xi_i,a_i)$ deux à deux distincts, $\xi_i$ autoduale symplectique si $a_i$ est impair et autoduale orthogonale si $a_i$ pair.

Suivant Mœglin, on définit alors $\Jord (\varphi)$ comme l'ensemble des couples $(\xi_i,a_i)$ -- la représentation $U_{a_i}$ de $\SU(2)$ est en effet complètement caractérisée par sa dimension $a_i$.
On peut d'ailleurs, en posant $d_i=\dim(\xi_i)$ utiliser la correspondance de Langlands locale pour $\GL_{d_i}$ pour identifier $\xi_i$ à une représentation irréductible du groupe de Weil $\W_F$ de dimension $d_i$, identification que l'on fera par la suite.

On peut définir de même $\Jord(\varphi)$ pour un paramètre de Langlands non discret, en prenant garde au fait qu'il faut alors considérer un multi-ensemble plutôt qu'un ensemble.

On sait que dans le cas d'un paramètre discret, le quotient $\Cent(\varphi)/Z(\widehat{G})$ est fini et donc égal à son groupe des composantes connexes $\mathcal{S}_\varphi$. Par ailleurs, $Z(\widehat{G})=Z(\Sp_{2n}(\C))=\{\pm 1 \}$. Il est donc équivalent de parler de centralisateur fini.\medskip

Soit maintenant $\pi$ une représentation lisse irréductible de $G$.
Soit $d \leq n$ et considérons le sous-groupe de Levi standard 
\begin{equation}\label{Levi_SO_2n+1}
M_d=\GL_d(F) \times \SO_{2(n-d)+1}(F)
\end{equation}
de $G$ et le sous-groupe parabolique standard associé $P_d=M_dN_d$. Alors le foncteur de Jacquet (normalisé) $\rr_{P_d}^G \pi$ est une représentation (lisse, admissible, de longueur finie) de $M_d$. On peut regarder la semi-simplifiée de cette représentation :
\begin{equation}\label{rpg_ss}
(\rr_{P_d}^G \pi)^\mathrm{ss}=\bigoplus_i \tau_i \otimes \sigma_i
\end{equation}
où $\tau_i$ (resp. $\sigma_i$) est une représentation irréductible de $\GL_d(F)$ (resp. de \linebreak$\SO_{2(n-d)+1}(F)$). On pose alors, pour $\rho$ une représentation supercuspidale \emph{unitaire} de $\GL_d(F)$, et pour $x \in \R$ :
\begin{equation*}
\Jac_{\rho|\cdot|^x}=\bigoplus
\limits_{\substack{i \\ \tau_i \simeq \rho|\cdot|^x}}
\sigma_i.
\end{equation*}


On remarque que, quand $\rho$ parcourt l'ensemble des supercuspidales unitaires et $x$ parcourt $\R$, $\rho|\cdot|^x$ parcourt bien l'ensemble des représentations supercuspidales de $\GL_d(F)$. En revanche, dès que $d>1$, il existe des représentations irréductibles non supercuspidales et la donnée des $\Jac_{\rho|\cdot|^x}$ est moins complète que celle de $(\rr_{P_d}^G \pi)^\mathrm{ss}$. Pour $d=1$, on a bien :
\begin{equation}\label{jac_pour_gl1}
(\rr_{P_1}^G \pi)^\mathrm{ss}=\bigoplus_i \chi_i \otimes \Jac_{\chi_i}.
\end{equation}

(L'ensemble d'indices est \emph{a priori} plus petit qu'en \eqref{rpg_ss} puisqu'on a regroupé les indices correspondant à des caractères $\chi_i$ isomorphes).
Il reste la question de la semi-simplifiée pour laquelle la proposition suivante précise les choses.

%

%
%

\begin{prop}\label{proposition_GC_ss}
Soient $G$ et $H$ les $F$-points de deux groupes réductifs sur $F$.
Soit $V$ une $\C$-représentation lisse, de longueur finie, de $G\times H$.
Soit $r$ une $\C$-représentation \emph{supercuspidale} de $G$. On définit alors les propriétés suivantes d'une $\C$-représentation lisse de longueur finie $W$ de $G \times H$ :
\begin{itemize}
\item $\mathcal{P}_r$ : tout sous-quotient irréductible de $W$ est de la forme $X \boxtimes Y$ avec $X$ isomorphe à $r$.
\item $\mathcal{P}'_r$ : tout sous-quotient irréductible de $W$ est de la forme $X \boxtimes Y$ avec $X$ non isomorphe à $r$.
\end{itemize}
Il existe une unique décomposition en somme directe $V = A \oplus B$ telle que $A$ satisfait $\mathcal{P}_r$, et $B$ satisfait $\mathcal{P}'_r$.
\end{prop}

\begin{proof} \textbf{Unicité.} 
On vérifie de suite que tout homomorphisme entre deux représentations de longueur finie vérifiant $\mathcal{P}_r$ pour l'une, et $\mathcal{P}'_r$ pour l'autre, est nul. Ainsi, si on a $A \oplus B = A' \oplus B'$ avec $A,A'$ vérifiant $\mathcal{P}_r$ et $B,B'$ vérifiant $\mathcal{P}'_r$, alors les projections $A \rightarrow B'$ et $B \rightarrow A'$ sont nulles. On en déduit $A$ inclus dans $A'$, $B$ inclus dans $B'$, puis par symétrie $A=A'$ et $B=B'$.\ps 

\textbf{Existence.} On va montrer que

(a) $V$ est dans une suite exacte $0 \rightarrow A \rightarrow V \rightarrow B' \rightarrow 0$ avec $A$ vérifiant $\mathcal{P}_r$ et $B'$ vérifiant $\mathcal{P}'_r$.

(b) $V$ est dans une suite exacte $0 \rightarrow B \rightarrow V \rightarrow A' \rightarrow 0$ avec $A'$ vérifiant $\mathcal{P}_r$ et $B$ vérifiant $\mathcal{P}'_r$.

On conclura ainsi : $B \cap A$ est un sous-objet de $A$ et de $B$, donc satisfait $\mathcal{P}_r$ et $\mathcal{P}'_r$, et est de longueur finie, donc nul. La projection $V \rightarrow B'$ injecte donc $B$ dans $B'$. Mais $B$ et $B'$ ont même semi-simplifiée à isomorphisme près par le théorème de Jordan-Hölder (\cite{Lang}, III, §8). En particulier, elles ont même longueur, et l'injection de $B$ dans $B'$ induite par $V \rightarrow B'$ est en fait un isomorphisme. Cela entraîne $V = A \oplus B$.

On remarque par ailleurs que si l'on montre (a) alors on montre (b), en remplaçant $r$ et $V$ par leur contragrédiente. Il suffit donc de montrer (a), ce que nous faisons en utilisant le lemme suivant :

\begin{lemme}
(On garde les notations de la Proposition.)
Supposons disposer d'une suite exacte $0 \rightarrow C \rightarrow V \rightarrow D \rightarrow 0$ avec $C,D$ irréductibles, et $D$ vérifiant $\mathcal{P}_r$.

Alors on a également
une suite exacte $0 \rightarrow C' \rightarrow V \rightarrow D' \rightarrow 0$ avec $C',D'$ irréductibles, et $C'$ vérifiant $\mathcal{P}_r$.
\end{lemme}
\begin{proof}
La représentation $C$ est irréductible : elle s'écrit donc $X_1 \boxtimes Y_1$ avec $X_1$ irréductible de $G$ et $Y_1$ irréductible de $H$. De même, $D=X_2 \boxtimes Y_2$ et, comme elle vérifie $\mathcal{P}_r$, on a $X_2 \simeq r$.

On souhaite oublier l'action de $H$ et considérer ces représentations comme des $\C[G]$-modules. Il faut toutefois faire attention au fait que ces $\C[G]$-modules ne sont pas admissibles, et envisager les choses un peu différemment.
La lissité de $Y$ nous fournit un sous-groupe compact ouvert $K_H$ de $H$ tel que $Y^{K_H} \neq \{0\}$. La considération des $K_H$-invariants nous fournit la suite exacte :
\[
0 \rightarrow X_1 \boxtimes Y_1^{K_H} \rightarrow V^{\{1\}\times K_H} \rightarrow X_2 \boxtimes Y_2^{K_H} \rightarrow 0,
\]
que l'on voit comme une suite exacte de $\C[G]$-modules, admissibles désormais. On peut alors utiliser le Lemme VI.3.6 de \cite{Ren} pour dire que $r$, supercuspidale et quotient de $V^{\{1\}\times K_H}$, en est aussi une sous-$G$-représentation. C'est \emph{a fortiori} une sous-$G$-représentation de $V$ et la composante $r$-isotypique $V_r$ de $V$ est non nulle. Mais $V_r$ est aussi stable par $H$. On peut ainsi trouver une sous-représentation irréductible de $G \times H$ dans $V_r$, nécessairement de la forme $r \boxtimes X$, \ie vérifiant $\mathcal{P}_r$.
\end{proof}
\emph{(Retour à la démonstration de la Proposition)} Soit donc maintenant une filtration
\[
\{0\} = V_0 \subset V_1 \subset \cdots \subset V_n = V
\] avec $V_i/V_{i-1}$ irréductible comme $G \times H$-module pour tout $i$.
Supposons que l'on ait un indice $i$ tel que $V_i/V_{i-1}$ vérifie $\mathcal{P}_r$ et $V_{i-1}/V_{i-2}$ vérifie $\mathcal{P}'_r$. Alors en appliquant le lemme à la suite exacte :
\[
0 \rightarrow V_{i-1}/V_{i-2} \rightarrow V_{i}/V_{i-2} \rightarrow V_{i}/V_{i-1} \rightarrow 0,
\]
on peut modifier $V_{i-1}$ de sorte que ce soit l'inverse. Si aucun des $V_i/V_{i-1}$ ne vérifie $\mathcal{P}_r$, alors $V = B$ vérifie $\mathcal{P}'_r$, ce qui, en prenant $A=0$, achève la preuve.

Sinon, on peut supposer que $V_1/V_0=V_1$ est $\mathcal{P}_r$. Par récurrence sur la longueur de $V$, on peut alors trouver $A_0$ ($\mathcal{P}_r$) et $B$ ($\mathcal{P}'_r$) tels que la suite :
\[
0 \rightarrow A_0 \rightarrow V/V_1 \rightarrow B_0 \rightarrow 0
\]
soit exacte.
En définissant alors $A$ comme l'extension naturelle de $A_0$ par $V_1$ (vérifiant toujours $\mathcal{P}_r$), on a bien une suite exacte
\[
0 \rightarrow A \rightarrow V \rightarrow B_0 \rightarrow 0
\]
avec les propriétés recherchées.
\end{proof}

On veut évidemment appliquer cette proposition à $\rr_{P_d}^G \pi$, le cas le plus intéressant étant celui de $P_1$ pour lequel on fait apparaître des caractères $\chi_i$ de $\GL_1(F)$ qui sont bien tous supercuspidaux. La longueur finie de $\rr_{P_1}^G \pi$ nous donne, par itération de la Proposition \ref{proposition_GC_ss} :
\[
\rr_{P_1}^G \pi=\bigoplus_i V(\chi_i),
\]
où $V(\chi_i)$ est telle que tout sous-quotient irréductible est de la forme $\chi_i \otimes Y$ (avec $Y$ représentation irréductible de $\SO_{2n-1}(F)$). On remarque que $V(\chi_i)$ est plus grande \emph{a priori} que la composante $\chi_i$-isotypique $V_{\chi_i}$ de la $\GL_1(F)$-représentation $V$, qui ne considère que les \emph{sous-représentations} de la forme $\chi_i\otimes Y$. Par identification avec l'écriture \eqref{jac_pour_gl1}, on a $V(\chi_i)^\mathrm{ss}=\Jac_{\chi_i}$.\medskip

Revenons à notre représentation irréductible $\pi$ et soit $\varphi=\Ll(\pi)$ son paramètre de Langlands. On a également un caractère $\eps$ du groupe $\mathcal{S}_\varphi$, qui caractérise $\pi$ à l'intérieur du paquet $\Pi_\varphi$. Le choix d'une représentation supercuspidale unitaire $\rho$ de $\GL_d(F)$ (ou de façon équivalente d'une représentation irréductible de dimension $d$ de $\W_F$) implique le choix du sous-groupe de Levi $M_d$ (au sens de \eqref{Levi_SO_2n+1}) et on peut alors considérer le foncteur de Jacquet correspondant et donc la quantité $\Jac_{\rho |\cdot|^x}$ (pour $x$ réel).

\begin{lemme} \emph{(M\oe{}glin-Tadi\'c, \cite{Moeg-Tad}, Lemma 3.6)}\label{Ato4.1} 

Soit $\pi$ une série discrète de $G$ et soit $\varphi$ son paramètre de Langlands. Alors
\[
\Jac_{\rho|\cdot|^x} \neq 0 \Rightarrow (\rho,2x+1) \in \Jord(\varphi)
\]
En particulier $x \in \frac{1}{2}\mathbb{Z}$ et $x \geq 0$.
\end{lemme}
\begin{proof}
On renvoie à \cite{BinXu}, Lemma 7.2 pour la preuve. 
\end{proof}

On peut en fait préciser les conditions de non nullité de $\Jac_{\rho|\cdot|^x}$. Introduisons auparavant une notation supplémentaire. Soit $\varphi$ un paramètre discret de $G$, d'après la Proposition \ref{decomposition_par_discret_sp}, on a : $\varphi=\bigoplus_{i=1}^r \varphi_i$ avec chaque $\varphi_i$ autoduale symplectique irréductible et les $\varphi_i$ deux à deux distinctes. On obtient alors immédiatement que $\Cent(\varphi)=\{\pm1\}^r$. 
On peut fixer une $\Z/2\Z$-base $(\alpha_{\varphi_i})_{1\leq i \leq r}$ de ce dernier groupe en associant $\alpha_{\varphi_i}$ à $\varphi_i$.

Par ailleurs, $\mathcal{S}_\varphi$ vaut (ici) $\Cent(\varphi)/\{\pm 1\}$. Un caractère de $\mathcal{S}_\varphi$ est donc la donnée d'un signe pour chaque $\alpha_{\varphi_i}$ tel que le produit des signes fasse 1 (\ie le caractère de $\Cent(\varphi)$ doit être trivial sur l'élément diagonal $\sum_i \alpha_i$). Étant donné $\eps$ un élément de $\widehat{\mathcal{S}_\varphi}$, on peut donc parler de sa valeur en $\alpha_{\varphi_i}$, que l'on notera $\eps(\alpha_{\varphi_i})$.

Enfin, si $\varphi=\bigoplus_i \sigma_i \otimes U_{a_i}$ est un paramètre de Langlands, on notera 
\begin{equation} \label{notation_varphi_ominus}
\varphi \ominus (\sigma_j \otimes U_{a_j})
\end{equation}
pour $\bigoplus_{i\neq j} \sigma_i \otimes U_{a_i}$.


\begin{thm}\label{Ato4.2} \emph{(\cite{Moeg-Tad} ; \cite{Moeg-courtenote} p.9)} 

Soit $\pi$ une série discrète de $G$ et soit $\varphi$ son paramètre de Langlands. On suppose que $(\rho,2x+1) \in \Jord(\varphi)$ et on note $\varphi_-=\varphi \ominus \rho \otimes U_{2x+1} \oplus \rho \otimes U_{2x-1}$. Alors $\Jac_{\rho|\cdot|^x} \neq 0$ exactement quand l'une des trois conditions suivantes est satisfaite :
\begin{enumerate}
\item $x>1/2$ et $(\rho,2x-1) \not\in \Jord(\varphi)$, si bien que $\varphi_-$ est un paramètre discret : c'est le paramètre de Langlands de $\Jac_{\rho|\cdot|^x}$ ;
\item $x>1/2$ et $(\rho,2x-1) \in \Jord(\varphi)$ avec $\eps(\alpha_{\rho \otimes U_{2x+1}})=\eps(\alpha_{\rho \otimes U_{2x-1}})$, auquel cas $\varphi_-$ est le paramètre de Langlands (non discret) de $\Jac_{\rho|\cdot|^x}$ ;
\item $x=1/2$ et $\eps(\alpha_{\rho \otimes U_{2}})=1$, auquel cas $\varphi_-$ est le paramètre de Langlands (discret) de $\Jac_{\rho|\cdot|^x}$.
\end{enumerate}
On peut en fait préciser le caractère $\eps_-$ correspondant à $\Jac_{\rho|\cdot|^x}$, si bien que cette dernière représentation est entièrement caractérisée.
\end{thm}
\begin{proof}
On renvoie à \cite{BinXu}, Lemma 7.3 pour la preuve. 
\end{proof}

La combinaison des Lemme \ref{Ato4.1} et Théorème \ref{Ato4.2} nous dit que, dans le cas d'une série discrète, $\Jac_{\rho|\cdot|^x}$ est soit nul, soit irréductible. On obtient alors le corollaire suivant.

\begin{cor}\label{gl1_jac_ss}
Soit $M_1=\GL_1(F)\times \SO_{2n-1}(F)$ sous-groupe de Levi de $G$ et $P_1$ le sous-groupe parabolique correspondant. Alors $\rr_{P_1}^G \pi$ est semi-simple.
\end{cor}
\begin{proof}
On a vu à la suite de la Proposition \ref{proposition_GC_ss} que $\rr_{P_1}^G \pi=\bigoplus_i V(\chi_i)$ avec $V(\chi_i)^\mathrm{ss}=\Jac_{\chi_i}$. Or $\Jac_\chi$ étant nul ou irréductible, on a $V(\chi_i)$ de longueur 0 ou 1, il est donc nul ou irréductible. Finalement, $\rr_{P_1}^G \pi=\bigoplus \chi \otimes \Jac_\chi$ est semi-simple.
\end{proof}

\section{Invariants paramodulaires des séries discrètes}
Il nous faut commencer par un résultat concernant les invariants paramodulaires des séries principales non ramifiées.

\begin{lemme}\label{lemme_inv_param_induites}
Soit $\chi$ un caractère non ramifié de $T$ et soit la représentation induite $\ii_B^G \chi$. Alors les espaces $(\ii_B^G \chi)^{(\J,+)}$ et $(\ii_B^G \chi)^{(\J,-)}$ sont tous deux de dimension 1, si bien que $\dim(\ii_B^G \chi)^{\J^+}=2$.
\end{lemme}
\begin{proof}
Pour $\eps\in\{+,-\}$, on a :
\[
(\ii_B^G \chi)^{(\J,\eps)}=\left\lbrace f \in {\rm C}^\infty (G,\C) \middle| f(bgj)= \delta_B(b)^{1/2}\chi(b) \alpha_\eps(j)f(g), \forall b \in B, \forall g \in G \right\rbrace
\]
avec $\alpha_+$ le caractère trivial et $\alpha_-$ le caractère $\eta \circ \nu$.
Dans les deux cas, la factorisation $G=B\J$ (du Corollaire \ref{Factorisation G=BJ}) nous dit que $f \in (\ii_B^G \chi)^{(\J,\eps)}$ est complètement déterminée par $f(1)$ : l'application $f \mapsto f(1)$ est une injection de $(\ii_B^G \chi)^{(\J,\eps)}$ dans $\C$.\ps 

Vérifions la surjectivité de cette application : soit donc $\lambda \in \C$, on définit 
\[
\begin{array}{ccccl}
f & : & G &\longrightarrow & \mathbb{C} \\
 & & bj & \longmapsto & \delta_B(b)^{1/2}\chi(b) \alpha_\eps(j)\lambda \\
\end{array}.
\]
Il faut vérifier que l'application est bien définie, \ie que la valeur de $f$ en un point $g$ ne dépend pas de l'écriture $g=bj$ choisie. Supposons donc avoir $b_1j_1=b_2j_2$, soit $b_2^{-1}b_1=j_2j_1^{-1}$. On remarque alors que $\delta_B$ est trivial sur le groupe compact $\J$ et, par le Lemme \ref{lemme_BcapJ}, que le caractère $\alpha_\eps$ et le caractère $\chi$ sont triviaux sur $B \cap \J$.

L'application $f$ ainsi définie est lisse (elle est invariante par translations à droite par le sous-groupe compact ouvert $\J^+$) et elle est évidemment dans $(\ii_B^G \chi)^{(\J,\eps)}$. Finalement $(\ii_B^G \chi)^{(\J,\eps)}$ est bien de dimension 1.
\end{proof}

\begin{lemme}\label{inv_param_Steinberg}
La représentation de Steinberg $\St_{\SO_{2n+1}(F)}$ de $\SO_{2n+1}(F)$ n'a pas de $\J$-invariants si $2n+1>1$ et n'a pas de $\J^+$-invariants si $2n+1>3$. De plus, $\St_{\SO_3(F)}^{\J^+}$ est de dimension 1.
\end{lemme}
\begin{proof}
On note $G=\SO_{2n+1}(F)$ et on a $\St_G=\Ind_B^G \1 / (\sum_{B \varsubsetneq P} \Ind_P^G \1) $, la somme portant sur les sous-groupes paraboliques de $G$ (standard, à cause de la condition qu'ils contiennent $B$). On prendra garde au fait qu'il s'agit ici d'induites non normalisées ne faisant pas intervenir la fonction modulaire.

Le Lemme \ref{lemme_inv_param_induites} nous dit\footnote{Le lemme travaille avec des induites normalisées, mais le résultat demeure, quitte à tordre par $\delta_B^{-1/2}$.} que la représentation $\Ind_B^G \1$ admet des $\J$-invariants (\ie des $(\J,+)$-invariants) de dimension 1. Or, si $2n+1>1$, $B \varsubsetneq G$ et $\Ind_G^G \1=\1$ apparaît dans ce par quoi on quotiente : représentation triviale qui a également des $\J$-invariants de dimension 1. Par exactitude de la prise d'invariants par un sous-groupe compact ouvert (Lemme \ref{5.2.1 de DeB}), les $\J$-invariants de $\St_G$ sont donc de dimension nulle.

En ce qui concerne les $\J^+$-invariants, il suffit d'après le Lemme \ref{lemme_decomposition_inv_param} de regarder simultanément les $(\J,+)$-invariants et les $(\J,-)$-variants : on a  $\pi^{\J^+}=\pi^{(\J,+)} \oplus \pi^{(\J,-)}$.

On sait déjà que, pour $2n+1>1$, il n'y a pas $(\J,+)$-invariants, considérons donc les $(\J,-)$-variants de $\St_G$ ou, de façon équivalente, les $(\J,+)$-invariants de $\eta \otimes \St_G= \Ind_B^G \eta / (\sum_{B \varsubsetneq P} \Ind_P^G \eta)$, où $\eta$ est l'unique caractère quadratique non ramifié non trivial introduit au paragraphe \ref{Norme spinorielle}.\ps 

Le Lemme \ref{lemme_inv_param_induites} nous donne que la représentation $\Ind_B^G \eta$ admet des $\J$-invariants de dimension 1. Si $2n+1>3$, alors le sous-groupe parabolique $P_n$ correspondant au sous-groupe de Levi $M_n=\GL_n(F) \times \SO_1(F)$ contient strictement le sous-groupe de Borel $B$. Or le Lemme \ref{lemme_inv_induite_induisante} nous donne $(\Ind_{P_n}^G \eta)^\J \simeq \eta_{|M_n}^{M_n \cap \J}$. Or $M_n \cap \J=\GL_n(\Ok) \times \{1\}$ et, par la Proposition \ref{prop_norme_spin_Levi}, $\nu_{|M_n \cap \J}=\det$, si bien que $ \nu(M_n \cap \J) \subset  \Ok^\times / (\Ok^\times)^2$. D'où, finalement $\eta_{|M_n \cap \J}$ trivial et $(\Ind_{P_n}^G \eta)^\J$ de dimension 1.

Par le même argument d'exactitude de la prise d'invariants, on en déduit qu'il n'y pas de $\J$-invariants (non triviaux) dans $\eta \otimes \St_G$ et finalement que $\St_G$ n'admet pas de $\J^+$-invariants si $2n+1>3$.

Il reste à conclure dans le cas $2n+1=3$. Or la Proposition \ref{prop_I_inclus_J+} nous dit que dans ce cas, on est en fait en train de prendre les invariants par le sous-groupe d'Iwahori standard, dont on sait qu'ils sont de dimension 1 (on peut penser à l'isomorphisme exceptionnel $\SO_3 \simeq \mathrm{PGL}_2$, commenté au §\ref{Isomorphisme exceptionnel}).
\end{proof}

\begin{cor}\label{cor_ser_disc_so(3)}
Les seules séries discrètes de $\SO_3(F)$ ayant des $\J^+$-invariants sont $\St_{\SO_3(F)}$ et $\eta \otimes \St_{\SO_3(F)}$.

La seule série discrète de $\SO_3(F)$ ayant des $\J$-invariants est $\eta \otimes \St_{\SO_3(F)}$.
\end{cor}

\begin{proof}
Raisonnons à l'aide de la correspondance de Langlands locale. À une telle série discrète doit correspondre, d'après \eqref{rep_disc_param_disc} un paramètre discret $\varphi$ de $\Sp_2(\C)=\SL_2(\C)$.\ps 
\begin{itemize}
\item Le cas $\varphi=\chi \oplus \chi^{-1}$ où $\chi$ est un caractère de $\W_F$ est exclu car ce paramètre se factorise par le tore et donc n'est pas discret.\ps 
\item Si $\varphi=\tau$ où $\tau$ est une représentation irréductible de $\W_F$ (nécessairement autoduale et symplectique), le paquet de Langlands associé contient un seul élément $\pi$ supercuspidal (on utilise encore l'isomorphisme $\SO_3 \simeq \mathrm{PGL}_2$). La représentation $\pi$ ne peut avoir d'invariants par le groupe $\J^+=\mathrm{I}$ où $\mathrm{I}$ désigne le sous-groupe d'Iwahori standard (Proposition \ref{prop_I_inclus_J+}) car on aurait alors (par \cite{Cass-art}, Proposition 2.6) que $\pi$ se plonge dans une série principale non ramifiée, ce qui contredit sa supercuspidalité.\ps 
\item Si $\varphi=\chi\otimes U_2$ avec $\chi$ nécessairement autodual donc quadratique, le paquet de Langlands associé contient un seul élément, qui est $\chi \otimes \St_{\SO_3(F)}$. Si cette dernière représentation a des $\J^+$-invariants, alors de même elle se plonge dans une série principale non ramifiée, donc $\chi$ est non ramifié. Or il existe exactement deux caractères quadratiques non ramifiés de $\W_F^{\rm ab} \simeq F^\times$ : le caractère trivial et le caractère $\eta$.\ps 
\end{itemize}

On a donc identifié $\St_{\SO_3(F)}$ et $\eta \otimes \St_{\SO_3(F)}$, qui ont bien des $\J^+$-invariants.\ps 

Par ailleurs, le groupe $\J$ agit par un signe sur les $\J^+$-invariants (car $\J^+$ est d'indice 2 dans $\J$). On sait que ce signe est \og $-$ \fg{} pour la représentation $\St_{\SO_3(F)}$ (sans quoi elle aurait des $\J$-invariants, ce qui n'est pas le cas d'après le Lemme \ref{inv_param_Steinberg}). Comme $\eta \circ \nu$ est \emph{le} caractère non trivial de $\J / \J^+=\Z/2\Z$, ce dernier groupe (et donc $\J$) agit par un signe \og + \fg{} sur les $\J^+$-invariants de $\eta \otimes \St_{\SO_3(F)}$, ce qui revient à dire qu'elle a en fait des $\J$-invariants.
\end{proof}

Pour traiter le cas $2n+1 \geq 5$, nous commençons par un lemme technique. 

\begin{lemme}\label{lemme_r_P1_G_pi}
Soit $\pi$ une série discrète de $\SO_{2n+1}(F)$ avec $2n+1 \geq 5$. On suppose que $\pi^\J \neq \{0\}$.

Alors, en notant $P_1=M_1N_1$ le sous-groupe parabolique standard, où $M_1=\GL_1 \times \SO_{2n-1}$, on a que $\rr_{P_1}^G (\pi)$ est une représentation non nulle de $M_1$. Mieux, 
\begin{equation*}
\rr_{P_1}^G (\pi)=\bigoplus_{i \in I} \chi_i \otimes \alpha_i,
\end{equation*}
où les $\chi_i$ sont des caractères non ramifiés de $F^\times$ et les $\alpha_i$ sont des représentations irréductibles de $\SO_{2n-1}(F)$ telles que $\alpha_i^{\J_{2n-1}}\neq \{0\}$, et $\alpha_i=\Jac_{\chi_i}$.
\end{lemme}
\begin{proof}
On a $\mathrm{I}$ inclus dans $\J$ par la Proposition \ref{prop_Iwahori_inclus_J} (et même dans $\J^+$ par la Proposition \ref{prop_I_inclus_J+}) où $\mathrm{I}$ désigne le sous-groupe d'Iwahori standard. La représentation $\pi$ a donc des invariants par $\mathrm{I}$, si bien que, par la Proposition 2.6 de \cite{Cass-art}, elle se plonge dans une série principale non ramifiée. Il existe $\psi$ caractère du tore maximal standard $T$ tel que $\pi \hookrightarrow \ii_B^G \psi $, soit, par réciprocité de Frobenius, $\rr_B^G \pi \twoheadrightarrow \psi$. En particulier, le foncteur de Jacquet (normalisé) par rapport à $B$ est non nul, et donc, par transitivité du foncteur de Jacquet, il en est de même pour tous les sous-groupes paraboliques standard, en particulier pour $P_1$.

Or, $\pi$ étant irréductible, elle est $G$-engendrée par n'importe quel vecteur non nul, en particulier par ses $\J$-invariants. Puisque $G=P_1\J$, le quotient $\rr_{P_1}^G (\pi)$ est donc $M_1$-engendré par ses $(\J \cap M_1)$-invariants, et il en est de même de tous les quotients irréductibles de $\rr_{P_1}^G (\pi)$. Or, on a vu au Corollaire \ref{gl1_jac_ss} que $\rr_{P_1}^G (\pi)$ était semi-simple, donc chaque composant irréductible peut être vu comme un quotient et possède en particulier des $(\J \cap M_1)$-invariants (qui l'engendrent). Finalement :
\begin{equation}\label{r_p_G_pi}
\rr_{P_1}^G (\pi)=\bigoplus_{i \in I} \chi_i \otimes \alpha_i,
\end{equation}
où $\chi_i$ est un caractère de $\GL_1$ et $\alpha_i$ une représentation irréductible de $\SO_{2n-1}$, tels que $(\chi_i \otimes \alpha_i)^{\J \cap M_1}=\chi_i^{\Ok^\times} \otimes \alpha_i^{\J_{2n-1}} \neq \{0\}$ pour tout $i$ (en utilisant la Proposition \ref{Param_inter_Levi}). En particulier, tous les caractères $\chi_i$ sont non ramifiés.
On a alors $\alpha_i=\Jac_{\chi_i}$ (on a bien vu au Théorème \ref{Ato4.2} qu'il y avait au plus un constituant irréductible dans $\Jac_{\chi_i}$).
\end{proof}

\begin{prop}\label{prop_ser_disc_so5_inv_param}
Il n'existe pas de série discrète de $\SO_5(F)$ ayant des invariants paramodulaires.
\end{prop}
\begin{proof} 

Selon la Remarque qui suit le Lemme \ref{lemme_decomposition_inv_param}, on peut supposer qu'il s'agit ici d'invariants par $\J$. Supposons donc disposer d'une série discrète $(\pi,V)$ de $\SO_5(F)$ telle que $\pi^\J \neq 0$.

Notons $\varphi$ le paramètre de Langlands de $\pi$. On sait, par \eqref{rep_disc_param_disc}, que c'est un paramètre discret si bien que l'on est dans une des cinq situations suivantes :
\begin{enumerate}
\item $\varphi=\theta$ avec $\theta$ autoduale symplectique irréductible ;
\item $\varphi=\tau \oplus \tau'$ avec $\tau$ et $\tau'$ autoduales symplectiques irréductibles distinctes ;
\item $\varphi=(\chi \otimes U_2) \oplus \tau' $, avec $\chi$  autodual \ie quadratique et $\tau'$ autoduale symplectique irréductible ;
\item $\varphi=\chi \otimes U_4$, avec $\chi$ quadratique ;
\item $\varphi=(\chi \otimes U_2) \oplus (\chi' \otimes U_2) $, avec $\chi$ et $\chi'$ quadratiques distincts.
\end{enumerate}

Or, le Lemme \ref{lemme_r_P1_G_pi} implique que $\rr_{P_1}^G(\pi)$ est non nul et s'écrit sous la forme $\bigoplus_{i \in I} \chi_i \otimes \alpha_i$, où les $\chi_i$ sont des caractères non ramifiés. Cela signifie donc qu'il existe au moins un $i$ -- et un caractère $\chi_i$ -- tel que $\alpha_i=\Jac_{\chi_i} \neq 0$. Or on sait que $\Jac_{\chi_i} \neq 0$ implique $(\rho,2x+1) \in \Jord(\varphi)$ où l'on a écrit $\chi_i=\rho |\cdot|^x$ avec $\rho$ caractère unitaire. \medskip

Ceci exclut donc les deux premiers cas qui ne font apparaître aucun caractère.

Pour le troisième cas, le seul $\chi_i$ qui puisse donner lieu à un $\Jac_{\chi_i}$ non nul est $\chi_i=\chi |\cdot|^{1/2}$, pour lequel on a $\Jac_{\chi |\cdot|^{1/2}}$ effectivement non nul d'après le Théorème \ref{Ato4.2}. Mieux, on connaît son paramètre de Langlands, qui est alors $\tau'$. On a donc $\rr_{P_1}^G(\pi)=\chi|\cdot|^{1/2} \otimes \pi'$ où $\pi'$ est une représentation de $\SO_3(F)$ de paramètre de Langlands $\tau'$. D'après le Lemme \ref{lemme_r_P1_G_pi}, $\pi'$ possède des invariants par $\J_3$ : 
c'est donc, d'après le Corollaire \ref{cor_ser_disc_so(3)}, $\eta \otimes \St_{\SO_3(F)}$. Or le paramètre de Langlands de cette dernière représentation est $\eta \otimes U_2 \neq \tau'$, ce qui nous fournit une contradiction.

Dans le quatrième cas, il est bien connu que le paquet de Langlands associé est un singleton contenant la seule représentation $\chi \otimes \St_{\SO_5(F)}$ qui n'a pas de $\J$-invariants d'après le Lemme \ref{inv_param_Steinberg}.

Enfin, dans le cinquième cas, puisqu'on sait que $\chi$ et $\chi'$ doivent être quadratiques, non ramifiés et distincts, on a $\{\chi,\chi'\}=\{\1,\eta\}$. Le Théorème \ref{Ato4.2} nous dit que $\Jac_{\eta|\cdot|^{1/2}} \neq 0$. Mieux, on connaît son paramètre de Langlands, c'est $\1 \otimes U_2$, qui est le paramètre de Langlands de la Steinberg de $\SO_3(F)$. Or cette représentation n'a pas de $\J_3$-invariants d'après le Corollaire \ref{cor_ser_disc_so(3)}.

Finalement, tous les cas sont exclus et il n'existe pas de série discrète de $\SO_5(F)$ avec des invariants paramodulaires. 
\end{proof}

\begin{prop}\label{ser_disc_>4_pas_d'inv}
Il n'existe pas de série discrète de $\SO_{2n+1}(F)$ pour $2n+1 \geq 5$ ayant des invariants paramodulaires.
\end{prop}
\begin{proof}

On se ramène ici aussi (toujours en utilisant le Lemme \ref{lemme_decomposition_inv_param}) au cas d'une série discrète $\pi$ ayant des $\J$-invariants. Il s'agit de se ramener par récurrence au cas où $2n+1=5$ en propageant des invariants (épi)paramodulaires. Soit $\varphi$ le paramètre de Langlands (discret) de $\pi$.


On peut encore écrire, suivant le Lemme \ref{lemme_r_P1_G_pi} :
\[
\rr_{P_1}^G (\pi)=\bigoplus_i \chi_i \otimes \Jac_{\chi_i},
\]
avec les conditions sur la non nullité de $\Jac_{\chi_i}$ données par le Théorème \ref{Ato4.2}. Mieux, on sait qu'alors $\Jac_{\chi_i}$ a des invariants par $\J_{2n-1}$ et on connaît son paramètre de Langlands par le Théorème \ref{Ato4.2}, $\varphi_-=\varphi\ominus \rho \otimes U_{2x+1} \oplus \rho \otimes U_{2x-1}$ avec les notations de \eqref{notation_varphi_ominus} et $\rho=\chi_i|\cdot|^{-x}$. 
 Si ce paramètre est discret, alors on a affaire (toujours par \eqref{rep_disc_param_disc}) à une série discrète de $\SO_{2n-1}(F)$ qui a des invariants paramodulaires, ce qui fait fonctionner la récurrence.\ps 

Il reste donc à montrer que le cas où le paramètre de $\Jac_{\chi_i}$ n'est pas discret est exclu.

D'après le Théorème \ref{Ato4.2}, ce cas ne se produit que si $\varphi=(\rho \otimes U_{2x+1}) \oplus (\rho \otimes U_{2x-1}) \oplus R$ avec $2x-1>0$ et $\rho=\chi_i|\cdot|^{-x}$. On a alors bien 
$\varphi_-=(\rho \otimes U_{2x-1}) \oplus (\rho \otimes U_{2x-1}) \oplus R$ paramètre non discret de $\Sp_{2n-2}(\C)$.
Notons $H$ le sous-groupe de Levi de $\Sp_{2n-2}(\C)$ défini par $H=\GL_{2x-1}(\C) \times \Sp_{2a}(\C)$ avec $2a=\dim R$ et $2a+2\times(2x-1)=2n-2$. On rappelle que la représentation $\rho \otimes U_{2x-1}$ est de dimension $2x-1$, qui est un entier pair. Le paramètre $\varphi_-$ est donc à valeurs dans le groupe $H$ et
on note $\phi_-$ ce même paramètre corestreint au groupe $H$. Alors $\phi_-$ est un paramètre, manifestement discret, pour le groupe $L=\GL_{2x-1}(F) \times \SO_{2a+1}(F)$ (qui est un sous-groupe de Levi de $\SO_{2n-1}(F)$) et on a bien $\widehat{L}=H$ (avec le même raccourci de notation qu'au début du paragraphe \ref{Paramètres de Langlands}).

On peut donc considérer le paquet de Langlands $\Pi(\phi_-)$ composé de séries discrètes pour $L$. Or $\phi_-=\varphi_1 \times \varphi_2$ avec $\varphi_1= \rho \otimes U_{2x-1}$ et $\varphi_2=R$ donc 
\[
\Pi(\phi_-)=\{ \pi_1 \boxtimes \pi_2 \ | \pi_1 \in \Pi(\varphi_1), \pi_2 \in \Pi(\varphi_2)\}.
\]

Or $\Pi(\varphi_1)$ contient un seul élément (c'est un paquet pour $\GL_{2x-1}(F)$) qui est la représentation de Steinberg de $\GL_{2x-1}(F)$ tordue par le caractère $\rho$ ; $\Pi(\varphi_2)$ est quant à lui un paquet de séries discrètes de $\SO_{2a+1}(F)$. On sait alors comment construire les représentations tempérées de $\SO_{2n-1}(F)$ correspondant au paramètre (tempéré) $\varphi_-$ à partir des séries discrètes de $L$ correspondant au paramètre (discret) $\phi_-$ : c'est l'induction parabolique qui est en jeu ((\cite{waldspurger_2003} Proposition III.4.1). On a donc :
\[
\Jac_{\chi_i} \hookrightarrow \boldsymbol{i}_L^{\SO_{2n-1}(F)} ((\rho \otimes \St_{\GL_{2x-1}(F)}) \boxtimes \pi_2).
\]
Les $\J_{2n-1}$-invariants s'injectent tout autant, ils sont non nuls pour $\Jac_{\chi_i}$ donc non nuls pour l'induite.
Par le Lemme \ref{lemme_inv_induite_induisante}, la représentation induisante a alors des $(\J_{2n-1} \cap L)$-invariants 
et donc, par la Proposition \ref{Param_inter_Levi}, $(\rho \otimes \St_{\GL_{2x-1}(F)})$ a des $\GL_{2x-1}(\Ok)$-invariants. Or, un paramètre discret et non ramifié pour $\GL_n$ n'existe que pour $n=1$. Puisque $2x-1$ est un entier pair strictement positif, nous obtenons donc une contradiction, ce qui exclut finalement bien le cas où $\Jac_{\chi_i}$ n'est pas discret.
\end{proof}

Nous avons finalement démontré le théorème suivant.
\begin{thm}\label{thm_pcpl_ser_disc}
Les seules séries discrètes de $\SO_{2n+1}(F)$ (pour $n$ variable) ayant des $\J^+$-invariants sont :
\begin{itemize}
\item la représentation triviale de $\SO_1(F)$ ;
\item la représentation $\St_{\SO_3(F)}$ ;
\item la représentation $\eta \otimes \St_{\SO_3(F)}$.\ps 
\end{itemize}

Les seules séries discrètes de $\SO_{2n+1}(F)$ (pour $n$ variable) ayant des $\J$-invariants sont :
\begin{itemize}
\item la représentation triviale de $\SO_1(F)$ ;
\item la représentation $\eta \otimes \St_{\SO_3(F)}$.\ps 
\end{itemize}

Dans tous les cas, les invariants considérés sont de dimension 1.
\end{thm}

\section{Représentations tempérées}\label{Représentations tempérées}
\subsection{Représentations non ramifiées}\label{Représentations non ramifiées}
Il est utile de rappeler ici les résultats concernant les représentations non ramifiées, ou $\K_0$-sphériques, \ie ayant des invariants par le sous-groupe compact maximal hyperspécial $\K_0$.

\begin{prop}\label{ser_disc_nr}
La seule série discrète de $\SO_{2n+1}(F)$ (pour $n$ variable) ayant des $\K_0$-invariants est la représentation triviale de $\SO_1(F)$.
\end{prop}
\begin{proof}
Cette proposition est \emph{parallèle} au Théorème \ref{thm_pcpl_ser_disc}. Il suffit de considérer le paramètre de Langlands $\varphi$ d'une telle représentation, qui doit être à la fois discret et non ramifié. C'est donc une somme de caractères autoduaux (quadratiques) distincts, et $\varphi$ est à valeurs dans $\Sp_{2n}(\C)$. Cela impose $n=0$ et on a bien affaire à la représentation triviale du groupe trivial $\SO_1(F)$.
\end{proof}

La proposition suivante est bien connue, nous en donnons néanmoins une démonstration. 
\begin{prop}\label{thm_rep_sph_temp}
Soit $(\pi,V)$ une représentation tempérée irréductible de $G$. On suppose que $\pi^{\K_0} \neq \{0\}$. Alors $\pi \simeq \ii_B^G \chi$ où $\chi$ est un caractère non ramifié unitaire du tore $T$.

De plus, $\pi$ est la seule représentation dans son paquet de Langlands, \ie en notant $\varphi=\Ll(\pi)$, $\Pi_\varphi$ est un singleton.
\end{prop}
\begin{proof}
Par la classification de Langlands (\cite{waldspurger_2003} Proposition III.4.1), il existe un sous-groupe de Levi standard $M=\GL_{n_1}(F)\times \cdots \times \GL_{n_r}(F) \times \SO_{2m+1}(F)$ avec $2(n_1+\cdots+n_r)+2m+1=2n+1$ et $P=MN$ sous-groupe parabolique standard ainsi que $\delta_i$ série discrète de $\GL_{n_i}(F)$ (pour $i \in \{1,\cdots,r\}$) et $\tau$ série discrète de $\SO_{2m+1}(F)$ tels que :
\[
\pi \hookrightarrow \ii_P^G (\delta_1 \boxtimes \cdots \boxtimes \delta_r \boxtimes \tau).
\]
Or $\pi$ a des $\K_0$-invariants donc l'induite également et la représentation induisante a des $\K_0\cap M$-invariants. On sait par ailleurs (Proposition \ref{Hypersp_inter_Levi}) que :
\[
\K_0 \cap M=\prod_{i=1}^r \GL_{n_i}(\Ok) \times \SO_{2m+1}(\Ok)
\]
donc $\delta_i$ est non ramifiée pour tout $i$ et $\tau$ également.
Une série discrète de $\GL_{n_i}(F)$ non ramifiée est un caractère 
unitaire donc les $\delta_i$ sont des caractères non ramifiés, notés $\chi_i$, et $n_i=1$. Quant à $\tau$, la Proposition \ref{ser_disc_nr} nous dit que c'est la représentation triviale de $\SO_1(F)$.

On a ainsi $m=0$ et $n_i=1$ pour tout $i$, si bien que l'on est en train d'induire le caractère $\chi_1 \boxtimes \cdots \boxtimes \chi_n$ non ramifié de $T$.

Il reste à voir que cette induite est irréductible : c'est l'objet d'un résultat de David Keys \cite{Keys1982} repris par Jian-Shu Li au Corollaire 2.6 de \cite{Li1992}, utilisant le fait que $G$ est adjoint et semi-simple.
\end{proof}

\emph{Remarque :} On peut également montrer l'irréductibilité de l'induite totale en utilisant la théorie d'Arthur. On induit unitairement un caractère unitaire donc la représentation induite est unitaire, partant semi-simple. Il suffit alors de montrer que l'induite ne possède qu'une sous-représentation irréductible pour conclure (par le résultat de multiplicité 1 de M\oe{}glin). Or on sait par la classification de Langlands pour les représentation tempérées que toutes les sous-représentations irréductibles de $\ii_B^G (\chi_1 \boxtimes \cdots \boxtimes \chi_n)$ sont dans le même paquet $\Pi_\varphi$ avec 
\[
\varphi=(\chi_1 \oplus \cdots \oplus \chi_n) \oplus (\chi_n^{-1} \oplus \cdots \oplus \chi_1^{-1}),
\]
où l'on a encore noté $\chi_i$ le caractère du groupe de Weil-Deligne correspondant (c'est la théorie du corps de classes).

Ce paquet ne contient pas d'autre élément et les éléments sont indexés par les caractères du groupe $\mathcal{S}_\varphi$ du Théorème \ref{LLC_SO}. Pour conclure, il faut et il suffit donc de montrer que ce dernier groupe est trivial : il n'y aura alors qu'un seul élément dans $\Pi_\varphi$ et donc qu'une seule sous-représentation irréductible dans l'induite qui sera finalement irréductible.

Pour déterminer le centralisateur de l'image du paramètre dans $\Sp_{2n}(\C)$, il faut isoler, parmi les caractères $\chi_i$, ceux qui sont à valeurs dans $\{\pm 1\}$, \ie de carré 1. Il y a exactement deux tels caractères : le caractère trivial et le caractère $\eta$ introduit au paragraphe \ref{Norme spinorielle}.  On a donc, après regroupement, $\varphi= (2a) \mathbf{1} \oplus (2b) \eta \oplus \bigoplus_i m_i (\chi_i \oplus \chi_i^{-1})$ avec $a+b+\sum_i m_i=n$ et
\[
\Cent(\varphi)=\Sp_{2a}(\C) \times \Sp_{2b}(\C) \times \prod_i \GL_{m_i} (\C).
\]
En particulier, ce centralisateur est connexe, si bien que $\mathcal{S}_\varphi$ est le groupe trivial.\newline

Terminons ce paragraphe par la précision suivante sur les invariants des séries principales non ramifiées.

\begin{lemme} \label{lemme_inv_sph_induites}
Soit $\chi$ un caractère non ramifié de $T$ et soit la représentation induite $\ii_B^G \chi$. Alors l'espace $(\ii_B^G \chi)^{\K_0}$ est de dimension 1.
\end{lemme}
\begin{proof}
Ce lemme est \emph{parallèle} au Lemme \ref{lemme_inv_param_induites} et se démontre avec les mêmes ingrédients, à l'aide de la décomposition $G=B\K_0$ du Corollaire \ref{Factorisation G=BK_0}.
\end{proof}

\subsection{Représentations \og paramodulaires \fg{} }
\begin{thm}\label{thm_rep_temp}
Soit $(\pi,V)$ une représentation tempérée irréductible de $G$. On suppose que $\pi^{\J^+} \neq \{0\}$. Alors \ps 
\begin{itemize}
\item soit $\pi$ est non ramifiée et $\pi \simeq \ii_B^G \chi$ où $\chi$ est un caractère non ramifié unitaire du tore $T$ ;\ps 
\item soit $\pi \simeq \ii_P^G (\psi \otimes \alpha \St_{\SO_3(F)})$ avec $P=MN$ où $M=\GL_1^{n-1}\times \SO_3$, $\psi$ un caractère non ramifié unitaire de $\GL_1^{n-1}$ et $\alpha \in \{\1,\eta\}$. \ps 
\end{itemize}

Dans tous les cas, $\pi$ est la seule représentation dans son paquet de Langlands, \ie en notant $\varphi=\Ll(\pi)$, $\Pi_\varphi$ est un singleton.
\end{thm}
\begin{proof}
Par la classification de Langlands (\cite{waldspurger_2003} Proposition III.4.1), il existe un sous-groupe de Levi standard $M=\GL_{n_1}(F)\times \cdots \times \GL_{n_r}(F) \times \SO_{2m+1}(F)$ avec $2(n_1+\cdots+n_r)+2m+1=2n+1$ et $P=MN$ sous-groupe parabolique standard ainsi que $\delta_i$ série discrète de $\GL_{n_i}(F)$ (pour $i \in \{1,\cdots,r\}$) et $\tau$ série discrète de $\SO_{2m+1}(F)$ tels que :
\[
\pi \hookrightarrow \ii_P^G (\delta_1 \boxtimes \cdots \boxtimes \delta_r \boxtimes \tau).
\]

On a alors également, par la Proposition \ref{prop_norme_spin_Levi} :
\[
\eta \otimes \pi \hookrightarrow \ii_P^G ((\eta\circ\det)\delta_1 \boxtimes \cdots \boxtimes (\eta\circ\det)\delta_r \boxtimes (\eta\circ\nu)\tau).
\]

Or, puisque $\pi$ a des $\J^+$-invariants, une des deux représentations $\pi$ ou $\eta\otimes\pi$ a des $\J$-invariants. Or ces $\J$-invariants s'injectent dans les induites respectives et, par le Lemme \ref{lemme_inv_induite_induisante}, correspondent à des $\J\cap M$-invariants de la représentation induisante.\ps 

Dans le premier cas, on a donc que chaque $\delta_i$ a des invariants (non triviaux) par $\GL_{n_i}(\Ok)$ et $\tau$ a des invariants par $\J_{2m+1}$.

Dans le second cas, on a que chaque $(\eta\circ\det)\delta_i$ a des invariants (non triviaux) par $\GL_{n_i}(\Ok)$ et $(\eta\circ\nu)\tau$ a des invariants par $\J_{2m+1}$. Or $(\eta\circ\det)$ est trivial sur $\GL_{n_i}(\Ok)$ donc chaque $\delta_i$ a encore des invariants par $\GL_{n_i}(\Ok)$.

Une série discrète de $\GL_{n_i}(F)$ non ramifiée est un caractère unitaire de $\GL_1(F)$ donc les $\delta_i$ sont des caractères non ramifiés de $F^\times$, notés $\chi_i$, et $n_i=1$.  Quant à $\tau$, c'est soit la représentation triviale de $\SO_1(F)$, soit $\St_{\SO_3(F)}$, soit $\eta \St_{\SO_3(F)}$ (Théorème \ref{thm_pcpl_ser_disc}).\ps




Il reste à voir que ces induites sont irréductibles. D'après les résultats d'Arthur, il suffit de montrer que $\mathcal{S}_\varphi$ est trivial avec
\[
\varphi=(\chi_1 \oplus \cdots \oplus \chi_r) \oplus \varphi_\tau \oplus (\chi_r^{-1} \oplus \cdots \oplus \chi_1^{-1}),
\]
où $\varphi_\tau=\Ll(\tau)$ et on l'on a encore noté $\chi_i$ le caractère du groupe de Weil-Deligne correspondant (c'est la théorie du corps de classes).

%
%

Le cas où $\tau$ est la représentation triviale est traité par la Proposition \ref{thm_rep_sph_temp}.

Il reste donc à considérer le cas où $\tau$ est la représentation $\alpha \St_{\SO_3(F)}$ avec $\alpha \in \{\1,\eta\}$.
On a alors $\varphi=(\chi_1 \oplus \cdots \oplus \chi_{n-1}) \oplus (\alpha\otimes U_2) \oplus (\chi_{n-1}^{-1} \oplus \cdots \oplus \chi_1^{-1})$ avec tous les $\chi_i$ non ramifiés.

Comme dans la Remarque qui suit la Proposition \ref{thm_rep_sph_temp}, les caractères non ramifiés vont faire apparaître des composantes du type $\GL_{m}(\C)$ ou $\Sp_{2a}(\C)$ dans le centralisateur, qui sont connexes. Reste à considérer $\alpha \otimes U_2$, représentation irréductible autoduale symplectique, dont le centralisateur est $\{\pm I_2\}$. Le centralisateur de $\varphi$ est donc le produit de $\{\pm I_2\}$ avec un groupe connexe et $\mathcal{S}_\varphi$ est à nouveau trivial (puisqu'on quotiente encore par le centre de $\Sp_{2n}(\C)$).
\end{proof}

On connaît les dimensions précises des invariants dans le premier cas du Théorème par le Lemme \ref{lemme_inv_param_induites}. Il nous faut encore gérer le second.

\begin{lemme}\label{lemme_inv_param_ind_St}
Soit $P=MN$ le sous-groupe parabolique standard avec $M=\GL_1(F)^{n-1} \times \SO_3(F)$. Soit $\psi$ un caractère non ramifié unitaire de $\GL_1(F)^{n-1}$ et $\alpha \in \{\1,\eta\}$.

Alors l'espace $(\ii_P^G (\psi \otimes \alpha \St_{\SO_3(F)}))^{(\J,\eps_\alpha)}$ est de dimension 1 et l'espace $(\ii_P^G (\psi \otimes \alpha \St_{\SO_3(F)}))^{(\J,-\eps_\alpha)}$ est nul, où $\eps_\1=-$ et $\eps_\eta=+$.
\end{lemme}
\begin{proof}
La traduction du Lemme \ref{lemme_inv_induite_induisante} dans le contexte présent nous donne :
\[
(\ii_P^G (\psi \otimes \alpha \St_{\SO_3(F)})^{(\J,\eps)} \simeq \psi^{(\Ok^\times)^{n-1}} \otimes (\alpha \St_{\SO_3(F)})^{(\J_3,\eps)}.
\]

Le premier facteur du membre de droite est une droite vectorielle puisque $\psi$ est non ramifié.
Enfin, le Corollaire \ref{cor_ser_disc_so(3)}, combiné au Lemme \ref{inv_param_Steinberg} nous indique que $(\alpha \St_{\SO_3(F)})^{(\J_3,\eps_\alpha)}$ est de dimension 1, tandis que $(\alpha \St_{\SO_3(F)})^{(\J_3,-\eps_\alpha)}$ est nul.
\end{proof}

\newpage
\chapter{Conducteur}\label{Conducteur}
Il s'agit maintenant d'introduire la notion de conducteur, qui va quantifier la ramification d'une représentation. Les résultats concernant le produit tensoriel seront cruciaux pour notre étude puisqu'ils vont nous permettre de généraliser au cas du conducteur $p$, les calculs qui étaient effectués dans \cite{Chen-Lannes} en conducteur 1. Le contrôle du conducteur d'un produit tensoriel est d'ailleurs un des éléments-clés de l'énoncé de finitude de \cite{Chen-HM}. Ce dernier article utilise les résultats de \cite{BH_cond} qui démontrent l'inégalité voulue \og côté automorphe\fg. \ps 

Nous introduisons ici le conducteur uniquement \og côté galoisien \fg et c'est par la correspondance de Langlands locale que cela définit le conducteur \og côté automorphe\fg.
\section{Conducteur d'une représentation galoisienne}\label{Conducteur d'Artin}
Comme précédemment, $F$ est un corps local non-archimédien de caractéristique nulle, $k$ son corps résiduel (fini), de caractéristique $p$. Pour la clarté, nous écrirons dans cette partie $k_F$ au lieu de $k$ quand cela paraîtra nécessaire.

Soit de plus $K$ une extension galoisienne finie de $F$. On désigne par $\mathcal{O}_K$ l'anneau des entiers de $K$ et par $\mathfrak{p}_K$ son unique idéal maximal. Le corps résiduel $k_K=\mathcal{O}_K/\mathfrak{p}_K$ est une extension finie de $k_F$. On pose $G=\Gal(K/F)$.
\subsection{Groupes de ramification}\label{Groupes de ramification}
Nous suivons ici la présentation de \cite{Se}, chapitre IV.

Il est facile de voir qu'un élément de $G$ stabilise $\mathcal{O}_K$ et $\mathfrak{p}_K$. Il induit donc un automorphisme du quotient $k_K$ qui fixe point par point $k_F$, \ie un élément de $\Gal(k_K/k_F)$. De la même façon, un élément de $G$ stabilise $\pkk^j$ pour $j>1$ et induit un automorphisme de l'anneau $\Okk / \pkk^{j}$. Précisons ceci.
\begin{defi}
Soit $i \geq -1$ un entier. On définit $G_i$ le $i$-ème groupe de ramification supérieure comme l'ensemble des éléments de $G$ agissant trivialement sur $\Okk / \pkk^{i+1}$.
\end{defi}
Puisque $G$ agit sur l'anneau quotient $\Okk / \pkk^{i+1}$ par automorphismes d'anneaux, le groupe $G_i$ est \emph{distingué} dans $G$. Par ailleurs les $G_i$ forment une suite décroissante.

\begin{ex}
On a $G_{-1}=G$ et $G_0=I_K$ (sous-groupe d'inertie)
\end{ex}

Nous avons $\Okk=\Ok [x]$ pour un certain élément primitif $x$ (\cite{Se} Chap. III, §6, Proposition 12). On a alors le
\begin{lemme}
Soit $s \in G$ et $i \geq -1$. Les conditions suivantes sont équivalentes :
\begin{enumerate}[(i)]
\item $s$ appartient à $G_i$,
\item $v_K(s(a)-a) \geq i+1$ pour tout $a \in \Okk$,
\item $v_K(s(x)-x) \geq i+1$.
\end{enumerate}
où $v_K$ désigne la valuation discrète canonique.
\end{lemme}
\begin{proof}
L'équivalence de \textit{(i)} et \textit{(ii)} est une simple traduction. Pour voir que \textit{(i)} et \textit{(iii)} sont équivalents, il suffit de remarquer que l'image $x_i$ de $x$ dans $\Okk / \pkk^{i+1}$ engendre $\Okk / \pkk^{i+1}$ comme $\Ok$-algèbre. Ainsi $s$ opère trivialement sur $\Okk / \pkk^{i+1}$ si, et seulement si $s(x_i)=x_i$.
\end{proof}

On en déduit alors immédiatement que $G_i=\{1\}$ pour $i$ assez grand.\medskip

Mentionnons sans démonstration quelques résultats concernant les quotients $G_i/G_{i+1}$ (\cite{Se} Chap. IV,§ 2).

Pour $i\geq 0$, le quotient $G_i/G_{i+1}$ s'injecte dans $U_i/U_{i+1}$, où $U_i$ est le sous-groupe des éléments de $\Okk^\times$ congrus à 1 modulo $\pkk^i$. Or on sait que le groupe $U_i/U_{i+1}$ est isomorphe au groupe multiplicatif $(k_K^\times,\cdot)$ pour $i=0$ (donc cyclique d'ordre premier à $p$), et au groupe additif $(k_K,+)$ (un $k_F$-espace vectoriel de dimension finie) pour $i>0$. Ainsi, $G_0/G_1$ est d'ordre premier à $p$ (et cyclique), alors que les $G_i/G_{i+1}$ pour $i>0$ sont des $p$-groupes (produits directs de groupes cycliques d'ordre $p$), ainsi donc que $G_1$ (appelé groupe d'inertie sauvage), qui est alors l'unique $p$-Sylow de $G$.

\subsection{Définition et premières propriétés}
Soit maintenant $V$ une représentation complexe de dimension finie (que l'on supposera toujours non nulle) du groupe de Galois fini $G$ de l'extension galoisienne finie $K/F$. Son conducteur d'Artin est par définition l'idéal $\pk^{\aar(V)}$, où l'exposant (dit \emph{exposant d'Artin} de $V$) $\aar(V)$ est
\[
\aar(V) = \frac{1}{|G_0|}  \sum_{i\geq 0} |G_i|\, \mathrm{codim} (V^{G_i}),
\]
où $V^{G_i}$ désigne le sous-espace de $V$ constitué des $G_i$-invariants de $V$, et \linebreak $\mathrm{codim} (V^{G_i})= \dim V - \dim V^{G_i}$. Cette quantité est bien définie car on a $G_i =\{1\}$ pour $i$ suffisamment grand et que dans ce cas on a $\mathrm{codim} (V^{G_i})=0$ (la somme est donc finie).

On a toujours $\aar(V) \geq \mathrm{codim} (V^{G_0})$, avec égalité si, et seulement si $V^{G_1}=V$ (\emph{i.e.} $V$ est \og modérément ramifiée \fg{}), autrement dit, si $V$ se factorise en une représentation de $G/G_1$). En particulier, on a $\aar(V) = 0$ si, et seulement si $V$ est non ramifiée, \emph{i.e.} $V^{G_0} = V$. On a aussi $\aar(V^\vee)=\aar(V)$ où $V^\vee$ désigne la représentation duale de $V$ (il n'y a pas de subtilité ici puisqu'il s'agit de la représentation d'un groupe fini) et, 
pour deux représentations $V$ et $W$ de $G$, $\aar(V \oplus W) = \aar(V)+ \aar(W)$.\medskip

Le résultat suivant se trouve dans la thèse de Guy Henniart (\cite{Hen} Théorème 7.15 p. 38), nous en détaillons néanmoins la preuve avec des notations légèrement différentes.

Notons $i(V)$ le plus grand entier $i$ tel que $G_i$ agisse non trivialement dans $V$ (et on pose $i(V) = -1$ si $G_i$ agit trivialement pour tout $i$).

On note aussi $f(j)$ la somme des $|G_0/G_k|^{-1}$ pour $0 \leq k\leq j$ (de sorte que $f(j)-1$ est aussi la valeur en $j$ de la fonction de Herbrand, selon \cite{Se} Chap. IV, §3). On a $f(-1)=0$ par convention, et $f(j) \geq 1$ si $j \geq 0$.

\begin{prop}\label{Exposant Henniart}
Supposons $V$ irréductible. Alors $\aar(V) = \dim V  \times f(i(V))$.
\end{prop}
\begin{proof}
On commence par exclure le cas où $V$ est la représentation triviale, auquel cas $i(V)=-1$ et les deux membres de l'égalité sont égaux à 0. Réciproquement, si $i(V)=-1$, $G$ agit trivialement et, par irréductibilité, $V$ est la représentation triviale.\ps 

Posons $i = i(V)$. L'action de $G_i$ se factorise à travers le groupe $G_i/G_{i+1}$, abélien (car $i\geq 0$) d'après les résultats de la fin du paragraphe \ref{Groupes de ramification}.

La $G_i$-représentation $V$ est donc somme de caractères de $G_i/G_{i+1}$, l'un au moins d'entre eux, disons $c$, étant non trivial. Notons $V(c)$ la composante isotypique de $c$. Par irréductibilité de $V$, on a $V = \sum_{g \in G} g \cdot V(c)$, et d'autre part $g \cdot V(c)$ est stable par $G_i$ : c'est en fait $V(c')$ où $c'$ est le conjugué extérieur du caractère $c$ de $G_i$ par $g^{-1}$. Il est donc de la forme $V(c')$ avec $c'$ éventuellement égal à $c$, mais en tout cas $c' \neq \1$.

Ainsi on constate que $V$, vue comme $G_i$-représentation, est somme de caractères non triviaux : $V^{G_i} = \{0\}$ et, plus généralement, $V^{G_j} = \{0\}$ pour $j \leq i$. D'autre part on a $V^{G_{i+1}} = V$. On a donc bien $\aar(V) = \frac{1}{|G_0|} \sum_{i\leq i(V)} |G_i| \dim V = \dim V \times f(i(V))$
\end{proof}
\begin{cor}\label{cor522}
Supposons encore $V$ irréductible. Si $V^{G_0} \neq V$, alors $\aar(V) \geq \dim V$ avec égalité si, et seulement si $V^{G_1}=V$ ($V$ est modérément ramifiée)
\end{cor}
\begin{proof}
L'hypothèse signifie que $i(V) \geq 0$. Il suffit alors d'appliquer la proposition précédente. Par ailleurs, la fonction $f$ étant strictement croissante, on a bien $f(i(V))=1 \Leftrightarrow i(V)=0 \Leftrightarrow V^{G_1}=V $.
\end{proof}
\subsection{Conducteur d'un produit tensoriel}
Notre objectif est de majorer la norme du conducteur du produit tensoriel de deux représentations. Nous allons raisonner, comme précédemment, sur l'exposant d'Artin.

\begin{lemme}\label{Exposant produit tensoriel d'irr}
Soient $V$ et $V'$ deux représentations irréductibles de $G$. Si $i(V) \leq i(V')$ (ce que l'on peut toujours supposer quitte à échanger les rôles de $V$ et $V'$), on a l'inégalité $\aar(V \otimes V') \leq (\dim V) \aar(V').$ 
\end{lemme}
\begin{proof}
Notons $i_0=i(V')$. Alors $G_{i_0}$ agit trivialement sur $V'$ (par définition) et sur $V$ (car $i(V)\leq i_0$), donc sur $V\otimes V'$. Ainsi
\begin{align*}
\aar(V \otimes V') &= \frac{1}{|G_0|}  \sum_{i\geq 0} |G_i|\, \mathrm{codim} ((V \otimes V')^{G_i}) \\
				   &=  \frac{1}{|G_0|}  \sum_{i\leq i_0} |G_i|\, \mathrm{codim} ((V \otimes V')^{G_i}) \\
				   &\leq \frac{1}{|G_0|}  \sum_{i\leq i_0} |G_i| \dim V \dim V' \\
				   &\leq (\dim V \dim V') f(i(V')) \\
				   &\leq (\dim V) \aar(V').
\end{align*}
%
On remarque qu'on a en fait seulement utilisé l'irréductibilité de $V'$.
\end{proof}

\begin{lemme}
On ne suppose plus $V$ irréductible. On a l'inégalité $\aar(V) \leq (\dim V) f(i(V))$, avec égalité si, et seulement si tous les facteurs irréductibles $W$ de $V$ ont même indice $i(W)$ (en particulier si $V$ est irréductible donc).
\end{lemme}
\begin{proof}
On décompose $V$ en somme d'irréductibles (représentation complexe d'un groupe fini) :
\[
V= \bigoplus_{j=1}^k W_j.
\]

Alors
$$
\aar(V)=\sum_{j=1}^k \aar(W_j)=\sum_{j=1}^k \dim W_j f(i(W_j))
$$
en utilisant la Proposition \ref{Exposant Henniart} pour chacune des représentations irréductibles $W_j$. Or, on a évidemment $i(W_j) \leq i(V)$ et la fonction $f$ est croissante donc $\aar(V) \leq (\sum_{j=1}^k \dim W_j) f(i(V))$, soit $\aar(V) \leq (\dim V) f(i(V))$.\ps 

Supposons maintenant que tous les facteurs irréductibles ont même indice $i(W_j)$. Alors la décomposition en somme directe nous assure que $i(V)$ est égal à cet indice commun et on a donc bien $\aar(V) = (\dim V) f(i(V))$. 

Dans le cas contraire, on a un facteur irréductible, disons $W_1$ tel que $i(W_1) < i(V)$ (puisqu'on voit qu'en fait $i(V)=\max_j i(W_j)$). Alors, en utilisant le fait que la fonction $f$ est strictement croissante, on voit que l'inégalité précédente est stricte.
\end{proof}

On en arrive à la proposition suivante :
\begin{prop}\label{exposant d'un produit tensoriel}
Soient $V$ et $V'$ deux représentations de $G$. Alors $$\aar(V \otimes V') \leq \aar(V) \dim V' + \aar(V') \dim V - \min (\aar(V),\aar(V')).$$
\end{prop}
\begin{proof}
Commençons par le cas où $V$ et $V'$ sont irréductibles. Alors, on peut reformuler le résultat du Lemme \ref{Exposant produit tensoriel d'irr} sous la forme suivante :
$$
\aar(V \otimes V') \leq (\dim V) (\dim V') \max(f(i(V)),f(i(V'))).
$$

Or, pour deux réels $\alpha$ et $\beta$, $\max(\alpha,\beta)=\alpha+\beta-\min(\alpha,\beta)$. Ceci nous donne, en utilisant que $\aar(V)=(\dim V) f(i(V))$ (irréductibilité de $V$) :
$$
\aar(V \otimes V') \leq  \aar(V) (\dim V') + \aar(V') (\dim V) - (\dim V) (\dim V') \min(f(i(V)),f(i(V'))).
$$

Pour conclure, il suffit donc de voir que $\min(df,d'f') \leq dd' \min (f,f')$ avec des notations évidentes, ce qui est vrai puisque $d$ et $d'$ sont des entiers supérieurs ou égaux à 1 (on a exclu depuis le début du paragraphe la représentation nulle).

La proposition est donc démontrée dans le cas où $V$ et $V'$ sont irréductibles.
L'application (symétrique) $(V,V') \mapsto \aar(V\otimes V')$ est bilinéaire : par définition, $\aar((U_1\oplus U_2)\otimes V')=\aar((U_1\otimes V')\oplus(U_2\otimes V'))=\aar(U_1\otimes V')\oplus\aar(U_2\otimes V')$. Pour déduire le cas général du cas irréductible, il suffit alors de constater que 
\[
\min(a_1,a') +\min(a_2,a') \geq \min (a_1+a_2,a')
\]
pour tout triplet d'entiers naturels $(a_1,a_2,a')$.
\end{proof}

\section{Conducteur d'une représentation du groupe de Weil}\label{Conducteur d'une représentation du groupe de Weil}

On peut étendre la définition du conducteur au cas d'une représentation continue du groupe de Galois absolu $\rho : \Gal(\overline{F} / F) \rightarrow \GL(V)$ de dimension finie. Par continuité, $\rho$ se factorise à travers $\Gal(K / F)$ où $K$ est une extension \emph{galoisienne finie} de $F$ et on note $\rho_K$ le morphisme induit. On pose alors $\aar(\rho)=\aar(\rho_K)$ et on vérifie que cette définition ne dépend pas du choix de l'extension $K$. Ceci vaut également pour une représentation (continue, de dimension finie) de $\W_F$ de type galoisien (\cf § \ref{Groupe de Weil}).\ps 

Pour les autres cas, on fait appel à la Proposition \ref{rep_WF_rep_type_galoisien} : soit $\rho$ une représentation \emph{irréductible} $\rho$ de $\W_F$, il existe alors $\chi$ un caractère non ramifié tel que $\rho \otimes \chi$ soit de type galoisien. On pose alors $\aar(\rho)=\aar(\rho \otimes \chi)$ et on vérifie que cela ne dépend pas du choix de $\chi$. 
On a d'ailleurs toujours $\aar(\rho^\vee)=\aar(\rho)$.
Par additivité, l'exposant d'Artin est alors défini pour toutes les représentations semi-simples du groupe de Weil (qui sont les seules que nous considérons).

Enfin, étant données deux représentations $(\rho,V)$ et $(\rho',V')$ du groupe de Weil, on peut trouver $\chi$ et $\chi'$ caractères non ramifiés tels que $(\rho \otimes \chi,V)$ et $(\rho' \otimes \chi',V')$ soient de type galoisien. Chacune se factorise par une extension galoisienne finie de $F$ et, en raisonnant dans l'extension produit (qui est toujours galoisienne finie), les résultats sur l'exposant d'Artin du produit tensoriel demeurent.

\section{Conducteur d'une représentation du groupe de Weil-Deligne}\label{Conducteur d'une représentation du groupe de Weil-Deligne}

Nous considérons ici le groupe de Weil-Deligne de $F$, $\WD_F$, comme le produit direct du groupe de Weil $\W_F$ et du groupe compact $\SU(2)$. Les représentations irréductibles de ce second groupe sont entièrement caractérisées par leur dimension. John Tate définit dans \cite{Tate-Corvallis} un exposant d'Artin \og élargi \fg{} pour les représentations de Weil-Deligne, représentations qui font intervenir un élément nilpotent plutôt que le groupe $\SU(2)$. Reprenons d'abord ses définitions.
\begin{defi}\label{def_rep_de_WD}
Une \emph{représentation de Weil-Deligne} du groupe de Weil $\W_F$ est la donnée d'un triplet $(\rho, V, N)$ où
\begin{enumerate}
\item $(\rho, V)$ est une représentation linéaire complexe de dimension finie du groupe de Weil $\W_K$,
\item $N$ est un endomorphisme nilpotent de $V$,
\item pour tout $w \in \W_K$, $\rho(w)N\rho(w)^{-1}=|w|N$.
\end{enumerate}
\end{defi}

Pour $(\rho,V,N)$ une telle représentation de Weil-Deligne, on définit :
\[
\aar_\mathrm{WD}((\rho,V,N))=\aar_\mathrm{W}((\rho,V))+ \dim V^I - \dim V_N^I
\]
où $(\rho,V)$ désigne la représentation du groupe de Weil sous-jacente, $\aar_\mathrm{W}$ son exposant d'Artin tel que défini au paragraphe précédent et où $V_N=\Ker N$.\medskip

Sans trop détailler la traduction classique entre représentations de Weil-Deligne et représentations de $\W_F \times \SU(2)$ (exposée avec précision dans \cite{GR}), voyons quelle formule on obtient pour une représentation de ce dernier groupe.

Toutes les représentations considérées seront semi-simples, il suffit donc, par additivité de l'exposant d'Artin, de le définir sur une représentation irréductible, \ie sur un certain $V=X\otimes U_d$ où $X$ est une représentation irréductible de $\W_F$ et $d$ est un entier strictement positif. L'élément nilpotent correspondant à la représentation $U_d$ est essentiellement un bloc de Jordan de taille $d$ ; le noyau associé est donc de dimension 1. On a ainsi :
\begin{align*}
V^I& =X^I \otimes U_d, \\
\dim(V^I) &=\dim X^I \times d, \\
\dim(V_N^I)&=\dim X^I.
\end{align*}

Il reste à comprendre $\aar_\mathrm{W}(X\otimes U_d)$. Or $X\otimes U_d$, vue comme représentation du groupe de Weil $\W_F$, consiste simplement en $d$ copies de la représentation $X$, d'où $\aar_\mathrm{W}(X\otimes U_d)=\aar_\mathrm{W}(X^{\oplus d})=d\aar_\mathrm{W}(X)$. Finalement, pour une représentation irréductible $V=X\otimes U_d$, on a :
\begin{equation}\label{awd_def}
\aar_\mathrm{WD}(X\otimes U_d)=d\aar_\mathrm{W}(X)+ (d-1)\dim X^I.
\end{equation}

On peut en fait être plus spécifique : si $\aar_\mathrm{W}(X)=0$, alors $X^I=X$ et la représentation non ramifiée $X$ se factorise donc en une représentation de $\W_F/I \simeq \Z$. Puisqu'elle est irréductible, elle est en fait de dimension 1 et $X$ est donc un caractère non ramifié. Si, inversement, $\aar_\mathrm{W}(X)>0$, la démonstration de la Proposition \ref{Exposant Henniart} nous montre que $X^I=\{0\}$. Finalement, pour $X\otimes U_d$ une représentation \emph{irréductible} de $\W_F \times \SU(2)$, on a 
\begin{equation}\label{awd_irr}
\aar_\mathrm{WD}(X\otimes U_d)=
\begin{cases}
d-1 &\text{ si } \aar_\mathrm{W}(X)=0 \\
d\aar_\mathrm{W}(X) &\text{ si } \aar_\mathrm{W}(X)>0
\end{cases}.
\end{equation}

On remarque par ailleurs que la formule \eqref{awd_def} vaut encore si $X$ n'est pas supposée irréductible (par additivité), 
et que $\aar_{\WD}((X\otimes U_d)^\vee)=\aar_{\WD}(X^\vee \otimes U_d)=\aar_{\WD}(X\otimes U_d)$.

\begin{prop}\label{exposant_produit_tensoriel_WD}
Soient $V$ et $V'$ deux représentations semi-simples de $\W_F \times \SU(2)$. Alors 
\[
\aar_\mathrm{WD}(V \otimes V') \leq \aar_\mathrm{WD}(V) \dim V' + \aar_\mathrm{WD}(V') \dim V - \min (\aar_\mathrm{WD}(V),\aar_\mathrm{WD}(V')).
\]
\end{prop}
\begin{proof}
Par semi-simplicité, le cas général se déduit du cas irréductible de la même manière que dans la preuve de la Proposition \ref{exposant d'un produit tensoriel}. 
Il suffit donc de prouver l'inégalité pour $V$ et $V'$ irréductibles. Notons donc $V=X\boxtimes U_d$ et $V'=Y \boxtimes U_e$ où $X,Y$ sont des représentations irréductibles de $\W_F$ et $d,e$ deux entiers strictement positifs (on choisit ici la notation du produit tensoriel extérieur pour des raisons de lisibilité).

On a $V\otimes V'=(X\otimes Y) \boxtimes (U_d \otimes U_e)$ et 
\begin{equation}\label{regle_prod_tens_rep_sym}
U_d \otimes U_e\simeq U_{d+e-1} \oplus U_{d+e-3} \oplus \cdots \oplus U_{|d-e|+1}=\bigoplus_{k=0}^{\min(d,e)-1} U_{|d-e|+1+2k}
\end{equation}
(règle sur le produit tensoriel de représentations symétriques de $\SU(2)$).\ps 

Nous avons trois cas à considérer selon que aucune, une ou deux des représentations du groupe de Weil $X$ et $Y$ sont ramifiées.\ps 

\textbf{Premier cas :} $X$ et $Y$ sont non ramifiées.

Ce sont alors des caractères, que l'on note $\chi$ et $\psi$, et on peut, sans nuire à la généralité, supposer que $d \geq e$.
\begin{align*}
\aar_\mathrm{WD}(V \otimes V')&=\aar_\mathrm{WD}\left(\chi \psi \boxtimes (\bigoplus_{k=0}^{e-1} U_{d-e+1+2k})\right) \\
&=\sum_{k=0}^{e-1} [(d-e+1+2k)-1] \\
&=(d-1)e
\end{align*}
en utilisant \eqref{awd_irr} ($\chi \psi$ est non ramifié). Or $d-1=\aar_\mathrm{WD}(\chi \boxtimes U_d)$ et $e=\dim(\psi \boxtimes U_e)$. Pour conclure, il suffit donc de voir que $\aar_\mathrm{WD}(\psi \boxtimes U_e)=e-1$, $\dim(\chi \boxtimes U_d)=d$ puis que $(e-1)d - \min(d-1,e-1) \geq 0$, ce qui est bien le cas puisque $d\geq 1$.\ps 

\textbf{Deuxième cas :} $X$ est non ramifiée et $Y$ est ramifiée.

Comme ci-dessus, notons $\chi$ pour $X$. On remarque déjà que $\aar_\mathrm{W}(\chi \otimes Y)=\aar_\mathrm{W}(Y)$.
\begin{align*}
\aar_\mathrm{WD}(V \otimes V')&=\aar_\mathrm{WD}\left((\chi \otimes Y) \boxtimes \left(\bigoplus_{k=0}^{\min(d,e)-1} U_{|d-e|+1+2k}\right)\right) \\
&=\sum_{k=0}^{\min(d,e)-1} \aar_\mathrm{WD} \left((\chi \otimes Y) \boxtimes U_{|d-e|+1+2k}\right) \\
&=\sum_{k=0}^{\min(d,e)-1} (|d-e|+2k+1) \aar_\mathrm{W}(\chi \otimes Y)  \\
&=de \aar_\mathrm{W}(Y).
\end{align*}
Or, toujours d'après \eqref{awd_irr}, $e\aar_\mathrm{W}(Y)=\aar_\mathrm{WD}(Y \boxtimes U_e)$, et $d=\dim(\chi \boxtimes U_d)$. Pour conclure, il suffit donc de voir que $\aar_\mathrm{WD}(\chi \boxtimes U_d)=d-1$, $\dim(Y \boxtimes U_e)=e\dim(Y)$ puis que $(d-1)e\dim(Y) \geq \min(d-1,e\aar_\mathrm{W}(Y))$, ce qui est bien le cas puisque $e\geq 1$ et $\dim(Y)\geq 1$.\ps 

\textbf{Troisième cas :} $X$ et $Y$ sont ramifiées.

On peut, sans nuire à la généralité, supposer que $d \geq e$.
\begin{align*}
\aar_\mathrm{WD}(V \otimes V')&=\aar_\mathrm{WD}\left((X\otimes Y) \boxtimes (\bigoplus_{k=0}^{e-1} U_{d-e+1+2k})\right) \\
&=\sum_{k=0}^{e-1} \aar_\mathrm{WD}((X\otimes Y) \boxtimes U_{d-e+1+2k}) \\
&=\sum_{k=0}^{e-1}\left( (d-e+1+2k)\aar_\mathrm{W}(X\otimes Y) + \dim (X\otimes Y)^I (d-e+2k) \right) \\
&\leq \left(\aar_\mathrm{W}(X\otimes Y) + \dim (X\otimes Y)^I\right)\sum_{k=0}^{e-1} (d-e+1+2k) \\
&\leq de\left(\aar_\mathrm{W}(X\otimes Y) + \dim (X\otimes Y)^I\right).
\end{align*}
Pour gérer la quantité entre parenthèses, il faut revenir à la définition de $\aar_\mathrm{W}$. On peut, quitte à tordre $X$ et $Y$ chacune par un caractère non ramifié, supposer qu'elles sont de type galoisien. Comme expliqué en fin du paragraphe \ref{Conducteur d'une représentation du groupe de Weil}, on peut trouver une extension galoisienne finie $K/F$ associée au produit tensoriel des représentations et on revient alors aux calculs du paragraphe \ref{Conducteur d'Artin} dont on reprend les notations. Alors
\begin{align*}
\aar_\mathrm{W}((X\otimes Y) + \dim (X\otimes Y)^I
&= \sum_{i\geq 0} \frac{|G_i|}{|G_0|} \mathrm{codim} ((X \otimes Y)^{G_i})+ \dim (X\otimes Y)^I \\
&= \sum_{i\geq 1} \frac{|G_i|}{|G_0|} \mathrm{codim} ((X \otimes Y)^{G_i})+ \dim (X\otimes Y) \\
&\leq \sum_{i\geq 1} \frac{|G_i|}{|G_0|} (\dim X)(\dim Y)+ \dim (X\otimes Y) \\
&\leq \sum_{i\geq 0} \frac{|G_i|}{|G_0|} (\dim X)(\dim Y)
\end{align*}
en utilisant $G_0=I$ et le même raisonnement que dans la preuve de la Proposition \ref{exposant d'un produit tensoriel}. Cette preuve nous dit également que, si $\aar_\mathrm{W}(Y)\geq \aar_\mathrm{W}(X)$, alors cette dernière quantité est égale à $(\dim X) \aar_\mathrm{W}(Y)$, et à $(\dim Y) \aar_\mathrm{W}(X)$ dans le cas contraire.\ps 

Si $\aar_\mathrm{W}(Y)\geq \aar_\mathrm{W}(X)$, alors on a montré que $\aar_\mathrm{WD}(V \otimes V') \leq de(\dim X) \aar_\mathrm{W}(Y)$, quantité que l'on reconnaît être égale à $\dim(X\boxtimes U_d) \aar_\mathrm{WD}(Y\boxtimes U_e)$. Pour conclure, il suffit donc de voir que $\aar_\mathrm{WD}(X \boxtimes U_d)=d\aar_\mathrm{W}(X)$, $\dim(Y \boxtimes U_e)=e(\dim Y)$ puis que $e(\dim Y)d\aar_\mathrm{W}(X) \geq \min(d\aar_\mathrm{W}(X),e\aar_\mathrm{W}(Y))$, ce qui est bien le cas puisque $e\geq 1$ et $\dim(Y)\geq 1$.

Le cas où $\aar_\mathrm{W}(Y)\geq \aar_\mathrm{W}(X)$ se gère de la même façon.
\end{proof}

\section{Conducteur des représentations locales}\label{Conducteur des représentations locales}

\begin{defi}\label{Cond_rep_locales_def}
Soit $\pi$ une représentation lisse, admissible, irréductible de $\GL_n(F)$. Alors la correspondance de Langlands locale pour $\GL_n$ (Théorème \ref{LLC_GLn}) lui associe un paramètre de Langlands $\Ll(\pi)$, (classe de conjugaison de) représentation de $\WD_F$ dans $\GL_n(\C)$. On définit l'exposant d'Artin de $\pi$ comme étant celui de $\Ll(\pi)$ -- dont on voit immédiatement qu'il ne dépend pas du choix de représentant dans la classe de conjugaison. 
La compatibilité de la correspondance de Langlands locale à la dualité nous donne d'ailleurs $\aar(\pi^\vee)=\aar(\pi)$.\ps 

Étant donnée une paire $(\pi,\pi')$ de représentations lisses, admissibles, irréductibles de $\GL_n(F)$ et $\GL_{n'}(F)$ respectivement, on définit l'exposant d'Artin de la paire $\pi \times \pi'$ par :
\[
\aar(\pi \times \pi')=\aar_{\WD}(\Ll(\pi) \otimes \Ll(\pi')).
\]
\end{defi}

%

On hérite donc en particulier l'inégalité (dite inégalité d'Henniart) :
\begin{equation}\label{exposant_prod_tens_rep_locales}
\aar(\pi \times \pi') \leq n' \aar(\pi)  + n \aar(\pi') - \min (\aar(\pi),\aar(\pi'))
\end{equation}
de la Proposition \ref{exposant_produit_tensoriel_WD}.\medskip

Considérons le cas plus général d'une représentation $\pi$ lisse, admissible, irréductible du groupe $G$ des $F$-points d'un groupe algébrique $\mathbf{G}$ réductif défini et déployé sur $F$. La correspondance de Langlands locale nous fournit alors (conjecturalement, en toute généralité) un paramètre de Langlands, à partir duquel on veut définir l'exposant d'Artin de $\pi$.

Il y a toutefois une subtilité : on fait ici apparaître un morphisme de $\WD_F$ dans $\widehat{G}$ et non pas dans un $\GL(V)$ où $V$ est un $\C$-espace vectoriel de dimension finie. C'est uniquement dans ce dernier cadre que l'exposant d'Artin d'une représentation du groupe de Weil-Deligne a été défini, il faut donc choisir un plongement de $\widehat{G}$ dans un groupe linéaire et de ce plongement dépend le calcul du conducteur.\ps

Précisons les choses dans le cas qui nous intéresse du groupe déployé $G=\SO_{2n+1}(F)$.

\begin{defi}\label{defi_aar_rep_tauto}
Soit $G$ le groupe $\SO_{2n+1}(F)$ (déployé) et soit $\pi$ une représentation lisse, admissible, irréductible de $G$. On a alors $\widehat{G}=\Sp_{2n}(\C)$ et, notant $\tau$ la représentation tautologique $\widehat{G} \hookrightarrow \GL_{2n}(\C)$, on définit $\aar(\Ll(\pi))$ (et donc $\aar(\pi)$) comme étant $\aar(\tau \circ \Ll(\pi))$.
\end{defi}



Nous cherchons désormais à classifier les paramètres de Langlands d'exposant d'Artin 0 ou 1 et à en déduire l'ensemble des représentations de $G$ de conducteur $\Ok$ (resp. $\pk$).

\subsection{Paramètres non ramifiés de $\SO_{2n+1}(F)$}\label{Paramètres non ramifiés}

On s'intéresse d'abord au cas d'un paramètre non ramifié, \ie d'exposant d'Artin nul.
\begin{prop}\label{prop_parametres_1}
Soit $\varphi$ un paramètre de Langlands de $G=\SO_{2n+1}(F)$ de conducteur $\Ok$. Alors il existe $n$ nombres complexes $s_1,\cdots,s_n$ avec $\re(s_1)\geq \cdots \geq \re(s_n)\geq 0$, correspondant à $n$ caractères non ramifiés $\chi_i=|\cdot|^{s_i}$ de $\W_F^{\rm ab}\simeq F^\times$, tels que $\varphi=\chi_1 \oplus \cdots \oplus \chi_n \oplus \chi_n^{-1} \oplus \cdots \oplus \chi_1^{-1}$.\ps 

Le paquet de Langlands $\Pi_\varphi$ contient alors un seul élément, à savoir le quotient de Langlands de $\ii_B^G (\chi_1 \boxtimes \cdots \boxtimes \chi_n)$.\ps  

Si $\varphi$ est supposé de plus tempéré, alors tous les $s_i$ sont de partie réelle nulle et le quotient de Langlands est égal à l'induite tout entière.
\end{prop}
\begin{proof}
La semi-simplicité de $\varphi$ nous permet d'écrire $\varphi=\bigoplus_i (X_i \otimes U_{d_i})$ avec $X_i$ représentation irréductible de $\W_F$ et $d_i$ entier strictement positif. Le calcul de l'exposant d'Artin se fait donc en utilisant \eqref{awd_irr} pour chaque terme, ledit exposant devant être nul. Cela donne $X_i$ caractère non ramifié et $d_i=1$ pour tout $i$. Le paramètre $\varphi$ se factorise donc par le tore standard de $\Sp_{2n}(\C)$ et le respect de la structure symplectique impose : $\varphi=\chi_1 \oplus \cdots \oplus \chi_n \oplus \chi_n^{-1} \oplus \cdots \oplus \chi_1^{-1}$ avec tous les $\chi_i$ non ramifiés.\ps 

Le calcul du groupe $\mathcal{S}_\varphi$ est le même que celui qui a été effectué dans la preuve du Théorème \ref{thm_rep_temp} : il est trivial, si bien que le paquet de Langlands associé au paramètre $\varphi$ ne contient qu'un élément.

Chaque $\chi_i$, non ramifié, est égal à $|\cdot|^{s_i}$ pour un certain nombre complexe $s_i$ dont la partie réelle est uniquement déterminée. On peut alors, quitte à échanger $\chi_i$ et $\chi_i^{-1}$, supposer que toutes ces parties réelles sont positives et, quitte à permuter les $\chi_i$, supposer que $\re(s_1) \geq \cdots \geq \re(s_n)$.
Puisque $\varphi$ se factorise par le tore standard de $\Sp_{2n}(\C)$, on a, par la classification de Langlands, que la représentation associée est l'unique quotient irréductible (dit quotient de Langlands) de $\ii_B^G (\chi_1 \boxtimes \cdots \boxtimes \chi_n)$ où $B$ est le sous-groupe de Borel standard de $G$ correspondant au sous-groupe de Levi (tore maximal déployé en fait) $T$. 
\end{proof}

%
%

Nous avons donc exposé la caractérisation complète (et classique) des représentations tempérées (voir la Remarque {\it infra} pour le cas non tempéré) de conducteur $\Ok$, qui traduit bien  que les représentations non ramifiées correspondent aux paramètres non ramifiés \eqref{LLC_resp_nr}.

\begin{thm}
Soit $\pi$ une représentation lisse irréductible tempérée de $G=\SO_{2n+1}(F)$ de conducteur $\Ok$. Alors il existe $n$ nombres complexes $s_1,\cdots,s_n$ imaginaires purs, correspondant à $n$ caractères non ramifiés $\chi_i=|\cdot|^{s_i}$ de $\W_F^{\rm ab}\simeq F^\times$, tels que $\pi$ soit égal à $\ii_B^G (\chi_1 \boxtimes \cdots \boxtimes \chi_n)$.

La représentation $\pi$ est alors la seule dans son paquet de Langlands et elle admet des $\K_0$-invariants non triviaux $\pi^{\K_0}$ de dimension 1.

Réciproquement, une représentation lisse irréductible tempérée de $G$ telle que $\pi^{\K_0}\neq \{0\}$ est de cette forme (et est de conducteur $\Ok$).
\end{thm}
\begin{proof}
Nous combinons les résultats de la Proposition \ref{prop_parametres_1} avec la Proposition \ref{thm_rep_sph_temp}.
\end{proof}

\subsection{Paramètres de conducteur $\pk$ de $\SO_{2n+1}(F)$}\label{Paramètres de conducteur p}

\begin{prop}\label{prop_parametres_p}
Soit $\varphi$ un paramètre de Langlands de $G=\SO_{2n+1}(F)$ de conducteur $\pk$. Alors il existe $n-1$ nombres complexes $s_1,\cdots,s_{n-1}$ avec $\re(s_1)\geq \cdots \geq \re(s_{n-1})\geq 0$, correspondant à $n-1$ caractères non ramifiés $\chi_i=|\cdot|^{s_i}$ de $\W_F^{\rm ab}\simeq F^\times$, et $\alpha \in \{\1,\eta\}$ tels que $\varphi=\bigoplus_{i=1}^{n-1} (\chi_i \oplus \chi_i^{-1}) \oplus (\alpha \otimes U_2)$.

Le paquet de Langlands $\Pi_\varphi$ contient alors un seul élément, à savoir le quotient de Langlands de $\ii_P^G (\chi_1 \boxtimes \cdots \boxtimes \chi_{n-1} \boxtimes \alpha \St_{\SO_3(F)})$. 

Si $\varphi$ est supposé de plus tempéré, alors tous les $s_i$ sont de partie réelle nulle et le quotient de Langlands est égal à l'induite tout entière.
\end{prop}
\begin{proof}
La semi-simplicité de $\varphi$ nous permet encore d'écrire $\varphi=\linebreak\bigoplus_i (X_i \otimes U_{d_i})$ avec $X_i$ représentation irréductible de $\W_F$ et $d_i$ entier strictement positif. Le calcul de l'exposant d'Artin se fait à nouveau en utilisant \eqref{awd_irr} pour chaque terme. On a donc $\aar_\mathrm{WD}(X_i \otimes U_{d_i})=0$ (ce qui équivaut à $X_i$ caractère non ramifié et $d_i=1$ d'après le paragraphe précédent) pour tous les termes sauf un exactement (pour l'indice $i_0$ disons) pour lequel l'exposant d'Artin vaut 1.\ps 

D'après \eqref{awd_irr} toujours, on peut atteindre $\aar_\mathrm{WD}(X_{i_0} \otimes U_{d_{i_0}})=1$ de deux façons distinctes :
\begin{itemize}
\item soit $X_{i_0}$ est un caractère non ramifié et $d_{i_0}=2$ ;
\item soit $\aar_\mathrm{W}(X_{i_0})=1$ et $d_{i_0}=1$.\ps 
\end{itemize}

Supposons être dans le deuxième cas, alors, $X_{i_0}$ étant irréductible ramifiée, le Corollaire \ref{cor522} nous dit que $\aar_\mathrm{W}(X_{i_0})\geq \dim X_{i_0}$ et donc que $X_{i_0}$ est un caractère (ramifié). Le paramètre $\varphi$ se factorise par le tore standard de $\Sp_{2n}(\C)$ et, comme dans la preuve de la Proposition \ref{prop_parametres_1}, le respect de la structure symplectique implique $\varphi=\chi_1 \oplus \cdots \oplus \chi_n \oplus \chi_n^{-1} \oplus \cdots \oplus \chi_1^{-1}$. Or $\aar_\mathrm{W}(\chi^{-1})=\aar_\mathrm{W}(\chi^\vee)=\aar_\mathrm{W}(\chi)$ pour tout caractère $\chi$, en particulier pour notre caractère ramifié $\chi_{i_0}$ si bien que le paramètre $\varphi$ fait apparaître deux caractères ramifiés, contradiction.\ps 

Examinons maintenant le premier cas : on a $\chi_{i_0} \otimes U_2$ avec $\chi_{i_0}$ qui doit être à valeurs dans $Z(\Sp_2(\C))=\{\pm I_2\}$ et non ramifié, donc $\chi_{i_0}$ est soit le caractère trivial, soit le caractère $\eta$. Finalement, on a
\[
\varphi=\bigoplus_{i=1}^{n-1} (\chi_i \oplus \chi_i^{-1}) \oplus (\alpha \otimes U_2)
\]
avec $\alpha \in \{\1,\eta\}$ et les $\chi_i$ non ramifiés.

Le calcul du groupe $\mathcal{S}_\varphi$ dans ce cas est le même que celui qui a été effectué dans la preuve du Théorème \ref{thm_rep_temp} : il est trivial, si bien que le paquet de Langlands associé au paramètre $\varphi$ ne contient qu'un élément.

Le paramètre $\varphi$ se factorise par le sous-groupe de Levi $\GL_1^{n-1}(\C) \times \Sp_2(\C)$ de $\Sp_{2n}(\C)$ correspondant au sous-groupe de Levi $M=\GL_1^{n-1}(F) \times \SO_3(F)$ de $G$. Puisqu'on reconnaît dans $\alpha \otimes U_2$ le paramètre de Langlands de $\alpha \St_{\SO_3(F)}$ (seule représentation du paquet associé) et quitte à permuter les $\chi_i$ (et à échanger $\chi_i$ et $\chi_i^{-1}$), la représentation considérée est donc l'unique quotient irréductible (quotient de Langlands) de l'induite $\ii_P^G (\chi_1 \boxtimes \cdots \boxtimes \chi_{n-1} \boxtimes \alpha \St_{\SO_3(F)})$, où $P$ est le sous-groupe parabolique standard associé à $M$.
\end{proof}

Nous avons donc exposé la caractérisation complète (et nouvelle dans ce degré de généralité) des représentations tempérées de conducteur $\pk$. Ce résultat sera d'ailleurs légèrement amélioré (dans le sens de la Conjecture formulée par Gross dans \cite{Gross} et rappelée en Introduction) au Théorème \ref{thm_J_variants_eps}.

\begin{thm}\label{thm_pcpl_temp}
Soit $\pi$ une représentation lisse irréductible tempérée de $G=\SO_{2n+1}(F)$ de conducteur $\pk$. Alors il existe $n-1$ nombres complexes $s_1,\cdots,s_{n-1}$ imaginaires purs, correspondant à $n$ caractères non ramifiés $\chi_i=|\cdot|^{s_i}$ de $\W_F^{\rm ab}\simeq F^\times$, et $\alpha \in \{\1,\eta\}$ tels que $\pi$ soit égal à $\ii_P^G (\chi_1 \boxtimes \cdots \boxtimes \chi_{n-1} \boxtimes \alpha \St_{\SO_3(F)})$.\ps 

La représentation $\pi$ est alors la seule dans son paquet de Langlands et elle admet des $\J^+$-invariants non triviaux $\pi^{\J^+}$ de dimension 1.\ps 

Réciproquement, une représentation lisse irréductible tempérée de $G$ telle que $\pi^{\J^+}\neq \{0\}$ et $\pi^{\K_0}=\{0\}$ est de cette forme (et est de conducteur $\pk$).
\end{thm}
\begin{proof}
Nous combinons les résultats de la Proposition \ref{prop_parametres_p} avec les Théorèmes \ref{thm_pcpl_ser_disc} et \ref{thm_rep_temp}.
\end{proof}

\emph{Remarque :} Dans le cas d'une représentation non tempérée, on sait par les résultats de Casselman \cite{Cass-art} qu'il y a encore équivalence entre le fait d'être de conducteur $\Ok$ et le fait d'avoir des $\K_0$-invariants non triviaux (et encore de dimension 1). Nous ne savons pas si l'équivalence analogue entre le fait d'être de conducteur $\pk$ et le fait d'avoir des $\J^+$-invariants non triviaux vaut encore pour une représentation non tempérée. 


%

\section{Lien avec les facteurs epsilon}\label{Lien avec les facteurs epsilon}
\subsection{Rappels}\label{Rappels}
Nous rappelons ici les résultats de \cite{Tate-Corvallis} pour introduire les facteurs epsilon de représentations (locales) du groupe de Weil et du groupe de Weil-Deligne.

Le groupe additif du corps local $F$ admet une mesure de Haar invariante (à gauche et à droite puisque le groupe est abélien), que l'on normalise en imposant que le volume de l'anneau des entiers $\Ok$ soit égal à 1. Notons $\dd x$ cette mesure de Haar -- désormais uniquement déterminée.

Fixons par ailleurs un caractère additif de $F$, de niveau 0 (\ie trivial sur $\Ok$ et non trivial sur $\pk^{-1}$), noté $\psi$. On peut penser, dans le cas de $\Q_p$, à $\psi :x \mapsto e^{-2i\pi\{x\}_p}$, où $\{x\}_p$ désigne la partie fractionnaire (vue comme nombre rationnel) de $x$, définie par $\{x\}_p=\sum_{n<0}a_n p^{-n}$ si $x=\sum_{n}a_n p^{-n}$. On a alors $x-\{x\}_p \in \Z_p$ puis $\{x+y\}_p\equiv \{x\}_p +\{y\}_p$ modulo $\Z_p$ (et donc modulo $\Z$), ce qui permet de voir que $\psi$ est bien un morphisme, de niveau 0.\medskip

Le Théorème 3.4.1 de \cite{Tate-Corvallis} nous assure l'existence d'un complexe non nul $\eps(V,\psi,\dd x)$ pour chaque représentation continue $(\rho,V)$ du groupe de Weil $\W_F$, telle que $V \mapsto \eps(V,\psi,\dd x)$ soit additive, ce qui permet de la définir pour des représentations virtuelles. On demande alors que  $V \mapsto \eps(V,\psi,\dd x)$ soit inductive en degré 0 et que, dans le cas où $V$ est de dimension 1, elle coïncide avec une définition \emph{ad hoc} pour les caractères.

Puisque $\dd x$ et $\psi$ sont fixés, on ne mentionnera plus la dépendance de $\eps$ en ces deux objets et l'on notera $\eps(V)$. On peut alors définir une \emph{fonction} (holomorphe) $\eps$ de la variable complexe $t$ en posant $\eps(t,V)=\eps(V \otimes |\cdot|^t)$. En particulier, $\eps(0,V)=\eps(V)$. Nous prenons la convention -- dite de Langlands au paragraphe 3.6 de \cite{Tate-Corvallis} -- et considérons plutôt la variable $s=t+\frac{1}{2}$, si bien que la constante $\eps(V)$ correspond à la valeur en $s=\frac{1}{2}$ de la fonction $s \mapsto \eps(s,V)$. Les deux résultats essentiels que nous utiliserons sont :\ps 
\begin{itemize}
\item $\eps(\chi)=1$, si $\chi$ est un caractère non ramifié (3.2.6.1 de \cite{Tate-Corvallis}) ; \ps 
\item $\eps(s,V)=\eps(V)q^{\aar(V)(\frac{1}{2}-s)}$, où $q$ désigne le cardinal du corps résiduel de $F$ (3.4.5 de \cite{Tate-Corvallis}).
\end{itemize}
La combinaison de ces deux résultats nous donne d'ailleurs $\eps(s,\chi)=1$ pour un caractère non ramifié $\chi$.\medskip

Tate étend ces définitions aux représentations de Weil-Deligne, telles qu'introduites à la Définition \ref{def_rep_de_WD}. On a alors, pour $(\rho,V,N)$ une telle représentation :
\[
\eps(s,(\rho,V,N))=\eps_\mathrm{W}(s,(\rho,V))\det\left(-q^{-s}\rho(\Fr)_{|V^I/V_N^I}\right),
\]
où $\eps_\mathrm{W}$ désigne le facteur epsilon de la représentation du groupe de Weil tel que défini ci-dessus, $\Fr$ un morphisme de Frobenius, $V_N= \Ker N$ et l'exposant $I$ indique les invariants par le sous-groupe d'inertie.

Nous travaillons avec des représentations du groupe de Weil-Deligne $\W_F \times \SU(2)$, dans le langage desquels nous voulons exprimer ces facteurs epsilon. On rappelle qu'on ne considère que des représentations semi-simples, il suffit donc, par additivité, de considérer le cas des représentations irréductibles et, d'après l'analyse effectuée au paragraphe \ref{Conducteur d'une représentation du groupe de Weil-Deligne}, il y a deux cas à considérer :\ps 
\begin{itemize}
\item celui d'une représentation $X \otimes U_d$ où $X$ est un caractère non ramifié ;\ps 
\item celui d'une représentation $X \otimes U_d$ où $X$ est irréductible ramifiée (et alors $X^I=\{0\}$).\ps 
\end{itemize}

À nouveau, nous renvoyons à \cite{GR} pour les détails de la traduction entre triplets $(\rho,V,N)$ et représentations de $\W_F \times \SU(2)$. Nous obtenons :\ps 
\begin{itemize}
\item si $X$ est un caractère non ramifié, $\eps(s,X \otimes U_d)=(-X(\Fr))^{d-1}q^{(d-1)(\frac{1}{2}-s)}$ ;\ps 
\item si $X$ est irréductible ramifiée, $\eps(s,X \otimes U_d)=\eps(s,X)^d$.\ps 
\end{itemize}

Avec \eqref{awd_irr}, on obtient bien $\eps(s,X\otimes U_d)=\eps(X \otimes U_d)q^{\aar_\mathrm{WD}(X \otimes U_d)(\frac{1}{2}-s)}$, où
\begin{equation}\label{eps_rep_wd}
\eps(X \otimes U_d)=
\begin{cases}
(-X(\Fr))^{d-1} &\text{ si } \aar_\mathrm{W}(X)=0 \\
\eps(X)^d &\text{ si } \aar_\mathrm{W}(X)>0
\end{cases}.
\end{equation}

\subsection{Calcul pour les représentations étudiées}\label{Calcul pour les représentations étudiées}
On fait maintenant l'analogue pour les facteurs epsilon de ce qui a été fait pour le conducteur au paragraphe \ref{Conducteur des représentations locales}.
\begin{defi}
Soit $\pi$ une représentation lisse, admissible, irréductible de $\GL_n(F)$. Alors la correspondance de Langlands locale pour $\GL_n$ (Théorème \ref{LLC_GLn}) lui associe un paramètre de Langlands $\Ll(\pi)$, (classe de conjugaison de) représentation de $\WD_F$ dans $\GL_n(\C)$. On définit le facteur espilon de $\pi$ comme étant celui de $\Ll(\pi)$ -- dont on voit immédiatement qu'il ne dépend pas du choix de représentant dans la classe de conjugaison.

Étant donnée une paire $(\pi,\pi')$ de représentations lisses, admissibles, irréductibles de $\GL_n(F)$ et $\GL_{n'}(F)$ respectivement, on définit l'exposant d'Artin de la paire $\pi \times \pi'$ par :
\[
\eps(\pi \times \pi')=\eps(\Ll(\pi) \otimes \Ll(\pi')).
\]
\end{defi}
\begin{defi}
Soit $\pi$ une représentation lisse, admissible, irréductible de $G=\SO_{2n+1}(F)$. Alors la correspondance de Langlands locale pour $\SO_{2n+1}$ (Théorème \ref{LLC_SO}) lui associe un paramètre de Langlands $\Ll(\pi)$, (classe de conjugaison de) représentation de $\WD_F$ dans $\Sp_{2n}(\C)$. On définit le facteur espilon de $\pi$ comme étant celui de $\tau \circ \Ll(\pi)$ où $\tau$ désigne la représentation tautologique de $\Sp_{2n}(\C)$ dans $\GL_{2n}(\C)$.
\end{defi}

On peut maintenant calculer les facteurs epsilon des représentations que l'on a fait apparaître aux paragraphes \ref{Paramètres non ramifiés} et \ref{Paramètres de conducteur p}.

Dans le cas d'une représentation de conducteur $\Ok$, on a un paramètre non ramifié que l'on peut écrire 
$\varphi=\chi_1 \oplus \cdots \oplus \chi_n \oplus \chi_n^{-1} \oplus \cdots \oplus \chi_1^{-1}$ avec tous les $\chi_i$ non ramifiés. Pour chacun de ces caractères, on a $\eps(\chi_i)=1$ et même $\eps(s,\chi_i)=1$, d'où, par additivité, $\eps(\varphi)=1$ (et $\eps(s,\varphi)=1$). On connaît donc le facteur local d'une telle représentation.

Dans le cas d'une représentation de conducteur $\pk$, on fait apparaître un paramètre $\varphi=\bigoplus_{i=1}^{n-1} (\chi_i \oplus \chi_i^{-1}) \oplus (\alpha \otimes U_2)$
avec $\alpha \in \{\1,\eta\}$ et les $\chi_i$ non ramifiés. Les caractères non ramifiés ne vont pas contribuer et on doit juste déterminer ce qu'il se passe pour $\alpha \otimes U_2$.\ps 

Commençons par remarquer que le caractère non ramifié $\eta$ peut s'écrire comme $|\cdot|^{t_0}$ avec $t_0=-\frac{i\pi}{\log q}$ donc $\eps(\eta \otimes U_2)=\eps(t_0+\frac{1}{2},U_2)=\eps(U_2)q^{\aar_\mathrm{WD}(U_2)(-t_0)}$. Or $\aar_\mathrm{WD}(U_2)=1$, si bien que $\eps(\eta \otimes U_2)=-\eps(U_2)$.

Il suffit donc de connaître $\eps(U_2)=\eps(\1 \otimes U_2)$. Le caractère $\1$ étant non ramifié, la formule \eqref{eps_rep_wd} nous dit qu'il faut considérer $(-\1(\Fr))^{2-1}=-1$. On a donc 
\begin{equation*}
\eps(\alpha \otimes U_2)=
\begin{cases}
-1 &\text{ si } \alpha=\1 \\
1 &\text{ si } \alpha=\eta
\end{cases}.
\end{equation*}

On a finalement le résultat suivant.
\begin{prop}
Soit $\pi$ une représentation de $G$ de conducteur $\pk$. Alors, en reprenant les notations de la Proposition \ref{prop_parametres_p}, $\pi$ est le quotient de Langlands de $\ii_P^G (\chi_1 \boxtimes \cdots \boxtimes \chi_{n-1} \boxtimes \alpha \St_{\SO_3(F)})$.\ps 

On a alors $\alpha=\1$ si, et seulement si $\eps(\pi)=-1$ et $\alpha=\eta$ si, et seulement si $\eps(\pi)=1$.
\end{prop}

En combinant avec le Lemme \ref{lemme_inv_param_ind_St}, on obtient le :
\begin{thm}\label{thm_J_variants_eps}
Soit $\pi$ une représentation lisse irréductible tempérée de $G$ de conducteur $\pk$. Alors $\pi$ a des invariants paramodulaires de dimension 1. Plus précisément, elle a des $(\J,+)$-invariants si, et seulement si $\eps(\pi)=1$ et des $(\J,-)$-variants si, et seulement si $\eps(\pi)=-1$.
\end{thm}
On peut donc lire le \og signe \fg{} des invariants paramodulaires sur le facteur epsilon local.

C'est un cas particulier de la Conjecture de Gross (\cite{Gross}) rappelée dans l'Introduction. Cette Conjecture porte sur le conducteur $\pk^m$ pour tout entier naturel $m$ et fait donc appel à un groupe noté $K_0(\varpi^m)$ \emph{loc. cit.} et $\K(\pk^m)$ dans \cite{Tsai-phd}, qui généralise notre groupe $\J^+$ (correspondant à $m=1$). La conjecture est démontrée par Tsai pour $m$ quelconque mais uniquement dans le cas où $\pi$ est supposée supercuspidale générique. \medskip 

Donnons une dernière formulation de nos résultats qui sera utile dans la Deuxième Partie.

\begin{cor}\label{cor_temp_dim_invariants_général}
Soit $(\pi,V)$ une représentation tempérée irréductible de $G$.\ps 

On suppose que $\pi^{\K_0} \neq \{0\}$ (ou, de façon équivalente, que $\pi$ est de conducteur $\Ok$). Alors $\dim \pi^{\K_0}=1$ et $\dim \pi^{\J^+}=2$. Plus précisément $\dim \pi^{(\J,+)}=\dim \pi^{(\J,-)}=1$.\ps 

On suppose que $\pi^{\J^+} \neq \{0\}$. Soit $\pi^{\K_0} \neq \{0\}$ et on relève du cas précédent, soit $\pi^{\K_0}=\{0\}$ et alors, il est équivalent de supposer que $\pi$ est de conducteur $\pk$. Dans ce deuxième cas $\dim \pi^{\J^+}=1$, plus précisément $\dim \pi^{(\J,\eps)}=1$ et $\dim \pi^{(\J,-\eps)}=0$, où $\eps=\eps(\pi)$.
\end{cor}

\newpage
\part{Étude globale}

\newpage
\chapter{Représentations automorphes cuspidales algébriques de $\GL_n$ de conducteur premier}\label{Chapitre_Représentations_automorphes...}

%

Nous allons maintenant introduire et étudier des objets globaux, dépendant d'un corps de nombres. {\bfseries Dans toute la suite, nous travaillons uniquement sur le corps $\Q$ des rationnels.}

Les enjeux de ce chapitre sont multiples. Il va d'abord s'agir de définir précisément chacun des termes des Théorèmes \ref{thm_2_w17}, \ref{thm_2_w19} et \ref{thm_p} de l'Introduction. Ce sera fait dans les paragraphes \ref{Représentations automorphes cuspidales}, \ref{Algébricité} et \ref{Conducteur premier}.

En vue des applications aux Chapitres \ref{La_formule explicite de Riemann-Weil-Mestre} et \ref{Chapitre_calculs}, nous considérons au §\ref{Fonctions $L$ et facteurs epsilon} les fonctions $L$ (ou $\Lambda$) et les facteurs epsilon associés aux représentations locales et globales. Au-delà des seuls rappels en la matière, nous calculons ces quantités dans les cas particuliers correspondant aux représentations étudiées.

Enfin, puisque les représentations \emph{autoduales} jouent un rôle prépondérant (à la fois à cause de leurs \og meilleures propriétés \fg{} et parce que ce sont celles que l'on trouve en pratique, \cf Théorèmes \ref{thm_2_w17} et \ref{thm_p}), nous présentons au paragraphe \ref{Alternative symplectique-orthogonale} l'alternative symplectique-orthogonale pour les représentations autoduales de $\GL_n$ selon la théorie d'Arthur.


\section{Représentations automorphes cuspidales de $\GL_n$ sur $\Q$}\label{Représentations automorphes cuspidales}
Nous rappelons ici sans démonstration des résultats de \cite{BJ-Corvallis}.

On fixe un caractère de Hecke $\omega : \Q^\times \backslash \aq^\times \rightarrow \C^\times$ et on peut alors définir l'espace des formes automorphes pour $\GL_n(\aq)$ de caractère central $\omega$, noté $\mathcal{A}^\omega(\GL_n)$. Cet espace a une structure de $(\mathfrak{gl}_n(\R),\Oo_n(\R))$-module (ou module de Harish-Chandra\footnote{Cette dernière terminologie a l'avantage de ne pas spécifier le choix du compact maximal, choix qui n'intervient pas en fait dans la plupart des énoncés.}) et une structure de $\GL_n(\mathbb{A}_f)$-module compatibles, où l'on note $\mathbb{A}_f=\prod'_p \qp$ l'ensemble des adèles finis (avec bien sûr $\aq=\R \times\mathbb{A}_f$).

On considère alors le sous-$(\mathfrak{gl}_n(\R),\Oo_n(\R))\times \GL_n(\mathbb{A}_f)$-module constitué des formes automorphes \emph{paraboliques} ou \emph{cuspidales}, noté $\mathcal{A}^\omega_\mathrm{cusp}(\GL_n)$, qui a la propriété remarquable d'être semi-simple.

\begin{defi}
Une représentation automorphe cuspidale $\pi$ de $\GL_n$ sur $\Q$ de caractère central $\omega$ est un constituant irréductible de l'espace $\mathcal{A}^\omega_\mathrm{cusp}(\GL_n)$.
\end{defi}

Si $\omega$ est un caractère de Hecke, alors $\omega  | \cdot |_{\mathbb{A}_\mathbb{Q}}^s$ où $s$ est un nombre complexe quelconque, en est encore un (on utilise le fait que les éléments de $\mathbb{Q}^\times$ plongés diagonalement dans $\mathbb{A}^\times_\mathbb{Q}$ sont de norme adélique 1).
On a, de façon simple, $\mathcal{A}^{\omega  | \cdot |_{\mathbb{A}_\mathbb{Q}}^s}(\GL_n) \simeq \mathcal{A}^{\omega}(\GL_n) \otimes |\det|^{\frac{s}{n}}$, donc, quitte à tordre par une puissance du déterminant, on peut supposer que $\omega_{|\mathbb{R}_{>0}}=1$. Cela motive la définition suivante.
\begin{defi}\label{defi_rep_centree}
Une représentation automorphe cuspidale de $\GL_n$ sur $\Q$ est \emph{centrée} si son caractère central est trivial sur $\R_{>0}$.
\end{defi}

Soit donc $\pi$ une représentation automorphe cuspidale de $\GL_n$ sur $\Q$. Nous avons classiquement la décomposition en produit tensoriel restreint :
\begin{equation}\label{prod_tens_res}
\pi \simeq \pi_\infty \otimes {\bigotimes_p}' \pi_p,
\end{equation}
où $\pi_\infty$ est un $(\mathfrak{gl}_n(\R),\Oo_n(\R))$-module admissible irréductible (\ie chaque représentation irréductible continue de $\Oo_n(\R)$ n'intervient qu'un nombre fini de fois dans $\pi_\infty$) et les $\pi_p$ sont des représentations admissibles (au sens de la Définition \ref{defi_rep_lisse_admissible}) irréductibles de $\GL_n(\Q_p)$, non ramifiées pour presque tout $p$. C'est d'ailleurs ce qui permet de définir le produit tensoriel restreint : pour chaque $p$ tel que $\pi_p$ est non ramifiée, on a une droite vectorielle distinguée $\pi_p^{\GL_n(\Z_p)}$.

Terminons par un résultat classique.
\begin{prop}\label{prop_613}
Soit $\pi$ une représentation automorphe cuspidale de $\GL_n$ sur $\Q$. Alors il existe une unique représentation automorphe cuspidale de $\GL_n$ sur $\Q$, notée $\pi^\vee$ et dite représentation contragrédiente (ou duale) de $\pi$, et vérifiant $(\pi^\vee)_v=(\pi_v)^\vee$ pour toute place $v$.\footnote{La notion de contragrédiente \emph{locale} a été définie dans le cas $p$-adique au paragraphe \ref{Généralités}, elle est classique pour la place archimédienne et nous ne la rappelons pas.}
\end{prop}

\section{Algébricité}\label{Algébricité}

Il nous faut commencer par quelques brefs rappels sur la correspondance de Langlands dans le cas local archimédien.

On pose, suivant \cite{Tate-Corvallis}, $\W_\C=\C^\times$ et $\W_\R=\C^\times \coprod j\C^\times$ avec $j^2=-1$ et $jzj^{-1}=\bar{z}$ pour $z\in \C^\times$ -- ainsi $\W_\R$ est une extension de $\Gal(\C/\R)$ par $\W_\C$. 

On introduit le morphisme de groupes $\eps_{\C / \R} : \W_\R \rightarrow \{\pm 1\}$ qui envoie les éléments de $\C^\times$ sur 1 et ceux de $j \C^\times$ sur $-1$.
Par ailleurs, on peut prolonger le morphisme \og module au carré \fg{} de $\C^\times$ à $\W_\R$ en imposant $|z|=z\overline{z}$ pour $z \in \C^\times$ et $|j|=1$. On a alors un morphisme de groupes $|\cdot | : \W_\R \rightarrow \R_+^\times$ et on peut  voir que $\eps_{\C / \R} \times | \cdot |$ donne un isomorphisme de groupes $\W_\R^\ab \xrightarrow{\sim} \R ^\times$.

Cela définit, dans le cas local archimédien, l'isomorphisme d'Artin (à rapprocher de \eqref{application d'Artin-groupe de Weil} dans le cas non archimédien) ${\rm Art}_K :K^\times \simeq \W_K^\ab$ (avec $K=\R$ ou $\C$).\ps

On s'intéresse ici au groupe $\GL_n(\R)$. Un paramètre de Langlands de $\GL_n(\R)$ est une classe de conjugaison de morphisme admissible (\ie continu et d'image constituée d'éléments semi-simples) $\varphi:\W_\R \rightarrow \GL_n(\C)$. Nous notons $\Phi(\GL_n)$ l'ensemble de tels paramètres.

Un $(\mathfrak{gl}_n(\R),\Oo_n(\R))$-module admissible irréductible $U$ admet un caractère infinitésimal ${\rm inf}(U)$ que l'on peut voir, suivant Harish-Chandra et Langlands (on renvoie à \cite{Lgl_Euler_products}), comme une classe de conjugaison semi-simple de $\mathrm{M}_n(\C)$. On appelle \emph{poids} de $U$ les valeurs propres de cette classe de conjugaison semi-simple. Par abus de langage, on parlera des poids d'une représentation automorphe cuspidale $\pi$ de $\GL_n$ pour désigner les poids de sa composante archimédienne $\pi_\infty$. On note 
${\rm Irr}(\GL_n)$ l'ensemble des classes d'équivalence de $(\mathfrak{gl}_n(\R),\Oo_n(\R))$-modules irréductibles.

\begin{thm} \emph{Correspondance de Langlands locale archimédienne pour $\GL_n$ (\cite{Lgl_73}, \cite{Knapp_Archimedean})} \label{LLC_arch}

La construction de \emph{\cite{Lgl_73}} définit une bijection canonique :
\[
\Ll : \mathrm{Irr}(\GL_n) \longrightarrow \Phi(\GL_n)
\]
telle que, pour tout $U \in {\rm Irr}(\GL_n)$ :
\begin{enumerate}
\item (dualité) $\Ll(U^\vee) \simeq \Ll(U)^\vee$ ;
\item (caractère central) $\det \Ll(U)=\omega_U \circ {\rm Art}_\R^{-1}$, où $\omega_U:\R^\times \rightarrow \C^\times$ désigne le caractère central de $U$;
\item (caractère infinitésimal) écrivant\footnote{\label{notation_evoc_Lgl}En restreignant $\Ll(U)$ au sous-groupe $\W_\C \simeq \C^\times$ d'indice 2, on obtient une somme de $n$ caractères de $\C^\times$. Or un tel caractère est de la forme $\chi_{s,m} : re^{i\theta} \mapsto r^s e^{im\theta}$ où $s \in \C$ et $m \in \Z$. On a donc $\chi_{s,m}(z)=|z|^s(\frac{z}{|z|})^m=|z|^{s-m}z^m$, que l'on peut encore écrire, suivant la notation évocatrice de Langlands, $z^\frac{s+m}{2}\bar{z}^\frac{z-m}{2}$ avec $|z|=(z\bar{z})^\frac{1}{2}$. Tout caractère de $\C^\times$ s'écrit alors, selon cette convention, $z \mapsto z^\lambda \bar{z}^\mu$ avec $\lambda, \mu$ deux nombres complexes tels que $\lambda-\mu \in \Z$.}
\begin{equation}\label{déc_algébricité}
\Ll(U)_{|\C^\times}=\bigoplus_{i=1}^n z^{\lambda_i} \bar{z}^{\mu_i}
\end{equation}
 avec $\lambda_i- \mu_i \in \Z$, les poids de $U$ soient les $\lambda_i$.
\end{enumerate}
\end{thm}

\begin{defi}
Un $(\mathfrak{gl}_n(\R),\Oo_n(\R))$-module irréductible est dit \emph{strictement algébrique} si ses poids sont dans $\Z$.



Une représentation automorphe cuspidale $\pi$ sur $\Q$ est dite \emph{strictement algébrique} si sa composante archimédienne $\pi_\infty$ l'est.
\end{defi}

En particulier, si $U$ est strictement algébrique, alors les $\lambda_i$ (et partant les $\mu_i$) selon l'écriture \eqref{déc_algébricité} pour $\Ll(U)$, sont tous entiers. Par ailleurs, puisque pour $z\in \C^\times$, on a $jzj^{-1}=\bar{z}$ dans $\W_\R$, les multi-ensembles des $\lambda_i$ et des $\mu_i$ sont égaux. 

\begin{prop} Lemme de pureté de Clozel \emph{(\cite{Clozel}, Lemme 4.9)}

Soit $\pi$ une représentation automorphe cuspidale \emph{strictement algébrique} de $\GL_n$ sur $\Q$. Notons $\pi_\infty$ sa composante archimédienne. Alors la quantité $\lambda_i +\mu_i$ dans la décomposition \eqref{déc_algébricité} pour $\Ll(\pi_\infty)$ est constante.

\end{prop}

\begin{prop-def}\label{defi_alg_centrée}
Une représentation automorphe cuspidale $\pi$ de $\GL_n$ sur $\Q$ est dite \emph{algébrique} si $\pi_\infty$ est unitaire et si les poids de $\pi_\infty$ sont tous dans $\Z$ ou tous dans $\frac{1}{2}\Z-\Z$. On a alors (par le lemme de pureté de Clozel) ${\rm inf}(\pi_\infty)=\mathrm{diag}(\pm \frac{w_1}{2},\cdots,\pm \frac{w_n}{2})$, avec les $w_i$ entiers naturels de la même parité et $w_1 \geq w_2 \geq \cdots \geq w_n \geq 0$ si on veut. Le \emph{poids motivique} de $\pi$, noté ${\rm w}(\pi)$, est alors égal à $w_1$.

On dit de plus que $\pi$ est \emph{régulière} (resp. \emph{très régulière}) si $i\neq j \Rightarrow w_i \neq w_j$ (resp. $|w_i-w_j|>1$).
\end{prop-def}
Puisque ce sont des propriétés qui ne dépendent que de $\pi_\infty$, on dira aussi, en tant que de besoin, que $\pi_\infty$ est algébrique (resp. algébrique régulière, resp. algébrique très régulière).

En particulier, une représentation automorphe cuspidale algébrique de $\GL_n$ sur $\Q$ est \emph{centrée} au sens de la Définition \ref{defi_rep_centree}.\newline

Nous introduisons encore quelques définitions qui nous seront utiles au §\ref{Facteur local}. Par la discussion précédente, il est opportun de s'intéresser aux représentations de $\W_\R$ qui sont triviales sur le sous-groupe $\R_{>0}$.
Notons $\xi$ le caractère unitaire $z \mapsto \frac{z}{|z|}$ (soit encore $z \mapsto z^\frac{1}{2}\bar{z}^{-\frac{1}{2}}$ selon la convention rappelée en note \ref{notation_evoc_Lgl} \emph{supra}).
\begin{prop}

Les caractères de $\W_\R$ triviaux sur le sous-groupe $\R_{>0}$ sont le caractère trivial et le caractère $\eps_{\C/\R}$ d'ordre $2$ et trivial sur $\W_\C$.

On pose $I_n=\Ind_{\W_\C}^{\W_\R}\xi^n$ pour $n \in \Z$. Alors
\begin{itemize}
\item $I_n \simeq I_{-n}$ ;
\item $I_0 \simeq \1 \oplus \eps_{\C/\R}$ ;
\item $I_n$ est irréductible pour $n>0$.
\end{itemize}

La donnée des deux caractères précédents et des représentations $I_n$ pour $n>0$ est exhaustive pour les représentations irréductibles de $\W_\R$ triviales sur $\R_{>0}$.

Les autres représentations irréductibles sont obtenues à partir de celles-ci en les tordant par une puissance (complexe) de la norme.
\end{prop}

Puisqu'on s'intéresse aux représentations semi-simples, il est naturel de considérer le groupe libre sur $\Z$ engendré par de telles représentations.

\begin{defi}\label{defi_G(R)_alg_K_infty}
On note $K_\infty$ l'anneau de Grothendieck sur $\Z$ de la catégorie des représentations de $\W_\R$ semi-simples, continues, complexes de dimension finie, triviales sur le sous-groupe $\R_{>0}$ de $\W_\R$.
\end{defi}

On constate immédiatement les identités :
\begin{itemize}
\item $I_n \otimes \eps_{\C / \R} \simeq I_n$,
\item $\eps_{\C / \R} \otimes \eps_{\C / \R} \simeq \mathbf{1}$,
\item $I_n \otimes I_m \simeq I_{n+m} \oplus I_{|n-m|}$,
\end{itemize}
et on remarque que chacune des $\mathbf{1},\, \eps_{\C / \R},\, I_n$ pour $n>0$, est autoduale et qu'il en est donc de même pour tous les éléments de $K_\infty$. 

\begin{cor}\label{cor_alg_duale_memes_poids}
Soit $\pi$ une représentation automorphe cuspidale algébrique de $\GL_n$ sur $\Q$. Alors $\pi^\vee$ est également algébrique et a les mêmes poids que $\pi$. (En fait, $\pi$ et $\pi^\vee$ ont des composantes archimédiennes isomorphes.)
\end{cor}

Pour des besoins futurs (\cf §\ref{Un résultat de finitude}), il est utile d'introduire la filtration suivante.

\begin{defi}\emph{(\cite{Chen-HM}, Definition 3.2)}\label{defin_K_infty_leq_w}

Soit $w$ un entier naturel. On définit $K_\infty^{\leq w}$ comme le sous-groupe de $K_\infty$ engendré par
\begin{itemize}
\item les $I_v$ avec $0< v \leq w$ et $v$ impair si $w$ est impair ;
\item les $I_v$ avec $0< v \leq w$ et $v$ pair ainsi que $\1$ et $\eps_{\C/\R}$ si $w$ est pair.
\end{itemize}
Ainsi $K_\infty^{\leq w}$ est un $\Z$-module libre de rang $\frac{w+1}{2}$ si $w$ est impair et de rang $\frac{w}{2}+2$ si $w$ est pair.

On dira qu'un élément de $K_\infty^{\leq w}$ est \emph{effectif} s'il est combinaison linéaire à coefficients positifs des vecteurs de base sus-cités.
\end{defi}

D'ores et déjà, nous avons la traduction suivante.

\begin{lemme}\emph{(\cite{Chen-HM}, Lemma 3.8)}\label{lemme_pi_infty_i_w}

Soit $\pi$ une représentation automorphe cuspidale algébrique de $\GL_n$ sur $\Q$ de poids motivique $w$. Alors $\Ll(\pi_\infty)$ appartient à $K_\infty^{\leq w}$ (et est effectif).
\end{lemme}

\section{Conducteur premier}\label{Conducteur premier}
Nous définissons le conducteur de façon \emph{ad hoc} à l'aide de la décomposition \eqref{prod_tens_res}.

\begin{defi}\label{defi_conducteur_global}
Soit $\pi$ une représentation automorphe cuspidale de $\GL_n$ sur $\Q$. On pose :
\begin{equation}\label{conducteur-global}
{\rm N}(\pi)=\prod_p p^{\mathrm{a}(\pi_p)}=\prod_p p^{\mathrm{a}(\Ll(\pi_p))},
\end{equation}
où $\pi_p$ est défini par la décomposition \eqref{prod_tens_res} et où $\Ll(\pi_p)$ est la représentation de $\WD_{\qp}$ associée à $\pi_p$ par la correspondance de Langlands locale pour $\GL_n(\Q_p)$. La deuxième égalité découle de la Définition \ref{Cond_rep_locales_def}.
\end{defi}

Ce produit est bien défini car $\aar(\pi_p)=0$ pour presque tout $p$.

\begin{cor}\label{cor_cond_contragrediente}
Soit $\pi$ une représentation automorphe cuspidale de $\GL_n$ sur $\Q$. Alors ${\rm N}(\pi^\vee)={\rm N}(\pi)$.
\end{cor}
\begin{proof}
Il suffit d'utiliser le fait que $(\pi^\vee)_p \simeq (\pi_p)^\vee$ et que $\aar((\pi_p)^\vee)=\aar(\pi_p)$.
\end{proof}

Dans le reste de ce paragraphe, on s'intéresse aux propriétés particulières des représentations automorphes cuspidales de $\GL_n$ sur $\Q$ de conducteur $p$ premier qui sont celles des Théorèmes \ref{thm_2_w17}, \ref{thm_2_w19} et \ref{thm_p}.

Soit donc $\pi$ une telle représentation.
Selon la formule \eqref{conducteur-global}, cela signifie que $\Ll(\pi_\ell)$ est d'exposant d'Artin nul pour $\ell\neq p$ et que $\Ll(\pi_p)$ est d'exposant d'Artin égal à 1.

Il s'agit alors d'identifier quels sont les paramètres de Langlands de $\GL_n(\Q_\ell)$ qui sont d'exposant d'Artin 0 ou 1 -- travail similaire à ce qui a été fait en \ref{Conducteur des représentations locales}. Ici, les paramètres sont à valeurs dans $\widehat{{\GL_n}}(\C)=\GL_n(\C)$. 


\begin{prop-def}\label{prop_types_I_et_II}
Les paramètres de Langlands de $\GL_n(\Q_\ell)$ d'exposant d'Artin nul sont de la forme $\varphi=\chi_1 \oplus \cdots \oplus \chi_n$, où les $\chi_i$ sont des caractères non ramifiés.

Les paramètres de Langlands de $\GL_n(\Q_\ell)$ d'exposant d'Artin égal à 1 sont de deux types :
\begin{enumerate}[(I)]
\item $\varphi=\chi_1 \oplus \cdots \oplus \chi_{n-2} \oplus (\psi \otimes U_2)$, où les $\chi_i$ \emph{et} $\psi$ sont des caractères non ramifiés.
\item $\varphi=\chi_1 \oplus \cdots \oplus \chi_{n-1} \oplus \psi$ où les $\chi_i$ sont des caractères non ramifiés, et où $\psi$ est un caractère \emph{modérément ramifié} (\ie $\psi_{|\Z_\ell^\times}$ non trivial et $\psi_{|1+\ell\Z_\ell}$ trivial). 
\end{enumerate}
Une représentation automorphe cuspidale de $\GL_n$ sur $\Q$ de conducteur $p$ est donc non ramifiée en toutes les places $\ell \neq p$ et sera dite de type (I) ou (II) selon la forme de son paramètre de Langlands (ramifié) en $p$.
\end{prop-def}
\begin{proof}
Soit donc $\varphi : \W_{\Q_\ell} \times \SU(2) \rightarrow\GL_n(\C)$ un tel paramètre. Par semi-simplicité, on a la décomposition en somme directe de représentations irréductibles : $\varphi=\bigoplus_i \sigma_i \otimes U_{a_i}$. Alors $\aar(\varphi)=\sum_i \aar(\sigma_i \otimes U_{a_i})$ et il suffit de déterminer à quelle condition on a un exposant d'Artin égal à 0 ou 1 pour une représentation \emph{irréductible} $(\sigma_i \otimes U_{a_i})$ de $\W_{\Q_\ell} \times \SU(2)$.

La formule \eqref{awd_irr} nous dit qu'on ne peut atteindre un exposant nul qu'avec un caractère non ramifié du groupe de Weil (et la représentation triviale de $\SU(2)$). Or un caractère de $\W_{\Q_\ell}$ trivial sur le sous-groupe d'inertie $I_{\Q_\ell}$ correspond, modulo la théorie du corps de classes local sous sa forme \eqref{application d'Artin-groupe de Weil}, à un caractère de $\Q_\ell^\times$ trivial sur $\Z_\ell^\times$.

La formule \eqref{awd_irr} nous dit encore que l'exposant 1 peut être atteint soit avec $\sigma_i$ caractère non ramifié et $a_i=2$ (I), soit avec un caractère $\sigma_i$ d'exposant d'Artin égal à 1 (vu comme représentation du groupe de Weil) et $a_i=1$ (II). Ce dernier cas correspond, par la théorie du corps de classes, à un caractère modérément ramifié au sens de l'énoncé : c'est la compatibilité des filtrations donnée par le Corollaire 3 de \cite{Se}, XV §2.
\end{proof}

Quoique immédiat, le lemme suivant nous sera très utile pour la suite.
\begin{lemme}\label{lemme_cond_2_implique_type_I}
Une représentation automorphe de $\GL_n$ sur $\Q$ de conducteur 2 est nécessairement de type (I).
\end{lemme}
\begin{proof}
On a $1+2\Z_2=\Z_2^\times$ donc il n'existe pas de caractère modérément ramifié de $\Q_2^\times$.
\end{proof}

\begin{lemme}\label{lemme_carac_central_triv}
Soit $\pi$ une représentation automorphe cuspidale de $\GL_n$ sur $\Q$. On suppose que $\pi$ est de conducteur 1 ou de conducteur $p$ et de type (I). Alors le caractère central $\omega_\pi$ est trivial sur $\widehat{\Z}^\times$. En particulier, si $\pi$ est centrée, alors $\omega_\pi$ est trivial. 
\end{lemme}
\begin{proof}
On a $\omega_\pi=\omega_\infty \prod_\ell \omega_{\pi_\ell}$ avec des notations évidentes. En tout nombre premier $\ell$ tel que $\pi_\ell$ est non ramifiée, $\omega_{\pi_\ell}$ est trivial sur $\Z_\ell^\times$. Dans le cas où $\pi_p$ est ramifiée de type (I), son caractère central $\omega_{\pi_p}$ est donné par le déterminant du paramètre de Langlands (c'est une propriété de la correspondance de Langlands locale pour $\GL_n$, rappelée au Théorème \ref{LLC_GLn}) et on voit alors qu'il s'agit encore d'un caractère trivial sur $\Z_p^\times$.

La dernière assertion découle de la factorisation $\aq^\times=\Q^\times \times (\R_{>0} \times \widehat{\Z}^\times)$.
\end{proof}

Dans la suite, on s'intéresse uniquement à ces deux classes de représentations automorphes cuspidales \emph{algébriques} (donc centrées) de $\GL_n$ sur $\Q$ (conducteur 1 et conducteur $p$ de type (I)), que l'on peut alors voir comme des représentations de $\PGL_n$ sur $\Q$.

\begin{cor}\label{cor_GL_2_autoduale}
Soit $\pi$ une représentation automorphe cuspidale algébrique de $\GL_2$ sur $\Q$. On suppose que $\pi$ est de conducteur 1 ou de conducteur $p$ et de type (I). Alors $\pi$ est autoduale.
\end{cor}
\begin{proof}
On voit $\pi$ comme une représentation de $\PGL_2$ sur $\Q$ et on utilise le fait que l'automorphisme $g \mapsto {}^t g^{-1}$ est alors intérieur.
\end{proof}

\section{Fonctions $L$ et facteurs epsilon}\label{Fonctions $L$ et facteurs epsilon}
Cette partie poursuit deux objectifs :
\begin{itemize}
\item rappeler comment sont définis en général les fonctions $L$ (ou $\Lambda$) et les facteurs epsilon de paires ;
\item les calculer dans les cas particuliers qui nous intéressent (ce qui sera utile pour les Chapitres \ref{La_formule explicite de Riemann-Weil-Mestre} et \ref{Chapitre_calculs}).
\end{itemize}
\subsection{Fonctions $\Lambda$ et facteurs epsilon globaux}\label{Fonctions Lambda et facteurs epsilon globaux}

Soit $\pi$ une représentation automorphe cuspidale algébrique de $\GL_n$ sur $\Q$. Nous rappelons dans la suite comment lui associer une fonction $\Lambda$ qui sera un produit de fonctions $L$ locales :
\begin{equation}\label{fonction_Lambda_1_rep}
\Lambda(s,\pi)=L_\infty(s,\pi) \prod_p L_p(s,\pi),
\end{equation}
le produit étant défini et convergent pour $\re(s)>>0$, le tout s'étendant en une fonction méromorphe sur $\C$ (nous donnons des énoncés précis dans ce même §\ref{Fonctions Lambda et facteurs epsilon globaux} au paragraphe {\bfseries Paires}).

Chaque fonction $L$ locale ne dépend en fait que de la composante locale associée, selon la décomposition \eqref{prod_tens_res}. Ainsi $L_\infty(s,\pi)=L_\infty(s,\pi_\infty)$ et il s'agit alors de définir une telle fonction $L_\infty$ à partir de la donnée d'un $(\mathfrak{gl}_n(\R),\Oo_n(\R))$-module admissible irréductible algébrique. C'est l'objet du paragraphe \ref{Facteur local}.

De même, pour chaque $p$ premier, nous avons $L_p(s,\pi)=L_p(s,\pi_p)$ et nous rappelons au paragraphe \ref{Facteurs locaux} comment définir une fonction $L_p$ associée à une représentation admissible irréductible de $\GL_n(\qp)$.

\subsubsection{Facteur epsilon}\label{Facteurs_epsilon_712}
Si $\pi$ est une représentation automorphe cuspidale de $\GL_n$ sur $\Q$, il existe une équation fonctionnelle reliant la fonction $\Lambda$ de $\pi$ à celle de sa contragrédiente $\pi^\vee$ :
\[
\Lambda(s,\pi)=\eps(\pi){\rm N}(\pi)^{\frac{1}{2}-s}\Lambda(1-s,\pi^\vee),
\]
où ${\rm N}(\pi)$ est le conducteur de $\pi$ introduit en \eqref{conducteur-global}, et $\eps(\pi)$ est un complexe non nul.

{\bfseries Nous adoptons la convention qui consiste à faire systématiquement apparaître le conducteur.} Ainsi le facteur epsilon est-il réellement une constante (correspondant à la dénomination anglaise \emph{root number}), contrairement au cas de la convention qui consiste à appeler facteur epsilon la \emph{fonction} $s \mapsto \eps(\pi){\rm N}(\pi)^{\frac{1}{2}-s}$.

Selon la formule \eqref{conducteur-global}, le conducteur de $\pi$ se décompose en produit de quantités locales. Il en est de même pour le facteur epsilon \emph{global} $\eps(\pi)$ qui se décompose en produit de facteurs epsilon \emph{locaux} $\eps(\pi_v)$ qui ne dépendent (presque) que de la composante locale $\pi_v$.\ps 

Soyons plus précis en suivant \cite{Tate-Corvallis} §3.5.

On se donne, pour chaque nombre premier $p$, le caractère additif $\psi_p$ de $\qp$ défini par $\psi_p:x_p \mapsto e^{-2i\pi\{x_p\}_p}$ (déjà introduit au paragraphe \ref{Rappels}). À la place archimédienne, on se donne le caractère additif $\psi_\infty$ de $\R$ défini par $\psi_\infty:x_\infty \mapsto e^{+2i\pi x_\infty}$.

On peut alors former le caractère $\psi=\psi_\infty\prod_p \psi_p$ des adèles de $\Q$, dont on voit immédiatement qu'il est trivial sur $\Q$ (plongé diagonalement dans $\aq$).

Il nous faut également une mesure de Haar $\dd x$ sur $\aq$ telle que $\int_{\aq/\Q} \dd x=1$. Nous prenons à chaque place finie la mesure de Haar $\dd x_p$ sur $\qp$ normalisée par le fait que le volume de $\zp$ soit égal à 1 (c'est la mesure définie au paragraphe \ref{Rappels}), et à la place archimédienne la mesure de Lebesgue $\dd x_\infty$ standard. La mesure produit $\dd x=\prod_v \dd x_v$ vérifie alors les conditions requises.

On peut alors, pour chaque place finie $p$ définir un facteur epsilon local $\eps_p(\pi_p,\psi_p,\dd x_p)$ selon les rappels du paragraphe \ref{Rappels}. Pour la place archimédienne, nous aurons une définition \emph{ad hoc} dans le paragraphe \ref{Facteur local} pour les représentations algébriques. On peut alors former le produit :
\[
\eps(\pi,\psi,\dd x)=\prod_v \eps(\pi_v,\psi_v,\dd x_v),
\]
qui a la propriété remarquable de ne pas dépendre des choix globaux (raisonnables) de caractère additif et de mesure de Haar. \textbf{Nous fixons pour toute la suite} ces choix globaux, si bien que chaque facteur epsilon local ne dépend plus que de la composante locale : nous écrirons ainsi $\eps(\pi_v)$ pour $\eps(\pi_v,\psi_v,\dd x_v)$ étant entendu que $\psi_v$ et $\dd x_v$ sont tels que définis ci-dessus.

\subsubsection{Paires}\label{Paires globales}

On se donne maintenant deux représentations automorphes cuspidales unitaires (par exemple algébriques) $\pi$ et $\pi'$ de $\GL_n$ (resp. $\GL_{n'}$) sur $\Q$ et on raisonne sur la paire\footnote{Pour être parfaitement exact, il faudrait parler de multi-paire puisqu'on veut pouvoir considérer le cas où $\pi=\pi'$.} $\{\pi,\pi'\}$. Tous les objets introduits généralisent les précédents (avec une seule représentation) : il suffit de considérer la paire $\{\pi,\1\}$.

Les fonctions $\Lambda$ de paires (ou fonctions $\Lambda$ de Rankin-Selberg) ont été définies par Jacquet, Piatetski-Shapiro et Shalika dans \cite{JS} et \cite{JPSS83}, à qui l'on doit également les propriétés énoncées (sauf mention contraire). Pour cette partie, nous nous sommes servi de l'article de synthèse \cite{Cog_book}.

De même qu'en \eqref{fonction_Lambda_1_rep}, on utilise la décomposition en produit tensoriel restreint \eqref{prod_tens_res} pour définir la fonction $\Lambda$ de la paire $\{\pi,\pi'\}$ :
\begin{equation}\label{prod_eul+archim}
\Lambda(s, \pi \times \pi')=L_\infty(s, \pi_\infty \times \pi'_\infty) \prod_p L_p(s, \pi_p \times \pi'_p).
\end{equation}

Ce produit eulérien converge pour $\re(s)>1$ et la fonction $\Lambda$ se prolonge analytiquement en une fonction méromorphe du plan complexe.

Nous savons plus précisément par l'article de M\oe{}glin et Waldspurger \cite{MW} que cette fonction est entière sauf dans le cas où $\pi' \simeq \pi^\vee$ (où $\pi^\vee$ désigne la représentation contragrédiente de $\pi$) et qu'alors il y a deux pôles simples en $s=0$ et $s=1$.

On a de plus l'équation fonctionnelle attendue :
\begin{equation}\label{eq_fonctionnelle_paire_globale}
\Lambda(s, \pi \times \pi')=\eps(\pi \times \pi'){\rm N}(\pi \times \pi')^{\frac{1}{2}-s}\Lambda(1-s, \pi^\vee \times (\pi')^\vee),
\end{equation}
où ${\rm N}(\pi \times \pi')$ est le conducteur de la paire $\{\pi,\pi'\}$ (défini à partir des exposants d'Artin place par place, de manière analogue à ce qui a été fait en \eqref{conducteur-global}) et $\eps(\pi \times \pi')$ le facteur epsilon global, produit des facteurs espilon locaux.\medskip

\emph{Remarque :} La substitution $s \leftrightarrow 1-s$ dans l'équation ci-dessus impose, avec le fait que $(\pi^\vee)^\vee \simeq \pi$, que le conducteur de la paire $\{\pi,\pi'\}$ est égal à celui de la paire $\{\pi^\vee,(\pi')^\vee\}$, ce qui est bien le cas.

\subsection{À la place archimédienne}
\label{Facteur local}

Soit $\pi$ une représentation automorphe cuspidale \emph{algébrique} de $\GL_n$ sur $\Q$. Il s'agit de lui associer un facteur local pour la place archimédienne $L_\infty(s,\pi)$ fonction holomorphe définie pour $\re(s)>>0$ et qui ne dépend en fait que de $\pi_\infty$. Plutôt que de définir ce facteur local directement, nous utilisons la correspondance de Langlands locale et sa compatibilité aux fonctions $L$ : nous aurons donc une définition explicite des facteurs \og côté galoisien \fg et utiliserons la correspondance pour en déduire une \emph{définition} des facteurs \og côté automorphe \fg. La même remarque préliminaire s'applique aux facteurs epsilon.\medskip


Nous suivons \cite{Chen-Lannes} Chap. VIII § 2.20 et posons :
\begin{enumerate}
\item $\eps(\mathbf{1})=1$ et $\Gamma(s,\mathbf{1})=\Gamma_\R (s)$,
\item $\eps(I_n)=i^{n+1}$ et $\Gamma(s,I_n)=\Gamma_\C (s+ \frac{n}{2})$ pour $n \geq 0$
\end{enumerate}
avec
\[
\Gamma_\R(s)=\pi^{-\frac{s}{2}} \Gamma(\frac{s}{2}) \; \text{ et }  \; \Gamma_\C(s)=2 (2\pi)^{-s} \Gamma(s)
\]
où $s \mapsto \Gamma(s)$ est la fonction Gamma d'Euler. On remarque que l'on a, via la formule de duplication, $\Gamma_\C(s)=\Gamma_\R(s) \Gamma_\R (s+1)$.

On impose alors que, pour tous $\rho, \rho' \in K_\infty$ (\cf Définition \ref{defi_G(R)_alg_K_infty}) :
\[
\eps(\rho \oplus \rho')=\eps(\rho)\eps(\rho') \; \text{ et }  \; \Gamma(s,\rho \oplus \rho')=\Gamma(s, \rho) \Gamma(s, \rho')
\]
Ceci permet en particulier, en utilisant $I_0 \simeq \mathbf{1} \oplus \eps_{\C / \R}$, de voir que $\eps(\eps_{\C / \R})=i$ et $\Gamma(s,\eps_{\C / \R})=\Gamma_\R (s+1)$.

On peut alors définir la fonction $\rho \mapsto \Gamma(s,\rho)$ sur tout 
$K_\infty$. 
 Pour $\pi$ et $\pi'$ deux représentations automorphes cuspidales algébriques, on \emph{définit}
\begin{align*}
L_\infty (s, \pi \times \pi')&=\Gamma(s, \Ll(\pi) \otimes \Ll(\pi')) ; \\
\eps_\infty (\pi \times \pi')&=\eps(\Ll(\pi) \otimes \Ll(\pi')).
\end{align*}
Par ailleurs, $\eps_\infty (\pi \times \pi')$ est un nombre complexe non nul, et plus précisément une puissance de $i$.

\subsection{Aux places finies}
\label{Facteurs locaux}
Comme dans le cas archimédien, même si l'on travaille avec des objets globaux, la définition des facteurs locaux ne fait intervenir que les composantes locales (y compris pour le facteur epsilon puisqu'on a fixé une fois pour toutes un caractère additif et une mesure locaux à chaque place au paragraphe \ref{Facteurs_epsilon_712}). Soit donc $p$ un nombre premier (\ie une place finie).

La même remarque préliminaire qu'au paragraphe \ref{Facteur local} s'applique : nous donnons une définition explicite des facteurs \og côté galoisien\fg et utilisons la correspondance pour en déduire une \emph{définition} des facteurs \og côté automorphe \fg. Plus précisément, il nous faut donc rappeler, étant donnée une représentation semi-simple $\varphi$ du groupe de Weil-Deligne de $\qp$, comment lui associer une fonction $L$ locale. On veut la propriété d'additivité suivante :
\[
L_p(s,\varphi \oplus \varphi')=L_p(s,\varphi)L_p(s,\varphi').
\]

Par semi-simplicité, on peut donc se contenter de définir une telle fonction $L$ dans le cas où $\varphi$ est irréductible, \ie de la forme $\sigma \otimes U_a$ où $\sigma$ est une représentation irréductible de $\W_{\qp}$ et où $a$ est un entier naturel non nul.

Plus classiquement (dans \cite{Tate-Corvallis} notamment), les facteurs sont définis pour des représentations de Weil-Deligne, au sens de la Définition \ref{def_rep_de_WD}. Soit donc $(\rho, V,N)$ une telle représentation. On pose :
\[
L_p(s,(\rho, V,N))=\det(1-p^{-s}\rho(\Fr)_{|V_N^I})^{-1},
\]
où $\Fr$ désigne un morphisme de Frobenius (au sens du paragraphe \ref{Groupe de Weil}), $V_N$ désigne le noyau de l'endomorphisme nilpotent $N$ et l'exposant $I$ le sous-espace des invariants pour l'action du sous-groupe d'inertie.

Comme au paragraphe \ref{Rappels}, nous utilisons l'article de Benedict Gross et Mark Reeder \cite{GR} pour donner la \og traduction \fg{} suivante dans le contexte d'une représentation irréductible $\varphi=\sigma\otimes U_a$ de $\WD_{\qp}$ :
\begin{equation}\label{defi_fonction_L_WD}
L_p(s,\varphi)=\det(1-p^{-s-\frac{a-1}{2}}\,\Fr_{|\sigma^I})^{-1}.
\end{equation}
\smallskip


Dans un objectif pédagogique, nous commençons par considérer les fonctions $L$ locales associées à une seule représentation, le cas des paires étant abordé \emph{infra}.

Soit donc $\pi_p$ une représentation admissible irréductible de $\GL_n(\qp)$, que l'on suppose d'abord non ramifiée.
Alors $\Ll(\pi_p) = \chi_1 \oplus \cdots \oplus \chi_n$, où les $\chi_i$ sont des caractères non ramifiés (c'est la Proposition-Définition \ref{prop_types_I_et_II} en exposant d'Artin nul).
Chaque caractère $\chi_i$, vu comme caractère de $\qp^\times$ non ramifié, contribue pour un facteur $(1-\chi_i(p)p^{-s})^{-1}$.

 La fonction $L$ associée à cette représentation (du groupe de Weil de $\qp$ en fait ici) est :
\begin{equation}\label{fonction_L_rep_loc_nr}
L_p(s,\Ll(\pi_p))=\prod_{i=1}^n \frac{1}{1-\chi_i(p)p^{-s}},
\end{equation}
et cela \emph{définit} la fonction $L_p(s,\pi_p)$.

Si l'on note $c_p(\pi_p)$ la classe de conjugaison semi-simple dans $\GL_n(\C)$ de ${\rm diag}(\chi_1(p),\cdots,\chi_n(p))$, on retrouve le paramètre de Satake d'une représentation non ramifiée et on a bien $L_p(s,\pi_p)=\det(1-p^{-s}c_p(\pi_p))^{-1}$.

Le seul autre cas que nous ayons à considérer est celui d'une représentation $\varpi_p$ de $\GL_n(\qp)$ de conducteur $p$ et de type (I) selon la Proposition-Définition \ref{prop_types_I_et_II}, \ie dont le paramètre de Langlands est de la forme
\[
\varphi=\chi_1 \oplus \cdots \oplus \chi_{n-2} \oplus (\psi \otimes U_2),
\]
où les $\chi_i$ \emph{et} $\psi$ sont des caractères non ramifiés. Il nous faut donc calculer la fonction $L$ associée à la représentation irréductible $\psi \otimes U_2$, selon la formule \eqref{defi_fonction_L_WD} :
\[
L_p(s,\psi \otimes U_2)=(1-\psi(p)p^{-s-\frac{1}{2}})^{-1}.
\]

Nous avons donc dans ce cas :
\begin{equation}\label{fonction_L_rep_loc_ram_type_I}
L_p(s,\varpi_p)=\left(\prod_{i=1}^{n-2} \frac{1}{1-\chi_i(p)p^{-s}}\right) \frac{1}{1-\psi(p)p^{-s-\frac{1}{2}}}.
\end{equation}

On note, par analogie avec le paramètre de Satake, $d_p(\varpi_p)$ la classe de conjugaison semi-simple  dans $\GL_{n-1}(\C)$ de ${\rm diag}(\chi_1(p),\cdots,\chi_{n-2}(p),\psi(p)p^{-\frac{1}{2}})$ -- remarquer que c'est une classe de conjugaison dans $\GL_{\bm{n-1}}(\C)$ pour une représentation de $\GL_{\bm{n}}(\qp)$. On a alors $L_p(s,\varpi_p)=\det(1-p^{-s}d_p(\varpi_p))^{-1}$. 

%

\subsubsection{Fonctions de paires}\label{Fonctions de paires}

Il nous faut maintenant considérer les fonctions $L$ locales associées à une \emph{paire} de représentations et donc, en dernier lieu, comprendre le produit tensoriel de représentations de $\WD_{\qp}$.

Soient donc $\pi_p$ et $\pi'_p$ deux représentations admissibles irréductibles de \linebreak $\GL_m(\qp)$ et $\GL_{m'}(\qp)$ respectivement, non ramifiées. Soient $\varpi_p$ et $\varpi'_p$ deux représentations admissibles irréductibles de $\GL_n(\qp)$ et $\GL_{n'}(\qp)$ respectivement, de conducteur $p$ et de type (I).

Nous fixons les notations suivantes :

\begin{equation}\label{Notations_pi,varpi_p}
\begin{array}{rl}\vspace{3 pt}
\Ll(\pi_p)&=\chi_1 \oplus \cdots \oplus\chi_m ,\\ \vspace{3 pt}
\Ll(\pi'_p)&=\chi'_1 \oplus \cdots \oplus\chi'_{m'} ,\\ \vspace{3 pt}
\Ll(\varpi_p)&=\eta_1 \oplus \cdots \oplus\eta_{n-2} \oplus (\psi \otimes U_2) ,\\ \vspace{3 pt}
\Ll(\varpi'_p)&=\eta'_1 \oplus \cdots \oplus\eta'_{n'-2} \oplus (\psi' \otimes U_2).
\end{array}
\end{equation}

\begin{prop}\label{prop_calcul_fonction_L_locale}
Avec les notations précédentes, nous avons :
\begin{align*}
L_p(\pi_p\times \pi'_p)&=\frac{1}{\det(1-p^{-s}c_p(\pi_p)\otimes c_p(\pi'_p))} ;\\
L_p(\pi_p\times \varpi_p)&=\frac{1}{\det(1-p^{-s}c_p(\pi_p)\otimes d_p(\varpi_p))} ;\\
L_p(\varpi_p\times \varpi'_p)&=\frac{1}{\det(1-p^{-s}d_p(\varpi_p)\otimes d_p(\varpi'_p))}\cdot \frac{1}{1-\psi(p)\psi'(p)p^{-s}}.
\end{align*}
\end{prop}

\begin{proof}
Il suffit de regarder la représentation du groupe de Weil-Deligne obtenue par tensorisation.

Dans le premier cas, nous avons :
\[
\Ll(\pi_p) \otimes \Ll(\pi'_p)= \bigoplus_{i=1}^m \bigoplus_{j=1}^{m'} \chi_i \chi'_j.
\]
C'est donc encore une somme de caractères non ramifiés, on en déduit :
\[
L_p(\pi_p\times \pi'_p)=\prod_{i=1}^m\prod_{j=1}^{m'} \frac{1}{1-\chi_i(p)\chi'_j(p)p^{-s}}.
\]

Dans le deuxième cas, nous avons :
\[
\Ll(\pi_p) \otimes \Ll(\varpi_p)= \bigoplus_{i=1}^m \bigoplus_{j=1}^{n-2} \chi_i \eta_j \oplus \bigoplus_{i=1}^m (\chi_i \psi) \otimes U_2.
\]
Les contributions de chacun de ces termes ont été calculées précédemment et on obtient :
\begin{align*}
L_p(\pi_p\times \varpi_p)&=\left(\prod_{i=1}^m\prod_{j=1}^{n-2} \frac{1}{1-\chi_i(p)\eta_j(p)p^{-s}}\right) \prod_{i=1}^m \frac{1}{1-\chi_i(p)\psi_j(p)p^{-s-\frac{1}{2}}} \\
&=\prod_{i=1}^m\left[\left(\prod_{j=1}^{n-2} \frac{1}{1-\chi_i(p)\eta_j(p)p^{-s}}\right) \frac{1}{1-\chi_i(p)\psi(p)p^{-s-\frac{1}{2}}} \right].
\end{align*}

Dans le troisième cas, nous avons :
\[
\Ll(\varpi_p) \otimes \Ll(\varpi'_p)= \bigoplus_{i=1}^{n-2} \bigoplus_{j=1}^{n'-2} \eta_i \eta'_j \oplus \bigoplus_{i=1}^{n-2} \eta_i \psi' \otimes U_2 \oplus \bigoplus_{j=1}^{n'-2} \eta'_j \psi \otimes U_2 \oplus (\psi \psi' \otimes (\1 \oplus U_3)),
\]
en utilisant $U_2 \otimes U_2=U_1 \oplus U_3$ (\cf \eqref{regle_prod_tens_rep_sym}). \ps

On pose alors
\begin{align*}
L_p^{\rm I}(s)&= \prod_{i=1}^{n-2} \prod_{j=1}^{n'-2} \frac{1}{1-\eta_i(p)\eta'_j(p) p^{-s}} ;\\
L_p^{\rm II}(s)&= \prod_{i=1}^{n-2} \frac{1}{1-\eta_i(p)\psi'(p) p^{-\frac{1}{2}-s}} \prod_{j=1}^{n'-2} \frac{1}{1-\eta'_j(p)\psi(p) p^{-\frac{1}{2}-s}} ;\\
L_p^{\rm III}(s)&= \frac{1}{1-\psi(p)\psi'(p) p^{-s}} \cdot \frac{1}{1-\psi(p)\psi'(p) p^{-1-s}}
\end{align*}

puis
\[
L_p(s,\varpi_p \times \varpi'_p)=L_p^{\rm I}(s)L_p^{\rm II}(s)L_p^{\rm III}(s).
\]

On voit alors que le deuxième facteur de $L_p^{\rm III}$ combiné à $L_p^{\rm I}$ et $L_p^{\rm II}$ nous donne $(\det(1-p^{-s}d_p(\varpi_p)\otimes d_p(\varpi'_p))^{-1}$. Il reste donc le premier facteur de $L_p^{\rm III}$.
\end{proof}

\subsubsection{Facteurs epsilon locaux}\label{Facteurs epsilon locaux}

Il nous faut enfin déterminer les facteurs epsilon locaux en $p$ d'une représentation et d'une paire de représentations (le premier cas relevant du second si l'on veut en considérant la représentation triviale).
Nous raisonnons de nouveau \og côté galoisien\fg sur les paramètres de Langlands.

\begin{prop}\label{prop_facteur_eps_p_paire}
Avec les notations \eqref{Notations_pi,varpi_p} et, en notant $\omega_{\mu}$ 
le caractère central d'une représentation locale $\mu$, on a :
\begin{align*}
\eps_p(\pi \times \pi')&=1 ; \\
\eps_p(\pi \times \varpi)&= (-1)^m \psi^m(p)\, \omega_{\pi_p}(p) ; \\
\eps_p(\varpi \times \varpi')&=(-1)^{n+n'}\psi(p)^{n'-2}\psi'(p)^{n-2}\,\omega_{\varpi_p}(p)\,\omega_{\varpi'_p}(p).
\end{align*}
\end{prop}

\begin{proof}
Puisque $\pi$ et $\pi'$ non ramifiées en $p$, $\Ll(\pi_p) \otimes \Ll(\pi'_p)$ est une somme de caractères non ramifiés. Or, d'après les rappels du paragraphe \ref{Rappels} qui s'appliquent ici encore, le facteur epsilon d'un caractère non ramifié est trivial et, par additivité, $\eps_p(\pi \times \pi')=1$.

Avant de considérer les deux autre paires, calculons déjà $\eps_p(\varpi)=\eps(\Ll(\varpi_p))$. On a : 
\[
\Ll(\varpi_p)=\eta_1 \oplus \cdots \oplus\eta_{n-2} \oplus (\psi \otimes U_2),
\]
avec les $\eta_i$ et $\psi$ non ramifiés. Par additivité, on ne doit en fait considérer que $\eps(\psi \otimes U_2)$. Le même calcul qu'au paragraphe \ref{Calcul pour les représentations étudiées} en écrivant le caractère non ramifié $\psi$ comme une puissance de la norme nous donne : $\eps(\psi \otimes U_2)=\psi(p)\eps(U_2)$. Or $\eps(U_2)$ a été calculé \emph{loc. cit.} et vaut $-1$. Nous avons finalement :
\begin{equation}\label{eps_local_varpi_p}
\eps_p(\varpi)=-\psi(p).
\end{equation}

Il nous reste à considérer le cas de la paire $\{\pi,\varpi\}$ et celui de la paire $\{\varpi,\varpi'\}$. Dans le premier cas, nous avons :
\[
\Ll(\pi_p) \otimes \Ll(\varpi_p)= \bigoplus_{i=1}^m \bigoplus_{j=1}^{n-2} \chi_i \eta_j \oplus \bigoplus_{i=1}^m (\chi_i \psi) \otimes U_2,
\]
et donc, par le calcul précédent :
\begin{align*}
\eps_p(\pi \times \varpi)&=\prod_{i=1}^m(-\chi_i\psi)(p) \\
						&=(-1)^m \psi^m(p) \prod_{i=1}^m \chi_i(p).
\end{align*}
On remarque alors que le dernier produit est en fait égal à $\det \Ll(\pi_p)(p)$, soit encore à $\omega_{\pi_p}(p)$ par la compatibilité de la correspondance de Langlands pour $\GL_n$ au caractère central (\cf Théorème \ref{LLC_GLn}).\medskip

Dans le second cas, nous avons :
\[
\Ll(\varpi_p) \otimes \Ll(\varpi'_p)= \bigoplus_{i=1}^{n-2} \bigoplus_{j=1}^{n'-2} \eta_i \eta'_j \oplus \bigoplus_{i=1}^{n-2} \eta_i \psi' \otimes U_2 \oplus \bigoplus_{j=1}^{n'-2} \eta'_j \psi \otimes U_2 \oplus (\psi \psi' \otimes (\1 \oplus U_3)).
\]

Il nous faut donc calculer $\eps(\psi\psi' \otimes U_3)$. Si l'on écrit $\psi=|\cdot|^t$ et $\psi'=|\cdot|^{t'}$, alors $\eps(\psi\psi' \otimes U_3)=\eps(|\cdot|^{t+t'} \otimes U_3)=\eps(U_3)p^{\aar_\mathrm{WD}(U_3)(-(t+t'))}$.

Or $\aar_\mathrm{WD}(U_3)=2$ et $\eps(U_3)=(-\1(\Fr))^{3-1}=1$ d'après \eqref{eps_rep_wd} et finalement $\eps(\psi\psi' \otimes U_3)=\psi(p)^2\psi'(p)^2$. Nous avons donc :
\begin{align*}
\eps_p(\varpi \times \varpi')&=\psi(p)^2\psi'(p)^2\prod_{i=1}^{n-2}(-\eta_i\psi')(p) \prod_{j=1}^{n'-2} (-\eta'_j\psi)(p) \\
						&=(-1)^{n+n'}\psi(p)^{n'-2}\psi'(p)^{n-2}\left(\psi(p)^2\prod_{i=1}^{n-2}\eta_i(p)\right) \left(\psi'(p)^2\prod_{j=1}^{n'-2}\eta'_j(p)\right) \\
						&=(-1)^{n+n'}\psi(p)^{n'-2}\psi'(p)^{n-2}\,\omega_{\varpi_p}(p)\,\omega_{\varpi'_p}(p).
\end{align*}
\end{proof}

\section{Alternative symplectique-orthogonale}\label{Alternative symplectique-orthogonale}
Soit $\mathbf{G}$ un groupe algébrique défini et quasi-déployé sur $\Q$ appartenant à une des familles de groupes classiques $\Sp_{2n}, \SO_{2n}$ et $\SO_{2n+1}$. 

On a introduit au paragraphe \ref{Paramètres de Langlands} la notion de groupe dual et en particulier des $\C$-points de ce groupe dual, qui permettent de définir les paramètres de Langlands. Notons $\mathrm{Std}_{\widehat{G}}$ la représentation standard de $\widehat{G}$ dans $\GL_n(\C)$, ce qui d'ailleurs fixe l'entier $n$ que l'on notera $\mathrm{n}(\widehat{G})$. 
 On a ainsi $\mathrm{n}(\widehat{G})=2n$ pour $\mathbf{G}=\SO_{2n}$ ou $\mathbf{G}=\SO_{2n+1}$ et $\mathrm{n}(\widehat{G})=2n+1$ pour $\mathbf{G}=\Sp_{2n}$.


\begin{thm} \emph{(Arthur, \cite{Art13} Theorem 1.4.1 \emph{et} Theorem 1.4.2)} \label{thm_alt_sp_orth}

Soit $n$ un entier strictement positif et soit $\pi$ une représentation automorphe cuspidale \emph{autoduale} de $\GL_n$ sur $\Q$. Alors il existe un groupe classique quasi-déployé $\bm{\mathbf{G}_\pi}$, unique à isomorphisme près, tel que $\mathrm{n}(\widehat{G_\pi})=n$
et une représentation automorphe $\tilde{\pi}$ discrète de $\bm{\mathbf{G}_\pi}$ telle que $\mathrm{Std}_{\widehat{G_\pi}} \circ \Ll(\tilde{\pi}_v)$ soit égale (comme classe de conjugaison) à $\Ll(\pi_v)$ pour toute place $v$ de $\Q$.
\end{thm}


On dit alors que $\pi$ est orthogonale si $\widehat{G_\pi}$ est (isomorphe à) un groupe spécial orthogonal et symplectique si $\widehat{G_\pi}$ est (isomorphe à) un groupe symplectique.
Si $n$ est impair, alors  la seule possibilité est $\bm{\mathbf{G}_\pi}=\Sp_{n-1}$, et $\pi$ est orthogonale. 

\begin{prop} \emph{(\cite{Chen-Lannes}, Proposition 3.3)} \label{CL VIII 3.3}

Soit $\pi$ une représentation automorphe cuspidale autoduale algébrique de $\GL_n$. On suppose que $\pi_\infty$ possède au moins un poids sans multiplicité. Alors $\pi$ est symplectique si, et seulement si son poids motivique $\mathrm{w}(\pi)$ est impair.
\end{prop}
\begin{proof}
La démonstration \emph{loc. cit.} suppose que le conducteur est égal à 1, ce qui n'intervient pas dans la preuve, elle s'applique donc \emph{verbatim}.
\end{proof}

\begin{prop}\label{type I sp}
Soit $\pi$ une représentation automorphe cuspidale autoduale de $\GL_n$ de conducteur $p$ et de type (I). Alors $\pi$ est symplectique et $n$ est pair.
\end{prop}
Cette Proposition explique pourquoi la famille de groupes classiques $\SO_{2n+1}$ joue un rôle prépondérant dans la suite de notre travail.
\begin{proof}
D'après le Théorème \ref{thm_alt_sp_orth}, il existe un groupe classique $\bm{\mathbf{G}_\pi}$ et une représentation discrète $\tilde{\pi}$ associés à $\pi$. On sait de plus que $\Ll(\pi_p)$ est égale, comme classe de conjugaison, à $\mathrm{Std}_{\widehat{G_\pi}} \circ \Ll(\tilde{\pi}_p)$. Or on connaît exactement $\Ll(\pi_p)$ qui est égal à $\chi_1 \oplus \cdots \oplus \chi_{n-2} \oplus (\psi \otimes U_2)$, où les $\chi_i$ \emph{et} $\psi$ sont des caractères non ramifiés. En particulier, le terme $(\psi \otimes U_2)$ est le seul de son type et $\Ll(\pi_p)$ ne peut donc pas se factoriser par un groupe orthogonal. Donc $\widehat{G_\pi}$ est nécessairement égal à $\Sp_{2m}(\C)$ et $\pi$ est ainsi une représentation symplectique avec $\bm{\mathbf{G}_\pi}=\SO_{2m+1},\, \widehat{G_\pi}=\Sp_{2m}(\C)$ et $n=2m$. 
\end{proof}

\begin{cor}\label{cor_654}
Soit $w$ un entier \emph{pair} et soit $V$ un élément de $K_\infty^{\leq w}$ sans multiplicité. Alors si $\pi$ est une représentation automorphe cuspidale algébrique de conducteur $p$ et de type (I) telle que $\Ll(\pi_\infty) = V$, $\pi$ n'est pas autoduale. On a alors $\pi^\vee \not\simeq \pi$ qui vérifie également $\Ll((\pi^\vee)_\infty) = V$.

En particulier le nombre de (classes d'isomorphie de) représentations automorphes cuspidales algébriques $\pi$ de conducteur $p$, et de type (I) vérifiant $\Ll(\pi_\infty)= V$ est pair.
\end{cor}
\begin{proof}
Supposons $\pi$ autoduale. Alors la Proposition \ref{CL VIII 3.3} nous dit que $\pi$ est orthogonale, ce qui contredit la Proposition \ref{type I sp}.
\end{proof}
\subsection{Facteur epsilon}
\begin{prop} \emph{(Arthur \cite{Art13}, Theorem 1.5.3)} \label{facteur_epsilon_alternative_sp_orth}

Soient $\pi$ et $\pi'$ deux représentations automorphes cuspidales autoduales de $\GL_m$ et $\GL_n$ respectivement. Alors, si $\pi$ et $\pi'$ sont du même type (\ie toutes deux orthogonales ou toutes deux symplectiques), le facteur global $\eps(\pi \times \pi')$ qui apparaît dans l'équation fonctionnelle de la fonction $\Lambda$ de paire :
\[
\Lambda(s, \pi \times \pi')=\eps(\pi \times \pi'){\rm N}(\pi \times \pi')^{\frac{1}{2}-s}\Lambda(1-s, \pi \times \pi'),
\]
est égal à 1.
\end{prop}

La suite de ce paragraphe est une forme de longue remarque consistant à vérifier que, dans les cas étudiés, les formules de la Proposition \ref{prop_facteur_eps_p_paire} sont bien cohérentes avec ce résultat.


D'après les classifications de \cite{Chen-Lannes} (Théorème F), la seule représentation algébrique orthogonale de conducteur 1 et de poids motivique inférieur à 21 est la représentation triviale. De plus, d'après la Proposition \ref{type I sp}, une représentation autoduale de conducteur $p$ et de type (I) est nécessairement symplectique. Nous allons donc considérer le seul cas du facteur epsilon d'une paire de représentations symplectiques (en dimension paire donc), que l'on supposera de plus régulières (ce qui sera le cas en pratique).\medskip

Soient $\pi$ et $\pi'$ deux représentations automorphes cuspidales de $\GL_{2m}$ et $\GL_{2m'}$ respectivement, algébriques, symplectiques, régulières et de conducteur 1. La Proposition \ref{prop_facteur_eps_p_paire} nous donne déjà $\eps_v(\pi \times \pi')=1$ pour toute place finie $v$.


Par l'hypothèse de régularité et par la Proposition \ref{CL VIII 3.3}, $\pi$ est de poids motivique impair, et on a $\Ll(\pi_\infty)=I_{w_1} \oplus \cdots \oplus I_{w_m}$ avec $w_1 > \cdots >w_m$ impairs (c'est le Lemme \ref{lemme_pi_infty_i_w}). Il en est de même pour $\pi'_\infty$ ; pour déterminer $\eps_\infty(\pi \times \pi')$, il faut donc calculer $\eps(I_w \otimes I_{w'})$ avec $w,w'$ impairs.

Supposons, sans nuire à la généralité $w \geq w'$. Alors, par un résultat déjà mentionné au paragraphe \ref{Facteur local}, nous avons :
\[
I_w \otimes I_{w'}=I_{w+w'} \oplus I_{w-w'},
\]
puis
\begin{align*}
\eps(I_w \otimes I_{w'})&=\eps(I_{w+w'}) \eps(I_{w-w'}) \\
						&=i^{w+w'+1} i^{w-w'+1} \\
						&=i^{2(w+1)} \\
						&=1,
\end{align*}
car $w$ est impair.

Ainsi, dans le cas de deux représentations automorphes cuspidales $\pi$ et $\pi'$, algébriques, symplectiques, régulières et de conducteur 1, nous avons $\eps_v(\pi \times \pi')=1$ à toutes les places et donc bien $\eps(\pi \times \pi')=1$ comme attendu.\newline

Considérons maintenant la paire $\{\pi,\varpi\}$ où $\varpi$ est une représentation automorphe cuspidale de $\GL_{2n}$ de conducteur $p$ et de type (I) -- et encore algébrique, symplectique, régulière. Alors tout se passe de la même façon aux places $v$ différentes de $p$ où l'on a encore $\eps_v(\pi \times \varpi)=1$. Pour conclure, il suffit de regarder ce qu'il se passe à la place $p$.

Il nous faut maintenant être plus précis sur la forme des paramètres de Langlands en $p$. On a vu que $\Ll(\pi_p)$ était une somme de caractères non ramifiés, mais par le Théorème \ref{thm_alt_sp_orth}, ce paramètre est symplectique, ce qui impose que chaque caractère non ramifié apparaît avec son inverse, nous avons finalement :
\[
\Ll(\pi_p)=\chi_1 \oplus \cdots \oplus \chi_m \oplus \chi_m^{-1} \oplus \cdots \oplus \chi_1^{-1}.
\]

Quant à $\Ll(\varpi_p)$, c'est une somme de $2n-2$ caractères non ramifiés et d'un terme $\alpha \otimes U_2$, où $\alpha$ est encore un caractère non ramifié. Or, là aussi, le fait d'être symplectique impose, d'une part, que $\alpha$ est à valeurs dans $\{ \pm 1 \}$ et donc $\alpha \in \{ \1,\eta \}$, et d'autre part que les autres caractères non ramifiés apparaissent chacun avec leur inverse. Finalement, nous avons :
\[
\Ll(\varpi_p)=\bigoplus_{j=1}^{n-1} (\psi_j \oplus \psi_j^{-1}) \oplus (\alpha \otimes U_2).
\]

Ainsi, la Proposition \ref{prop_facteur_eps_p_paire} nous donne :
\[
\eps_p(\pi \times \varpi)=(-1)^{2m} \alpha^{2m}(p) \prod_{i=1}^m \chi_i(p) \chi_i^{-1}(p) =1,
\]
ce qui conclut encore dans ce cas. \newline

Considérons enfin le cas de deux représentations automorphes cuspidales de $\GL_{2n}$ et $\GL_{2n'}$ respectivement, $\varpi$ et $\varpi'$, algébriques, symplectiques, régulières,  de conducteur $p$ et de type (I).
Là encore, il suffit de regarder ce qu'il se passe à la place $p$ pour conclure. On a :
\begin{align*}
\Ll(\varpi_p)&=\bigoplus_{j=1}^{n-1} (\psi_j \oplus \psi_j^{-1}) \oplus (\alpha \otimes U_2), \\
\Ll(\varpi'_p)&=\bigoplus_{j=1}^{n'-1} (\psi'_j \oplus {\psi'_j}^{-1}) \oplus (\alpha' \otimes U_2),
\end{align*}
avec tous les caractères intervenant non ramifiés et $\alpha,\alpha' \in \{\1,\eta\}$.

La Proposition \ref{prop_facteur_eps_p_paire} nous donne alors :
\[
\eps_p(\varpi \times \varpi')=(-1)^{2n+2n'} \alpha^{2n'}(p) {\alpha'}^{2n}(p) \prod_{i=1}^{n-1} \psi_i(p) \psi_i^{-1}(p) \prod_{j=1}^{n'-1} \psi'_j(p) {\psi'_j}^{-1}(p),
\]
avec tous les facteurs égaux à 1, ce qui conclut.


\newpage
\chapter{Théorie d'Arthur pour $\SO_{2n+1}$}\label{Théorie d'Arthur pour SO-2n+1}
La théorie d'Arthur pour paramétrer les représentations des groupes classiques est subtile et demande de nombreuses précautions techniques. Comme nous ne l'utilisons que dans des cas qui simplifient considérablement les énoncés, nous nous contentons ici de rappels dans un cadre \emph{ad hoc}.

Nous aurons besoin de considérer deux types de groupes, présentés au paragraphe \ref{Groupes étudiés et leurs représentations}.
Dans le {\bfseries Cas 1} (déployé), tous les résultats énoncés se trouvent dans le livre de James Arthur \cite{Art13}. Le Chapitre 9 de ce même ouvrage discute de la généralisation au cas des formes intérieures des groupes classiques et donc en particulier au cas des groupes classiques non quasi-déployés, mais ne la démontre pas. 
Sous certaines hypothèses supplémentaires, Olivier Taïbi a pu démontrer les résultats souhaités dans \cite{Taibi_cpctmult}, à l'excellente introduction duquel nous renvoyons. Nous nous contentons ici de remarquer que le {\bfseries Cas 2} relève des résultats démontrés par Taïbi, si bien que nous pourrons appliquer le formalisme d'Arthur indifféremment aux deux cas.

\section{Groupes étudiés et leurs représentations}\label{Groupes étudiés et leurs représentations}

Soit $(V,q)$ un espace quadratique non dégénéré de dimension impaire $2n+1$ sur $\Q$ et soit $\mathbf{G}=\bm{\SO_V}$ le groupe (algébrique) spécial orthogonal associé. Pour chaque place $v$, on notera $V_v$ l'espace quadratique $(V \otimes_\Q \Q_v, q\otimes_\Q \Q_v)$ et on allégera les notations en ne mentionnant pas systématiquement la forme quadratique associée. On fait l'hypothèse que $q \otimes_\Q \Q_p$ est d'indice de Witt maximal $n$ pour tout nombre premier $p$. Ainsi le groupe spécial orthogonal local associé $\bm{\SO_{V_p}}$ est déployé sur $\qp$ pour tout nombre premier $p$ et on peut utiliser les résultats de la Première Partie. Il reste alors à examiner ce qu'il se passe à l'unique place archimédienne. Nous aurons deux cas à considérer :
\begin{itemize}
\item \textbf{Cas 1} : la forme quadratique $q \otimes_\Q \R$ est de signature $(n+1,n)$ (ou l'inverse) donc d'indice de Witt maximal et le groupe spécial orthogonal associé $\bm{\SO_{V_\infty}} \simeq \SO_{n+1,n}/\R \simeq \SO_{n,n+1}/\R$ est encore déployé.  
\item \textbf{Cas 2} : la forme quadratique $q \otimes_\Q \R$ est de signature $(2n+1,0)$ (resp. $(0,2n+1)$) donc définie positive (resp. définie négative) et le groupe spécial orthogonal associé $\bm{\SO_{V_\infty}} \simeq \SO_{2n+1,0}/\R \simeq \SO_{0,2n+1}/\R$ est compact. \newline
\end{itemize}

On trouve le {\bfseries Cas 1} pour tout $n$, il suffit de considérer l'espace quadratique $\Q^{2n+1}$ muni de la forme $x\mapsto x_1x_2+\cdots+x_{2n-1}x_{2n}+x_{2n+1}^2$. Le {\bfseries Cas 2} ne peut en revanche se produire que si $2n+1 \equiv \pm 1 \mod 8$ (\cf Proposition \ref{ImQp}) \footnote{Nous définirons au §\ref{Conjectures pour le poids motivique $21$} un {\bf Cas 3} qui se produit quand $2n+1 \equiv \pm 3 \mod 8$, pour lequel les énoncés sont conjecturaux.}.\newline

Nous pouvons maintenant faire quelques rappels sur les représentations automorphes de $\mathbf{G}$, étant entendu que $\mathbf{G}$ relève d'un des deux cas ci-dessus. De même qu'au paragraphe \ref{Représentations automorphes cuspidales}, nous renvoyons à \cite{BJ-Corvallis}.

Le centre de $\mathbf{G}$ est ici trivial, il n'est donc pas question de caractère central. On peut considérer l'espace $\mathcal{A}^2(\mathbf{G})$ des formes automorphes de carré intégrable pour $\mathbf{G}$ et sa partie discrète. Une représentation automorphe discrète est alors un des constituants irréductibles de cette partie discrète. On a ainsi
\begin{equation}\label{A_disc_SO}
\mathcal{A}_{\rm disc}(\mathbf{G})=\bigoplus_{\pi \in \Pi_{\rm disc}} \m(\pi) \pi,
\end{equation}
où $\m(\pi)$ est la multiplicité de $\pi$ dans cette décomposition, qui est finie et strictement positive par définition de $\Pi_{\rm disc}(\mathbf{G})$.

On peut également définir l'ensemble $\mathcal{A}_{\rm cusp}(\mathbf{G})$ des formes automorphes \emph{paraboliques} (ou \emph{cuspidales}) comme sous-espace de $\mathcal{A}^2(\mathbf{G})$. On a alors $\mathcal{A}_{\rm cusp}(\mathbf{G}) \subset \mathcal{A}_{\rm disc}(\mathbf{G})$ et on note $\Pi_{\rm cusp}(\mathbf{G})$ l'ensemble des représentations intervenant dans la décomposition de $\mathcal{A}_{\rm cusp}(\mathbf{G})$. 

Si $\pi$ est une représentation automorphe discrète de $\mathbf{G}$, on a encore la décomposition en produit tensoriel restreint :
\begin{equation}\label{prod_tens_res_so}
\pi \simeq \pi_\infty \otimes {\bigotimes_p}' \pi_p,
\end{equation}
où $\pi_\infty$ est un module de Harish-Chandra irréductible unitaire\footnote{Dans le {\bfseries Cas 2}, c'est simplement une représentation irréductible unitaire de dimension finie du groupe compact $\mathbf{G}(\R)$.} de $\mathbf{G}(\R)$  et les $\pi_p$ sont des représentations admissibles irréductibles unitaires de $\mathbf{G}(\Q_p)$, non ramifiées pour presque tout $p$. 

Le module de Harish-Chandra $\pi_\infty$ admet un caractère infinitésimal $\mathrm{inf}(\pi_\infty)$ que l'on peut voir, via l'isomorphisme de Harish-Chandra, comme une classe de conjugaison semi-simple dans $\mathfrak{sp}_{2n}(\C)$, dont les $2n$ valeurs propres sont encore appelées poids. En particulier, le multi-ensemble des poids est stable par $X \mapsto -X$.

\begin{defi}\label{defi_alg_reg_so_2n+1}
Soit $U$ un module de Harish-Chandra irréductible unitaire de $\mathbf{G}(\R)$. On dit que $U$ est \emph{algébrique} si ses poids sont dans $\frac{1}{2}\Z-\Z$, ce qui revient à demander que, pour tous poids $x,y$ de $U$, $x-y \in \Z$, ou encore à ce que $\mathrm{inf}(U)=\mathrm{diag}(\pm \frac{a_1}{2},\cdots,\pm \frac{a_n}{2})$, avec les $a_i$ entiers naturels impairs.

Un module de Harish-Chandra algébrique $U$ de $\mathbf{G}(\R)$ sera dit \emph{régulier} (resp. \emph{très régulier}) si tous ses poids sont distincts (resp. si $i\neq j \Rightarrow |a_i-a_j|>2$).
\end{defi}


Par abus de langage, on parlera encore des poids d'une représentation automorphe discrète $\pi$ (pour les poids de $\pi_\infty$), de son algébricité et de sa régularité.

\begin{lemme}\label{lemme_disc_cusp}
\emph{(Wallach, \cite{Wallach} Theorem 4.3)}

Soit $\pi$ une représentation automorphe discrète de $\mathbf{G}$.
On suppose que $\pi_\infty$ est une série discrète. Alors $\pi$ est en fait cuspidale.
\end{lemme}

On note $\Pi_{\rm alg}(\mathbf{G})$ l'ensemble des représentations automorphes discrètes algébriques de $\mathbf{G}$.
Si $\underline{a}=(a_1,\cdots,a_n)$ est un $n$-uplet d'entiers naturels impairs rangés par ordre décroissant, alors on note $\Pi_{\underline{a}}(\mathbf{G})$ l'ensemble des représentations automorphes discrètes algébriques dont les poids sont $(\pm \frac{a_1}{2},\cdots,\pm \frac{a_n}{2})$. 


\section{Paramètre d'Arthur global}\label{Paramètre d'Arthur global}
On conserve les notations du paragraphe précédent, en particulier $\mathbf{G}$ relève d'un des deux {\bfseries Cas} exposés ci-dessus.

\begin{defi}
Un paramètre d'Arthur global symplectique de dimension $2n$ est une somme formelle\footnote{Il serait équivalent de parler d'une collection de $(\pi_i,d_i)$.}
\[
\psi=\bigoplus_i \pi_i[d_i],
\]
où les $\pi_i$ sont des représentations automorphes cuspidales \emph{autoduales} de $\GL_{n_i}$ et les $d_i$ des entiers naturels strictement positifs tels que :

\emph{(GAP 1)} \; $\sum_i n_i d_i=2n$ ;

\emph{(GAP 2)} \; si $\pi_i$ est symplectique (resp. orthogonale, au sens du paragraphe \ref{Alternative symplectique-orthogonale}), alors $d_i$ est impair (resp. pair);

\emph{(GAP 3)} \; les couples $(\pi_i,d_i)$ sont distincts.

On dit que $\psi$ est \emph{générique} (resp. \emph{tempéré}) si tous les $d_i$ sont égaux à 1 (resp. et si, de plus, chaque représentation $\pi_i$ est tempérée à toutes les places).
\end{defi}

\emph{Remarque 1 :} Dans le formalisme général des paramètres d'Arthur globaux, il existe une condition supplémentaire de compatibilité des caractères centraux, à savoir $\prod_i \omega_{\pi_i}^{d_i}=1$, qui est ici automatiquement remplie à cause de (GAP 2).\medskip

\emph{Remarque 2 :} Si la conjecture de Ramanujan généralisée pour $\GL_n$ est vraie, alors tout paramètre générique est automatiquement tempéré. Nous utiliserons dans ce texte les deux notions distinctement. \newline

Il y a alors deux questions à examiner.
\begin{itemize}
\item Étant donnée $\pi$ une représentation automorphe discrète algébrique de $\mathbf{G}$, on veut lui associer un paramètre d'Arthur global symplectique, noté $\psi(\pi)$. Outre l'existence et l'unicité d'un tel paramètre (Théorème \ref{thm_Arthur_Taibi}), on cherche à savoir dans quelle mesure les propriétés de $\pi$ correspondent à des propriétés de son paramètre d'Arthur global $\psi(\pi)$ (§\ref{Paramètres locaux associés à un paramètre global} et §\ref{Invariants aux places finies}).
\item À l'inverse, construisant un paramètre d'Arthur global à partir de représentations automorphes cuspidales autoduales du groupe linéaire, on peut se demander s'il existe une représentation automorphe discrète de $\mathbf{G}$ dont c'est le paramètre : c'est l'objet de la formule de multiplicité d'Arthur (§\ref{Formule de multiplicité dans le cas tempéré}).
\end{itemize}


\section{Paramètres locaux associés à un paramètre global}\label{Paramètres locaux associés à un paramètre global}
Soit donc $\psi=\bigoplus_i \pi_i[d_i]$ un paramètre d'Arthur global symplectique de dimension $2n$.
Arthur lui associe des \emph{paramètres d'Arthur locaux} en utilisant la correspondance de Langlands locale pour les groupes linéaires. En effet, pour toute place $v$, $\pi_{i,v}$ admet un paramètre de Langlands $\Ll(\pi_{i,v})$ et on peut alors construire
\[
\psi_v=\bigoplus_i \Ll(\pi_{i,v}) \boxtimes \mathrm{Sym}^{d_i-1}
\]
représentation de $\WD_{\Q_v} \times \SL_2(\C)$, à valeurs dans $\GL_{2n}(\C)$ d'après (GAP 1).
La condition (GAP 2), combinée au fait que ${\rm Sym}^k$ est orthogonale (resp. symplectique) si $k$ est pair (resp. impair), et au Théorème \ref{thm_alt_sp_orth} nous dit que chacun de ces paramètres d'Arthur locaux se factorise par la représentation standard du groupe symplectique $\Sp_{2n}(\C) \rightarrow \GL_{2n}(\C)$. Ainsi, à un paramètre d'Arthur global $\psi$, symplectique de dimension $2n$, on a associé des paramètres locaux :
\begin{equation}\label{param_Arthur_locaux_sp}
\psi_v : \WD_{\Q_v} \times \SL_2(\C) \longrightarrow \Sp_{2n}(\C)
\end{equation}
de manière unique, à conjugaison près.\medskip



Dans le cas particulier où $\psi$ est générique, 
on remarque que chaque $\psi_v$ est alors un paramètre de Langlands pour $\SO_{2n+1}(\Q_v)$. On peut alors considérer le paquet \emph{de Langlands} $\Pi_{\psi_v}$, tel que défini au Théorème \ref{LLC_SO} quand $v=p$ est un nombre premier, ou tel qu'il est défini par Langlands \cite{Lgl_73} lorsque $v=\infty$. Dans le cas général, il est encore possible de définir des paquets locaux associés à chaque $\psi_v$ dits \emph{paquets d'Arthur locaux} et notés $\Pi_{\psi_v}(\mathbf{G})$. 
\begin{itemize}
\item Dans le {\bfseries Cas 1}, cela est fait par Arthur dans \cite{Art13}.
\item Dans le {\bfseries Cas 2}, la construction d'Arthur définit encore des paquets aux places finies (les groupes locaux $\mathbf{G}(\Q_p)$ sont encore déployés) ; à la place archimédienne, \cite{AMR} et \cite{Taibi_cpctmult} montrent que les paquets définis par Adams et Johnson \cite{Adams-Johnson} ont les propriétés requises.
\end{itemize}
Néanmoins le seul résultat dont nous aurons besoin dans ce qui suit concerne la place archimédienne.\medskip

{\bfseries Notation :}  Soit $E=\{a_1,\cdots,a_{n} \}$ un multi-ensemble de nombres complexes et soit $d$ un entier naturel non nul. On définit $E[d]$ le multi-ensemble des $x+k$ où $x$ parcourt $E$ et $k$ parcourt $\left\lbrace \frac{d-1}{2},\frac{d-3}{2},\cdots,\frac{3-d}{2},\frac{1-d}{2} \right\rbrace$. C'est un multi-ensemble à $nd$ éléments.

\begin{prop}\label{prop_AMRT} 
Soit $\mathbf{G}$ un groupe spécial orthogonal sur $\Q$ en $2n+1$ variables, relevant d'un des \emph{{\bfseries Cas}} du §\ref{Groupes étudiés et leurs représentations}.

Soit $\psi=\bigoplus_i \pi_i[d_i]$ un paramètre d'Arthur global symplectique de dimension $2n$. Alors tous les éléments du paquet $\Pi_{\psi_\infty}(\mathbf{G})$ ont le même caractère infinitésimal : c'est la classe de conjugaison semi-simple dans $\mathfrak{sp}_{2n}(\C)$ donnée par l'union des $P_i[d_i]$, où $P_i$ désigne le multi-ensemble des poids de $\pi_{i,\infty}$. 
\end{prop}

Par abus, on parlera encore des poids de $\psi$ (ou de $\psi_\infty$).

\begin{thm}\label{thm_Arthur_Taibi} \emph{(Arthur, Taïbi)}

Soit $\mathbf{G}$ un groupe spécial orthogonal sur $\Q$ en $2n+1$ variables, relevant d'un des \emph{{\bfseries Cas}} du §\ref{Groupes étudiés et leurs représentations}.
Soit $\pi$ une représentation automorphe discrète de $\mathbf{G}$. Alors il existe un unique paramètre d'Arthur global, symplectique de dimension $2n$, noté $\psi(\pi)$ vérifiant 
\begin{equation}\label{lien_rep_param_Arthur_paquets_locaux}
\pi={\bigotimes_v}' \pi_v, \text{ avec } \pi_v \in \Pi_{\psi_v}(\mathbf{G}) \text{ pour tout } v,
\end{equation}
où $\psi=\psi(\pi)$.

Si on suppose de plus $\pi$ algébrique très régulière, alors $\psi(\pi)=\bigoplus_i \pi_i[1]$ est générique et chaque $\pi_i$ est algébrique très régulière.
\end{thm}
\begin{proof}
Nous renvoyons à \cite{Art13}, Theorem 1.5.2 et \cite{Taibi_cpctmult}, Theorem 4.0.1 (voir aussi Remark 4.0.2).

Le deuxième point découle de la Proposition \ref{prop_AMRT} combinée au lemme combinatoire suivant.
\end{proof}
\begin{lemme}
Soit $I$ un ensemble d'indices. On se donne, pour chaque $i \in I$, un multi-ensemble de nombres complexes $E_i$ et un entier naturel non nul $d_i$. On pose $E=\bigcup_I E_i[d_i]$.

Si les éléments de $E$ sont distincts, alors chaque $E_i$ est un ensemble (sans multiplicité).

Si pour $x,y \in E, x \neq y \Rightarrow |x-y|>1$, alors chaque $E_i$ est un ensemble et tous les $d_i$ sont égaux à 1.
\end{lemme}

Retenons en particulier que chaque composante locale de $\pi$ est dans le paquet d'Arthur local associé à $\psi_v$ (et en dernier lieu, associé à $\psi$).
Plus précisément, selon le formalisme d'Arthur, étant donné $\mathbf{G}$ un groupe spécial orthogonal sur $\Q$, relevant d'un des {\bfseries Cas} du §\ref{Groupes étudiés et leurs représentations} on définit un paquet d'Arthur global $\Pi_\psi(\mathbf{G})$ associé à un paramètre d'Arthur global $\psi$ et défini comme suit.
\begin{equation*}
\Pi_\psi(\mathbf{G})=\left\lbrace {\bigotimes_v}' \pi_v \, \middle| {\begin{array}{l} \pi_v \in \Pi_{\psi_v}(\mathbf{G}) \text{ pour tout } v \\
\pi_v \text{ non ramifiée pour presque tout } v \end{array}} \right\rbrace,
\end{equation*}
ce que nous abrégeons par la notation (classique) $\Pi_\psi(\mathbf{G})={\bigotimes}'\Pi_{\psi_v}(\mathbf{G})$. Le Théorème \ref{thm_Arthur_Taibi} se reformule en disant (sous les mêmes hypothèses) qu'il existe un unique $\psi$ paramètre d'Arthur global, symplectique 
tel que $\pi \in \Pi_{\psi}(\mathbf{G})$. \newline

Nous rappelons maintenant un résultat crucial dû à Sug-Woo Shin \cite{Shin} et Ana Caraiani \cite{Caraiani}
\begin{thm}\label{thm_shin_caraiani}
Soit $\pi$ une représentation automorphe cuspidale algébrique autoduale régulière de $\GL_n$ sur $\Q$. Alors $\pi_v$ est tempérée pour toute place $v$ (\ie $\pi$ vérifie la conjecture de Ramanujan généralisée).
\end{thm}

\emph{Remarque :} Les énoncés sus-cités s'appliquent au sens strict à des représentations automorphes cuspidales algébriques régulières $\Pi$ de $\GL_n$ sur un corps CM $F$ telles que $\Pi^\vee \simeq \Pi^c$, où $\Pi^c$ désigne le conjugué extérieur de $\Pi$ par la conjugaison complexe du corps CM $F$.
On peut néanmoins se ramener à ce cas en utilisant le changement de base d'Arthur-Clozel, une méthode qui remonte au moins aux travaux de Harris et Taylor \cite{HT} (voir par exemple le §4.3 de l'article \cite{CHT} pour une discussion détaillée). Concrètement, on peut faire comme suit. Fixons $\pi$ comme dans l'énoncé et, étant donné $p$ un nombre premier, on va montrer que $\pi_p$ est tempérée. Pour cela, on choisit un nombre premier auxiliaire $q$ tel que $\pi_q$ est non ramifiée, puis un corps quadratique imaginaire $K$ tel que $p$ est décomposé dans $K$ et ramifié en $q$ (de tels $q$ et $K$ existent évidemment).
Notons $\chi$ le caractère de Hecke de $\Q$ d'ordre 2 correspondant à l'extension quadratique $K/\Q$. Par hypothèse sur $K$ et $q$, les représentations $\Pi_q \otimes \chi_q$ et $\Pi_q$ ne sont pas isomorphes (la première est ramifiée, pas la seconde). D'après le théorème principal de \cite{Arthur-Clozel}, on en déduit que si $\Pi$ désigne la représentation obtenue par changement de base de $\pi$ à $K$, alors $\Pi$ est cuspidale. Elle vérifie trivialement $\Pi^c \simeq \Pi$, et aussi $\Pi^\vee \simeq \Pi$ (à cause de cette mème hypothèse sur $\pi$). Ainsi, $\Pi$ est cohomologique à l'unique place archimédienne de $K$, car $\pi_\infty$ est algébrique régulière.
D'après \cite{Shin} et \cite{Caraiani}, $\Pi_v$ est tempérée pour tout $v$. Mais si $v$ est l'une des deux places de $K$ au dessus de $p$, alors $\GL_n(\Q_p)$ s'identifie à $\GL_n(K_v)$ et via cette identification on a $\Pi_v \simeq \pi_p$, si bien que $\pi_p$ est tempérée.

\begin{cor}\label{cor_shin_caraiani}
Soit $\mathbf{G}$ un groupe spécial orthogonal sur $\Q$, relevant d'un des \emph{{\bfseries Cas}} du §\ref{Groupes étudiés et leurs représentations}. Soit $\pi$ une représentation automorphe discrète de $\mathbf{G}$ algébrique très régulière. Alors $\pi_v$ est tempérée pour toute place $v$.
\end{cor}
\begin{proof}
C'est immédiat en utilisant le Théorème \ref{thm_Arthur_Taibi} (le paramètre $\psi(\pi)$ est alors non seulement générique, mais encore tempéré).
\end{proof}
Il est d'ailleurs remarquable que la seule connaissance de ce qu'il se passe à la place archimédienne affecte ce qu'il se passe à toutes les places. \newline

{\bfseries Notation :} Dans l'écriture d'un paramètre d'Arthur global, lorsque $d_i=1$, nous notons $\pi_i$ plutôt que $\pi_i[1]$. Ainsi un paramètre générique s'écrit $\psi=\bigoplus_i \pi_i$.

\section{Invariants aux places finies}\label{Invariants aux places finies}

Ici encore $\mathbf{G}$ relève d'un des deux {\bfseries Cas} du paragraphe \ref{Groupes étudiés et leurs représentations}. On remarque que, pour toute place finie $p$, $\mathbf{G}\otimes_\Q \qp$ s'identifie au $\qp$-groupe algébrique que nous avons noté $\SO_{2n+1}$ au paragraphe \ref{Notations}. En particulier, cela a bien un sens de parler de sous-groupe compact maximal hyperspécial (\cf §\ref{Le groupe K_0}), de sous-groupe épiparamodulaire (\cf §\ref{Le_groupe_J}) et de sous-groupe paramodulaire (\cf §\ref{J^+}).

\begin{prop}\label{prop_dim_invariants_so}
Soit $\pi$ une représentation automorphe discrète algébrique très régulière de $\mathbf{G}$ et soit $\psi(\pi)=\bigoplus_i \pi_i$ son paramètre d'Arthur global (qui existe bien et qui est générique d'après le Théorème \ref{thm_Arthur_Taibi}).

Soit $p$ un nombre premier et $K_p$ un sous-groupe compact ouvert de $\mathbf{G}(\qp)$ tels que $\pi_p^{K_p} \neq 0$.
\begin{itemize}
\item si $K_p$ est un sous-groupe compact maximal hyperspécial $\K_0(p)$, alors tous les $\pi_i$ sont non ramifiés en $p$.
\item si $K_p$ est un sous-groupe paramodulaire $\J^+(p)$, alors tous les $\pi_i$ sont non ramifiés en $p$, sauf éventuellement un, qui est, localement en $p$, d'exposant d'Artin égal à 1.
\end{itemize}
Mieux, on connaît la dimension de l'espace $\pi_p^{K_p}$ des invariants.
\begin{itemize}
\item Si tous les $\pi_i$ sont non ramifiés en $p$, alors $\dim \pi_p^{\K_0(p)}=\dim \pi_p^{(\J(p),+)}=\dim \pi_p^{(\J(p),-)}=1$.
\item Si l'un des $\pi_i$ exactement vérifie $\aar(\pi_{i,p})=1$, 
 disons $\pi_1$, alors $\dim \pi_p^{(\J(p),\eps)}=1$ et $\dim \pi_p^{\K_0(p)}=\dim \pi_p^{(\J(p),-\eps)}=0$, où $\eps=\eps_p(\pi_1)$ est le signe local en $p$ de la représentation $\pi_1$.
\end{itemize}
\end{prop}

\begin{proof}
On rappelle que, par le Corollaire \ref{cor_shin_caraiani}, $\pi_p$ est tempérée.

Supposons d'abord que $\pi_p$ a des invariants par $\K_0(p)$.

Le Corollaire \ref{cor_temp_dim_invariants_général} nous donne directement les dimensions d'invariants. De plus, combiné avec la Proposition \ref{prop_parametres_1}, il nous
 indique qu'alors $\Ll(\pi_p)=\bigoplus_j (\chi_j \oplus \chi_j^{-1})$ où les $\chi_j$, pour $j$ variant de 1 à $n$, sont des caractères non ramifiés de $\W_{\qp}$ (ou de $\qp^\times$). Or $\Ll(\pi_p)=\psi_p=\bigoplus_i \Ll(\pi_{i,p})$. Ainsi chaque $\Ll(\pi_{i,p})$ est lui-même une somme de caractères non ramifiés et est donc un paramètre non ramifié. Finalement tous les $\pi_i$ sont non ramifiés en $p$.


Supposons maintenant que $\pi_p$ a des invariants paramodulaires. Alors le même Corollaire \ref{cor_temp_dim_invariants_général}, combiné à la Proposition \ref{prop_parametres_p}
 nous indique deux possibilités. La première correspond à la situation précédente où tous les $\pi_i$ sont non ramifiés en $p$. La seconde correspond au paramètre de Langlands $\Ll(\pi_p)=\bigoplus_j (\chi_j \oplus \chi_j^{-1}) \oplus (\alpha \otimes U_2)$ où les $\chi_j$, pour $j$ variant de 1 à $n-1$, sont des caractères non ramifiés de $\W_{\qp}$ et où $\alpha \in \{\1,\eta\}$. Dans ce dernier cas, un et un seul des $\pi_i$, disons $\pi_1$, contient ce terme $\alpha \otimes U_2$ dans son paramètre local en $p$. Nous avons donc tous les $\pi_i$ non ramifiés, sauf un, qui est d'exposant d'Artin égal à 1 et donc de conducteur $pk$ avec $p \nmid k$. La représentation $\pi_{1,p}$ étant symplectique, on peut lui appliquer les résultats du paragraphe \ref{Lien avec les facteurs epsilon}. On sait alors déterminer $\alpha$ à partir du signe local $\eps_p(\pi_1)$ et on a bien (avec le Corollaire \ref{cor_temp_dim_invariants_général}) les dimensions d'invariants voulues.
\end{proof}

Nous avons immédiatement une réciproque.
\begin{defi}
Soit $\psi=\bigoplus \pi_i$ un paramètre d'Arthur global symplectique générique. On définit le conducteur de $\psi$, noté ${\rm N}(\psi)$ comme le produit des ${\rm N}(\pi_i)$.
\end{defi}

\begin{prop}\label{prop_def_param_sporadique}

Soit $\mathbf{G}$ un groupe spécial orthogonal sur $\Q$, relevant d'un des \emph{{\bfseries Cas}} du §\ref{Groupes étudiés et leurs représentations}.

Soit $\psi=\bigoplus \pi_i$ un paramètre d'Arthur global symplectique \emph{tempéré}. On suppose que le conducteur de $\psi$ est sans facteur carré. 

Soit $\pi$ un élément du paquet d'Arthur global $\Pi_\psi(\mathbf{G})$ (sans hypothèse d'automorphie, qui sera gérée par le Théorème \ref{thm_mult_Arthur}). Alors ${\rm N}(\pi)={\rm N}(\psi)$ et $\pi_p$ admet des $\J^+(p)$-invariants (resp. des $\K_0(p)$-invariants) non triviaux pour tout $p$ (resp. pour tout $p \nmid {\rm N}(\psi)$), ce que l'on peut traduire de façon globale en disant que 
$\pi$ admet des invariants par le sous-groupe compact ouvert $\mathsf{K}({\rm N}(\psi))$ de $\mathbf{G}(\aq)$ avec $\mathsf{K}({\rm N}(\psi))=\prod_p K_p$, où $K_p$ est paramodulaire (au sens du §\ref{J^+}) si $p \mid {\rm N}(\psi)$ et hyperspécial sinon (au sens du §\ref{Le groupe K_0}).
\end{prop}

\section{Formule de multiplicité dans le cas générique}\label{Formule de multiplicité dans le cas tempéré}
Soit $\psi=\bigoplus_i \pi_i[d_i]$ un paramètre d'Arthur global. On note $I$ l'ensemble d'indexation des $(\pi_i,d_i)$ et on pose $\mathcal{C}_\psi=\{\pm 1 \}^I$. Pour $i \in I$, on note $s_i$ l'élément de $\mathcal{C}_\psi$ tel que la $j$-ème composante de $s_i$ est $-1$ si $j=i$ et $1$ sinon. Les $s_i$ pour $i \in I$ engendrent alors le 2-groupe fini $\mathcal{C}_\psi$.\ps 


On peut définir un caractère $\eps_\psi : \mathcal{C}_\psi \rightarrow \{\pm 1 \}$ par :
\[
\eps_\psi(s_i)=\prod_{j\neq i} \eps(\pi_i \times \pi_j)^{\min(d_i,d_j)},
\]
où $\eps(\pi_i \times \pi_j)$ est le facteur epsilon de la paire $\{\pi_i,\pi_j\}$ qui intervient dans l'équation fonctionnelle \eqref{eq_fonctionnelle_paire_globale} de la fonction $\Lambda$ correspondante.

On s'intéresse ici au cas tempéré, où tous les $d_i$ sont égaux à 1, et où tous les $\pi_i$ sont donc symplectiques au sens du paragraphe \ref{Alternative symplectique-orthogonale}. La Proposition \ref{prop_facteur_eps_p_paire} nous indique qu'alors $\eps(\pi_i \times \pi_j)=1$ pour tous $i$ et $j$ si bien que le caractère $\eps_\psi$ est trivial.\medskip

Nous avons vu au paragraphe \ref{Paramètres locaux associés à un paramètre global} qu'au paramètre d'Arthur global $\psi$ et au groupe spécial orthogonal $\mathbf{G}$ (relevant d'un des {\bfseries Cas} du §\ref{Groupes étudiés et leurs représentations}) était associé un paquet d'Arthur global $\Pi_\psi(\mathbf{G})$. Nous allons maintenant associer un caractère à chaque élément de ce paquet.

Soit donc $\pi={\bigotimes}'\pi_v$ un élément de $\Pi_\psi(\mathbf{G})$. On a, pour tout $v$, $\pi_v \in \Pi_{\psi_v}(\mathbf{G})$ où $\Pi_{\psi_v}(\mathbf{G})$ est le paquet d'Arthur local associé au paramètre d'Arthur local $\psi_v$ -- lui-même déduit du paramètre d'Arthur global $\psi$, selon \eqref{param_Arthur_locaux_sp}. Dans le cas qui nous intéresse où tous les $d_i$ sont égaux à 1, on a vu au paragraphe \ref{Paramètres locaux associés à un paramètre global} que les paquets locaux étaient en fait des paquets de Langlands. On sait alors par le Théorème \ref{LLC_SO} (Arthur-M\oe{}glin) qu'à toute place finie $v=p$, les éléments du paquet de Langlands $\Pi_{\psi_p}$ sont indexés par les caractères du groupe (abélien fini) $\mathcal{S}_{\psi_p}$. Ainsi, à notre élément $\pi$ du paquet global $\Pi_\psi(\mathbf{G})$, est associé à chaque place finie un caractère $\chi_p$ de $\mathcal{S}_{\psi_p}$, correspondant à $\pi_p$. Les choses se passent de la même façon à la place archimédienne, et nous expliciterons à la Proposition \ref{prop_calcul_chi_infty} le calcul de ce caractère. Il est encore possible, dans le cas non générique d'associer de tels caractères à chaque place, ce que nous ne détaillons pas.\medskip

La dernière étape consiste à définir un morphisme \emph{de localisation} $\mathcal{C}_\psi \rightarrow \mathcal{S}_{\psi_v}$ pour toute place $v$, ce qui nous permettra de voir (par composition) chacun des caractères $\chi_v$ comme un caractère de $\mathcal{C}_\psi$. On a, selon \eqref{param_Arthur_locaux_sp}, $\psi_v : \WD_{\Q_v}  \rightarrow \Sp_{2n}(\C)$, qui se factorise en fait par le sous-groupe diagonal par blocs $\bigoplus_i \Sp_{n_id_i}(\C)$. On peut alors définir, pour chaque $i$, l'élément $\sigma_i=\bigoplus_{j\neq i} {\rm id}_{n_jd_j} \oplus (-{\rm id}_{n_id_i})$. Ces éléments centralisent manifestement $\psi_v$ et on peut alors définir le morphisme de localisation $\mathcal{C}_\psi \rightarrow \mathcal{S}_{\psi_v}$ qui envoie $s_i$ sur la classe de $\sigma_i$ dans $\mathcal{S}_{\psi_v}=\Cent(\psi_v)/Z(\Sp_{2n}(\C))$.

\begin{thm} \emph{Formule de multiplicité d'Arthur (\cite{Art13}, Theorem 1.5.2 et \cite{Taibi_cpctmult}, Theorem 4.0.1)} \label{thm_mult_Arthur}

Soit $\mathbf{G}$ un groupe spécial orthogonal sur $\Q$ en $2n+1$ variables, relevant d'un des \emph{{\bfseries Cas}} du §\ref{Groupes étudiés et leurs représentations}.

Soit $\psi=\bigoplus_i \pi_i[d_i]$ un paramètre d'Arthur global symplectique de dimension $2n$ et soit $\pi \in \Pi_{\psi}(\mathbf{G})$. Notons $\chi_v$ le caractère de $\mathcal{S}_{\psi_v}$ associé à $\pi_v$.

Alors $\pi$ est automorphe discrète (et de multiplicité un dans la décomposition du spectre discret) si, et seulement si les caractères $\eps_\psi$ et $\prod_v \chi_v$ sont égaux sur $\mathcal{C}_\psi$. 
On remarque que le premier caractère (trivial dans le cas générique) ne dépend que du paramètre global $\psi$, tandis que le second dépend de la représentation $\pi$.
\end{thm}

Dans la suite, on s'intéresse à la situation suivante.

\begin{prop}\label{prop_chi_p_trivial}
Soit $\mathbf{G}$ un groupe spécial orthogonal sur $\Q$, relevant d'un des \emph{{\bfseries Cas}} du §\ref{Groupes étudiés et leurs représentations}.
Soit $\psi=\bigoplus_i \pi_i$ un paramètre d'Arthur global symplectique tempéré de conducteur ${\rm N}(\psi)$ sans facteur carré.

Alors, pour tout $p$ premier, le paquet local $\Pi_{\psi_p}(\mathbf{G})$ est un singleton $\{\mu\}$ et le caractère $\chi_p$ de $\mathcal{S}_{\psi_p}$ associé à $\mu$ est trivial.
\end{prop}\label{prop_752}
\begin{proof}
Il suffit de constater que le paramètre (de Langlands) local, $\psi_p=\bigoplus_i \Ll(\pi_{i,p})$ vérifie $\aar(\psi_p)=\sum_i \aar(\pi_{i,p}) \leq 1$, pour tout $p$ (et est tempéré car les $\pi_{i,p}$ le sont par hypothèse sur $\psi$). Le résultat suit alors des Propositions \ref{prop_parametres_1} et \ref{prop_parametres_p}.
\end{proof}

Tout se passe donc comme dans le cas du conducteur 1 étudié dans \cite{Chen-Ren}, le seul caractère à considérer est celui de l'unique place archimédienne. En ce qui concerne cette unique place archimédienne, il faut être un peu plus vigilant car l'application d'indexation $\widehat{\mathcal{S}_{\psi_\infty}} \rightarrow \Pi_{\psi_\infty}$ du paquet de Langlands n'est plus nécessairement bijective (comme dans le cas $p$-adique), mais elle est toujours surjective. Nous allons donc déterminer ce qu'il se passe dans les cas qui nous intéressent.

\begin{prop}\label{prop_calcul_chi_infty}
Soit $\psi=\bigoplus_i \pi_i$ un paramètre d'Arthur global symplectique de dimension $2n$, générique. On suppose que $\psi_\infty$ est algébrique régulier, \ie est de la forme $I_{w_1} \oplus \cdots \oplus I_{w_n}$ avec $w_1 > \cdots > w_n>0$ entiers impairs. 
 On pose $J_i$ l'ensemble des $i\in\{1,\cdots,n\}$ tels que $\frac{w_i}{2}$ est un poids\footnote{$-\frac{w_i}{2}$ est alors également un poids de $\pi_i$ du fait qu'une représentation algébrique est supposée centrée.} de $\pi_i$. Ainsi les $J_i$ définissent une partition de $\{1,\cdots,n\}$. \ps 

On définit ci-dessous un $n$-uplet de signes $\underline{\eps}=(\eps_1,\cdots,\eps_n)$. La valeur de $\chi_\infty$ sur un générateur $s_i$ de $\mathcal{C}_\psi$ est alors égale au produit des signes $\eps_{j}$ pour $j$ parcourant $J_i$. \ps 
%

\begin{itemize}
\item Pour le groupe déployé $\SO_3\simeq \PGL_2$, le paquet $\Pi_{\psi_\infty}(\SO_3)$ est réduit à un seul élément et $\underline{\eps}=(+)$.
\item Pour le groupe déployé $\SO_5\simeq {\rm PGSp}_4$, le paquet $\Pi_{\psi_\infty}(\SO_5)$ contient deux éléments : la série discrète générique et la série discrète holomorphe. C'est la seconde qui nous intéresse, pour laquelle on a $\underline{\eps}=(-,-)$.
\item Pour le groupe $\SO_{2n+1}$ compact à l'infini (c'est le \emph{{\bfseries Cas 2}} du §\ref{Groupes étudiés et leurs représentations}), le paquet $\Pi_{\psi_\infty}(\SO_{2n+1})$ est réduit à un seul élément et $\underline{\eps}=((-1)^n,\cdots,+,-)$. 
\end{itemize}

\end{prop}
\begin{proof}
Dans tous les cas, il faut examiner le paquet (de Langlands) $\Pi_{\psi_\infty}(\mathbf{G})$ qui est un paquet de séries discrètes.
Les groupes $\SO_3$ et $\SO_5$ déployés sur $\Q$ relèvent bien sûr du {\bfseries Cas 1} du §\ref{Groupes étudiés et leurs représentations} et les résultats, classiques, sont exposés dans \cite{Chen-Ren}, §4.1 et §4.2. En ce qui concerne le {\bfseries Cas 2}, nous renvoyons à \cite{Chen-Lannes}, chapitre VIII, §5.5. 
\end{proof}

%
%
%
%

\section{Paramètres à considérer}
\begin{cor}\label{cor_761}
Soit $(\pi_i)$ une collection de représentations automorphes cuspidales algébriques régulières symplectiques de $\GL_{n_i}$ (les $n_i$ sont donc tous pairs) telle que :
\begin{itemize}
\item pour tout $i$, le conducteur de $\pi_i$ est sans facteur carré ;
\item les conducteurs des $\pi_i$ sont deux à deux premiers entre eux ;
\item les ensembles de poids des $\pi_i$ sont deux à deux disjoints.
\end{itemize}
On peut alors construire le paramètre d'Arthur global symplectique de dimension $2n=\sum_i n_i$, manifestement tempéré, $\psi=\bigoplus_i \pi_i$, et de conducteur ${\rm N}(\psi)$ sans facteur carré.

Soit $\mathbf{G}$ un groupe spécial orthogonal sur $\Q$ en $2n+1$ variables, relevant d'un des \emph{{\bfseries Cas}} du §\ref{Groupes étudiés et leurs représentations}.
Soit $\pi$ \emph{l'unique élément d'intérêt} du paquet d'Arthur global $\Pi_\psi(\mathbf{G})$ (les paquets locaux $p$-adiques sont des singletons selon la Proposition \ref{prop_chi_p_trivial} et le paquet local archimédien également, sauf dans le cas $\SO_5$ où l'on choisit pour $\pi_\infty$ la série discrète holomorphe). Alors $\pi$ est automorphe discrète (et de multiplicité 1) si, et seulement si on est dans un des cas suivants :
\begin{itemize}
\item pour le groupe déployé $\SO_3\simeq \PGL_2$, $\psi=\pi^w$ ;
\item pour le groupe déployé $\SO_5\simeq {\rm PGSp}_4$, $\psi=\pi^{w,v}$ avec $w>v$ ;
\item pour le groupe $\SO_7$ compact à l'infini,
\[
\psi= \begin{cases}
\pi^{w,v,u}\\
\pi^{w,u}\oplus\pi^v
\end{cases}
\]
avec $w>v>u$ ;
\item pour le groupe $\SO_9$ compact à l'infini, 
\[
\psi= \begin{cases}
\pi^{w,v,u,t}\\
\pi^{w,u}\oplus\pi^{v,t} \\
\pi^w \oplus \pi^{v,t} \oplus \pi^u \\
\pi^w \oplus \pi^{v,u,t} \\
\pi^{w,v,t} \oplus \pi^u
\end{cases}
\]
avec $w>v>u>t$ ;
\end{itemize}
où $\pi^{x_1,\cdots,x_r}$ désigne une représentation automorphe cuspidale algébrique symplectique de $\GL_{2r}$ arbitraire de poids $(\pm \frac{x_1}{2},\cdots, \pm \frac{x_r}{2})$.

De plus, $\pi$ admet des $\mathsf{K}({\rm N}(\psi))$-invariants (avec les notations de la Proposition-Définition \ref{prop_def_param_sporadique}).
\end{cor}

\begin{proof}
Il suffit de considérer tous les paramètres d'Arthur possibles. Hormis le cas $\SO_5$, le paquet d'Arthur global est un singleton, et dans le cas $\SO_5$ (où le paquet d'Arthur global est un doubleton), on sait encore à quel élément du paquet d'Arthur global on s'intéresse. Il suffit alors de calculer $\chi_\infty$ pour l'unique élément d'intérêt du paquet d'Arthur global selon la Proposition \ref{prop_calcul_chi_infty} et l'on retient alors, suivant le Théorème \ref{thm_mult_Arthur}, les éléments pour lesquels ce caractère $\chi_\infty$ est trivial. On retrouve les \emph{tempered cases} de \cite{Chen-Ren} §4.2 (pour $\SO_5$), §5.3 (pour $\SO_7$) et §6.2 (pour $\SO_9$).
\end{proof}


La combinaison des Théorème \ref{thm_Arthur_Taibi}, Proposition \ref{prop_dim_invariants_so} et Corollaire \ref{cor_761} nous donne alors le résultat suivant. 
\begin{cor}\label{cor_762}
Soit $\mathbf{G}$ un groupe spécial orthogonal sur $\Q$ en $2n+1$ variables, relevant d'un des \emph{{\bfseries Cas}} du §\ref{Groupes étudiés et leurs représentations}.
Soit $N$ un entier naturel non nul sans facteur carré.

Il y a {\bfseries autant} de représentations automorphes discrètes algébriques {\bfseries très régulières} de $\mathbf{G}$ de poids $\underline{w}=(\pm\frac{w_1}{2},\cdots,\pm\frac{w_n}{2})$ avec $w_1>\cdots>w_n$ et admettant des $\mathsf{K}(N)$-invariants que de paramètres d'Arthur globaux $\psi$ symplectiques de dimension $2n$ tempérés, tels que considérés au Corollaire \ref{cor_761} avec $\psi_\infty$ de poids $\underline{w}$ et ${\rm N}(\psi)=N$.

Il y a {\bfseries au moins autant} de représentations automorphes discrètes algébriques {\bfseries régulières} de $\mathbf{G}$  de poids $\underline{w}=(\pm\frac{w_1}{2},\cdots,\pm\frac{w_n}{2})$ avec $w_1>\cdots>w_n$ et admettant des $\mathsf{K}(N)$-invariants que de paramètres d'Arthur globaux $\psi$ symplectiques de dimension $2n$ tempérés, tels que considérés au Corollaire \ref{cor_761} avec $\psi_\infty$ de poids $\underline{w}$ et ${\rm N}(\psi)=N$.


\end{cor}
\begin{proof}
Dans le second cas, les paramètres du Corollaire \ref{cor_761} correspondent bien aux représentations automorphes discrètes recherchées, il se peut néanmoins qu'il existe des paramètres \emph{non génériques} correspondant à ces mêmes représentations.
\end{proof}

\newpage
\chapter{Lien avec des objets classiques}\label{Lien avec des objets classiques}

\section{Formes modulaires}
\subsection{Isomorphisme exceptionnel}\label{Isomorphisme exceptionnel}
On considère $\Z^3$ réalisé comme l'ensemble des matrices de ${\rm M}_2(\Z)$ de trace nulle. On a alors une forme quadratique naturelle donnée par l'opposé du déterminant $q: \begin{pmatrix}
a & b \\
c & -a
\end{pmatrix}  \mapsto a^2+bc$. La conjugaison permet de faire agir $\GL_2(\Z)$ sur ${\rm M}_2(\Z)^{{\rm tr}=0}$, en préservant le déterminant. On a donc un morphisme de groupes $\GL_2(\Z) \rightarrow \Oo({\rm M}_2(\Z)^{{\rm tr}=0},q)$, plus précisément $\PGL_2(\Z) \rightarrow \SO({\rm M}_2(\Z)^{{\rm tr}=0},q)$ et en fait mieux : un \emph{isomorphisme de schémas en groupes} : $\PGL_2 \simeq \SO_3$, où $\SO_3$ désigne donc le schéma en groupes déployé sur $\Z$, relevant du {\bf Cas 1} du §\ref{Groupes étudiés et leurs représentations}.

On remarque d'ailleurs que selon cet isomorphisme et, reprenant les notations du Chapitre \ref{Le groupe paramodulaire}, $\PGL_2(\zp)$ s'envoie sur un sous-groupe hyperspécial de $\SO_3(\qp)$ tandis que les sous-groupes d'Iwahori s'envoient sur des sous-groupes paramodulaires et les normalisateurs de ces sous-groupes d'Iwahori sur des sous-groupes épiparamodulaires. Les détails de cette \og traduction \fg{}  se trouvent dans \cite{Tsai-phd}, §6.1.

À cause de cet isomorphisme exceptionnel, on s'intéresse maintenant à $\PGL_2$. Une représentation automorphe de $\PGL_2$ est une représentation automorphe de $\GL_2$ de caractère central trivial. Les représentations automorphes discrètes non cuspidales de $\PGL_2$ sont bien connues, elles sont de dimension 1 et nous ne les considérons pas.

La théorie d'Arthur pour les représentations automorphes cuspidales est tautologique puisque, par l'isomorphisme exceptionnel $\PGL_2 \simeq \SO_3$, les représentations considérées sont bien symplectiques au sens du Théorème \ref{thm_alt_sp_orth}. De façon plus élémentaire, la condition de caractère central trivial nous indique que les paramètres locaux sont tous à valeurs dans $\SL_2(\C)=\Sp_2(\C)$. Ainsi le paramètre d'Arthur d'une représentation automorphe cuspidale de $\SO_3$ est simplement son image par l'isomorphisme exceptionnel. On remarque enfin que la multiplicité 1 (au sens de la décomposition \eqref{A_disc_SO}) est simplement le résultat classique de multiplicité 1 pour $\GL_2$ (en cohérence avec le Théorème \ref{thm_mult_Arthur}).

\subsection{En niveau $\Gamma_0(N)$ avec $N$ sans facteur carré}\label{En niveau Gamma_0_N}

Le lien entre formes modulaires pour des sous-groupes de congruence de $\SL_2(\Z)$ et formes automorphes adéliques pour $\GL_2$ est classique. Nous renvoyons à \cite{Gelbart_book} pour ce lien ainsi que pour les définitions usuelles suivantes.\ps 

%
%
Si $\Gamma$ est un sous-groupe de $\SL_2(\Z)$, on note ${\rm S}_{k}(\Gamma)$ le $\C$-espace vectoriel des formes modulaires paraboliques (ou cuspidales) de poids (modulaire) $k$ pour le groupe $\Gamma$.
On s'intéresse ici à une seule classe de sous-groupes.
Soit $N$ un entier naturel non nul, que l'on suppose sans facteur carré. On pose
\begin{equation*}
\Gamma_0(N)=\left\lbrace \begin{pmatrix}
a & b \\
c & d
\end{pmatrix} \in \SL_2(\Z) \; \middle| \; c \in N\Z \right\rbrace,
\end{equation*}
et l'on remarque que $\Gamma_0(1)=\SL_2(\Z)$ (\emph{full modular group}).

Il est possible d'associer à chaque forme modulaire parabolique normalisée de ${\rm S}_k(\Gamma_0(N))$ une représentation automorphe cuspidale algébrique $\pi$ de $\GL_2$, de caractère central trivial (donc de $\PGL_2$ si l'on veut), de conducteur divisant $N$. Plus précisément,\smallskip
\begin{itemize}
\item $\pi_\infty$ est isomorphe à la série discrète holomorphe ${\rm D}_{k}$ \emph{de poids $k$} (au sens classique), algébrique de poids $\{\pm\frac{k-1}{2}\}$, et on a $\Ll({\rm D}_{k})=I_{k-1}$ ; \smallskip
\item $\pi_p$ admet des invariants par $\PGL_2(\zp)$ pour $p\nmid N$ ;\smallskip
\item $\pi_p$ admet des invariants par un sous-groupe d'Iwahori ${\rm I}_p$ pour $p\mid N$.
\end{itemize}
\smallskip

On peut d'ailleurs regrouper les deux dernières propriétés en une seule propriété adélique : $\pi_f$ admet des invariants non triviaux par le sous-groupe compact ouvert $K_f(N)=\prod_p K_p(N) $ avec $K_p(N)=\PGL_2(\zp)$ si $p\nmid N$ et$K_p(N)={\rm I}_p$ si $p \mid N$, où ${\rm I}_p$ désigne un sous-groupe d'Iwahori fixé de $\GL_2(\Q_p)$.\medskip


En fixant une droite de plus haut poids dans ${\rm D}_k$, on a finalement un isomorphisme d'espaces vectoriels :
\begin{equation}\label{isom_S_k_Gamma}
{\rm S}_{k}(\Gamma_0(N))	\simeq \bigoplus\limits_{\substack{\pi_\infty \simeq \mathrm{D}_{k} \\ \pi_f^{K_f(N)}\neq 0}} \m(\pi) \pi_f^{K_f(N)}
\simeq \bigoplus\limits_{\substack{\pi_\infty \simeq \mathrm{D}_{k} \\ \pi_f^{K_f(N)}\neq 0}}  \left(\bigotimes_p \pi_p^{K_p(N)} \right).
\end{equation}


(On utilise le fait que $\m(\pi)=1$).
La dimension de l'espace de droite dépend alors du nombre de $\pi$ intervenant dans la somme et de la dimension de $\bigotimes_p \pi_p^{K_p(N)}$, \ie de celle de chaque $\pi_p^{K_p(N)}$.

\begin{prop}\label{prop_811}
Soit $p$ un nombre premier, et soit $\varpi$ une représentation admissible irréductible tempérée de $\PGL_2(\qp)$.

On suppose que $\varpi^{\PGL_2(\zp)}\neq \{0 \}$. Alors $\varpi$ est non ramifiée et $\dim(\varpi^{\PGL_2(\zp)})=1$.

On suppose que $\varpi^{{\rm I}(\zp)}\neq \{0 \}$. Alors
\begin{itemize}
\item soit $\varpi$ est non ramifiée et $\dim(\varpi^{{\rm I}(\zp)})=2$,
\item soit $\varpi=\alpha \St_{\PGL_2}$ avec $\alpha \in \{\1,\eta\}$ et $\dim(\varpi^{{\rm I}(\zp)})=1$.
\end{itemize}
\end{prop}
\begin{proof}
Ce sont des résultats très classiques, que l'on peut également voir comme la \emph{traduction}, via l'isomorphisme exceptionnel $\SO_3 \simeq \PGL_2$ des résultats du §\ref{Représentations tempérées} (voir aussi le Corollaire \ref{cor_temp_dim_invariants_général}). 
\end{proof}

Pour simplifier la discussion, et puisque ce sont les seuls cas que nous aurons à considérer en pratique, restreignons-nous maintenant à $N=1$ et $N=p$.

\begin{defi}
Soit $N\in\{1\}\cup \mathscr{P}$ où $\mathscr{P}$ désigne l'ensemble des nombres premiers. Soit $m$ un entier naturel.

On note $\Pi_{m}^N(\PGL_2)$ l'ensemble des représentations automorphes cuspidales algébriques $\pi$ de $\PGL_2$ de poids $\{\pm \frac{m}{2}\}$ (et donc telles que $\Ll(\pi_\infty)= I_m$), de conducteur $N$.
\end{defi}

\begin{cor}\label{cor_dim_s_k}
Pour tout entier $k\geq 2$, on a :
\begin{align*}
&\dim {\rm S}_{k}(\SL_2(\Z))=\left|\Pi_{k-1}^1(\PGL_2)\right| ;\\
&\dim {\rm S}_{k}(\Gamma_0(p))=2\left|\Pi_{k-1}^1(\PGL_2)\right| + \left|\Pi_{k-1}^p(\PGL_2)\right|.
\end{align*}
\end{cor}
\begin{proof}
On considère simplement les dimensions dans \eqref{isom_S_k_Gamma}.
\end{proof}
Dans le cas de ${\rm S}_{k}(\Gamma_0(p))$, il existe une théorie des formes anciennes et nouvelles, chaque type correspondant à un des termes de la somme.
On retrouve ainsi la décomposition 
\[
\dim {\rm S}_{k}(\Gamma_0(p))=\dim {\rm S}_{k}^{\rm old}(\Gamma_0(p))+\dim {\rm S}_{k}^{\rm new}(\Gamma_0(p)),
\]
où $\dim {\rm S}_{k}^{\rm old}(\Gamma_0(p))=2 \dim {\rm S}_{k}(\SL_2(\Z))$.

Puisque l'on peut, via la théorie classique, accéder directement à $\dim {\rm S}_{k}^{\rm new}(\Gamma_0(p))$, nous connaîtrons ainsi $\left|\Pi_{k-1}^p(\PGL_2)\right|$. 

\subsubsection{Signe}\label{Signe}
Selon la théorie classique toujours, les formes modulaires nouvelles de niveau $\Gamma_0(p)$ admettent un signe d'Atkin-Lehner. On dispose alors des espaces ${\rm S}_{k}^{\mathrm{new,+}}(\Gamma_0(p))$ et ${\rm S}_{k}^{\mathrm{new,-}}(\Gamma_0(p))$ Ce signe est lié à l'involution du même nom et donc à l'action du normalisateur d'un sous-groupe d'Iwahori. Selon la traduction donnée au §\ref{Isomorphisme exceptionnel}, il s'agit en fait de considérer l'action de $\J$ sur l'espace $\varpi^{\J^+}$ des $\J^+$-invariants d'une représentation $\varpi$ de $\SO_3(\qp)$. Les résultats sont alors donnés par le Lemme \ref{lemme_inv_param_ind_St} et le Théorème \ref{thm_J_variants_eps}, et nous avons le \og dictionnaire \fg suivant :

\begin{center}
\begin{tabular}{ | *{4}{c|}}
   \hline
   \multirow{2}{*}{Caractère $\alpha$}  & Signe & Variants & Facteur epsilon \\
   & d'Atkin-Lehner & paramodulaires & local \\
   \hline
   $\1$ & $-$ & $(\J,-)$ & $-1$ \\
   \hline
   $\eta$ & $+$ & $(\J,+)$ & $1$ \\
   \hline
\end{tabular}
\end{center}

On peut alors définir $\Pi_{k-1}^{2,+}(\PGL_2)$ comme l'ensemble des représentations cuspidales de $\PGL_2$ de conducteur 2, de poids motivique $k-1$ et de facteur epsilon local en 2 égal à 1. Le \emph{dictionnaire} nous donne alors $\dim {\rm S}_{k}^{\mathrm{new,+}}(\Gamma_0(2))=\left|\Pi_{k-1}^{2,+}(\PGL_2)\right|$. On fait de même pour $\Pi_{k-1}^{2,-}(\PGL_2)$.

On a d'ailleurs, pour tout $m$ :
\[
\Pi_{m}^{2}(\PGL_2)=\Pi_{m}^{2,+}(\PGL_2) \coprod \Pi_{m}^{2,-}(\PGL_2).
\]

\section{Formes modulaires de Siegel}
\subsection{Un autre isomorphisme exceptionnel}\label{Un autre somorphisme exceptionnel}
On dispose encore d'un \emph{isomorphisme de schémas en groupes} : $\PGSp_4 \simeq \SO_5$, où $\SO_5$ désigne le schéma en groupes déployé sur $\Z$, relevant du {\bf Cas 1} du §\ref{Groupes étudiés et leurs représentations}. Nous ne détaillons pas ici cet isomorphisme exceptionnel, mais renvoyons à \cite{Tsai-phd}, §6.2 (qui le démontre au niveau des $k$-points, où $k$ est un corps).

On remarque d'ailleurs que selon cet isomorphisme et, reprenant les notations du Chapitre \ref{Le groupe paramodulaire}, $\PGSp_4(\zp)$ s'envoie sur un sous-groupe hyperspécial de $\SO_5(\qp)$ tandis que les sous-groupes paramodulaires de niveau $p$, $\K(p)$ (au sens de \cite{RS_book}) s'envoient sur des sous-groupes paramodulaires $\J^+(p)$ tels qu'introduits au paragraphe \ref{J^+} et les sur-groupes de ces sous-groupes paramodulaires (contenant \og l'élément d'Atkin-Lehner \fg{}) sur des groupes épiparamodulaires $\J(p)$ tels qu'introduits au paragraphe \ref{Le_groupe_J}. Les détails de cette \og traduction \fg{}  se trouvent dans \cite{Tsai-phd}, §6.2.

À cause de cet isomorphisme exceptionnel, on s'intéresse maintenant à $\PGSp_4$. Une représentation automorphe de $\PGSp_4$ est une représentation automorphe de ${\rm GSp}_4$ de caractère central trivial. On a une notion de conducteur qui vient de celle que l'on a développée pour $\SO_5$ (localement) à la Définition \ref{defi_aar_rep_tauto}.

\subsection{En niveau $\Gamma^{\rm para}(N)$ avec $N$ sans facteur carré}\label{En niveau Gamma_para_N}

Le lien entre formes modulaires de Siegel de degré 2 pour le groupe modulaire $\Sp_4(\Z)$ et formes automorphes adéliques pour ${\rm PGSp}_4$ est classique et est détaillé par exemple dans \cite{Asgari-Schmidt}. Nous considérons aussi le cas du sous-groupe paramodulaire (au sens de \cite{RS-art}, \ie vu comme sous-groupe de $\Sp_4(\Q)$) de niveau $N$ (sans facteur carré), $\Gamma^{\rm para}(N)$ pour lequel les mêmes arguments que \cite{Asgari-Schmidt} s'appliquent (le point-clé étant qu'on a encore un caractère central trivial).\ps

Soit donc $N$ un entier naturel non nul, que l'on suppose sans facteur carré. On pose :
\[
\Gamma^{\rm para}(N)= \begin{pmatrix}
\Z 	& N\Z	& \Z	& \Z \\
\Z 	& \Z		& \Z	& N^{-1}\Z \\
\Z 	& N\Z	& \Z	& \Z \\
N\Z	& N\Z	& N\Z	& \Z \\
\end{pmatrix} \cap \Sp_4(\Q)
\]
et l'on remarque que $\Gamma^{\rm para}(1)=\Sp_4(\Z)$.

On choisit deux entiers naturels $j,k$ avec $j$ pair et $k\geq 3$. On note alors ${\rm S}_{j,k}(\Gamma^{\rm para}(N))$ le $\C$-espace vectoriel des formes modulaires de Siegel paraboliques (ou cuspidales) pour le groupe $\Gamma^{\rm para}(N)$ à coefficients dans la représentation ${\rm Sym}^j \otimes \det^k$ de $\GL_2(\C)$.

Il est possible d'associer à chaque forme modulaire de Siegel parabolique propre pour les opérateurs de Hecke de ${\rm S}_{j,k}(\Gamma^{\rm para}(N))$ une représentation automorphe cuspidale $\pi$ de $\GSp_4$, de caractère central trivial (donc de $\PGSp_4$ si l'on veut) et de conducteur divisant $N$. 
Plus précisément, \smallskip
\begin{itemize}
\item $\pi_\infty$ est isomorphe à la série discrète holomorphe ${\rm D}_{j,k}$, vérifiant $\Ll({\rm D}_{j,k})=I_{j+2k-3}\oplus I_{j+1}$ ; \smallskip 
\item $\pi_p$ admet des invariants par $\PGSp_4(\zp)$ pour $p\nmid N$ ; \smallskip
\item $\pi_p$ admet des invariants par le sous-groupe paramodulaire $\K(p)$ (au sens de \cite{RS_book}) pour $p\mid N$.
\end{itemize}
\smallskip

On peut d'ailleurs regrouper les deux dernières propriétés en une seule propriété adélique : $\pi_f$ admet des invariants non triviaux par le sous-groupe compact ouvert $K_f(N)=\prod_p K_p(N) $ avec $K_p(N)=\PGSp_4(\zp)$ si $p\nmid N$ et $K_p(N)=\K(p)$ si $p \mid N$.

On a, de façon analogue à \eqref{isom_S_k_Gamma} :
\begin{equation}\label{isom_S_jk_Gamma_para}
{\rm S}_{j,k}(\Gamma^{\rm para}(N))	\simeq \bigoplus\limits_{\substack{\pi_\infty \simeq \mathrm{D}_{j,k} \\ \pi_f^{K_f(N)}\neq 0}} \m(\pi) \left(\bigotimes_p \pi_p^{K_p(N)} \right),
\end{equation}
et la dimension de l'espace de droite dépend encore du nombre de $\pi$ intervenant dans la somme et de la dimension de $\bigotimes_p \pi_p^{K_p(N)}$, \ie de celle de chaque $\pi_p^{K_p(N)}$ (on a toujours $\m(\pi)=1$ par les travaux d'Arthur pour le groupe $\SO_5$).

Pour simplifier la discussion, et puisque ce sont les seuls cas que nous aurons à considérer en pratique, restreignons-nous ici encore à $N=1$ et $N=p$.

\begin{defi}
Soit $N\in\{1\}\cup \mathscr{P}$ où $\mathscr{P}$ désigne l'ensemble des nombres premiers. Soient $w\geq v$ deux entiers naturels.

On note $\Pi_{w,v}^N(\PGSp_4)$ l'ensemble des représentations automorphes cuspidales $\pi$ de $\PGSp_4$ telles que $\pi_\infty$ est la série discrète holomorphe vérifiant $\Ll(\pi_\infty)= I_w\oplus I_v$), de conducteur $N$.
\end{defi}

\begin{cor}\label{cor_dim_s_k_j}
Soient $j,k$ deux entiers naturels avec $j$ pair et $k\geq 3$. Alors
$$\dim {\rm S}_{j,k}(\Sp_4(\Z))=\left|\Pi_{j+2k-3,\,j+1}^1(\PGSp_4)\right|.$$
Si l'on suppose de plus que $j\neq 0$ et $k>3$, alors 
$$\dim {\rm S}_{j,k}(\Gamma^{\rm para}(p))=2\left|\Pi_{j+2k-3,\,j+1}^1(\PGSp_4)\right| + \left|\Pi_{j+2k-3,\,j+1}^p(\PGSp_4)\right|.$$
\end{cor}
\begin{proof}
On considère simplement les dimensions dans \eqref{isom_S_jk_Gamma_para}, dont on garde les notations.

Pour la première égalité, on utilise alors que $\pi_\ell^{\K_0(\ell)}$ est de dimension 1 puisque $\pi_\ell$ est non ramifiée pour tout $\ell$ (sans autre hypothèse sur $\pi_\ell$).

Pour la seconde égalité, on a encore $\pi_\ell^{\K_0(\ell)}$ de dimension 1 pour $\ell \neq p$. L'hypothèse supplémentaire sur $j$ et $k$ entraîne que $\pi_\infty$ est très régulière et donc en particulier que $\pi_p$ est tempérée par le Corollaire \ref{cor_shin_caraiani} ($\pi_\ell$ aussi pour $\ell \neq p$ d'ailleurs). Le résultat découle alors des calculs de dimensions d'invariants de la Première Partie (voir le Corollaire \ref{cor_temp_dim_invariants_général}).
\end{proof}

\subsubsection{Signe}\label{Signe_4}
Dans le cas des formes modulaires de Siegel de degré 2 et de niveau paramodulaire $\Gamma^{\rm para}(p)$, Roberts et Schmidt ont développé (et démontré) une théorie des formes anciennes et nouvelles, ces dernières étant dotées d'un \emph{signe d'Atkin-Lehner} étendant ainsi la théorie des formes modulaires classiques (voir \cite{RS_book}).

Le \emph{dictionnaire} du paragraphe \ref{Signe} est d'ailleurs toujours valable.\bigskip

Nous utilisons pour calculer $\dim {\rm S}_{j,k}(\Gamma^{\rm para}(p))$ les formules de dimensions de \cite{Ibu-Kita}\footnote{En fait, sauf pour la Conjecture \ref{thm_2_w21}, on aurait pu utiliser les formules de \cite{Ibu_seul} qui correspondent à des poids $w>v$ avec $v$ impair et $w>v+4$. Ce ne sont néanmoins pas celles que nous avons implémentées.}. Ces formules ne permettent hélas pas d'accéder à ce signe d'Atkin-Lehner.
Nous disposons cependant de deux manières (plus une troisième conjecturale) de \emph{capturer} ce signe.

Le paramètre d'Arthur d'une représentation automorphe cuspidale très régulière $\pi_{\PGSp_4}$ de $\PGSp_4$ de niveau paramodulaire $\Gamma^{\rm para}(p)$ est une représentation automorphe cuspidale $\pi_{\GL_4}$ (très régulière autoduale symplectique) de $\GL_4$ (c'est le Corollaire \ref{cor_761}) de conducteur 1 ou $p$. En particulier, le signe local en $p$ des deux représentations automorphes cuspidales en question coïncide.
\begin{itemize}
\item Ce signe local intervient dans la fonction $L$ de $\pi_{\GL_4}$. Nous verrons au Chapitre \ref{La_formule explicite de Riemann-Weil-Mestre} comment exploiter cette information, ce qui nous permettra, connaissant l'existence d'une certaine $\pi_{\PGSp_4}$ (nouvelle), de tester le signe de $\pi_{\GL_4}$ par la formule explicite, et parfois d'en interdire un des deux, ce qui conclura (voir un exemple au paragraphe \ref{En_poids_motivique 15}).
\item Une telle $\pi_{\GL_4}$ intervient également dans le paramètre d'Arthur de certaines représentations automorphes cuspidales très régulières de $\SO_7$ ou $\SO_9$ ({\bf Cas 2}), toujours selon le Corollaire \ref{cor_761}. Nous disposons pour ces deux groupes de formules de dimensions \emph{signées} (voir la Proposition \ref{dim_inv_2_so}) qui nous permettent alors également de déterminer le signe local de $\pi_{\GL_4}$ (et partant celui de $\pi_{\PGSp_4}$ si l'on veut). Un exemple se trouve au paragraphe \ref{En_poids_motivique 19}.
\item Selon les conjectures globales d'Arthur dans un {\bf Cas 3} présentées au paragraphe \ref{Heuristique pour m=5}, et une formule de multiplicité d'Arthur également conjecturale, on peut obtenir un analogue du Corollaire \ref{cor_761} pour le $\SO_5$ du {\bf Cas 3}. Cela nous permet alors, en utilisant les tables \cite{tabledimSOchenevier} d'accéder encore (par simple lecture) au signe local de $\pi_{\GL_4}$.\smallskip
\end{itemize}

Dans tous les cas où l'on obtient ainsi un signe, ces différentes méthodes donnent des résultats cohérents.

%
%
%
%

\section{Formes automorphes de conducteur $2$ pour ${\rm SO}_{2n+1}$ sur $\Q$}\label{Formes automorphes de conducteur 2_GC}
L'ensemble de cette section (jusqu'au paragraphe \ref{Interprétations automorphes} inclus) correspond à des résultats inédits de Gaëtan Chenevier. Comme nous les utilisons de façon cruciale pour notre travail, ce dernier nous a aimablement fourni la présentation et les démonstrations correspondantes, qui se voient donc incluses dans la présente thèse.
\subsection{Rappels sur les réseaux unimodulaires impairs}

	On fixe un entier $m \geq 1$ 
et on se place dans l'espace euclidien standard $\R^m$ de dimension $m$, 
dont on note $x \cdot y\,=\,\sum_i x_iy_i$ le produit scalaire. 
On rappelle qu'un r\'eseau $L \subset \R^m$ est dit {\it entier} 
si l'on a $x \cdot y \in \Z$ pour tous $x,y \in L$. 
Le {\it d\'eterminant} de $L$ est 
le d\'eterminant d'une matrice de Gram (quelconque) d'une $\Z$-base de $L$, 
c'est aussi le carr\'e du covolume de $L$ 
et on le note $\det L$. 
On dit que le r\'eseau $L \subset \R^m$ est unimodulaire 
si on a $\det L=1$. 
Enfin, on dit que $L$ est {\it pair} 
si on a $x \cdot x \in 2\Z$ pour tout $x \in L$, 
impair sinon. 
Par exemple, le {\it r\'eseau standard} 
\begin{equation} \label{remdefIm}  {\rm I}_m := \Z^m \end{equation}
est trivialement entier, unimodulaire et impair.
\ps\ps

\begin{defi}\label{deflm}  
On note $\mathcal{L}_m$ l'ensemble des r\'eseaux entiers unimodulaires impairs de $\R^m$. 
Il est muni d'une action naturelle du groupe orthogonal euclidien ${\rm O}(m)$, 
et on pose $\mathcal{X}_m = {\rm O}(m) \backslash \mathcal{L}_m$.
\end{defi}

La th\'eorie de la r\'eduction (Hermite, Minkowski) montre que $|\mathcal{X}_m|< \infty$ : 
il n'y a qu'un nombre fini de classes d'isom\'etrie 
de r\'eseaux unimodulaires impairs de rang $m$. 
Par les travaux de nombreux auteurs (dont Witt, Kneser, Niemeier, 
Conway-Sloane, Borcherds), leur classification est connue pour $m \leq 25$ : 
nous renvoyons \`a  \cite[Ch. 16]{conwaysloane} pour un \'etat de l'art. 
Dans cette th\`ese, le r\'esultat de classification suivant, 
d\'emontr\'e par exemple dans \cite[\S 106.F]{OMeara}, suffira. 
Il sera commode d'introduire d'abord le r\'eseau 
\begin{equation} \label{defDm} 
{\rm D}_m = \{ (x_1,\cdots,x_m) \in \Z^m\, \, |\, \, \sum_i x_i \equiv 0 \bmod 2\}, 
\end{equation}
\noindent qui est entier, pair et de d\'eterminant $4$, 
ainsi que le r\'eseau unimodulaire pair
\begin{equation}  
{\rm E}_8 = {\rm D}_8 + \Z e, \, \, \, \, e = \frac{1}{2}(1,1,\cdots,1).
\end{equation}

\begin{prop}\label{classXn9} Pour $m \leq 8$, tout r\'eseau entier unimodulaire impair de $\R^m$ 
est isom\'etrique \`a ${\rm I}_m$. 
Autrement dit, on a $\mathcal{X}_m = \{ {\rm I}_m\}$. 
Pour $m=9$, il existe \`a isom\'etrie pr\`es 
exactement deux r\'eseaux unimodulaires impairs de rang $m$ : 
le r\'eseau ${\rm I}_9$ et le r\'eseau $\Z \times {\rm E}_8$ 
{\rm (}aussi not\'e ${\rm I}_1 \oplus {\rm E}_8${\rm )}.
\end{prop}

\subsection{Genre des réseaux unimodulaires impairs de $\R^m$}

Soit $A$ un anneau commutatif. 
Un {\it $A$-module bilin\'eaire} de rang $m$ 
est la donn\'ee d'un $A$-module libre $M$ de rang $n$ 
et d'une forme bilin\'eaire sym\'etrique $M \times M \rightarrow A$, $(x,y) \mapsto x\cdot y$. 
Lorsque $2$ est inversible dans $A$, 
il est \'equivalent de se donner un $A$-module quadratique au sens du §\ref{Rappels_sur_les_FQ}, 
mais pas en g\'en\'eral. 
De mani\`ere similaire au cas quadratique, 
un $A$-module bilin\'eaire $M$ poss\`ede un d\'eterminant $\det M \in A/ (A^{\times})^2$. 
De plus, $M$ est dit non d\'eg\'en\'er\'e si on a $\det M \in A^\times$ 
ou, ce qui revient au m\^eme, si l'application naturelle $M \rightarrow {\rm Hom}_A(M,A)$, 
$x \mapsto (y \mapsto x \cdot y)$, est bijective.
Il y a une notion \'evidente de somme directe orthogonale de deux modules bilin\'eaires, 
d'extension des scalaires d'un module bilin\'eaire, 
d'isom\'etrie entre modules bilin\'eaires, 
etc.\ps

Tout r\'eseau entier $L$ de $\R^m$, 
muni du produit scalaire de $\R^m$, 
est un $\Z$-module bilin\'eaire, 
qui est non d\'eg\'en\'er\'e si, et seulement si $L \in \mathcal{L}_m$. 
Par exemple, 
deux \'el\'ements de $\mathcal{L}_m$ sont isom\'etriques (au sens bilin\'eaire) 
si, et seulement s'ils ont m\^eme classe dans $\mathcal{X}_m$. 
Un fait remarquable est que pour $p$ premier et $L \in \mathcal{L}_m$, 
le $\Z_p$-module bilin\'eaire $L \otimes \Z_p$ ne d\'epend pas du choix de $L$  
(on dit aussi que les \'el\'ements de $\mathcal{L}_m$ forment un unique {\it genre}, 
au sens de Gauss):

\begin{prop} \label{genusIm}
Pour tout $L$ dans $\mathcal{L}_m$ 
et tout nombre premier $p$, les $\Z_p$-modules bilin\'eaires $L \otimes \Z_p$ et ${\rm I}_m \otimes \Z_p$ 
sont isom\'etriques. 
\end{prop}

C'est un r\'esultat bien connu : 
voir par exemple  \cite[Chap. 15]{conwaysloane} et la Table 15.4 {\it loc. cit.} 
Les propositions classiques suivantes d\'ecrivent  la structure de ${\rm I}_m \otimes \Z_p$ 
pour $m$ impair, 
qui est le cas dont nous aurons besoin.

\begin{prop} \label{ImQp} 
On suppose $m=2n+1$ impair et l'on fixe $p$ un nombre premier.
L'espace quadratique ${\rm I}_m \otimes \Q_p$ est d'indice de Witt $n$, 
sauf pour $p=2$ et $m \equiv \pm 3 \bmod 8$ auquel cas il est d'indice de Witt $n-1$.
\end{prop}

Nous donnons une d\'emonstration pour la commodit\'e du lecteur. 
Pour un anneau commutatif $A$, 
on notera ${\rm H}(A)$ le $A$-module bilin\'eaire non d\'eg\'en\'er\'e $A e  \oplus A f$ 
avec $e\cdot e=f\cdot f=0$ et $e \cdot f=1$ (plan hyperbolique sur $A$). 
Pour $a \in A$ on notera aussi $\langle a \rangle$ 
le $A$-module bilin\'eaire $A e$ avec $e\cdot e=a$. 
\ps\ps

\begin{proof} Nous allons utiliser la classification des formes quadratiques sur $\Q_p$, 
pour laquelle nous renvoyons au Cours d'arithm\'etique de Serre \cite[Chap. IV]{serre}.
Consid\'erons le $\Q_p$-module bilin\'eaire\footnote{Pour coller aux conventions de Serre, 
il faut munir $V_p$ de la forme quadratique $q(x) = x\cdot x$, 
contrairement \`a notre convention du §\ref{Rappels_sur_les_FQ} 
qui serait $q(x) = \frac{x\cdot x}{2}$. 
C'est de toute fa\c{c}on sans cons\'equence car $q$ et $2q$ ont m\^eme indice de Witt.}  
$V_p = {\rm I}_m \otimes \Q_p$.  
Par d\'efinition, on a $\det V_p = 1$ (dans $\qp/(\qp^\times) ^2$)
car $V_p$ poss\`ede une base orthogonale constitu\'ee de vecteurs $v$ tels que $v\cdot v=1$. 
Pour la m\^eme raison, l'invariant de Hasse de $V_p$, not\'e $\epsilon(V_p)$, 
est aussi trivialement égal à $1$. 
Posons $$W_p = {\rm H}(\Q_p)^{\oplus n} \overset{\perp}{\oplus} \langle (-1)^n \rangle.$$ 
Il s'agit de montrer que  $W_p \simeq V_p$ 
si, et seulement si $p \neq 2$ ou $m \equiv \pm 1 \bmod 8$. 
En effet, une forme quadratique non d\'eg\'en\'er\'ee de dimension impaire $2n+1$ sur $\Q_p$
est d'indice de Witt $n$ ou $n-1$. \ps

On a $\det W_p = 1 = \det V_p$ 
donc, par la classification des formes quadratiques sur $\Q_p$, 
il est \'equivalent de demander $\epsilon(W_p)=\epsilon(V_p)=1$. 
Mais on a ${\rm H}(\Q_p) \simeq \langle 1 \rangle \overset{\perp}{\oplus} \langle -1 \rangle $, 
et donc (en notant $\left(\frac{a,b}{\Q_p}\right)$ le symbole de Hilbert de $\{a,b\}$ sur $\Q_p$) 
  $$\epsilon(W_p) =  
  \left(\frac{-1,-1}{\Q_p}\right)^{\frac{n(n-1)}{2}} \left(\frac{-1,(-1)^n}{\Q_p} \right)^n.$$ 
  Le signe $\left(\frac{-1,(-1)^r}{\Q_p} \right)$ vaut $1$, 
  sauf pour $p=2$ et $r \equiv 1 \bmod 2$. 
 On a donc $\epsilon(W_p)=1$ pour $p>2$, $\epsilon(W_2)=1$ pour $n \equiv 0, 3 \bmod 4$, et $\epsilon(W_2)=-1$ sinon.
\end{proof}

Notons $P$ le $\Z_2$-module bilin\'eaire $\Z_2 e \oplus \Z_2 f$ 
avec $e\cdot e=f\cdot f=2$ et $f\cdot f=1$. 

\begin{prop}\label{genreIm} On suppose $m=2n+1$ impair et l'on fixe $p$ premier. 
\begin{itemize}
\item[(i)] Pour $p\neq 2$ ou $m\equiv \pm 1 \bmod 8$, 
le $\Z_p$-module bilin\'eaire ${\rm I}_m\otimes \Z_p$ est isom\'etrique \`a 
${\rm H}(\Z_p)^{\oplus n}  \overset{\perp}{\oplus} \langle (-1)^n \rangle$ ;\ps
\item[(ii)] Pour $m \equiv \pm 3 \bmod 8$ 
on a ${\rm I}_m\otimes \Z_2 \simeq {\rm H}(\Z_2)^{\oplus n-1} \overset{\perp}{\oplus} P  \overset{\perp}{\oplus} \langle 3(-1)^{n-1}\rangle$.
\end{itemize}
\end{prop}

Faisons une remarque avant d'entamer la d\'emonstration. 
Parall\`element au cas des r\'eseaux de $\R^m$, 
un $\Z_2$-module bilin\'eaire $M$ sera dit {\it pair} 
si on a $x \cdot x \equiv 0 \bmod 2$ pour tout $x$ dans $M$, 
et {\it impair} sinon. 
Par exemple, ${\rm I}_m \otimes \Z_2$ est impair, 
et ${\rm H}(\Z_2)$ et $P$ sont pairs.  
On v\'erifie ais\'ement qu'\`a isom\'etrie pr\`es,
$P$ est l'unique $\Z_2$-module bilin\'eaire pair, non d\'eg\'en\'er\'e,
de rang $2$, et sans vecteur isotrope non nul (voir aussi \cite[\S 93]{OMeara}).

\begin{proof} Soit $M$ un $\Z_p$-module bilin\'eaire non d\'eg\'en\'er\'e, 
suppos\'e pair si $p=2$. 
On suppose que $M$ poss\`ede un vecteur isotrope non nul. 
Observons d'abord que l'on a 
une isom\'etrie $M \simeq {\rm H}(\Z_p) \overset{\perp}{\oplus} M'$, 
o\`u $M'$ est un $\Z_p$-module bilin\'eaire 
(n\'ecessairement non d\'eg\'en\'er\'e, pair si $p=2$). 
En effet, soit $v$ dans $M-\{0\}$ tel que $v\cdot v=0$. 
On peut supposer $v$ primitif. 
Comme $M$ est non d\'eg\'en\'er\'e, 
il existe un vecteur $w \in M$ v\'erifiant $v \cdot w=1$. 
Quitte \`a remplacer $w$ par $w- \frac{w\cdot w}{2} v$, 
ce qui est loisible car $M$ est pair si $p=2$, 
et $2 \in \Z_p^\times$ sinon, 
on peut supposer $w\cdot w=0$. 
Ainsi, $\Z_p v \oplus \Z_p w$ est isom\'etrique \`a ${\rm H}(\Z_p)$, 
qui est non d\'eg\'en\'er\'e sur $\Z_p$, 
ce qui conclut.  \ps 

Montrons maintenant la Proposition. 
Si $p$ est impair, on pose $M={\rm I}_m \otimes \Z_p$. 
Si $p=2$, 
on constate que l'\'el\'ement $e=(1,1,\cdots.,1)$ de ${\rm I}_m$ v\'erifie 
$e\cdot e = m \in \Z_2^\times$, 
de sorte que l'on a 
\begin{equation} \label{ecarIm} {\rm I}_m \otimes \Z_2  = \Z_2 e \overset{\perp}{\oplus} M \end{equation}
en posant $M = e^\perp$. 
De plus, $M$ est pair car $(\sum_i x_i)^2 \equiv \sum_i x_i \bmod 2$. 
Dans tous les cas relevant de l'assertion (i), d'apr\`es la Proposition \ref{ImQp}, 
$M \otimes \Q_p$ est donc d'indice de Witt $n$. 
En appliquant $n$ fois successivement 
l'observation du premier paragraphe de cette d\'emonstration, 
on en d\'eduit $M \simeq {\rm H}(\Z_p)^n \overset{\perp}{\oplus} N$, 
avec n\'ecessairement $N=\{0\}$ si $p=2$, 
et $N \simeq \langle (-1)^n \rangle$ si $p\neq 2$. 
Cela conclut le (i). Dans le cas (ii), l'indice de Witt de $M \otimes \Q_2$ est $n-1$,
et le même argument montre $M \simeq {\rm H}(\Z_2)^{n-1} \overset{\perp}{\oplus} N$
avec $N$ pair, de rang $2$, non d\'eg\'en\'er\'e, sans vecteur isotrope.
On a donc $N \simeq P$, et $\det N = \det P =3$.
\end{proof}

\subsection{Le réseau pair d'un $L \in \mathcal{L}_m$, avec $m$ impair, et le groupe ${\rm O}(L)^+$}

Si $L$ est un r\'eseau entier quelconque de $\R^m$, 
on constate que $L$ admet un plus grand sous-r\'eseau pair
d\'efini par $L_{\rm even} = \{ x \in L \, \, |\, \, x\cdot x \equiv 0 \bmod 2\}$.
Le point est que $L \rightarrow \Z/2\Z, \,x \mapsto x\cdot x \bmod 2$ est {\it lin\'eaire}.
Par exemple, on a 
\begin{equation}\label{ImevenDm} ({\rm I}_m)_{\rm even} = {\rm D}_m.\end{equation}
On a toujours $|L/L_{\rm even}|\leq 2$ et $L = L_{\rm even}$ si et seulement si $L$ est pair. 
Ces d\'efinitions et remarques s'appliquent \emph{verbatim}
au cas o\`u $L$ est remplac\'e par un $\Z_2$-module bilin\'eaire. 
On a en tous les cas $ L \otimes \Z_p = L_{\rm even} \otimes \Z_p$ pour $p\neq 2$, et 
$$ (L \otimes \Z_2)_{\rm even} = L_{\rm even} \otimes \Z_2.$$

\begin{prop-def}\label{defKL} Soit $L \in \mathcal{L}_m$ avec $m$ impair. 
On a $L/L_{\rm even} \simeq \Z/2\Z$ et
$L_{\rm even} \otimes \Z_2 \simeq ({\rm I}_m \otimes \Z_2)_{\rm even} \simeq {\rm D}_m \otimes \Z_2$. 
En particulier, $L_{\rm even} \otimes \Z/4\Z$ est d\'eg\'en\'er\'e, et son noyau, que l'on notera $\kappa(L)$,
est isomorphe \`a $\Z/4\Z$.\end{prop-def}

\begin{proof} La premi\`ere assertion d\'ecoule de la Proposition \ref{genusIm} et de \eqref{ImevenDm} (et vaut pour tout $m$).
On peut donc supposer $L={\rm I}_m$ et $L_{\rm even}={\rm D}_m$. Posons $e=(1,\cdots,1) \in {\rm I}_m$. On a $e\cdot e=m \in \Z_2^\times$ et on a d\'ej\`a vu en \eqref{ecarIm}
que l'on a ${\rm I}_m \otimes \Z_2 = \Z_2 e \overset{\perp}{\oplus} M$ avec $M$ pair et non d\'eg\'en\'er\'e. En prenant les parties paires, on en d\'eduit
${\rm D}_m \otimes \Z_2 \,=\, \Z_2 \,2e\, \overset{\perp}{\oplus} M$ : le noyau de ${\rm D}_m \otimes \Z/4\Z$ 
est donc libre de rang $1$ sur $\Z/4\Z$, engendr\'e par la classe de $2e=(2,2,\cdots,2)$.
\end{proof}

Pour $L \in \mathcal{L}_m$, on notera ${\rm O}(L)$ le groupe (fini) d'isom\'etries de $L$, 
i.e. des \'el\'ements $g \in {\rm O}(m)$ v\'erifiant $g(L)=L$. 
On note ${\rm SO}(L)$ le sous-groupe des $g \in {\rm O}(L)$ avec $\det g=1$. 
Dans le cas $m$ impair on a ${\rm O}(L) = \{ \pm 1\} \times {\rm SO}(L)$. 
Toujours dans ce cas, et conform\'ement \`a la d\'efinition ci-dessus,
tout \'el\'ement de ${\rm O}(L)$ induit un automorphisme du groupe $\kappa(L) \simeq \Z/4\Z$. 
Comme ${\rm Aut}(\kappa(L))=\{\pm 1\}$ (canoniquement) on a d\'efini un morphisme de groupes
\begin{equation} \chi_L : {\rm O}(L) \rightarrow \{ \pm 1\}.\end{equation}
Le morphisme $\chi_L$ est toujours surjectif car on a $\chi_L(-1)=-1$.

\begin{defi} \label{defetaL}
Soit $L \in \mathcal{L}_m$ avec $m$ impair. 
On note ${\rm O}(L)^+$  le noyau du morphisme $\chi_L$ ci-dessus,
et l'on pose ${\rm SO}(L)^+= {\rm O}(L)^+ \cap {\rm SO}(L)$.\par
On dira que $L$ est {\rm ambivalent} si on a ${\rm SO}(L)/{\rm SO}(L)^+ \simeq \Z/2\Z$ 
ou, ce qui revient au m\^eme, si $\chi_L({\rm SO}(L)) = \{ \pm 1\}$.
\end{defi}

Donnons deux exemples:

\begin{ex} \label{exemplen79}
\begin{itemize}
\item[(i)] Posons $L={\rm I}_m$. 
Les \'el\'ements $x \in L$ tels que $x\cdot x=1$ sont les $\pm \eps_i$, $i=1,\cdots,m$, 
o\`u $(\eps_1,\cdots,\eps_m)$ d\'esigne la base canonique de $\R^m$. 
Le groupe ${\rm O}({\rm I}_m)$ coïncide donc avec le groupe
$\{ \pm 1\}^m \rtimes \got{S}_m$ des permutations sign\'ees des $\eps_i$.  
Supposons $m$ impair et fixons $a\sigma \in {\rm O}({\rm I}_m)$ avec $a=(a_i) \in \{ \pm 1\}^m$ et $\sigma \in \got{S}_m$. Montrons que $\chi_L( a \sigma) = \prod_{i=1}^m a_i$. 
En effet, le groupe $\kappa({\rm I}_m)$ est engendr\'e par $e=(2,2,\cdots,2) \in {\rm D}_m \otimes \Z/4\Z$. 
D'une part, ce vecteur est manifestement invariant par $\got{S}_m$ et, d'autre part,  si $s_i \in {\rm O}(L)$ d\'esigne la réflexion orthogonale par rapport \`a $\eps_i$, 
on a $s_i(e)-e=-4 \eps_i \notin 4 {\rm D}_m$, ce qui conclut. 
En particulier, ${\rm O}({\rm I}_m)^+$ 
est le sous-groupe de ${\rm O}({\rm I}_m)$ 
constitu\'e des permutations sign\'ees ayant un nombre pair de $-1$, 
et donc ${\rm I}_m$ est ambivalent pour $m\neq 1$.\ps
\item[(ii)] Posons $L\,=\,{\rm I}_1\, \oplus\, {\rm E}_8 \,=\, \Z \,\eps_1 \,\overset{\perp}{\oplus} \,{\rm E}_8$ 
(voir Proposition \ref{classXn9}). Les seuls \'el\'ements $x$ de $L$ avec $x\cdot x=1$ sont $\pm \eps_1$ car ${\rm E}_8$ est pair, 
on a donc 
\begin{equation} \label{idI1E8} {\rm O}(L) = {\rm O}({\rm I}_1) \times {\rm O}({\rm E}_8) = \{\pm 1\} \times {\rm O}({\rm E}_8).\end{equation}
On a aussi $L_{\rm even} \,=\, \Z \,2 \eps_1 \,\oplus\, {\rm E}_8$, 
et donc $\kappa(L) \,=\, \Z/4\Z\, 2\eps_1$. 
Ainsi, le morphisme $\chi_L : {\rm O}(L) \rightarrow \{\pm 1\}$ est 
la premi\`ere projection dans \eqref{idI1E8},
et donc ${\rm O}(L)^+=1 \times {\rm O}({\rm E}_8)$. 
Comme $\det {\rm O}({\rm E}_8) = \{ \pm 1\}$ (r\'eflexions orthogonales par rapports aux ``racines''), 
${\rm I}_1 \oplus {\rm E}_8$ est ambivalent. 
\end{itemize}
\end{ex}
 
Dans le reste de ce paragraphe nous allons d\'egager les quelques \'enonc\'es 
qui nous permettront de faire le lien avec les groupes paramodulaire et épiparamodulaire
\'etudi\'es au Chapitre \ref{Le groupe paramodulaire}. Supposant toujours $m=2n+1$ impair, 
la Proposition \ref{genreIm} entra\^ine
\begin{equation} 
\label{2adicDm}
{
{\rm D}_m \otimes \Z_2 = ({\rm I}_m \otimes \Z_2)_{\rm even} \simeq
\left\{ \begin{array}{ll}  {\rm H}(\Z_2)^n \overset{\perp}{\oplus} \langle (-1)^{n} 4 \rangle & {\textrm{si}}\,\, m \equiv \pm 1 \bmod 8, 
\\  \\ {\rm H}(\Z_2)^{n-1} \overset{\perp}{\oplus} P \overset{\perp}{\oplus} \langle (-1)^{n-1} 12  \rangle  & {\textrm{si}}\,\, m \equiv \pm 1 \bmod 8. \end{array} \right.
}
\end{equation}
\ps 
En effet, il suffit d'observer que ${\rm H}(\Z_2)$ et $P$ sont pairs, 
et que l'on a $\langle a \rangle_{\rm even} \simeq \langle 4a \rangle$ pour $a \in \Z_2^\times$. 
On a montr\'e le :

\begin{cor} \label{cortype2}
\label{type12} Soient $L \in \mathcal{L}_m$, avec $m \equiv \pm 1 \bmod 8$, et $p$ un nombre premier.
Au sens\footnote{Au sens strict, pour se ramener au $\Q_p$-espace quadratique $V_n$ fix\'e au Chapitre \ref{Le groupe paramodulaire}, 
dont le d\'eterminant est $(-1)^n$, 
il faut multiplier la forme $x\cdot y$ sur $L \otimes \Q_p$ par $(-1)^n$.} 
de la Définition \ref{Defi_reseau_type_a}, le r\'eseau $L_{\rm even} \otimes \Z_p$ de $L \otimes \Q_p$ 
est de type $1$ si $p$ est impair, de type $2$ si $p=2$.
\end{cor}

Pour $L \in \mathcal{L}_m$, le morphisme $\chi_L$ de la D\'efinition \ref{defetaL} 
s'\'etend naturellement en un morphisme de groupes
\begin{equation}\label{locetaL} \chi_L : {\rm SO}(L \otimes \Z_2) \rightarrow {\rm Aut}(\kappa(L))=\{\pm 1\}. \end{equation}
On note $\nu_L : {\rm SO}(L \otimes \Q_2) \rightarrow \Q_2^\times/(\Q_2^{\times})^2$ la norme spinorielle, $\val_2 : \Q_2^\times \rightarrow \Z$ la valuation $2$-adique, et on pose $\eta_L(g) =  (-1)^{\val_2(\nu_L(g))}$ pour $g \in {\rm SO}(L \otimes \Q_2)$.

\begin{prop} \label{chiLetaL} Soit $L \in \mathcal{L}_m$ avec $m$ impair. On a 
$\chi_L(g)=\eta_L(g)$ pour tout $g \in {\rm SO}(L \otimes \Z_2)$,  et \eqref{locetaL} est surjective si $m>1$.
\end{prop}

\begin{proof}  
On voit $M= L_{\rm even}\otimes \Z_2$ comme un $\Z_2$-module quadratique. 
Le noyau de $M \otimes \Z/4\Z$ est $\kappa(L)$, on a donc
une suite exacte ${\rm O}(M)$-stable
$0 \rightarrow K(L) \rightarrow M \otimes \Z/4\Z \rightarrow M' \rightarrow 0$ 
avec $M'$ un $\Z/4\Z$-module quadratique non d\'eg\'en\'er\'e. Pour $g \in {\rm O}(M)$, 
on note $g'$ l'\'el\'ement de ${\rm O}(M')$ induit par $g$, et $\bar{g}$ l'\'el\'ement de ${\rm O}(M' \otimes \Z/2\Z)$
d\'eduit par r\'eduction modulo $2$ de $g'$. \ps
Montrons d'abord la formule pour $\chi_L$ sous l'hypoth\`ese $m \equiv \pm 1 \bmod 8$. 
D'apr\`es le Corollaire \ref{cortype2}, ${\rm SO}(M)$ est un sous-groupe \'epiparamodulaire de
${\rm SO}(M \otimes \Q_2)$.
D'apr\`es le Théorème \ref{thm_J+_noyau_norme_spin}, on a donc $\eta_L(g)=\det_{\rm DD}(\bar{g})$ 
(o\`u $\det_{\rm DD}$ est le d\'eterminant de Dickson-Dieudonn\'e de $\bar{g}$, voir §\ref{Rappels_sur_les_FQ}).
Mais $M'$ \'etant non d\'eg\'en\'er\'e sur $\Z/4\Z$, la ``constance'' du d\'eterminant de Dickson-Dieudonn\'e
entra\^ine $\det_{\rm DD}(\bar{g})=\det_{\rm DD}(g')$. D'autre part, on a 
$\det_{\rm DD}(g')\,\equiv\, \det g' \,\, \bmod 4$, d'apr\`es 
le diagramme central de \cite[Chap. 2]{Chen-Lannes} p. 26. On conclut car $1=\det g \equiv  \chi_L(g) \det g' \bmod 4$. \ps
Supposons maintenant $m$ impair quelconque. On peut supposer $L \otimes \Z_2 = {\rm I}_m \otimes \Z_2$.
Soit $r$ un entier pair tel que $m+r \equiv \pm 1 \bmod 8$. On pose
$$L' \,:=\, {\rm I}_m \otimes \Z_2 \overset{\perp}{\oplus} {\rm I}_r \otimes \Z_2 = {\rm I}_{m+r} \otimes \Z_2.$$
Tout \'el\'ement $g \in {\rm SO}(L \otimes \Z_2)$ s'\'etend  en 
un \'el\'ement $\widetilde{g}$ de ${\rm SO}(L')$ par l'identit\'e sur ${\rm I}_r \otimes \Z_2$.
On a \'evidemment $\eta_L(g) = \eta_{L'}(\widetilde{g})$. 
D'apr\`es le cas pr\'ec\'edent, on sait que $\eta_{L'}(\widetilde{g}) = \chi_{L'}(\widetilde{g})$.
Il ne reste qu'\`a montrer $\chi_L(g)= \chi_{L'}(\widetilde{g})$.
Pour cela, on consid\'ere l'\'el\'ement $e_s=(1,1,\cdots,1)$ de ${\rm I}_s$ ;
on rappelle que $2e_s$ engendre $\kappa({\rm I}_s)$ pour $s$ impair.
On a $e_{m+r}=e_m + e_r$. 
Par d\'efinition de $\chi_L$, on a $\widetilde{g}( 2 e_m) - \chi_L(g) \,2 e_m \in 4{\rm D}_m \otimes \Z_2$.
De mani\`ere triviale, on a aussi 
$\widetilde{g}( 2 e_r) - \chi_L(g) \,2e_r \,= \, (1-\chi_L(g)) \,2 e_r \in 4 {\rm D}_r \otimes \Z_2$, car $r$ est pair. 
On a bien montr\'e $\widetilde{g}( 2 e_{m+r} )\, \equiv \,\chi_L(g) \,2e_{m+r} \,\bmod \,4 {\rm D}_{m+r} \otimes \Z_2$. \ps 

	Pour la surjectivit\'e de \eqref{locetaL} pour $m>1$, consid\'erons par exemple une d\'ecomposition 
$L \,=\, \Z_2 e\, \overset{\perp}{\oplus} (\Z_2 u \oplus \Z_2 v) \overset{\perp}{\oplus} N$ avec $e\cdot e \in \Z_2^\times$, 
$u\cdot u=v\cdot v$, $u\cdot v=1$, et $N$ réseau pair, donn\'ee par la Proposition \ref{genreIm}. On d\'efinit un \'el\'ement 
$g \in {\rm SO}(L)$ en posant $g(e)=-e$, $g(u)=v$, $g(v)=u$ et $g_{|N}= {\rm id}$.
L'\'egalit\'e d\'ej\`a vue $\kappa(L) \,= \,\Z/4\Z \,2e$ entra\^ine $\chi_L(g)=-1$.
\end{proof}


\begin{lemme}
\label{OdmegalOim}
Si $m$ est impair, l'application $g \mapsto g_{|{\rm D}_m \otimes \Z_2}$ induit un isomorphisme de groupes 
${\rm O}({\rm I}_m \otimes \Z_2) \isomo {\rm O}({\rm D}_m \otimes \Z_2)$.
\end{lemme}

\begin{proof} Il est clair que toute isom\'etrie d'un $\Z_2$-module bilin\'eaire $L$
pr\'eserve $L_{\rm even}$, de sorte que le morphisme de l'\'enonc\'e
est bien d\'efini. Il est injectif car on a $L \subset L[1/2]=L_{\rm even}[1/2]$. 
Reste \`a justifier sa surjectivit\'e.
Rappelons que si $M$ est un sous-$\Z_2$-module de $V={\rm I}_m \otimes \Q_2$,
son {\it dual} est le $\Z_2$-module $M^\sharp = \{ x \in V\, \, |\, \, x\cdot M \subset \Z_2\}$.
Il est stable par toute isom\'etrie de $M$.
Dans le cas $L= {\rm I}_m \otimes \Z_2$, on constate $L=L^\sharp$ et, pour $M=L_{\rm even} = {\rm D}_m \otimes \Z_2$, 
 $$M^\sharp \,=\, L \,+ \,\Z_2 \,e/2 \, = \,M \,+\, \Z_2 \,\eps_1 \,+ \,\Z_2 \,e/2,$$
 avec $e=(1,\cdots,1)$ et $\eps_1=(1,0,\cdots,0)$. On en d\'eduit
\begin{equation} \label{residuMeven} M_{\rm even}^\sharp/M_{\rm even} = \langle e/2 \rangle \simeq \Z/4\Z \end{equation}
car on a $e-\eps_1 \in M$ pour $m$ impair. 
Tout $\Z_2$-sous-module bilin\'eaire de $V$ contenant $M$ \'etant inclus dans $M^\sharp$, 
on en d\'eduit que le seul tel sous-module $N$ avec $M \subsetneq N$ est $N=L$,
qui est donc n\'ecessairement stable par ${\rm O}(M)$, 
ce qui est la surjectivit\'e cherch\'ee.
\end{proof}

\emph{Remarque :} La formule \eqref{residuMeven} montre que, pour $L  \in \mathcal{L}_m$ avec $m$ impair, 
le morphisme $\chi_L$ coïncide avec le morphisme naturel ${\rm O}(L \otimes \Z_2) \rightarrow 
{\rm Aut}(M^\sharp/M)=\{ \pm 1\}$, o\`u $M = L_{\rm even}$.

\subsection{Réseaux unimodulaires impairs marqués et fonctions de réseaux}
\label{reseauxunimmarques}

Supposons encore $m$ impair. Soit $L \in \mathcal{L}_m$.
Nous appellerons {\it marquage} de $L$ 
la donn\'ee d'un g\'en\'erateur $\kappa$ du groupe $\kappa(L) \simeq \Z/4\Z$ (Proposition \ref{defKL}). 
Chaque $L$ a exactement deux marquages, 
de la forme $\kappa$ et $-\kappa$. 
Tout \'el\'ement $g \in {\rm O}(m)$ induit une isom\'etrie $L \rightarrow g(L)$, 
et donc un isomorphisme de groupes ab\'eliens $\kappa(L) \isomo \kappa(g(L))$ encore not\'e $g$ par abus.

\newcommand{\tlm}{\widetilde{\mathcal{L}_m}}
\begin{prop-def} \label{deftlm} Pour $m$ impair, on note $\tlm$ l'ensemble des couples $(L,\kappa)$ avec $L \in \mathcal{L}_m$
et $\kappa$ un marquage de $L$. Il est muni d'une involution naturelle $\tau$, d\'efinie par $\tau(L,\kappa)=(L,-\kappa)$, ainsi que d'une action naturelle de ${\rm O}(m)$ qui commute \`a $\tau$, d\'efinie par $g(L,\kappa)=(g(L),g(\kappa))$ pour $g \in {\rm O}(m)$. Par construction, le stabilisateur de $(L,\kappa) \in \tlm$ dans ${\rm O}(m)$ est le groupe ${\rm O}(L)^+$ de la D\'efinition \ref{defetaL}.
\end{prop-def}

La d\'efinition suivante est importante.
On entendra par {\it repr\'esentation} de ${\rm SO}(m)$ 
une repr\'esentation continue sur un espace vectoriel complexe de dimension finie
que nous noterons en g\'en\'eral $U$. Pour $\Gamma$ un sous-groupe de ${\rm SO}(m)$ on pose 
$U^\Gamma = \{ u \in U \, |\, \gamma u=u,\,\, \forall \gamma \in \Gamma\}$ ({\it $\Gamma$-invariants de $U$}).

\begin{prop-def} 
\label{definvariantsunim}
Si $U$ est une repr\'esentation de ${\rm SO}(m)$ ($m$ impair) on pose
$${\rm S}_U(m) = \{  f : \mathcal{L}_m \rightarrow U\,\,|\,\, f(g(L)) = g.f(L),\,\,\,\forall L \in \mathcal{L}_m, \,\,\forall g \in {\rm SO}(m)\},$$
$$\widetilde{{\rm S}}_U(m) = \{  f : \tlm \rightarrow U\,\,|\,\, f(g(L,\kappa)) = g.f(L,\kappa),\,\,\,\forall (L,\kappa) \in \widetilde{\mathcal{L}_m}, \,\,\forall g \in {\rm SO}(m)\}.$$
Ce sont des $\C$-espaces vectoriels de dimension finie.
\end{prop-def}

L'assertion sera \'evidente sur les formules \eqref{SU+concret} et \eqref{SUconcret} plus bas. 
Observons que $\tau$ induit une involution de 
$\widetilde{{\rm S}}_U(m)$, d\'efinie par $\tau(f)(L,\kappa)=f(L,-\kappa)$, 
de sorte que l'on a une d\'ecomposition en somme directe
\begin{equation}\label{Spm} \widetilde{{\rm S}}_U(m) = {\rm S}_U(m)^+ \oplus {\rm S}_U(m)^- \end{equation}
o\`u ${\rm S}_U(m)^\eps$ est le sous-espace des fonctions $f \in \widetilde{{\rm S}}_U(m)$ v\'erifiant 
$f(L,\eps \kappa)=\eps f(L,\kappa)$. On a donc une identification \'evidente 
\begin{equation}\label{SUplus} {\rm S}_U(m) = {\rm S}_U(m)^+.\end{equation}

{\it Description concr\`ete:} Soient $L_1,\cdots ,L_h$ des repr\'esentants de $\mathcal{X}_m$, 
\ie des classes d'isom\'etrie de r\'eseaux unimodulaires de rang $m$. 
Tout \'el\'ement $f \in {\rm S}_U(m)$ est uniquement d\'etermin\'e par ses $h$ valeurs 
$f(L_i) \in U$ pour $i\in\{1,\cdots,h\}$, 
qui doivent de plus satisfaire $f(L_i) \in U^{{\rm SO}(L_i)}$.
Il est \'evident que ce sont les seules contraintes sur les $f(L_i)$, 
de sorte que l'on a un isomorphisme
\begin{equation}
\label{SU+concret}
{\rm S}_U(m) \isomo \prod_{i=1}^h U^{{\rm SO}(L_i)}  
\end{equation}
De m\^eme, si $\kappa_1,\cdots,\kappa_n$ sont des marquages respectifs fix\'es des $L_i$, 
tout \'el\'ement $f \in \widetilde{{\rm S}}_U(m)$ est uniquement d\'etermin\'e par ses $2h$ valeurs 
$f(L_i,\pm \kappa_i) \in U$ pour $i\in\{1, \cdots,h\}$, 
qui doivent de plus satisfaire $f(L_i,g(\kappa_i))=g.f(L_i,\kappa_i)$ pour tout $i\in\{1,\cdots,h\}$ 
et $g \in {\rm SO}(L_i)$. 
Si $L_i$ est ambivalent, $f(L_i,- \kappa_i)$ se d\'eduit de $f(L_i,\kappa_i)$, 
et ce dernier est un \'el\'ement arbitraire de $U^{{\rm SO}(L_i)^+}$. 
Si $L_i$ n'est pas ambivalent, $f(L_i,\kappa_i)$ et 
$f(L_i,-\kappa_i)$ sont deux \'el\'ements arbitraires de $U^{{\rm SO}(L_i)^+}$. 
Pour $L \in \mathcal{L}_m$, posons ${\rm t}(L)=1$ si $L$ est ambivalent, 
et ${\rm t}(L)=2$ sinon. On a donc un isomorphisme
\begin{equation}
\label{SUconcret}
\widetilde{{\rm S}}_U(m) \isomo \prod_{i=1}^h (U^{{\rm SO}(L_i)^+})^{{\rm t}(L_i)} 
\end{equation}

\subsection{Dimension des espaces ${\rm S}_{U}(m)^\pm$.}

Des formules exactes pour $\dim {\rm S}_{U}(m)^{\pm}$ ont \'et\'e d\'etermin\'ees
par Chenevier pour tout $m\leq 23$, en appliquant les algorithmes de l'article \cite{chniemeier}: 
nous allons rappeler plus bas ce r\'esultat.
D'apr\`es les formules  \eqref{Spm}, \eqref{SUplus}, \eqref{SU+concret} et \eqref{SUconcret}, il s'agit 
de d\'eterminer $\dim U^{\Gamma}$ avec $\Gamma={\rm SO}(L)$ et $\Gamma={\rm SO}(L)^+$, 
lorsque $L$ parcourt les classes d'isom\'etrie de r\'eseaux unimodulaires impairs de rang $\leq 23$.
Ainsi qu'il est expliqu\'e {\it loc. cit.}, l'ingr\'edient crucial est de lister les {\it masses caract\'eristiques} 
de ces $\Gamma$, \ie  les polyn\^omes caract\'eristiques de leurs \'el\'ements et leurs multiplicit\'es. 
Dans cette th\`ese, seuls les cas $m=7$ et $m=9$ joueront un r\^ole, aux 
Chapitres \ref{La_formule explicite de Riemann-Weil-Mestre} et \ref{Chapitre_calculs} (voir néanmoins le traitement de la Conjecture \ref{thm_2_w21} au paragraphe \ref{Conjectures pour le poids motivique $21$}). Dans ces cas, les groupes ${\rm SO}(L)$ et ${\rm SO}(L)^+$ 
ont \'et\'e d\'etermin\'es concr\`etement dans l'Exemple \ref{exemplen79}, et le listage ci-dessus 
se d\'eduit du \S 3.2 de \cite{chniemeier} et de la Table C.12. {\it loc. cit.}
\ps

\'Etant donn\'e que l'on a, pour $\ast\in\{+,-\}$, la d\'ecomposition 
${\rm S}_{U \oplus U'}(m)^\ast\simeq {\rm S}_U(m)^\ast \oplus {\rm S}_{U'}(m)^\ast$,
on peut supposer que $U$ est une repr\'esentation irr\'eductible de ${\rm SO}(m)$ 
avec, rappelons-le, $m=2n+1$ impair.
La th\'eorie classique de Cartan-Weyl classifie ces repr\'esentations 
en terme de leur plus haut poids 
ou, ce qui revient au m\^eme, par leur caract\`ere infinit\'esimal ; 
et la th\'eorie de Langlands les met en bijection avec 
les classes de conjugaisons de morphismes discrets et semi-simples 
${\rm W}_\R \rightarrow {\rm Sp}_{2n}(\C)$ 
(\og param\`etres de Langlands de ${\rm SO}(2n+1)$ \fg). 
Avec les notations du §\ref{Algébricité}, 
et d'apr\`es la Proposition \ref{decomposition_par_discret_sp}, 
un tel param\`etre est de la forme $\oplus_{i=1}^n I_{w_i}$ 
pour une unique suite d\'ecroissante d'entiers impairs positifs $w_1>w_2>\cdots>w_n>0$. 
Le dictionnaire, 
qui d\'ecoule de la compatibilit\'e de la correspondance de Langlands 
au caract\`ere infinit\'esimal, 
est le suivant. On note ${\rm W}_{m}$ l'ensemble des $n$-uplets 
$\underline{w}=(w_1,w_2,\cdots,w_n)$ o\`u les $w_i$ sont des entiers impairs 
tels que $w_1>w_2>\cdots>w_n$, avec $m=2n+1$.

\begin{defi} \label{defUw} 
Pour tout $\underline{w} \in {\rm W}_{2n+1}$ il existe une repr\'esentation 
irr\'eductible $U_{\underline{w}}$, unique \`a isomorphisme pr\`es, 
v\'erifiant les conditions \'equivalentes suivantes:
\begin{itemize}
\item[(i)] dans les conventions standard pour ${\rm SO}(2n+1)$, le plus haut poids de $U_{\underline{w}}$ est la suite d\'ecroissante des entiers $w'_1\geq w'_2 \geq \cdots \geq w'_n\geq 0$ avec $w'_i=\frac{w_i-2n+2i-1}{2}$, \ps 
\item[(ii)] le caract\`ere infinit\'esimal de $U_{\underline{w}}$ est ${\rm diag}(\pm \frac{w_1}{2}, \pm \frac{w_2}{2}, \cdots, \pm \frac{w_n}{2})$,\ps 
\item[(iii)] le param\`etre de Langlands de $U_{\underline{w}}$ est $\oplus_{i=1}^n I_{w_i}$.\ps 
\end{itemize}
\end{defi}


\begin{thm} {\rm (Chenevier)} On suppose $m=2n+1 \leq 23$. Les entiers $\dim {\rm S}_{U_{\underline{w}}}(m)^\pm$,  pour  $\underline{w} \in {\rm W}_m$ avec $w_1 \leq 21$, sont donn\'es dans la table 
\emph{\cite{tabledimSOchenevier}}.
\end{thm}\ps

\emph{Remarque :}
Tout r\'eseau $L \in \mathcal{L}_m$ poss\'edant une racine, 
c'est-\`a-dire un vecteur $v \in L$ v\'erifiant $v \cdot v=2$, 
est ambivalent, car si $s$ d\'esigne la sym\'etrie orthogonale par rapport \`a cette racine, 
alors on a : $-s \in {\rm SO}(L)$ et $\chi_L(-s)=\eta_L(-s)=-1$ par la Proposition \ref{chiLetaL}. 
En dimension $1<m\leq 23$, il r\'esulte de la classification que tout $L \in \mathcal{L}_m$ poss\`ede des racines, except\'e une unique classe d'isom\'etrie
pour $m=23$ : celle du {\rm r\'eseau de Leech court}. 
Ce r\'eseau est non ambivalent : son groupe sp\'ecial orthogonal est le groupe simple de Conway ${\rm Co}_2$.

\subsection{Interprétations automorphes}\label{Interprétations automorphes}
 
Soient $m \equiv \pm 1\bmod 8$,  
$V$ le $\Q$-espace vectoriel quadratique ${\rm I}_m \otimes \Q$ 
et $\mathbf{G}= \bm{\SO_V}$ le $\Q$-groupe alg\'ebrique sp\'ecial orthogonal de $V$. 
On sp\'ecifie ici les objets automorphes \'evoqu\'es au §\ref{Groupes étudiés et leurs représentations} 
dans le cas du groupe $\mathbf{G}$. 
D'apr\`es la Proposition \ref{ImQp}, l'espace $V$ relève du
 {\bf Cas 2} consid\'er\'e dans le paragraphe {\it loc. cit.} 
L'inclusion ${\rm I}_m \subset \R^m$ permet d'identifier naturellement 
l'espace quadratique r\'eel $V \otimes \R$ \`a l'espace euclidien $\R^m$, et donc $\mathbf{G}(\R)$ \`a ${\rm SO}(m)$.
\ps

Le groupe $\mathbf{G}(\R)$ \'etant compact, 
la notion de forme automorphe pour $\mathbf{G}$ \cite{BJ-Corvallis} est 
particuli\`erement \'el\'ementaire (voir par exemple \cite{grossalg}). 
En effet, par d\'efinition, l'espace des formes automorphes de $\mathbf{G}$ est l'espace 
$\mathcal{A}(\mathbf{G})$ de toutes les fonctions $\mathbf{G}(\Q) \backslash \mathbf{G}(\aq) \rightarrow \C$ 
qui sont d'une part invariantes \`a droite par un sous-groupe compact ouvert de $\mathbf{G}(\af)$, 
et d'autre part qui engendrent, pour l'action de $\mathbf{G}(\R)$ par translations \`a droite, 
une repr\'esentation de dimension finie de $\mathbf{G}(\R)$. 
L'espace vectoriel $\mathcal{A}(\mathbf{G})$ est muni d'une repr\'esentation de $\mathbf{G}(\aq)$ 
par translations \`a droite. 
Les inclusions $\mathcal{A}_{\rm cusp}(\mathbf{G}) \subset \mathcal{A}(\mathbf{G}) \subset \mathcal{A}_{\rm disc}(\mathbf{G})$ 
sont des \'egalit\'es, car $\mathbf{G}$ n'admet pas de sous-groupe parabolique propre d\'efini sur $\Q$ (ni m\^eme sur $\R$). 
On a une d\'ecomposition 
$$\mathcal{A}(\mathbf{G}) = \mathcal{A}_{\rm disc}(\mathbf{G}) \,=\, \bigoplus {\rm m}(\pi) \pi$$
la somme portant sur toutes les repr\'esentations irr\'eductibles admissibles de $\mathbf{G}(\aq)$,  
${\rm m}(\pi)$ d\'esignant leur multiplicit\'e (finie) dans $\mathcal{A}(\mathbf{G})$. 
Une telle $\pi$ est dite discr\`ete si l'on a ${\rm m}(\pi) \neq 0$. 
Les r\'esultats d'Arthur et Ta\"ibi rappel\'es au Chapitre \ref{Théorie d'Arthur pour SO-2n+1} (Théorème \ref{thm_mult_Arthur})
montrent qu'alors ${\rm m}(\pi)=1$. \ps\ps

Toute repr\'esentation irr\'eductible admissible de $\mathbf{G}(\aq)$ s'\'ecrit sous la forme 
$\pi = \pi_\infty \otimes \pi_f$, 
o\`u $\pi_\infty$ est une repr\'esentation continue irr\'eductible {\it de dimension finie} 
du groupe compact $\mathbf{G}(\R)={\rm SO}(m)$. 
Soient $U$ une telle repr\'esentation, ainsi que $K$ un sous-groupe compact ouvert de $\mathbf{G}(\AAA_f)$. 
{\it L'espace des formes automorphes pour $\mathbf{G}$ de poids $U$ et de niveau $K$} est d\'efini par la formule 
\begin{equation} 
\label{defSwGK} 
{\rm S}_{U}(\mathbf{G},K) := 
{\rm Hom}_{\mathbf{G}(\R)}(U, \mathcal{A}(\mathbf{G})^K) 
\simeq \bigoplus\limits_{\substack{\pi_\infty \simeq U \\ \pi_f^{K}\neq 0}}{\rm m}(\pi)\, \pi_f^K,
\end{equation}
o\`u $\mathcal{A}(\mathbf{G})^K$ est le sous-espace des fonctions $K$-invariantes (\`a droite) de $\mathcal{A}(\mathbf{G})$.
\ps

Posons $L = {\rm I}_m$, ainsi que $L_p = L \otimes \Z_p$ et $V_p = L\otimes \Q_p$ pour $p$ premier.
On consid\`ere le sous-groupe compact ouvert ${\rm K}(2) \subset \mathbf{G}(\AAA_f)$, 
choisi sous la forme ${\rm K}(2) = \prod_p {\rm K}(2)_p$, et avec pour tout $p$ premier
$${\rm K} (2)_p = \{ g \in \mathbf{G}(\Q_p)\, \, |\, \, g(L_p) = L_p\} \,\,\,\,{\rm avec}\,\,\, \, L={\rm I}_m.$$
On note aussi ${\rm K}(2)^+=\prod_p {\rm K}(2)^+_p$ le sous-groupe des \'el\'ements $(g_p)_p$ de ${\rm K}(2)$ 
v\'erifiant $\chi_{L}(g_2)=1$ (Formule \eqref{locetaL}). On a ${\rm K}(2)^+_p={\rm K}(2)_p$ pour $p\neq 2$,
et ${\rm K}(2)^+_2$ d'indice $2$ dans ${\rm K}(2)_2$ pour $m>1$ (Proposition \ref{chiLetaL}).
La combinaison du Corollaire \ref{cortype2} avec le Lemme \ref{OdmegalOim}, la 
Proposition \ref{chiLetaL} et la Proposition-D\'efinition \ref{prop-def_J+} pour $p=2$, entraînent le :
\begin{fact} Le sous-groupe compact ouvert ${\rm K}(2)_p$ de ${\rm SO}(V_p)$ 
est hypersp\'ecial pour $p \neq 2$, \'epiparamodulaire pour $p=2$. De plus, ${\rm K}(2)^+_2$
est le sous-groupe paramodulaire de ${\rm SO}(V_2)$ inclus dans ${\rm K}(2)_2$. 
\end{fact}

Par construction, l'ensemble $\mathbf{G}(\Q) \backslash \mathbf{G}(\aq) / {\rm K}(2)$ s'identifie 
naturellement \`a l'ensemble des r\'eseaux de $\R^m$ qui sont dans le m\^eme genre que $L$
(voir par exemple \cite[Prop. 2.3]{borelfinitude} pour un \'enonc\'e g\'en\'eral), c'est-\`a-dire \`a $\mathcal{L}_m$ 
d'apr\`es la Proposition \ref{genusIm}.
Rappelons comment (voir aussi \cite[\S 4.4.4]{Chen-Lannes}).\ps

 Fixons $g=(g_\infty,g_f) \in \mathbf{G}(\aq) = \mathbf{G}(\R) \times \mathbf{G}(\AAA_f)$ et \'ecrivons $g_f=(g_p)_p$ 
 avec $g_p \in \mathbf{G}(\Q_p)$ (et donc $g_p \in {\rm K}(2)_p$ pour presque tout $p$). 
 D'apr\`es le point de vue local-global sur les r\'eseaux d'un $\Q$-espace vectoriel (Eichler, Weil),
il existe un unique r\'eseau de $V$, not\'e $g_f(L)$, tel que pour tout premier $p$
on a $g_f(L) \otimes \Z_p = g_p( L_p)$. En particulier, c'est un r\'eseau de $\R^m$, 
et il y a donc un sens \`a consid\'erer $g_\infty^{-1}(g_f(L))$. 
Par construction, $g_f(L)$ est dans le m\^eme genre que $L$, donc unimodulaire impair, 
et il en va de m\^eme du r\'eseau $g_\infty^{-1}(g_f(L))$ qui lui est isomorphe. 
On a donc d\'efini une application 
\begin{equation} \label{iotagenus}
\begin{array}{ccccc}
\iota & : &  \mathbf{G}(\Q) \backslash \mathbf{G}(\aq)/ {\rm K}(2) &\longrightarrow & \mathcal{L}_m \\
 & & (g_\infty,g_f) & \longmapsto & g_\infty^{-1}(g_f({\rm I}_m)) \\
\end{array}.
\end{equation}

\begin{prop}\label{propiota}  L'application \eqref{iotagenus} est bijective et $\mathbf{G}(\R)$-\'equivariante. \end{prop}

Dans cet \'enonc\'e, $\mathbf{G}(\Q) \backslash \mathbf{G}(\aq) / K$ est vu comme un $\mathbf{G}(\R)$-ensemble \`a gauche 
en utilisant la translation inverse \`a droite. 

\begin{proof} L'injectivit\'e est triviale. V\'erifions la surjectivit\'e. Soit $L' \in \mathcal{L}_m$.
D'apr\`es la Proposition \ref{genusIm}, $L'$ est dans le m\^eme genre que $L={\rm I}_m$.
D'apr\`es le théorème de Hasse-Minkowski, cela entra\^ine que $L' \otimes \Q$
est isom\'etrique \`a $L \otimes \Q$. 
Une telle isom\'etrie est toujours induite par un \'el\'ement $h$ de ${\rm O}(m)$, 
car on a $L \otimes \R = L' \otimes \R = \R^m$. 
Quitte \`a la remplacer par $-h$, on peut la supposer de d\'eterminant $1$ car $m$ est impair.
On a donc montr\'e l'existence de $g_\infty \in {\rm SO}(m)$ avec $g_\infty(L') \subset L \otimes \Q$.
Toujours par la Proposition \ref{genusIm}, il existe $g_f \in \mathbf{G}(\AAA_f)$ avec $g_\infty(L')=g_f(L)$
(noter aussi que $\det {\rm O}(L_p) = \{ \pm 1\}$ pour tout $p$), ce qui conclut.
Enfin, l'assertion de $\mathbf{G}(\R)$-\'equivariance signifie que pour tout $h \in \mathbf{G}(\R)$, et tout $x \in \mathbf{G}(\aq)$, 
on a $\iota (x (h^{-1} \times 1)) = h(\iota(x))$, ce qui est évident. 
\end{proof}

D\'ecrivons enfin l'analogue de \eqref{iotagenus} pour ${\rm K}(2)^+$. On consid\`ere cette fois-ci
l'ensemble $\tlm$ des r\'eseaux unimodulaires marqu\'es d\'efini au \S \ref{reseauxunimmarques}.
On fixe un marquage $\kappa$ pour le r\'eseau $L = {\rm I}_m$, par exemple $\kappa:=2e=(2,2,\cdots,2)$ 
pour fixer les id\'ees. On dispose alors d'une application naturelle 
\begin{equation} \label{iotagenusplus} 
\begin{array}{ccccc}
 \widetilde{\iota} & : &  \mathbf{G}(\Q) \backslash \mathbf{G}(\aq)/ {\rm K}(2)^+ &\longrightarrow & \tlm \\
 & & (g_\infty,g_f) & \longmapsto & g_\infty^{-1}(g_f({\rm I}_m),g_2(\kappa)) \\
\end{array}.
\end{equation}
Le m\^eme argument que dans la Proposition \ref{propiota} d\'emontre la :
\begin{prop} \label{propiotatilde}  L'application \eqref{iotagenusplus} est $\mathbf{G}(\R)$-\'equivariante, bijective pour $m>1$.\end{prop}

L'\'enonc\'e final suivant clôt le lien entre les espaces concrets d\'efinis au \S \ref{reseauxunimmarques}, et les objets g\'en\'eraux du monde automorphe dans le cas du groupe $\mathbf{G}$ et en niveaux ${\rm K}(2)$ et ${\rm K}(2)^+$.

\begin{prop}\label{prop_8.3.21} Pour toute repr\'esentation irr\'eductible $U$ de ${\rm SO}(m)$, avec $m \equiv \pm 1 \bmod 8$, les bijections $\iota$ et $\widetilde{\iota}$ induisent des isomorphismes d'espaces vectoriels ${\rm S}_{U}(m) \isomo {\rm S}_{U}(\mathbf{G},{\rm K}(2))$ et $\widetilde{{\rm S}}_{U}(m) \isomo {\rm S}_{U}(\mathbf{G},{\rm K}(2)^+)$ (pour $m>1$).
\end{prop}

\begin{proof} D'apr\`es la Proposition \ref{propiota} (resp. \ref{propiotatilde}), $\iota$ (resp. $\widetilde{\iota}$) induit un isomorphisme $\mathbf{G}(\R)$-\'equivariant 
entre $\mathcal{A}(\mathbf{G})^{{\rm K}(2)}$ (resp. $\mathcal{A}(\mathbf{G})^{{\rm K}(2)^+}$) et l'espace des fonctions $\mathcal{L}_m \rightarrow \C$ (resp. $\tlm \rightarrow \C$). On conclut par l'isomorphisme 
	$${\rm Hom}_{\mathbf{G}(\R)}(U,\mathcal{A}(\mathbf{G})^K) \,\simeq \,(U^\vee \otimes_\C \mathcal{A}(\mathbf{G})^K)^{\mathbf{G}(\R)}$$
qui vaut pour tout sous-groupe compact ouvert $K$ de $\mathbf{G}(\AAA_f)$. On a \'egalement utilis\'e le fait que $U^\vee \simeq U$ pour toute représentation irr\'eductible $U$ de ${\rm SO}(m)$. 
\end{proof}

\subsection{Formules de dimensions}
\begin{defi}
Soient $m \equiv \pm 1\bmod 8$, $V$ le $\Q$-espace vectoriel quadratique ${\rm I}_m \otimes \Q$ et $\mathbf{G}= \bm{\SO_V}$ le $\Q$-groupe algébrique spécial orthogonal de $V$ (relevant donc du {\bf Cas 2}).

Soit $\underline{w} \in {\rm W}_{m}$. On note $\Pi_{\underline{w}}^1(\mathbf{G})$ l'ensemble des représentations automorphes cuspidales algébriques $\pi$ de $\mathbf{G}$, de poids $\underline{w}$ et de conducteur $1$. Plus précisément, 
\begin{itemize}
\item $\pi_\infty \simeq U_{\underline{w}}$ au sens de la Définition \ref{defUw} ; \smallskip
\item $\pi_p$ est non ramifiée pour tout $p$. \smallskip
\end{itemize}

Soit de plus $\eps \in \{+,-\}$. On note $\Pi_{\underline{w}}^{2,\eps}(\mathbf{G})$ l'ensemble des représentations automorphes cuspidales algébriques $\pi$ de $\mathbf{G}$, de poids $\underline{w}$ de conducteur $2$, de signe local $\eps$. Plus précisément,
\begin{itemize}
\item $\pi_\infty \simeq U_{\underline{w}}$ au sens de la Définition \ref{defUw} ;\smallskip
\item $\pi_p$ est non ramifiée pour $p>2$ ; \smallskip
\item $\pi_2^{\K_0(2)}=\{0\}$ et $\pi_2^{(\J(2),\eps)}\neq \{0\}$. \smallskip
\end{itemize}

On pose $$\Pi_{\underline{w}}^{2}(\mathbf{G})=\Pi_{\underline{w}}^{2,+}(\mathbf{G}) \coprod \Pi_{\underline{w}}^{2,-}(\mathbf{G})$$ et $$\Pi_{\underline{w}}(\mathbf{G})=\Pi_{\underline{w}}^{1}(\mathbf{G}) \coprod \Pi_{\underline{w}}^{2}(\mathbf{G}).$$
\end{defi}

\begin{prop}\label{dim_inv_2_so}
Soient $m \equiv \pm 1\bmod 8$, $\underline{w} \in {\rm W}_{m}$ et $\eps \in \{+,-\}$. Alors
\[
{\rm S}_{U_{\underline{w}}}(m)^\eps \simeq \bigoplus_{\pi \in \Pi_{\underline{w}}(\mathbf{G})} \pi_2^{(\J(2),\eps)}.
\]

Si l'on suppose de plus que $\underline{w}$ est \emph{très régulier} (\ie $|w_i-w_j|>2$ pour $i\neq j$), alors 
\[
\dim {\rm S}_{U_{\underline{w}}}(m)^\eps=\left| \Pi_{\underline{w}}^1(\mathbf{G}) \right| + \left|\Pi_{\underline{w}}^{2,\eps}(\mathbf{G}) \right|.
\]
\end{prop}
\begin{proof}
Le premier isomorphisme vient de \eqref{defSwGK} combinée à \eqref{SU+concret}, \eqref{SUconcret} et la Proposition \ref{prop_8.3.21}. On utilise alors que $\pi_p^{\K_0(p)}$ est de dimension 1 si $\pi_p$ est non ramifiée (sans autre hypothèse sur $\pi_p$).

Pour la seconde égalité, l'hypothèse \og très régulier \fg{} 
entraîne que $\pi_2$ est tempérée par le Corollaire \ref{cor_shin_caraiani}. Les résultats découlent alors des calculs de dimensions d'invariants de la Première Partie (voir le Corollaire \ref{cor_temp_dim_invariants_général}).
\end{proof}

À noter enfin que l'ensemble $\Pi_{\underline{w}}^{1}(\mathbf{G})$ est déterminé en poids motivique $\leq 23$ par les travaux de Chenevier et Renard \cite{Chen-Ren}. Nous retrouvons d'ailleurs lesdits éléments avec les tables de dimension de \cite{tabledimSOchenevier}.

\chapter{La formule explicite de Riemann-Weil-Mestre}\label{La_formule explicite de Riemann-Weil-Mestre}
L'enjeu de ce Chapitre va être de mettre en place la \og machinerie \fg{} de la formule explicite de Riemann-Weil dans son cadre général donné par Mestre \cite{Mes}. 
Il faudra donc vérifier au §\ref{Les fonctions Lambda sont des fonctions Lambda} que les fonctions $\Lambda$ des représentations automorphes cuspidales algébriques, et plus précisément des \emph{paires} de telles représentations, telles qu'introduites au paragraphe \ref{Fonctions Lambda et facteurs epsilon globaux} relèvent bien du formalisme de l'article \cite{Mes}. Nous donnerons alors l'expression de la formule explicite au §\ref{Énoncés de la formule explicite de Riemann-Weil-Mestre} en général puis dans un langage \emph{ad hoc} pour notre étude.
Nous préciserons enfin, outre de précieux raffinements, une \emph{ingénierie} de cette formule qui nous permettra d'en extraire toute la puissance et qui gouvernera les calculs effectifs énoncés dans le chapitre suivant, à l'origine de nos démonstrations.

Nous utilisons à plusieurs endroits dans ce chapitre (§\ref{Une spectaculaire illustration}, §\ref{Cas du conducteur 1}, §\ref{Un exemple en conducteur 1}) des exemples en conducteur 1. Ces exemples ne sont pas originaux et sont tirés de \cite{Chen-Lannes}. Il nous a néanmoins paru judicieux de les exposer en détail tant ils sont illustratifs et -- puisque le fait de travailler en conducteur 2 ajoute une certaine technicité -- préférables à des exemples originaux en conducteur 2. On trouvera néanmoins de tels exemples originaux dans l'ensemble du Chapitre \ref{Chapitre_calculs}.

\section{Les fonctions $\Lambda$ (de paires) sont des fonctions $\Lambda$ (au sens de Mestre)}\label{Les fonctions Lambda sont des fonctions Lambda}
Nous commençons par introduire la notion de fonction $\Lambda$ au sens de \cite{Mes}.
\begin{defi}\label{defi préfonction lambda}
Soient $M,M' \in \mathbb{N} $, soit $A \in \mathbb{R}^*$, soit $c \in \mathbb{R}_+$.

On se donne $(a_i)_{1 \leq i \leq M} \in (\mathbb{R}_+)^M$, $(a'_i)_{1 \leq i \leq M} \in \mathbb{C}^M$ tels que $\mathrm{Re}(a'_i) \geq 0$ et $\mathrm{Re}(a_i+a'_i) > 0$ pour tout $i$. 

On se donne, pour tout $p$ premier $(\alpha_j(p))_{1 \leq j \leq M'} \in \mathbb{C}^{M'}$ tel que, pour tout $j$, $| \alpha_j(p)| \leq p^c$.

Une \emph{pré-fonction $\Lambda$} est une fonction méromorphe sur $\mathbb{C}$ vérifiant
\begin{enumerate}
\item $\Lambda$ n'a qu'un nombre fini de pôles.
\item La fonction $s\mapsto \Lambda(s)$ diminuée de ses parties singulières est bornée dans toute bande verticale
\begin{equation*}
- \infty < \sigma_0 \leq \mathrm{Re}(s) \leq \sigma_1 < + \infty.
\end{equation*}
\item Pour $\mathrm{Re}(s) > 1+c$, on a le développement en produit absolument convergent suivant :
\begin{equation} 
\Lambda(s)=A^s \prod_{i=1}^M \Gamma(a_i s+a'_i) \prod_p \prod_{j=1}^{M'} \frac{1}{1-\alpha_j(p)p^{-s}}.
\end{equation}
\end{enumerate}

\end{defi}

Ce qu'il manque, par rapport aux fonctions $\Lambda$ usuelles, c'est l'équation fonctionnelle. On va ici l'introduire par le biais d'un couple de pré-fonctions $\Lambda$.
\begin{defi}\label{Defi fonctions lambda}

Considérons un couple de pré-fonctions $\Lambda$, $\Lambda_1$ et $\Lambda_2$. Pour fixer les notations, écrivons la condition 3. de la Définition \ref{defi préfonction lambda} pour chacune d'entre elles. 
Pour $\mathrm{Re}(s) > 1+c_1$ (resp. $\mathrm{Re}(s) > 1+c_2$), on a les développements en produit absolument convergent suivants :
\begin{align*}
\Lambda_1(s)&=A^s \prod_{i=1}^{M_1} \Gamma(a_i s+a'_i) \prod_p \prod_{j=1}^{M'_1} \frac{1}{1-\alpha_j(p)p^{-s}} , \\
\Lambda_2(s)&=B^s \prod_{i=1}^{M_2} \Gamma(b_i s+b'_i) \prod_p \prod_{j=1}^{M'_2} \frac{1}{1-\beta_j(p)p^{-s}}.
\end{align*}
On dit alors que $(\Lambda_1,\Lambda_2)$ forme un couple de fonctions $\Lambda$ s'il existe un complexe non nul $w$ tel que :

\begin{enumerate}[(i)]

\item $\Lambda_1(s)=w \Lambda_2(1-s)$ \hfill (équation fonctionnelle)
\item $\sum_{i=1}^{M_1} a_{i}= \sum_{i=1}^{M_2} b_{i}$ \hfill (compatibilité des facteurs archimédiens)
\end{enumerate}
\end{defi}
On remarque tout de suite que l'on peut choisir le même réel positif $c$ pour les deux pré-fonctions $\Lambda$, en considérant le plus grand des deux. On peut également choisir les mêmes entiers $M$ et $M'$, quitte à poser des $a_i=0$, $ a'_i=1$ et $\alpha_j(p)=0$. Selon la même logique, on pourrait d'ailleurs considérer un seul entier $M=M'$, ce que nous ne faisons pas pour pouvoir conserver des notations cohérentes avec l'Annexe \ref{Démonstration de la formule explicite de Riemann-Weil-Mestre}.


\begin{prop}
Soient $\pi$ et $\pi'$ deux représentations automorphes cuspidales algébriques de $\GL_n$ et $\GL_{n'}$ respectivement. Nous avons défini au paragraphe \ref{Paires globales} la fonction $\Lambda$ de la paire $\{\pi,\pi'\}$. Posons :
\begin{align*}
\tilde{\Lambda}(s, \pi \times \pi')  &={\rm N}(\pi \times \pi')^{\frac{s}{2}}\Lambda(s,\pi \times \pi'), \numberthis \label{Lambda_tilde} \\
 \tilde{\Lambda}(s, \pi^\vee \times (\pi')^\vee)  &={\rm N}(\pi \times \pi')^{\frac{s}{2}}\Lambda(s,\pi^\vee \times (\pi')^\vee),
\end{align*}
où ${\rm N}(\pi,\pi')$ est le conducteur de la paire $\{\pi,\pi'\}$ -- dont on a remarqué \emph{loc. cit.} qu'il était bien égal à celui de la paire $\{\pi^\vee,(\pi')^\vee\}$.

Alors le couple $(\tilde{\Lambda}(\cdot, \pi^\vee \times \pi'),\tilde{\Lambda}(\cdot, \pi \times (\pi')^\vee))$ est un couple de fonctions $\Lambda$ au sens de la Définition \ref{Defi fonctions lambda}.
\end{prop}
\begin{proof}
Commençons par vérifier qu'il s'agit de pré-fonctions $\Lambda$ au sens de la Définition \ref{defi préfonction lambda}.

Le point 1. est vérifié par le résultat de \cite{MW} cité au §\ref{Fonctions Lambda et facteurs epsilon globaux}. Le point 2. est un résultat de Gelbart et Shahidi \cite{Gelbart-Shahidi}. Pour le point 3., il faut s'assurer d'avoir un produit eulérien avec le facteur archimédien \emph{et} les facteurs non archimédiens qui ont la bonne forme.

Par les propriétés analytiques déjà énoncées au §\ref{Fonctions Lambda et facteurs epsilon globaux}, on a une décomposition du type \eqref{prod_eul+archim} pour $\re(s)>1$ et il s'agit donc de regarder ce qu'il se passe à chaque place.\medskip

Commençons par le cas d'une place finie non ramifiée. On a vu au §\ref{Facteurs locaux} qu'alors :
\[
L_p(s,\pi_p^\vee \times \pi'_p)=\det (1-p^{-s}c_p(\pi_p^\vee) \otimes c_p(\pi'_p))^{-1}
\]
où $c_p(\pi_p)$ est le paramètre de Satake de la représentation locale non ramifiée $\pi_p$.
Or, puisque $\pi$ est une représentation automorphe cuspidale de $\GL_n$ de caractère central unitaire (en fait trivial), toutes les $\pi_v$ sont unitaires et on a $\Ll(\pi_v^\vee)=\Ll(\pi_v)^\vee=\overline{\Ll(\pi_v)}$, donc \emph{a fortiori} $c_p(\pi_p^\vee)=c_p(\pi_p)^{-1}=\overline{c_p(\pi_p)}$. 

On a alors :
$$
L_p(s,\pi_p^\vee \times \pi'_p)=\prod_{i,j} \frac{1}{1-p^{-s}\overline{\lambda_i} \mu_j}
$$
où $c_p(\pi_p)$ est la classe de conjugaison de $\mathrm{diag}((\lambda_i))$ et $c_p(\pi'_p)$ celle de $\mathrm{diag}((\mu_j))$.

La conjecture de Ramanujan généralisée prévoit que les $\lambda_i$ et les $\mu_j$ soient de module 1. Celle-ci est connue dans des cas particuliers (voir par exemple le Théorème \ref{thm_shin_caraiani}), mais en général les estimées de Jacquet et Shalika \cite{JS} donnent, puisque $\pi$ est centrée donc unitaire, $p^{-\frac{1}{2}}< |\lambda_i|< p^\frac{1}{2}$ (et de même pour $\mu_j$ bien sûr), si bien que l'on a écrit
\[
L_p(s,\pi_p^\vee \times \pi'_p)=\prod_{k=1}^{nn'} \frac{1}{1-p^{-s}\alpha_k(p)}
\]
avec $|\alpha_k(p)|\leq p$ pour tout $k\in\{1,\cdots,nn'\}$.

Puisqu'il n'y a qu'un nombre fini de places ramifiées pour $\pi$ ou $\pi'$, on trouve bien un réel $c (\geq 1)$ et un entier $M' (\geq nn')$ qui conviennent.\medskip


Concernant l'unique place archimédienne, il faut nous rappeler comment est défini $L_\infty(s, \pi_\infty^\vee\times \pi_\infty')$. Les calculs du paragraphe \ref{Facteur local} nous assurent qu'on a bien le produit de (composées affines de) fonctions $\Gamma$ d'Euler recherché. On a d'ailleurs $M=nn'$ et tous les $a_i$ égaux à $\frac{1}{2}$ si l'on choisit d'appliquer la formule de duplication à tous les facteurs $\Gamma_\C$ qui apparaissent (voir \emph{loc. cit.}). En particulier, on a bien $\mathrm{Re}(a'_i) \geq 0$ et $\mathrm{Re}(a_i+a'_i) > 0$ pour tout $i$.\medskip 

Nous avons donc bien affaire à deux \emph{pré-fonctions} $\Lambda$, il reste à vérifier les deux conditions de la Définition \ref{Defi fonctions lambda}. 
La première est la réécriture (en intégrant le conducteur) de l'équation fonctionnelle \eqref{eq_fonctionnelle_paire_globale} :
\begin{equation}\label{eq_fonct_normalisée}
\tilde{\Lambda}(s, \pi \times \pi')=\eps(\pi \times \pi')\tilde{\Lambda}(1-s, \pi^\vee \times (\pi')^\vee).
\end{equation}

 La seconde consiste en une \og compatibilité \fg{} des facteurs archimédiens.
Ici, nous avons plus précisément 
\begin{align*}
L_\infty(s, \pi_\infty^\vee \times \pi'_\infty)
&= \Gamma(s, \Ll(\pi_\infty^\vee) \otimes \Ll(\pi'_\infty)) \\
&= \Gamma(s, \Ll(\pi_\infty)^\vee \otimes \Ll(\pi'_\infty)) \\
&= \Gamma(s, \Ll(\pi_\infty) \otimes \Ll(\pi'_\infty)^\vee) \\
&= \Gamma(s, \Ll(\pi_\infty) \otimes \Ll((\pi')^\vee_\infty)) \\
&= L_\infty(s, \pi_\infty \times (\pi')^\vee_\infty)
\end{align*}
en utilisant la compatibilité de la correspondance de Langlands locale archimédienne pour $\GL_n$ à la dualité (\cf Théorème \ref{LLC_arch}) et le fait que tous les éléments de $K_\infty$ sont autoduaux (\cf Définition \ref{defi_G(R)_alg_K_infty}).
En particulier, les facteurs archimédiens de $\tilde{\Lambda}(s, \pi^\vee \times \pi')$ et $\tilde{\Lambda}(s, \pi \times (\pi')^\vee)$ sont bien compatibles (ce sont les mêmes) et nos deux \emph{pré-fonctions} $\Lambda$ forment finalement bien un couple de fonctions $\Lambda$.



\end{proof}
\color{black}

\textbf{Dans toute la suite, nous omettrons le tilde} pour alléger les notations, les fonctions $\Lambda$ seront donc toujours supposées intégrer le conducteur de façon normalisée.

\section{Énoncés de la formule explicite de Riemann-Weil-Mestre}\label{Énoncés de la formule explicite de Riemann-Weil-Mestre}

On a introduit au paragraphe précédent les couples de fonctions $\Lambda$ au sens de \cite{Mes}, on énonce ici la formule explicite dans ce même cadre général, avant de la préciser au cas qui nous intéresse qui est celui de fonctions $\Lambda$ de paires de représentations automorphes cuspidales algébriques du groupe général linéaire.
\subsection{Énoncé général}\label{Enonce_general}
On se donne donc un couple de fonctions $\Lambda$ au sens de la Définition \ref{Defi fonctions lambda} que l'on note $\Lambda_1$ et $\Lambda_2$. On a donc un réel $c \geq 0$, deux entiers $M$ et $M'$ et on fixe pour l'ensemble du chapitre les notations liées au développement en produit pour $\mathrm{Re}(s)>1+c$ (soumises aux hypothèses de la Définition \ref{defi préfonction lambda}) :
\begin{align*}
\Lambda_1(s)&=A^s \prod_{i=1}^M \Gamma(a_i s+a'_i) \prod_p \prod_{j=1}^{M'} \frac{1}{1-\alpha_j(p)p^{-s}} , \\
\Lambda_2(s)&=B^s \prod_{i=1}^M \Gamma(b_i s+b'_i) \prod_p \prod_{j=1}^{M'} \frac{1}{1-\beta_j(p)p^{-s}}.
\end{align*}

Puisque les fonctions $\Lambda_1$ et $\Lambda_2$ forment un couple de fonctions $\Lambda$ au sens de la Définition \ref{Defi fonctions lambda}, on a de plus un complexe non nul $w$ tel que $\Lambda_1(s)=w\Lambda_2(1-s)$ et on sait que $\sum_{i=1}^M a_i=\sum_{i=1}^M b_i$.

Il faut également se donner une fonction d'une classe particulière, qu'on nommera \emph{fonction test}.

\begin{defi} \emph{(\cite{Mes})}\label{defi fonction explicite}

Soit $c\geq 0$. Soit $F:\R \rightarrow \R$, mesurable, paire. On dit que $F$ est une fonction test (de niveau $c$) si : 
\begin{enumerate}[(i)]
\item il existe $\varepsilon >0$ tel que $x \mapsto F(x)e^{(\frac{1}{2}+c+\varepsilon)|x|} \in {\rm L}^1(\mathbb{R})$,
\item il existe $\varepsilon >0$ tel que $x \mapsto F(x)e^{(\frac{1}{2}+c+\varepsilon)|x|} $ soit à variation bornée et normalisée (i.e. égale en chaque point à la moyenne de ses limites à gauche et à droite),
\item la fonction $x \mapsto \dfrac{F(x)-F(0)}{x}$ est à variation bornée.
\end{enumerate}
\end{defi}

En pratique, nous n'utiliserons comme fonction test que des fonctions de classe ${\rm C}^2$ à support compact, pour lesquelles ces hypothèses sont trivialement vérifiées.

On pose, pour $s$ complexe avec $-c\leq \re(s) \leq 1+c$ :
\begin{equation}\label{defi_Phi}
\Phi(s)= \int_{- \infty}^{+ \infty} F(x) e^{(s-\frac{1}{2})x} \dd x.
\end{equation}

\begin{thm} \emph{Formule explicite de Riemann-Weil-Mestre \cite{Mes}} \label{formule_explicite}

Avec les notations précédentes, on a, pour toute fonction test $F$ de niveau $c$ : 
\begin{multline*}
\sum \Phi(\rho) - \sum \Phi(\mu)=F(0) \log(AB)\\
-\sum_p \sum_{j=1}^{M'} \sum_{k=1}^{+ \infty} F(k \log p) \frac{\log(p)}{p^{\frac{k}{2}}}(\alpha_j(p)^k+\beta_j(p)^k)\\
-\sum_{i=1}^M (I(a_i,a'_i)+I(b_i,b'_i))
\end{multline*}
où $\rho$ (resp. $\mu$) parcourt les zéros (resp. les pôles) de $\Lambda_1$ de partie réelle comprise entre $-c$ et $1+c$ comptés avec multiplicité, et où 
\[
I(a,a')=\int_0^{+ \infty} \left( \frac{F(ay)e^{- (\frac{a}{2}+a') y}}{1-e^{-y}} - F(0) \frac{e^{-y}}{y} \right) \dd y.
\]
\end{thm}

Cette formule provient d'une formule des résidus le long d'un contour bien choisi puis d'un passage à la limite. La démonstration de Jean-François Mestre mentionne \og des techniques classiques \fg{} d'analyse complexe, desquelles nous n'étions pas familier. Il a donc paru utile à notre compréhension de développer en détail la démonstration de \cite{Mes} et notamment de comprendre les hypothèses sur la fonction test. Elle ne contient aucune nouveauté par rapport à la démonstration originale, nous avons néanmoins cru judicieux de l'incorporer en Annexe \ref{Démonstration de la formule explicite de Riemann-Weil-Mestre} de cette thèse.

\subsection{Énoncé particulier}\label{Énoncé particulier}
Nous nous intéressons ici à des fonctions $\Lambda$ très spécifiques puisqu'il s'agit de fonctions $\Lambda$ de paires de représentations automorphes cuspidales algébriques du groupe général linéaire. 
Plus précisément, étant données $\pi$ et $\pi'$ deux représentations automorphes cuspidales algébriques de $\GL_n$ et $\GL_{n'}$ sur $\Q$ respectivement, on considère le couple de fonctions $\Lambda$ (normalisées au sens du §\ref{Les fonctions Lambda sont des fonctions Lambda}, \ie intégrant le conducteur) donné par $\Lambda(\cdot, \pi^\vee \times \pi')$ et $\Lambda(\cdot, \pi \times (\pi')^\vee)$.


À ce couple correspond un réel positif $c$ qui intervient à deux endroits dans la définition (convergence du produit eulérien et majoration des coefficients) et à deux endroits dans l'utilisation de la formule explicite (bande des zéros considérés et choix de la fonction test).
\begin{itemize}
\item Concernant la convergence du produit eulérien, un résultat de Jacquet et Shalika (\cite{JS}) déjà mentionné au paragraphe \ref{Paires globales} affirme qu'elle a lieu pour $\re(s)>1$ (correspondant donc à $c=0$).
\item Concernant la majoration des coefficients, on a vu au paragraphe \ref{Les fonctions Lambda sont des fonctions Lambda} que l'on pouvait choisir $c=1$ dans le cas non ramifié et qu'on avait bien un $c\geq 1$ dans le cas général.
\item Concernant la somme sur les zéros, on sait par les résultats de Jacquet et Shalika que tous les zéros de $\Lambda(s, \pi^\vee \times \pi')$ sont situés dans la bande $0 \leq \re(s) \leq 1$. Cette somme est donc indépendante de la valeur de $c \geq 0$.
\item Concernant le choix de la fonction test, elle est supposée de niveau $c$, elle est donc de niveau $c'$ pour tout réel positif $c' \leq c$.
\end{itemize}\smallskip

Ce fameux réel $c$ ne joue donc un rôle que dans le choix de la fonction test, avec d'autant moins de fonctions test que $c$ est plus grand. En pratique (voir le paragraphe \ref{Choix de fonction test}), on considérera une fonction régulière à support compact qui sera donc de niveau $c$ pour \emph{tout} réel $c$. Dans ce cadre-là et au vu des considérations précédentes, \emph{tout se passe donc comme si on avait} $c=0$ dans la formule explicite du Théorème \ref{formule_explicite}, ce que nous supposerons sans plus de précaution dans la suite. Nous renvoyons à la démonstration détaillée en Annexe \ref{Démonstration de la formule explicite de Riemann-Weil-Mestre} de la formule explicite, où le rôle joué par les différents paramètres apparaît plus clairement.\medskip

Il nous faut par ailleurs préciser ce que l'on entend par \og somme sur les zéros de $\Lambda(\cdot, \pi^\vee \times \pi')$ \fg{} (cela n'est pas nécessaire pour les pôles qui sont en nombre fini par hypothèse) : on désigne donc par $\sum_\rho \Phi(\rho)$ la limite quand $T$ tend vers $+\infty$ de $\sum_{|\Imm(\rho)|\leq T} \Phi(\rho)$ et c'est d'ailleurs un des résultats de la démonstration de la formule explicite que cette limite existe.\ps 

Remarquons enfin que l'on a :
\[
\Lambda(1-\overline{s},\pi^\vee \times \pi')=\eps(\pi^\vee \times \pi')\Lambda(\overline{s},\pi \times (\pi')^\vee)=\eps(\pi^\vee \times \pi')\overline{\Lambda(s,\pi^\vee \times \pi')},
\]
la dernière égalité provenant du fait déjà mentionné que chaque $\pi_v$ (resp. $\pi'_v$) est unitaire. 
Donc, si $s$ est un zéro de $\Lambda$, alors $1-\overline{s}$ en est un également, avec même multiplicité.

Nous pouvons maintenant reprendre les termes de la formule explicite :
\begin{multline*}
\sum \Phi(\rho) - \sum \Phi(\mu)=F(0) \log(AB)\\
-\sum_p \sum_{j=1}^{M'} \sum_{k=1}^{+ \infty} F(k \log p) \frac{\log(p)}{p^{\frac{k}{2}}}(\alpha_j(p)^k+\beta_j(p)^k)\\
-\sum_{i=1}^M (I(a_i,a'_i)+I(b_i,b'_i))
\end{multline*}

\begin{itemize}
\item On a vu que si $s$ est un zéro de $\Lambda$, alors $1-\overline{s}$ également, avec même multiplicité. Or $\Phi(1-\overline{s})=\Phi(\overline{s})=\overline{\Phi(s)}$ par \eqref{Phi(s)=Phi(1-s)} et \eqref{defi_Phi}. Donc $\sum \Phi(\rho)$ est une quantité réelle, nous la notons $\mathrm{Z}^F(\pi,\pi')$.
\item La fonction $\Lambda$ n'a de pôles que si $\pi' \simeq \pi$, auquel cas ils sont situés en $s=0$ et $s=1$ et ils sont simples. La contribution $\sum \Phi(\mu)$ vaut donc 0 si $\pi' \not\simeq \pi$ et $\Phi(0)+\Phi(1)$ dans le cas contraire. Or la parité de $F$ entraîne (\cf \eqref{Phi(s)=Phi(1-s)}) que $\Phi(s)=\Phi(1-s)$. Finalement $\sum \Phi(\mu)=2\Phi(0)\delta_{\pi,\pi'}$ (où $\delta$ correspond au symbole de Kronecker)
\item La dualité entre $(\pi^\vee,\pi')$ et $(\pi,(\pi')^\vee)$ implique que $\alpha_j(p)=\overline{\beta_j(p)}$ pour tout $j$ et tout $p$. La somme sur les nombres premiers $p$ (correspondant à la partie ultramétrique) est donc une quantité \emph{réelle} et on la note\footnote{le choix du facteur 2 permet que pour une place $p$ non ramifiée, le terme correspondant à la contribution en $p$ soit $\sum_k F(k \log p) \frac{\log p}{p^{k/2}} \re[ \overline{\mathrm{tr}(c_p(\pi)^k)}\mathrm{tr}(c_p(\pi')^k)]$} $2 \mathrm{B}_f^F(\pi,\pi')$ (qui est évidemment symétrique, et qui vérifie $\mathrm{B}_f^F(\pi,\pi')=\mathrm{B}_f^F(\pi^\vee,(\pi')^\vee)$)
\item Les quantités $I(a_i,a'_i)$ et $I(b_i,b'_i)$ correspondent aux facteurs $\Gamma_\R$ et $\Gamma_\C$ introduits dans le paragraphe \ref{Facteur local}. Or, quand on a défini \emph{loc. cit.} l'application $K_\infty  \ni V \mapsto \Gamma(\cdot,V)$, on a également fait apparaître des fonctions \og puissance de $\pi$ \fg{} (le \emph{réel} $\pi$) si bien que l'on va voir des facteurs $\pi$ apparaître dans le $A$ et le $B$ au sens de la Définition \ref{Defi fonctions lambda}. Nous suivons, pour la clarté, le choix de \cite{Chen-Lannes} Chap. IX Proposition-Définition 3.7 et introduisons la fonction $\mathrm{J}^F : K_\infty \rightarrow \R$ qui regroupe les contributions de ces puissances du réel $\pi$ avec les quantités $I(a_i,a'_i)$.
\end{itemize}
\begin{defi}\label{defi_J_B_infty}
On pose :
\begin{enumerate}
\item $\mathrm{J}^F(\mathbf{1})=\frac{1}{2} \log (\pi)F(0) + I(\frac{1}{2},0)$,
\item $\mathrm{J}^F(\eps_{\C / \R})=\frac{1}{2} \log (\pi)F(0) + I(\frac{1}{2},\frac{1}{2})$,
\item $\mathrm{J}^F(I_w)=\log (2\pi)F(0) + I(1,\frac{w}{2})$ pour $w>0$
\end{enumerate}
ce qui définit $\mathrm{J}^F$ sur $K_\infty$ par $\Z$-linéarité.\ps 

On définit alors l'application $\Z$-bilinéaire $\mathrm{B}_\infty^F$ sur $K_\infty$ par $\mathrm{B}_\infty^F(U,V)=\mathrm{J}^F(U^\vee \otimes V)$. Or on a aussi $\mathrm{B}_\infty^F(U,V)=\mathrm{J}^F(U \otimes V)$ puisque les éléments de $K_\infty$ sont autoduaux et on a ainsi défini une application $\Z$-bilinéaire \emph{symétrique}.

On note encore $\mathrm{B}_\infty^F(\pi,\pi')$ pour $\mathrm{B}_\infty^F(\Ll(\pi_\infty),\Ll(\pi'_\infty))$.
\end{defi}
\begin{itemize}
\item Il reste alors les contributions $N^{\frac{s}{2}}$ de chacune des fonctions $\Lambda$ et $\Lambda^\vee$ (modifiées), qui font donc apparaître un $F(0) \log (N^\frac{1}{2} N^\frac{1}{2})$.
\end{itemize}

On a donc :
\begin{thm} \emph{Formule explicite pour les fonctions $\Lambda$ de paires}\label{Formule explicite pour les fonctions Lambda de paires}

Soit $\{\pi,\pi'\}$ une paire de représentations automorphes cuspidales algébriques du groupe linéaire, soit $F$ une fonction test. Alors
\[
\mathrm{B}_f^F(\pi,\pi')+\mathrm{B}_\infty^F(\pi,\pi')+\frac{1}{2}\mathrm{Z}^F(\pi,\pi')=\Phi_F(0)\delta_{\pi,\pi'}+\frac{1}{2} F(0) \log {\rm N}(\pi^\vee \times \pi')
\]
\end{thm}




\section{Choix de fonction test}\label{Choix de fonction test}
\subsection{Propriétés de positivité}
Tous les termes de la formule explicite pour les fonctions $\Lambda$ de paires sont réels. À défaut de pouvoir les calculer explicitement, il peut être utile de connaître leur signe, ce qui donnera lieu à de fructueuses inégalités.

Nous rappelons donc deux propriétés de positivité pour une fonction test $F$. La première consiste simplement à supposer que cette fonction prend des valeurs positives. Les valeurs de la fonction $F$ apparaissent en effet dans la quantité $\B_f^F$.

La seconde est reliée à la quantité $\Phi_F$ qui intervient notamment dans la quantité ${\rm Z}^F$.
\begin{defi}
Soit $F$ une fonction test. On dit que $F$ est $\Phi$-positive si :
\[
\re \Phi_F(s) \geq 0 \text{ pour tout } s \in \C \text{ tel que } 0\leq \re (s) \leq 1.
\]
\end{defi}
C'est la propriété notée (T4) dans \cite{Chen-Lannes} IX.3.4. En particulier, nous avons le 
\begin{lemme}\label{lemme_F_Phi_positive}
Soit $F$ une fonction test $\Phi$-positive. Alors, pour toutes $\pi$ et $\pi'$, 
on a ${\rm Z}^F(\pi,\pi')\geq 0$.
\end{lemme}

La propriété notée (POS) dans \cite{Chen-Ta}, Définition 2.1, correspond à une fonction test positive \emph{et} $\Phi$-positive.

\subsection{La fonction d'Odlyzko}\label{La fonction d'Odlyzko}
On rappelle ici la définition et les propriétés de la fonction d'Odlyzko, 
qui sera la seule que l'on utilisera en pratique (notons néanmoins qu'elle dépend d'un paramètre $\lambda \in \R_{>0}$).  Cette fonction (ou famille de fonctions) s'avère optimale en un certain sens (\emph{cf.} \cite{Poit} §3) pour utiliser la formule explicite.\ps 

On définit la fonction $u: \R \rightarrow \R$ par 
\[
u(x)= \begin{cases}
\cos(\pi x) &\text{ si } |x| \leq \frac{1}{2} \\
0 &\text{ sinon}
\end{cases}
\]
puis on pose $g=2 u \ast u$ (produit de convolution de ${\rm L}^1(\R)$).

On obtient immédiatement que $g$ est paire, positive, à support dans $[-1;1]$, de classe ${\rm C}^2$, vérifiant $g(0)=1$ et $g$ est évidemment une fonction test (au sens de la Définition \ref{defi fonction explicite}) de niveau $c$ pour tout réel positif $c$.

On définit alors, pour $\lambda \in \R_{>0}$, la fonction d'Odlyzko de paramètre $\lambda$ :
\begin{equation}\label{Odlyzko}
{\rm F}_\lambda(x)=\frac{g(x/\lambda)}{\mathrm{ch}(x/2)}.
\end{equation}

Ce choix vient du résultat suivant, connu d'Odlyzko et de Poitou (\cite{Poit}).
\begin{lemme}\label{Positivité_Z_F}
Soit $h: \R \rightarrow \R$ une fonction paire, sommable et de carré sommable. On suppose que $h$ est une fonction test de niveau $c$. Sa transformée de Fourier $\widehat{h} \geq 0$ est bien définie et à valeurs réelles sur $\R$. Supposons $\widehat{h} \geq 0$ sur $\R$ et considérons la fonction $F : \R \rightarrow \R$ définie par 
\[
F(x)=\frac{h(x)}{\mathrm{ch}(x/2)}.
\]
Alors $F$ est une fonction test de niveau $c$ et $\Phi$-positive. 

\end{lemme}
\begin{proof}
Le fait que $F$ soit une fonction test de niveau $c$ est immédiat à partir de la régularité et de la minoration de la fonction cosinus hyperbolique sur $\R$. Nous renvoyons à (\cite{Chen-Lannes} IX. Lemme 3.15) pour le reste de la démonstration. 
\end{proof}

En remarquant que la transformée de Fourier de la fonction $g$ est $2 \widehat{u}^2$ qui est positive sur $\R$ car $u$ est réelle, paire et sommable, on obtient que ${\rm F}_\lambda$ est $\Phi$-positive. Par ailleurs, elle est évidemment positive et on remarque également que ${\rm F}_\lambda(0)=1$.

%



\subsection{Une spectaculaire illustration}\label{Une spectaculaire illustration}
Nous reprenons ici la discussion \og Cas $n=2$ ou $w\leq 10$ \fg{} du § IX.3.18 de \cite{Chen-Lannes} qui nous paraît particulièrement significative de la \emph{précision} de la formule explicite.

La fonction ${\rm F}_\lambda$ est à support compact, plus exactement est nulle en dehors de $[-\lambda,\lambda]$. Ainsi, la série définissant $\mathrm{B}_f^{{\rm F}_\lambda}$ dans les notations du Théorème \ref{Formule explicite pour les fonctions Lambda de paires} est en fait une somme finie ; plus exactement n'interviennent que les indices $p$ et $k$ tels que $k \log p \in [-\lambda,\lambda]$, soit encore $p^k < e^\lambda$. En particulier, si on choisit $\lambda \leq \log 2$, alors la somme est vide, si bien que $\mathrm{B}_f^{{\rm F}_\lambda}$ est nul.

Soit $\pi$ une représentation automorphe cuspidale algébrique de $\GL_n$ de conducteur 1 (cas traité dans \cite{Chen-Lannes}, remarquer que \emph{l'on ne suppose rien sur $n$}) et appliquons ceci à la paire $\{\1,\pi\}$ où $\mathbf{1}$ est la représentation triviale. On obtient alors, pour tout $\lambda \leq \log 2$ :
\[
\mathrm{B}_\infty^{{\rm F}_\lambda}(\1,\pi) +\frac{1}{2}\mathrm{Z}^{{\rm F}_\lambda}(\1,\pi)=\Phi_{{\rm F}_\lambda}(0)\delta_{\1,\pi}.
\]

Si l'on cherche d'autres représentations que la représentation triviale, le terme de droite est nul. De plus, par le Lemme \ref{Positivité_Z_F}, la quantité $\mathrm{Z}^{{\rm F}_\lambda}(\1,\pi)$ est réelle positive. Enfin, $\mathrm{B}_\infty^{{\rm F}_\lambda}(\1,\pi)=\mathrm{J}_{{\rm F}_\lambda} (\Ll(\pi_\infty))$, si bien que l'on obtient :
\[
\mathrm{J}_{{\rm F}_\lambda} (\Ll(\pi_\infty)) \leq 0
\]

À $w$ fixé, la fonction $\lambda \mapsto \mathrm{J}_{{\rm F}_\lambda}(I_w)$ est croissante (revenir à la définition, en remarquant que ${\rm F}_\lambda \leq {\rm F}_{\lambda'}$ si $\lambda \leq \lambda'$), on tirera donc le meilleur parti de cette inégalité en choisissant $\lambda=\log 2$.

C'est ce qu'on fait en calculant les valeurs pour les représentations irréductibles de $\W_\R$ triviales sur $\R_{>0}$ que l'on distingue selon leur parité : les $I_w$ avec $w>0$ impair d'un côté, les $I_w$ avec $w>0$ pair ainsi que $\1$ et $\eps_{\C/\R}$ de l'autre.
\begin{center}
\begin{tabular}{ c || *{5}{c|} c}
   $V$ & $I_1$ & $I_3$ & $I_5$ & $I_7$ & $I_9$ & $I_{11}$\\
   \hline
   $\mathrm{J}_{F_{\log 2}} (V)$ & 0.848 & 0.611 & 0.408 & 0.230 & 0.074 & -0.065  \\
 \end{tabular}

Table à $10^{-3}$ près pour $V$ \og impair \fg{} 
\end{center}

\begin{center}
\begin{small}
\begin{tabular}{ c || *{7}{c|} c}
   $V $ & $\1$ & $\eps_{\C/\R}$ & $I_2$ & $I_4$ & $I_6$ & $I_8$ & $I_{10}$ & $I_{12}$\\
   \hline
   $\mathrm{J}_{F_{\log 2}} (V)$ & 0.560 & 0.421 & 0.725 & 0.506 & 0.316 & 0.150 & 0.003 & -0.129\\
 \end{tabular}
\end{small}
Table à $10^{-3}$ près pour $V$ \og pair \fg{} 
\end{center}

On voit ainsi qu'en conducteur 1, le plus petit poids motivique qui puisse apparaître (et ce, \emph{quelle que soit la dimension}) pour une représentation non triviale est 11.

Par le Lemme \ref{lemme_carac_central_triv}, on peut voir $\pi$ comme une représentation automorphe cuspidale algébrique de $\PGL_n$. Dans le cas où $n=2$, on sait alors, via les résultats rappelés au paragraphe \ref{En niveau Gamma_0_N} qu'une représentation automorphe cuspidale de $\PGL_2$ de poids $\{\pm \frac{w}{2}\}$ correspond à une forme modulaire parabolique normalisée pour $\SL_2(\Z)$ de poids modulaire $w+1$.

Cela redémontre qu'il n'existe pas de forme modulaire parabolique de poids modulaire <12 pour $\SL_2(\Z)$. La précision de la formule est d'ailleurs étonnante puisque l'on sait qu'il existe bien une telle forme en poids modulaire 12. Nous renvoyons à la discussion \cite{Chen-Lannes} Chap. IX §3.18 qui montre davantage et détermine entre autres les dimensions des espaces de formes modulaires paraboliques (pour les poids 12 à 22) à l'aide de la formule explicite (nous détaillons au paragraphe \ref{Un exemple en conducteur 1} le cas du poids modulaire 14).

\paragraph{En conducteur 2}

Nous considérons maintenant une représentation $\pi$ automorphe cuspidale algébrique de $\GL_n$ de conducteur 2 (là encore, sans rien supposer sur $n$). La même discussion que ci-dessus pour la paire $\{\1,\pi\}$ amène, pour tout $\lambda \leq \log 2$ :
\[
\mathrm{B}_\infty^{{\rm F}_\lambda}(\1,\pi) +\frac{1}{2}\mathrm{Z}^{{\rm F}_\lambda}(\1,\pi)=\Phi_{{\rm F}_\lambda}(0)\delta_{\1,\pi} + \frac{1}{2} \log {\rm N}(\1^\vee \times \pi)
\]
d'où l'on tire
\begin{equation}\label{inegalite_J_F_cond_2}
\mathrm{J}_{{\rm F}_\lambda} (\Ll(\pi_\infty)) \leq \frac{\log 2}{2}.
\end{equation}

\begin{prop}\label{prop_facile_w<6}
Il n'existe pas de représentation automorphe cuspidale algébrique de $\GL_n$ de conducteur $2$, de poids motivique inférieur strictement à $6$.
\end{prop}
\begin{proof}
On utilise les tables précédentes et la valeur de $\log 2 \simeq 0,693$ à $10^{-3}$ près.
\end{proof}

\section{Un résultat de finitude}\label{Un résultat de finitude}

On se donne à nouveau une représentation $\pi$ automorphe cuspidale algébrique de $\GL_n$ sur $\Q$, sans rien supposer sur la dimension $n$, ni sur le conducteur ${\rm N}(\pi)$ de la représentation. Plutôt que de s'intéresser à la paire $\{\1,\pi\}$ comme au paragraphe précédent, on peut considérer la (multi-)paire $\{\pi,\pi\}$. On obtient alors :
\[
\mathrm{B}_f^F(\pi,\pi)+\mathrm{B}_\infty^F(\pi,\pi)+\frac{1}{2}\mathrm{Z}^F(\pi,\pi)=\Phi_F(0)+\frac{1}{2} F(0) \log {\rm N}(\pi^\vee \times \pi)
\]
pour toute fonction test $F$.

\begin{lemme}\label{lemme_F_positive_B_f}
Soit $F$ une fonction test \emph{positive}. Alors, quelle que soit $\pi$, on a $\B_f^F(\pi,\pi) \geq 0$.
\end{lemme}
\begin{proof}
Notons $L(s,\pi^\vee \times \pi)$ la partie finie de la fonction $\Lambda$, \ie le produit sur tous les $p$ des $L_p(s,\pi^\vee \times \pi)$.
Alors le Lemme 2.a de \cite{HR} énonce que, pour $\re (s)>1$,
\[
-\frac{L'}{L}(s, \pi^\vee \times \pi)=\sum_p \sum_{k=1}^{+\infty} x_{p^k} (\pi^\vee \times \pi) \frac{\log p}{p^{ks}},
\]
avec $x_{p^k} (\pi^\vee \times \pi)$ réel positif pour tout $p$ et pour tout $k$. Cela revient à dire, dans les notations de la Définition \ref{defi préfonction lambda} que, pour tout $p$ et pour tout $k$, la quantité $\sum_{j=1}^{M'} \alpha_j(p)^k$ est réelle positive. 
Si l'on suppose de plus $F$ positive, alors la quantité
\[
\mathrm{B}_f^F(\pi,\pi)=\sum_p \sum_{j=1}^{M'}\sum_{k=1}^{+ \infty} F(k \log p) \frac{\log(p)}{p^{\frac{k}{2}}}\alpha_j(p)^k
\]
 est bien réelle (ce qu'on savait déjà) positive.
\end{proof}


La combinaison de ce dernier lemme avec le Lemme \ref{lemme_F_Phi_positive} nous donne l'essentielle

\begin{prop} \emph{(\cite{Chen-HM}, Proposition 4.4)} \label{Prop_inegalite_pi_times_pi}

Soit $\pi$ une représentation automorphe cuspidale algébrique de $\GL_n$ sur $\Q$ et soit $F$ une fonction test positive et $\Phi$-positive. Alors
\begin{equation}\label{inegalite_pi_times_pi}
\mathrm{B}_\infty^F (\pi,\pi) \leq \Phi_F(0) +\frac{1}{2} F(0) \log {\rm N}(\pi^\vee \times \pi).
\end{equation}
\end{prop}

On sait que la fonction d'Odlyzko vérifie les hypothèses requises (pour tout paramètre). Ce sera la seule fonction test que nous utiliserons en pratique.

\subsection{Cas du conducteur 1}\label{Cas du conducteur 1}
Dans le cas où l'on suppose de plus que $\pi$ est de conducteur 1, le deuxième terme du membre de droite de \eqref{inegalite_pi_times_pi} est nul et on retombe sur le travail de Chenevier et Lannes \cite{Chen-Lannes}. En particulier, les Lemmes IX.3.34 et IX.3.36 \emph{op. cit.} nous disent que $\mathrm{B}_\infty^F$, vue comme forme $\Z$-bilinéaire symétrique sur $K_\infty$, est définie positive pour un bon choix de $F$ sur les sous-$\Z$-modules $K_\infty^{\leq 21}$ et $K_\infty^{\leq 22}$ respectivement. On a donc une équation du type $q(V) \leq C$ où $q$ est une forme quadratique définie positive sur un $\Z$-module libre de type fini et $C$ est une constante. Il y a donc un nombre fini de $V$ qui conviennent, et on se restreint en fait à ceux qui ont des coefficients positifs (\ie effectifs au sens de la Définition \ref{defin_K_infty_leq_w}).

Or, un résultat célèbre de Harish-Chandra dans \cite{HC68} démontre qu'à $V$ (effectif) et à conducteur $N$ fixés, il n'existe qu'un nombre fini de $\pi$ automorphes discrètes de conducteur $N$ telles que $\Ll(\pi_\infty) \simeq V$.

La combinaison de ces deux finitudes nous donne donc le
\begin{thm} \emph{(\cite{Chen-Lannes}, Théorème F)}

Il existe un nombre fini (à torsion près) de représentations automorphes cuspidales algébriques de $\GL_n$ sur $\Q$ de conducteur 1 et de poids motivique $\leq 22$.
\end{thm}

Ce théorème est en fait donné \emph{op. cit.} avec une liste explicite desdites représentations. Il est généralisé au poids motivique 23 (Theorem A) et même 24  (Theorem B) sous l'hypothèse de Riemann généralisée dans \cite{Chen-HM} en ce qui concerne la finitude, la liste explicite des représentations se trouvant, sous certaines hypothèses supplémentaires, dans \cite{Chen-Ta} (Theorems 3, 4 and 5).


\subsection{Cas d'un conducteur quelconque}
Dans le cas d'un conducteur quelconque, il faut pouvoir contrôler la quantité ${\rm N}(\pi^\vee \times \pi)$ en fonction du conducteur ${\rm N}(\pi)$ de $\pi$, représentation automorphe cuspidale de $\GL_n$. On a, par définition  :
\[
{\rm N}(\pi^\vee \times \pi)=\prod_p p^{\aar_p (\pi_p^\vee \times \pi_p)}.
\]

Or, l'inégalité d'Henniart \eqref{exposant_prod_tens_rep_locales} nous donne, pour tout $p$,
\[
\aar_p (\pi_p^\vee \times \pi_p) \leq (2n-1) \aar_p(\pi_p)
\]
puisqu'on a $ \aar_p(\pi_p^\vee)= \aar_p(\pi_p)$, comme rappelé à la Définition \ref{Cond_rep_locales_def}. On a finalement :
\begin{align*}
\log {\rm N}(\pi^\vee \times \pi) &= \sum_p  \aar_p(\pi_p^\vee \times \pi_p) \log p\\
						   & \leq \sum_p (2n-1) \aar_p(\pi_p) \log p\\
						   & \leq (2n-1) \log {\rm N}(\pi).
\end{align*}

Si l'on raisonne à conducteur $N$ fixé, l'inégalité \eqref{inegalite_pi_times_pi} devient :
\[
\mathrm{B}_\infty^F (\pi,\pi) \leq \Phi_F(0) +\frac{1}{2} F(0) (2n-1)\log N.
\]

Le $n$ est celui du $\GL_n$ dont $\pi$ est une représentation automorphe. On remarque que c'est aussi la dimension de la représentation semi-simple $\Ll(\pi_v)$ de $\WD_{\Q_v}$ donnée par la correspondance de Langlands locale à chaque place $v$ de $\Q$.

En particulier, puisqu'on s'intéresse à des représentations algébriques, la représentation de $\W_\R$ associée est un élément de $K_\infty$ et l'application qui à un élément de $K_\infty$ associe sa dimension est \emph{linéaire}.
Nous avons donc une équation du type $q(V) \leq C + L(V)$ où $q$ (resp. $L$) est une forme quadratique (resp. linéaire) sur le $\Z$-module libre $K_\infty$ et $C$ est une constante. 
Si l'on se restreint à un sous-$\Z$-module de rang fini sur lequel la forme quadratique est définie positive, alors on n'a encore qu'un nombre fini de $V$ qui conviennent (une forme quadratique définie positive croît plus vite que n'importe quelle forme linéaire). Via le même résultat de Harish-Chandra, on obtient ainsi le Théorème A de \cite{Chen-HM}.

\subsection{Cas du conducteur 2}
Dans le cas qui nous intéresse principalement, celui du conducteur 2, on sait qu'on n'a affaire qu'à un seul type de ramification (c'est le Lemme \ref{lemme_cond_2_implique_type_I}), on peut donc calculer les conducteurs de paires explicitement. Cela ne change évidemment rien à l'énoncé \emph{théorique} de finitude, mais pour les calculs pratiques que nous ferons, toute amélioration est bonne à prendre.

Au-delà du seul conducteur ${\rm N}(\pi^\vee \times \pi)$ où $\pi$ est une représentation automorphe algébrique cuspidale de conducteur 2, nous considérons toutes les paires faisant intervenir des représentations automorphes cuspidales du groupe linéaire de conducteur 1 ou 2. On raisonne alors sur l'exposant d'Artin de la composante locale en 2 de chacune de ces représentations.

\begin{prop}\label{prop_cond_2}
Soient $\pi$ et $\pi'$ deux représentations automorphes de $\GL_m$ et $\GL_{m'}$ respectivement, de conducteur $1$.

Soient $\varpi$ et $\varpi'$ deux représentations automorphes de $\GL_n$ et $\GL_{n'}$ respectivement, de conducteur $2$.

Alors :
\begin{itemize}
\item $\aar_2(\pi \times \pi')=0$ ;
\item $\aar_2(\pi \times \varpi)=m$ ;
\item $\aar_2(\varpi \times \varpi')=n+n'-2$.
\end{itemize}
\end{prop}

\begin{proof}
On note bien sûr $\aar_2(\pi \times \pi')$ pour $\aar_2(\pi_2 \times \pi'_2)$ \emph{et caetera similia}.

Nous avons le résultat bien connu $\aar_2(\pi \times \pi')=0$, dont on remarque d'ailleurs que c'est la borne donnée par l'inégalité d'Henniart.\ps 

On sait que $\varpi_2$ est de type (I) au sens de la Proposition-Définition \ref{prop_types_I_et_II} (avec un léger abus, \emph{stricto sensu} c'est $\varpi$ qui est de type (I)). Il suffit donc de considérer $\Ll(\varpi_2)$ restreint au sous-groupe $\{1\}\times \SU(2)$ du groupe de Weil-Deligne de $\Q_2$. 
On a alors :
\begin{align*}
(\Ll(\pi_2) \otimes \Ll(\varpi_2))_{|\SU(2)} &= m\1  \otimes ((n-2)\1\oplus U_2) \\
&= m(n-2)\1 \oplus m U_2
\end{align*}
dont l'exposant d'Artin est égal à $m(n-2)(1-1)+m(2-1)=m$ (c'est la formule \eqref{awd_irr} et l'additivité de l'exposant d'Artin). Cela donne donc $\aar_2(\pi \times \varpi)=m$, qui est là encore la borne donnée par l'inégalité d'Henniart.\ps 

Pour le troisième cas, considérons de même 
\begin{align*}
(\Ll(\varpi_2) \otimes \Ll(\varpi'_2))_{|\SU(2)} &= ((n-2)\1 \oplus U_2)  \otimes ((n'-2)\1\oplus U_2) \\
&= (n-2)(n'-2)\1 \oplus (n+n'-4) U_2 \oplus (U_2 \otimes U_2) \\
&= (n-2)(n'-2)\1 \oplus (n+n'-4) U_2 \oplus (\1 \oplus U_3) \\
&= [(n-2)(n'-2)+1]\1 \oplus (n+n'-4) U_2 \oplus U_3 
\end{align*}
dont l'exposant d'Artin est égal à $(n+n'-4)(2-1)+ (3-1)=n+n'-2$. Cela donne donc $\aar_2(\varpi \times \varpi')=n+n'-2$, là où la borne donnée par l'inégalité d'Henniart est $n+n'-1$.
\end{proof}

On obtient donc la précision suivante du Theorem A de \cite{Chen-HM} (voir aussi la Proposition \ref{Prop_inegalite_pi_times_pi}). 
\begin{cor}\label{cor_finitude_cond_2}

Soit $w$ un entier inférieur à $23$. Alors il n'existe qu'un nombre fini de $V \in K_\infty$ 
tels que $V \simeq \pi_\infty$ 
où $\pi$ est une représentation automorphe cuspidale algébrique de $\GL_n$ de conducteur $2$ et de poids motivique $w$.
En particulier, $n=\dim V$ et $V \in K_\infty^{\leq w}$. De tels $V$ vérifient :
\[
\mathrm{B}_\infty^F (V,V) \leq \Phi_F(0) + F(0) (\dim(V)-1)\log 2,
\]
pour toute fonction test $F$ positive et $\Phi$-positive.
\end{cor}
\begin{proof}
Par le Lemme \ref{lemme_pi_infty_i_w}, de tels $V$ appartiennent plus précisément à $K_\infty^{\leq w}$ et sont effectifs. La finitude vient alors du fait que, pour un bon choix de $F$ (on renvoie à \cite{Chen-HM}), $\mathrm{B}_\infty^F$ est définie positive sur le réseau $K_\infty^{\leq w}$. On impose ici de plus que la composante en $I_w$ soit non nulle (puisque $\pi$ est de poids motivique exactement $w$), ce qui ne fait que renforcer la finitude. Cette dernière une fois établie, l'inégalité vaut alors pour \emph{toute} fonction test positive et $\Phi$-positive.
\end{proof}

\section{Raffinements}
À partir de maintenant, on se restreint aux représentations automorphes cuspidales algébriques de $\GL_n$ de conducteur 1 ou 2. Cela nous permet d'alléger les notations, quoique la discussion vaille de façon beaucoup plus générale. Le cas du conducteur $p$ de type (I) se traite rigoureusement de la même manière (\cf §\ref{En conducteur p>2}).

\begin{defi}
Soit $n$ un entier naturel strictement positif. On désigne par $\Pi_{\rm alg}^1(\GL_n)$ (resp. $\Pi_{\rm alg}^2(\GL_n)$) l'ensemble des représentations automorphes cuspidales algébriques de $\GL_n$ et de conducteur $1$ (resp. $2$). On pose alors 
\[
\Pi_{\rm alg}^i=\coprod_{n\geq 1} \Pi_{\rm alg}^i (\GL_n),
\]
avec $i \in \{1,2\}$ et $\Pi_{\rm alg}=\Pi_{\rm alg}^1 \coprod \Pi_{\rm alg}^2$.
\end{defi}

On fixe une fonction test $F$. La formule explicite s'écrit alors, pour un couple $(\pi,\pi') \in \Pi_{\rm alg} \times \Pi_{\rm alg}$ :
\[
\mathrm{B}_f^F(\pi,\pi')+\mathrm{B}_\infty^F(\pi,\pi')+\frac{1}{2}\mathrm{Z}^F(\pi,\pi')=\Phi_F(0)\delta_{\pi,\pi'}+\frac{\log 2}{2} F(0) \aar_2(\pi^\vee \times \pi').
\]

On peut étendre chacune des cinq fonctions de $\Pi_{\rm alg} \times \Pi_{\rm alg}$ dans $\R$ envoyant respectivement $(\pi,\pi')$ sur $\mathrm{B}_f^F(\pi,\pi'), \mathrm{B}_\infty^F(\pi,\pi'), \mathrm{Z}^F(\pi,\pi'), \delta_{\pi,\pi'}, \aar_2(\pi^\vee \times \pi')$ en une application bilinéaire symétrique sur le groupe libre $\Z[\Pi_{\rm alg}]$ et même sur le $\R$-espace vectoriel $\R[\Pi_{\rm alg}]$. On a alors l'égalité entre formes bilinéaires :
\begin{equation}\label{Formule_explicite_fbs}
\mathrm{B}_f^F+\mathrm{B}_\infty^F+\frac{1}{2}\mathrm{Z}^F=\Phi_F(0) \delta + \frac{\log 2}{2} F(0) \aar_2.
\end{equation}

On dira enfin qu'un élément de $\R[\Pi_{\rm alg}]$ est effectif s'il est à coefficients positifs.

\subsection{La partie finie}\label{La partie finie}
Le Lemme \ref{lemme_F_positive_B_f} nous a montré que, pour $\pi \in \Pi_{\rm alg}$ et $F$ fonction test positive, on avait $\mathrm{B}_f^F(\pi,\pi) \geq 0$. Nous pourrions en déduire aisément que, sous ces mêmes hypothèses, la \emph{forme bilinéaire symétrique} $\B_f^F$ définie sur $\R[\Pi_{\rm alg}]$ est positive. Puisque nous considérons des cas particuliers (représentations locales non ramifiées ou ramifiées de type (I) au sens de la Proposition-Définition \ref{prop_types_I_et_II}), il est intéressant d'étudier en détail cette forme bilinéaire et ainsi de \emph{raffiner} la positivité.\ps 

On se donne donc une fonction test \emph{positive} $F$ et on omet dans la suite de ce paragraphe l'exposant $F$. La première chose à remarquer est que l'on peut décomposer $\mathrm{B}_f$ selon chaque nombre premier $p$. On pose, avec les notations du paragraphe \ref{Énoncé particulier}, pour un nombre premier $p$, et pour deux représentations $\pi$ et $\pi'$ de $\Pi_{\rm alg}$ :
\begin{equation}\label{defi_B_p}
\mathrm{B}_p(\pi,\pi')= \sum_{j=1}^{M'} \sum_{k=1}^{+ \infty} F(k \log p) \frac{\log(p)}{p^{\frac{k}{2}}} \re(\alpha_j(p)^k),
\end{equation}
ce qui définit immédiatement $\mathrm{B}_p$ comme une forme bilinéaire symétrique sur $\R[\Pi_{\rm alg}]$. On a alors $\B_f=\sum_p \B_p$ en tant que forme bilinéaire symétrique.\medskip

\emph{Remarque :} Si $F$ est de support compact, alors la quantité $\mathrm{B}_p(\pi,\pi')$ est nulle pour tout $p$ suffisamment grand.

\begin{prop}\emph{(\cite{Chen-Lannes}, Corollaire 3.11)} \label{B_p_sym_pos}

Soit $F$ une fonction test positive.
Alors, à toute place finie $p\neq 2$, $\B_p$ est une forme bilinéaire symétrique \emph{positive} sur $\R[\Pi_{\rm alg}]$.
\end{prop}
\begin{proof}
Soit $\pi \in \R[\Pi_{\rm alg}]$. Alors $\pi$ s'écrit sous la forme $\sum_{i=1}^r \lambda_i \pi_i$ avec $\lambda_i \in \R$ et $\pi_i \in \Pi_{\rm alg}$. 
En dernier lieu, il faut donc calculer $\B_p(\pi_i,\pi_j)$, pour tous $i,j \in \{1,\cdots,r\}$.

Cette dernière quantité ne dépend que des représentations locales $\pi_{i,p}$ et $\pi_{j,p}$, non ramifiées par hypothèse. La correspondance de Langlands associe donc $\Ll(\pi_{i,p})=\bigoplus_{a=1}^{n_i} \chi_a$ somme de caractères non ramifiés de $\Q_p^\times$. De même, $\Ll(\pi_{j,p})=\bigoplus_{b=1}^{n_j} \psi_b$. On a donc, selon les calculs effectués au §\ref{prop_calcul_fonction_L_locale} :
\[
L_p(s,\pi_i^\vee \times \pi_j)=\prod_{a=1}^{n_i} \prod_{b=1}^{n_j} \frac{1}{1-\overline{\chi_a(p)}\psi_b(p) p^{-s}}=\frac{1}{\det(1-p^{-s}\overline{c_p(\pi_i)}\otimes c_p(\pi_j))},
\]
en utilisant de nouveau le fait que $\pi_i$ est unitaire à toute place et avec la notation $c_p$ de \eqref{fonction_L_rep_loc_nr}.\ps 

Alors 
\begin{align*}
\B_p(\pi_i,\pi_j)&=\sum_{a=1}^{n_i} \sum_{b=1}^{n_j} \sum_{k=1}^{+ \infty} F(k \log p) \frac{\log(p)}{p^{\frac{k}{2}}} \re(\overline{\chi_a(p)}^k \psi_b(p)^k) \\
				&=\sum_{k=1}^{+ \infty} F(k \log p) \frac{\log(p)}{p^{\frac{k}{2}}}
\re \left[\overline{\mathrm{tr} (c_p(\pi_i)^k)} \mathrm{tr} (c_p(\pi_j)^k) \right],
\end{align*}
puisque $c_p(\pi_i)$ est la classe de conjugaison semi-simple dans $\GL_{n_i}(\C)$ de l'élément $\mathrm{diag}(\chi_1(p), \cdots, \chi_{n_i}(p))$, et de même pour $c_p(\pi_j)$.\ps 

%
%

On a donc :
\begin{align*}
\B_p (\pi,\pi) &=\sum_{i,j=1}^r \lambda_i \lambda_j \B_p(\pi_i,\pi_j) \\
			   &=\sum_{i,j=1}^r \lambda_i \lambda_j \sum_{k=1}^{+ \infty} F(k \log p) \frac{\log(p)}{p^{\frac{k}{2}}} \re \left[\overline{\mathrm{tr} (c_p(\pi_i)^k)} \mathrm{tr} (c_p(\pi_j)^k) \right] \\
			   &=\sum_{k=1}^{+ \infty} F(k \log p) \frac{\log(p)}{p^{\frac{k}{2}}} \re \sum_{i,j=1}^r  \left[\overline{\lambda_i \mathrm{tr} (c_p(\pi_i)^k)}  \lambda_j \mathrm{tr} (c_p(\pi_j)^k) \right] \\
			   &=\sum_{k=1}^{+ \infty} F(k \log p) \frac{\log(p)}{p^{\frac{k}{2}}} \left| \sum_{i=1}^r  \lambda_i \mathrm{tr} (c_p(\pi_i)^k) \right|^2, 
\end{align*}
quantité qui est bien positive sous l'hypothèse que $F$ l'est.
\end{proof}

Cette proposition nous montre que, dans le cadre du conducteur 1 étudié dans \cite{Chen-Lannes} et dans \cite{Chen-Ta}, la forme bilinéaire symétrique $\B_f$ est positive sur $\R[\Pi_{\rm alg}^1]$. Nous voulons ici considérer également les représentations de conducteur 2, il nous faut donc étudier en détail ce qu'il se passe à la place 2.


\begin{defi}\label{defi_B_2_calc}
Soient $\pi$ et $\pi'$ deux représentations automorphes de $\GL_m$ et $\GL_{m'}$ respectivement, de conducteur $1$.

Soient $\varpi$ et $\varpi'$ deux représentations automorphes de $\GL_n$ et $\GL_{n'}$ respectivement, de conducteur $2$.

On sait que $\varpi_2$ et $\varpi'_2$ sont de type (I) donc on peut écrire :
\begin{align*}
\Ll(\varpi_2)&=\eta_1 \oplus \cdots \oplus \eta_{n-2} \oplus (\psi \otimes U_2), \\
\Ll(\varpi'_2)&=\eta'_1 \oplus \cdots \oplus \eta'_{n'-2} \oplus (\psi' \otimes U_2),
\end{align*}
avec tous les caractères intervenant non ramifiés.\ps 

On définit alors l'application $\B_2^{\rm calc}$ par :
\begin{align*}
&\B_2^{\rm calc}(\pi,\pi')=0, \\
&\B_2^{\rm calc}(\pi,\varpi)=\B_2^{\rm calc}(\varpi,\pi)=0, \\
&\B_2^{\rm calc}(\varpi,\varpi')=\sum_{k=1}^{+ \infty} F(k \log 2) \frac{\log(2)}{2^{\frac{k}{2}}} \re (\overline{\psi(2)^k}\psi'(2)^k).
\end{align*}
Cette application est symétrique et s'étend en une forme bilinéaire symétrique sur $\R[\Pi_{\rm alg}]$.
\end{defi}

Il faut remarquer que la quantité $\B_2^{\rm calc}$ est liée aux caractères non ramifiés $\psi$ et $\psi'$ que l'on peut également \emph{lire} sur le facteur epsilon (\cf \eqref{eps_local_varpi_p} et Proposition \ref{prop_facteur_eps_p_paire}).

\begin{lemme}\label{B_2_calc_pos}
Soit $F$ une fonction test positive. Alors, la forme bilinéaire symétrique $\B_2^{\rm calc}$ est positive sur $\R[\Pi_{\rm alg}]$.
\end{lemme}
\begin{proof}
Comme dans la preuve de la Proposition \ref{B_p_sym_pos}, on se donne $\pi=\sum \lambda_i \pi_i + \sum \mu_j \varpi_j$ où pour tout $i,\, \lambda_i \in \R ,\, \pi_i \in \Pi_{\rm alg}^1$ et pour tout $j,\, \mu_j \in \R ,\, \varpi_j \in \Pi_{\rm alg}^2$.

On peut développer linéairement $\B_2^{\rm calc} (\pi,\pi)$ et, par définition seuls les termes qui font intervenir deux représentations de conducteur 2 sont non nuls.
On a alors, en reprenant les notations de la Définition \ref{defi_B_2_calc} :
\begin{align*}
\B_2^{\rm calc} (\pi,\pi) &=\sum_{j,j'} \mu_j \mu_{j'} \B_2^{\rm calc}(\varpi_j,\varpi_{j'}) \\
			   &=\sum_{j,j'} \mu_j \mu_{j'} \sum_{k=1}^{+ \infty} F(k \log 2) \frac{\log(2)}{2^{\frac{k}{2}}} \re (\overline{\psi_j(2)^k}\psi_{j'}(2)^k) \\
			   &=\sum_{k=1}^{+ \infty} F(k \log 2) \frac{\log(2)}{2^{\frac{k}{2}}} \re \sum_{j,j'} (\overline{\mu_j \psi_j(2)^k} \mu_{j'}\psi_{j'}(2)^k) \\
			   &=\sum_{k=1}^{+ \infty} F(k \log 2) \frac{\log(2)}{2^{\frac{k}{2}}} \left| \sum_{j} \mu_j \psi_j(2)^k \right|^2, 
\end{align*}
quantité qui est bien positive sous l'hypothèse que $F$ l'est.
\end{proof}


\begin{prop}
Soit $F$ une fonction test positive.
On peut décomposer la forme bilinéaire symétrique $\B_2$ définie sur $\R[\Pi_{\rm alg}]$ en somme de deux formes bilinéaires symétriques positives $\B_2^{\rm non\,calc}$ et $\B_2^{\rm calc}$, où $\B_2^{\rm calc}$ est donnée par la Définition \ref{defi_B_2_calc}.
\end{prop}
\begin{proof}
Les formes bilinéaires symétriques $\B_2$ et $\B_2^{\rm calc}$ ont été définies en \eqref{defi_B_p} et Définition \ref{defi_B_2_calc} respectivement. La forme bilinéaire symétrique $\B_2^{\rm non\,calc}$ est donc entièrement déterminée par leur différence. Il reste à voir qu'elle est positive.\ps 

Considérons d'abord ses valeurs sur la base naturelle de $\R[\Pi_{\rm alg}]$. Soient $\pi$ et $\pi'$ deux représentations automorphes de $\GL_m$ et $\GL_{m'}$ respectivement, de conducteur 1.
Soient $\varpi$ et $\varpi'$ deux représentations automorphes de $\GL_n$ et $\GL_{n'}$ respectivement, de conducteur 2.\ps 

On a $\B_2^{\rm non\,calc}(\pi,\pi')=\B_2(\pi,\pi')$ car $\B_2^{\rm calc}$ est nulle sur ce couple. De plus, $\pi$ et $\pi'$ sont non ramifiées en 2, et tout se déroule alors comme dans la Proposition \ref{B_p_sym_pos}.\ps 

On a $\B_2^{\rm non\,calc}(\pi,\varpi)=\B_2(\pi,\varpi)$ car $\B_2^{\rm calc}$ est nulle sur ce couple. Néanmoins, puisque $\varpi$ est ramifiée en 2, les calculs de la Proposition \ref{B_p_sym_pos} ne s'appliquent pas. On sait que l'on peut écrire :
\begin{align*}
\Ll(\pi_{2})&=\chi_1 \oplus \cdots \oplus \chi_{m}, \\
\Ll(\varpi_{2})&=\eta_1 \oplus \cdots \oplus \eta_{n-2} \oplus (\psi \otimes U_2),
\end{align*}
avec tous les caractères intervenant non ramifiés.

On a alors, selon la Proposition \ref{prop_calcul_fonction_L_locale} :
\begin{align*}
L_2(s,\pi^\vee \times \varpi)&=\prod_{a=1}^{m} \left( \prod_{b=1}^{n-2} \frac{1}{1-\overline{\chi_a(2)}\eta_b(2) 2^{-s}} \right) \frac{1}{1-\overline{\chi_a(2)}\psi(2) 2^{-\frac{1}{2}-s}} \\
							&=\frac{1}{\det(1-2^{-s}c_2(\pi)\otimes d_2(\varpi))},
\end{align*}
où l'on rappelle que, par analogie avec $c_2(\pi)$ qui désigne la classe de conjugaison semi-simple dans $\GL_{m}(\C)$ de $\mathrm{diag}(\chi_1(2),\cdots,\chi_{m}(2))$, on note $d_2(\varpi)$ la classe de conjugaison semi-simple dans $\GL_{n-1}(\C)$ de $\mathrm{diag}(\eta_1(2),\cdots,\eta_{n-2}(2),\psi(2)2^{-\frac{1}{2}})$ (\cf \eqref{fonction_L_rep_loc_nr}).

On peut donc écrire
\begin{align*}
\B_2^{\rm non\,calc}(\pi,\varpi)&=\sum_{a=1}^{m} \sum_{k=1}^{+ \infty} F(k \log 2) \frac{\log(2)}{2^{\frac{k}{2}}}    \re \left( \sum_{b=1}^{n-2}\overline{\chi_a(2)}^k \eta_b(2)^k + \overline{\chi_a(2)}^k 2^{-k/2}\psi(2)^k \right) \\
									&=\sum_{k=1}^{+ \infty} F(k \log 2) \frac{\log(2)}{2^{\frac{k}{2}}}
\re \left[\overline{\mathrm{tr} (c_2(\pi)^k)} \mathrm{tr} (d_2(\varpi)^k) \right].
\end{align*}\ps 

Enfin, il nous faut déterminer $\B_2^{\rm non\,calc}(\varpi,\varpi')=\B_2(\varpi,\varpi')-\B_2^{\rm calc}(\varpi,\varpi')$. Le second terme est donné par la Définition \ref{defi_B_2_calc}. On a :
\begin{align*}
\Ll(\varpi_{2})&=\eta_1 \oplus \cdots \oplus \eta_{n-2} \oplus (\psi \otimes U_2), \\
\Ll(\varpi'_{2})&=\eta'_1 \oplus \cdots \oplus \eta'_{n'-2} \oplus (\psi' \otimes U_2),
\end{align*}
puis, selon la Proposition \ref{prop_calcul_fonction_L_locale}
\[
L_2(\varpi_2\times \varpi'_2)=\frac{1}{\det(1-2^{-s}d_2(\varpi)\otimes d_2(\varpi'))}\cdot \frac{1}{1-\psi(2)\psi'(2)2^{-s}},
\]
%
%
d'où l'on tire
\begin{align*}
\B_2(\varpi,\varpi')&=\sum_{k=1}^{+ \infty} F(k \log 2) \frac{\log(2)}{2^{\frac{k}{2}}}
\re \left[\overline{\mathrm{tr} (d_2(\varpi)^k)} \mathrm{tr} (d_2(\varpi')^k)  + (\overline{\psi(2)}\psi'(2))^k\right] \\
						&=\B_2^{\rm non\,calc}(\varpi,\varpi')+\B_2^{\rm calc}(\varpi,\varpi')
\end{align*}
avec 
\[
\B_2^{\rm non\,calc}(\varpi,\varpi')=\sum_{k=1}^{+ \infty} F(k \log 2) \frac{\log(2)}{2^{\frac{k}{2}}}
\re \left[\overline{\mathrm{tr} (d_2(\varpi)^k)} \mathrm{tr} (d_2(\varpi')^k) \right].
\]

Ainsi, nous avons déterminé $\B_2^{\rm non\,calc}$ sur la base naturelle de $\R[\Pi_{\rm alg}]$, il nous reste à voir qu'elle est positive comme forme bilinéaire symétrique. On se donne donc, comme dans la preuve du Lemme \ref{B_2_calc_pos}, $\pi=\sum \lambda_i \pi_i + \sum \mu_j \varpi_j$ où pour tout $i \in \{1,\cdots r\}, \lambda_i \in \R , \pi_i \in \Pi_{\rm alg}^1$ et pour tout $j \in \{1,\cdots s\}, \mu_j \in \R , \varpi_j \in \Pi_{\rm alg}^2$. On trouve alors :
\[
\B_2^{\rm non\,calc}(\pi,\pi)=\sum_{k=1}^{+ \infty} F(k \log 2) \frac{\log(2)}{2^{\frac{k}{2}}} \left| \sum_{i=1}^r  \lambda_i \mathrm{tr} (c_2(\pi_i)^k) + \sum_{j=1}^s \mu_j \mathrm{tr} (d_2(\varpi_j)^k) \right|^2,
\]
qui est bien une quantité positive, sous l'hypothèse que $F$ l'est.

Ainsi $\B_2$, somme de deux formes bilinéaires symétriques positives, est positive.
\end{proof}

Nous avons finalement montré en détail la 

\begin{prop}\label{cor_B_f_pos}
Soit $F$ une fonction test positive.
La forme bilinéaire symétrique $\B_f$ est positive sur $\R[\Pi_{\rm alg}]$.
\end{prop} 

\subsection{La multiplicité de Taïbi}
\begin{defi}
Soit $V$ un élément effectif de $K_\infty$. On note $\m_1(V)$ (resp. $\m_2(V)$) le nombre de représentations $\pi$ dans $\Pi_{\rm alg}^1$ (resp. dans $\Pi_{\rm alg}^2$) vérifiant $\Ll(\pi_\infty) \simeq V$. Le résultat de Harish-Chandra déjà cité nous dit que $\m_1(V)$ et $\m_2(V)$ sont finis.
\end{defi}

Nous allons voir une version \emph{quantitative} de ce résultat

Fixons une fonction test $F$ positive et $\Phi$-positive (comme la fonction d'Odlyz\-ko). On suppose de plus que $F(0) \neq 0$ et donc, sans nuire à la généralité, que $F(0)=1$ (ce qui est le cas de la fonction d'Odlyzko). Nous omettons dans la suite de ce paragraphe les exposants et indices $F$.

\begin{thm} \emph{Multiplicité de Taïbi (Généralisation de \cite{Chen-Lannes} IX. Corollaire 3.13)} \label{Multiplicité de Taibi}

Soit $V$ un élément effectif de $K_\infty$. Alors :
\begin{enumerate}
\item $\m_1(V) \mathrm{B}_\infty(V,V) \leq \Phi(0)$ ;\ps 
\item $\m_2(V) (\mathrm{B}_\infty(V,V)-\log 2 (\dim V-1)) \leq \Phi(0)$ ;\ps 
\item Si $\m_1(V) + \m_2(V)>0$,
\begin{multline*}
(\m_1(V) + \m_2(V)) \mathrm{B}_\infty^F(V,V) \leq \Phi_F(0) \\
+ (\log 2) \m_2(V)\left(\dim V-\frac{\m_2(V)}{\m_1(V)+\m_2(V)}\right).
\end{multline*}

\end{enumerate}
%
\end{thm}
\begin{proof}
Supposons disposer de $r$ représentations distinctes $\pi_1, \cdots, \pi_{r}$ de $\Pi_{\rm alg}^1$ et de $s$ représentations distinctes $\varpi_1, \cdots, \varpi_{s}$ de $\Pi_{\rm alg}^2$ telles que $\Ll(\pi_{i,\infty})\simeq \Ll(\varpi_{j,\infty})\simeq V$ pour tous $i,j$. En particulier, ce sont des représentations automorphes de $\GL_n$ sur $\Q$ pour $n= \dim V$. On considère l'élément $\pi=\pi_1 + \cdots + \pi_{r} + \varpi_1 + \cdots + \varpi_{s}$ de $\R[\Pi_{\rm alg}]$.

La formule explicite nous donne :
\begin{equation*}
\mathrm{B}_f(\pi,\pi)+\mathrm{B}_\infty(\pi,\pi)+\frac{1}{2}\mathrm{Z}(\pi,\pi)=\Phi_F(0) \delta(\pi,\pi) + \frac{\log 2}{2} \aar_2(\pi,\pi).
\end{equation*}

Examinons chacun des termes de cette égalité.
\begin{itemize}
\item La Proposition \ref{cor_B_f_pos} nous dit que la quantité $\mathrm{B}_f(\pi,\pi)$ est positive.
\item La quantité $\mathrm{B}_\infty(\pi,\pi)$ ne dépend que de $\pi_\infty=V^{\oplus(r+s)}$ donc $\mathrm{B}_\infty(\pi,\pi)=(r+s)^2 \mathrm{B}_\infty(V,V) $.
\item Le Lemme \ref{lemme_F_Phi_positive} et le fait que l'élément $\pi$ soit effectif nous indiquent que la quantité $\mathrm{Z}(\pi,\pi)$ est positive. 
\item Puisque les représentations $\pi_i$ et $\varpi_j$ sont supposées deux à deux distinctes, on a $\delta(\pi,\pi)=r+s$.
\item La Proposition \ref{prop_cond_2} nous indique que 
\begin{align*}
\aar_2(\pi,\pi)&=\sum_{i,i'} \aar_2(\pi_i,\pi_{i'}) + 2 \sum_{i,j} \aar_2(\pi_i,\varpi_j) + \sum_{j,j'} \aar_2(\varpi_j,\varpi_{j'}) \\
			   &=\sum_{i,i'} 0 + 2 \sum_{i,j} n + \sum_{j,j'} (2n-2) \\
			   &= 2 n rs + s^2 (2n-2).
\end{align*}
\end{itemize}

Nous avons donc :
\begin{equation}\label{temporaire}
(r+s)^2 \mathrm{B}_\infty(V,V) \leq (r+s)\Phi(0) + (\log 2)s((r+s)n-s).
\end{equation}

L'inégalité 1. s'obtient donc avec $s=0$, \ie en ne considérant que les représentations de conducteur 1 et en divisant par $r$ (si $r>0$ ; si $\m_1(V)=0$, il n'y a rien à montrer) : c'est le Corollaire IX.3.13 de \cite{Chen-Lannes}.\ps 

L'inégalité 2. s'obtient avec $r=0$, \ie en ne considérant que les représentations de conducteur 2, et en divisant par $s$ (si $s>0$ ; si $\m_2(V)=0$, il n'y a rien à montrer).\ps 

L'inégalité 3. s'obtient en divisant \eqref{temporaire} par $r+s$.
\end{proof}

Nous avons mentionné au paragraphe \ref{Cas du conducteur 1} que l'on pouvait trouver une fonction $F$ telle que la forme bilinéaire symétrique $\B_\infty^F$ soit définie positive sur $K_\infty^{\leq 23}$ (voir \cite{Chen-HM}). En considérant des éléments \emph{effectifs} dans ces sous-$\Z$-modules, on redémontre donc dans le cas du conducteur 1 la finitude de Harish-Chandra déjà mentionnée, dans une version \emph{quantitative}.

Pour le conducteur 2, il faut trouver une fonction $F$ telle que la quantité $(\mathrm{B}_\infty^F(V,V)-\log 2 (\dim V-1))$ soit strictement positive, nous verrons que c'est possible en pratique pour les cas étudiés (jusqu'au poids motivique pair 16 et jusqu'au poids motivique impair 21).

Le corollaire suivant, qui repose sur la simple observation, déjà faite, que $\Ll(\pi_\infty) \simeq \Ll(\pi_\infty^\vee)$ nous rendra de grands services. On le rapprochera du Corollaire \ref{cor_654}.

\begin{cor}\label{cor_mult_autoduale}

\begin{enumerate}
\item Soit $V$ un élément effectif de $K_\infty$. On suppose que $\m_i(V) \leq 1$. Alors s'il existe $\pi \in \Pi_{\rm alg}^i$ telle que $\Ll(\pi_\infty) \simeq V$ (\ie $m_i(V) \geq 1$), $\pi$ est nécessairement autoduale.
\item Soit $\pi \in \Pi_{\rm alg}^i$. On suppose que $\pi$ n'est pas autoduale. Alors $\m_i(V)\geq 2$ où $V=\Ll(\pi_\infty)$.
\end{enumerate}
\end{cor}

\section{Un élément global}\label{Un élément global}
Nous avons, pour deux représentations $\pi$ et $\pi'$ de $\Pi_{\rm alg}$, l'équation fonctionnelle suivante :
\[
\Lambda(s,\pi^\vee \times \pi')=\eps(\pi^\vee \times \pi')\Lambda(1-s,\pi \times (\pi')^\vee)
\]
avec les bonnes normalisations qui font de $\eps(\pi^\vee \times \pi')$ une constante. La formule explicite fait, entre autres, intervenir les zéros de la fonction méromorphe $\Lambda(\cdot,\pi^\vee \times \pi')$. 

Les résultats de Jacquet et Shalika mentionnés au paragraphe \ref{Énoncé particulier} nous indiquent que tous les zéros de $\Lambda(s, \pi^\vee \times \pi')$ sont situés dans la bande $0 \leq \re(s) \leq 1$. L'hypothèse de Riemann généralisée (GRH) conjecture qu'ils sont en fait tous sur la droite critique $\re(s)=\frac{1}{2}$. Nous ne faisons pas ici cette hypothèse mais renvoyons à \cite{Chen-HM} pour voir comment elle permet de choisir une \og meilleure \fg{} fonction test et d'obtenir des résultats plus forts que ceux que nous avons avec la fonction d'Odlyzko.

Il est néanmoins un zéro que l'équation fonctionnelle peut nous donner : si les représentations $\pi$ et $\pi'$ sont autoduales, alors les fonctions $\Lambda$ sont les mêmes dans chaque membre de l'équation et donc $\eps(\pi\times\pi')\in\{\pm 1\}$. Si l'on suppose que ce facteur epsilon est  égal à $-1$, alors la fonction $\Lambda$ s'annule en $s=\frac{1}{2}$.


\begin{prop-def}
Soient $\pi$ et $\pi'$ deux représentations de $\Pi_{\rm alg}$. On pose ${\rm e}^\perp(\pi,\pi')=1$ si $\pi$ et $\pi'$ sont autoduales \emph{et si} $\eps(\pi\times\pi')=-1$ (ce qui implique que $\pi$ et $\pi'$ soient de nature différente selon l'alternative symplectique-orthogonale, \cf Théorème \ref{thm_alt_sp_orth}). Dans le cas contraire, on pose ${\rm e}^\perp(\pi,\pi')=0$.

Soit $F$ une fonction test $\Phi$-positive. On pose : $$\tilde{\mathrm{Z}}^F(\pi,\pi')=\mathrm{Z}^F(\pi,\pi')-\Phi_F(\frac{1}{2}){\rm e}^\perp(\pi,\pi').$$

Les quantités ${\rm e}^\perp$ et $\tilde{\mathrm{Z}}^F$ s'étendent immédiatement en des formes bilinéaires symétriques sur $\R[\Pi_{\rm alg}]$. Elles sont toutes les deux positives sur les éléments effectifs.
\end{prop-def}
\begin{proof}
La discussion précédente nous indique que, si ${\rm e}^\perp(\pi,\pi')=1$, alors on a un zéro en $s=\frac{1}{2}$ et $\mathrm{Z}^F(\pi,\pi')=\Phi_F(\frac{1}{2}) + \tilde{\mathrm{Z}}^F(\pi,\pi')$ où le deuxième terme correspond à la somme sur les autres zéros (qui peuvent encore contenir $s=\frac{1}{2}$ si c'est un zéro d'ordre multiple). La forme bilinéaire symétrique ${\rm e}^\perp$ est immédiatement positive (pas seulement sur les éléments effectifs d'ailleurs). Quant à la forme bilinéaire symétrique $\tilde{\mathrm{Z}}^F$, elle est bien positive sur les éléments effectifs sous l'hypothèse de $\Phi$-positivité de $F$, comme pour le Lemme \ref{lemme_F_Phi_positive}. 
\end{proof}


\subsection{Un exemple en conducteur 1}\label{Un exemple en conducteur 1}
Nous avons vu au paragraphe \ref{Une spectaculaire illustration} l'étonnante précision de la formule explicite qui permet, en conducteur 1, de (re)démontrer qu'il n'existe pas de forme modulaire pour le groupe $\SL_2(\Z)$ de poids modulaire <12. On sait classiquement qu'il en existe bien une pour le poids modulaire 12, et pour tous les poids modulaires pairs supérieurs, à l'exception notable du poids modulaire 14. Nous suivons ici encore la discussion \og Cas $n=2$ ou $w\leq 10$ \fg{} du § IX.3.18 de \cite{Chen-Lannes}.\ps 

Supposons disposer d'une telle forme, il lui correspond alors une représentation automorphe cuspidale algébrique $\pi$ de $\GL_2$, telle que $\Ll(\pi_\infty)=I_{13}$, de caractère central trivial. Cette représentation est alors nécessairement autoduale\footnote{Le calcul de la multiplicité selon le Théorème \ref{Multiplicité de Taibi} nous donne $\m_1(I_{13}) \leq 1$ pour le choix de $F={\rm F}_1$, ce qui donne un autre argument avec le Corollaire \ref{cor_mult_autoduale}.} par le Corollaire \ref{cor_GL_2_autoduale}. 

On peut donc calculer $\eps(\1 \times \pi)=\eps_\infty(\pi)=\eps_\infty(I_{13})=i^{14}=-1$. Comme $\pi$ et $\1$ sont autoduales, on a ${\rm e}^\perp(\1,\pi)=1$ et la fonction $\Lambda(s, \1 \times \pi)$ admet un zéro en $s=\frac{1}{2}$. Comme au paragraphe \ref{Une spectaculaire illustration}, en considérant ${\rm F}_{\log 2}$, on fait disparaître la quantité ${\rm B}_f^F$, on obtient alors :

\[
\mathrm{B}_\infty^{{\rm F}_\lambda}(\1,\pi) +\frac{1}{2}\mathrm{Z}^{{\rm F}_\lambda}(\1,\pi)=\Phi_{{\rm F}_\lambda}(0)\delta_{\1,\pi}=0,
\]
ce qui, puisque $\mathrm{Z}^{{\rm F}_\lambda}(\1,\pi) \geq \Phi_F(\frac{1}{2})$, impose finalement $\mathrm{J}^F(I_{13}) \leq - \frac{1}{2}\Phi_{{\rm F}_\lambda}(\frac{1}{2})$.\ps

Or $\mathrm{J}^F(I_{13}) \simeq -0,189$ à $10^{-3}$ près et $\frac{1}{2}\Phi_{{\rm F}_\lambda}(\frac{1}{2}) \simeq 0.279$ à $10^{-3}$ près, ce qui nous fournit une contradiction.

Du point de vue de la formule explicite, il n'y a donc \og pas de place \fg{} pour un zéro en $\frac{1}{2}$ en poids motivique 13, ce qui interdit donc l'existence d'une forme modulaire pour $\SL_2(\Z)$ de poids modulaire 14.

Nous verrons ce phénomène de \emph{contrainte sur les zéros} se produire également en conducteur 2 et, sur l'ensemble des représentations déterminées, par la rareté générale des représentations de signe epsilon global non trivial (voir les Tables en Annexe \ref{Tables_de_representations}).
\section{Méthode géométrique}\label{Version géométrique}
\subsection{Quantités calculables}\label{Quantites_calculables}

La formule explicite dans sa version \eqref{Formule_explicite_fbs} est une égalité entre formes bilinéaires 
certes explicites, mais sous réserve de connaître certaines informations sur les représentations de $\Pi_{\rm alg}$.

En pratique, nous utiliserons comme fonction test la fonction d'Odlyzko de paramètre $\lambda$, ${\rm F}_\lambda$, qui a les propriétés de positivité requises et pour laquelle tout est réellement explicite (aussi bien ses valeurs que les valeurs de $\Phi_{{\rm F}_\lambda}$, \cf néanmoins \cite{Chen-Lannes} §IX.3.16 et \cite{Chen-Ta} Remark 2.10 pour la question délicate de justification inconditionnelle des calculs). On a dans ce cas ${\rm F}_\lambda(0)=1$ et, pour alléger les notations, on ne fera plus apparaître les exposants et indices indiquant la dépendance en la fonction test. On a donc :

\[
\mathrm{B}_f+\mathrm{B}_\infty+\frac{1}{2}\tilde{\mathrm{Z}}+\frac{1}{2}\Phi(\frac{1}{2}){\rm e}^\perp=\Phi(0) \delta + \frac{\log 2}{2} \aar_2.
\]

On peut alors considérer les formes bilinéaires symétriques :
\begin{align*}
{\rm C^o}&=\Phi(0) \delta + \frac{\log 2}{2} \aar_2-\mathrm{B}_\infty ; \\
{\rm C}&=\Phi(0) \delta + \frac{\log 2}{2} \aar_2-\mathrm{B}_\infty-\frac{1}{2}\Phi(\frac{1}{2}){\rm e}^\perp ; \\
{\rm C^s}&=\Phi(0) \delta + \frac{\log 2}{2} \aar_2-\mathrm{B}_\infty-\frac{1}{2}\Phi(\frac{1}{2}){\rm e}^\perp - \mathrm{B}_2^{\rm calc}.
\end{align*}

Ces formes correspondent à la partie \emph{calculable} de la formule explicite, la quantité ${\rm C}$ généralise celle introduite sous le même nom dans le Corollaire IX.3.11 de \cite{Chen-Lannes}. On remarque que le calcul de ${\rm C^o}$ nécessite seulement de connaître les conducteurs et les composantes archimédiennes des différentes représentations, tandis que le calcul de ${\rm C}$ et de ${\rm C^s}$ requiert de connaître leur caractère autodual et le facteur epsilon en 2.

\begin{prop}\label{prop_c_c_s_positives}
Soit $F$ une fonction test positive et $\Phi$-positive (telle que $F(0)=1$).
Les formes bilinéaires symétriques ${\rm C^o}$, ${\rm C}$ et ${\rm C^s}$ sont positives sur les éléments effectifs de $\R[\Pi_{\rm alg}]$.
\end{prop}
\begin{proof}
Il suffit de remarquer que ${\rm C^o}={\rm B}_f +\frac{1}{2}\mathrm{Z}$ et ${\rm C}={\rm B}_f +\frac{1}{2}\tilde{\mathrm{Z}}$. Or nous avons vu à la Proposition \ref{cor_B_f_pos} et au Lemme \ref{lemme_F_Phi_positive} (resp. au paragraphe \ref{Un élément global}) que ce sont des formes positives sur les éléments effectifs.

Quant à ${\rm C^s}$, elle est égale à 
\[
\sum_{p>2}{\rm B}_p + {\rm B}_2^\mathrm{non\,calc} +\frac{1}{2}\tilde{\mathrm{Z}}
\]
qui est une somme de formes positives sur les éléments effectifs (\cf § \ref{La partie finie}).
\end{proof}

Cette proposition va nous permettre de déployer la stratégie suivante pour \emph{interdire} l'existence de certaines représentations.

\paragraph{Stratégie}

Considérons donc $n$ représentations (putatives ou avérées) $\pi_1,\cdots,\linebreak \pi_n$ de conducteurs (1 ou 2) et de composantes archimédiennes algébriques $\pi_{i,\infty}$ fixés. Alors, la Proposition \ref{prop_c_c_s_positives} affirme que, si ces $n$ représentations existent simultanément, ${\rm C^o}(\sum t_i \pi_i,\sum t_i \pi_i) \geq 0$, pour tout $\underline{t}=(t_1,\cdots,t_n) \in (\R_{\geq 0})^n$. Et de même pour ${\rm C}$ et ${\rm C^s}$, si l'on fixe également le caractère autodual ainsi que le facteur epsilon en 2 de chaque $\pi_i$.\ps 

Pour \emph{interdire} l'existence simultanée des $\pi_i$, il suffit donc de trouver \emph{un} $\underline{t}=(t_1,\cdots,t_n) \in (\R_{\geq 0})^n$ tels que ${\rm C^*}(\sum t_i \pi_i,\sum t_i \pi_i) < 0$. 
On a :
\[
{\rm C^s} \leq {\rm C} \leq {\rm C^o},
\]
c'est donc avec ${\rm C^s}$ que l'on trouvera le plus souvent des valeurs strictement négatives. L'intérêt de ${\rm C^o}$ est qu'elle nécessite moins d'informations, mais on peut alors s'interroger sur la pertinence de ${\rm C}$ qui demande \emph{a priori} autant d'informations que ${\rm C^s}$ pour des résultats \og moins négatifs \fg{}.\ps 

Le lemme suivant précise les choses dans la seule situation que nous utilisons, à savoir celle de $n-1$ représentations avérées et d'une représentation putative.

\begin{lemme}\label{lemme_C_flexible}
Soient $\pi_2,\cdots,\pi_n$ $n-1$ représentations de $\Pi_{\rm alg}$ connues. Soit $\pi=\pi_{1}$ une représentation putative de conducteur et de composante archimédienne algébrique $\pi_\infty$ fixés. On précise alors pour le calcul de ${\rm C}$ que $\pi$ est non autoduale et de facteur epsilon en $2$ égal à $1$. On suppose que l'on trouve $\underline{t}=(t_1,\cdots,t_{n})$ tel que ${\rm C}(\sum t_i \pi_i,\sum t_i \pi_i) < 0$.

Cela interdit alors l'existence d'une telle $\pi$ et, plus généralement, de toute représentation de mêmes conducteur et composante archimédienne, indépendamment de son caractère autodual et de son facteur epsilon en $2$.
\end{lemme}
\begin{proof}
Puisque l'on suppose dans un premier temps que $\pi$ est non autoduale, la quantité ${\rm e}^\perp(\pi,\cdot)$ est identiquement nulle, et ce quelle que soit la valeur du facteur epsilon en 2. On interdit donc bien l'existence de toute telle représentation non autoduale, indépendamment de la valeur de son facteur epsilon en 2.

Si on considère désormais que $\pi$ est autoduale, alors pour certaines valeurs\footnote{En réalité au plus une car, si $\pi$ est de conducteur 1, on a $\eps_2(\pi)=1$ et, si $\pi$ est de conducteur 2, l'autodualité impose $\eps_2(\pi)\in\{\pm 1\}$.} du facteur epsilon en 2, et pour certains indices\footnote{En pratique au plus un car, d'après la Proposition \ref{facteur_epsilon_alternative_sp_orth}, cela ne peut se produire que si $\pi$ et $\pi_i$ sont de nature (symplectique ou orthogonale) différente ; la seule représentation orthogonale que nous aurons à considérer est la représentation triviale.} $i$, on aura éventuellement ${\rm e}^\perp(\pi,\pi_i)=1$ et donc le nouveau calcul de ${\rm C}$ est \og plus négatif \fg, maintenant la contradiction.
\end{proof}

On comprend alors que la quantité ${\rm C}$ nous permet de \emph{profiter} de plus d'information sur les représentations connues, tout en se restreignant à ne connaître pas davantage que le conducteur et la composante archimédienne de la représentation putative (comme pour ${\rm C^o}$). Cela repose \emph{in fine} sur le fait que la quantité $\frac{1}{2}\Phi(\frac{1}{2}){\rm e}^\perp$ est positive sur les éléments effectifs. Il n'en est pas de même de $\B_2^{\rm calc}$, on ne peut donc pas en faire autant avec ${\rm C^s}$.

\subsection{La question de la multiplicité}\label{La question de la multiplicité}
Soit $V$ un élément effectif de $K_\infty$. On sait par les travaux de Harish-Chandra et, de façon quantitative, par le Théorème \ref{Multiplicité de Taibi}, que le nombre de représentations de conducteur 1 ou 2 et de composante archimédienne isomorphe à $V$, est fini. Il faut maintenant chercher à exploiter cette information du point de vue des quantités calculables ${\rm C}^*$. L'objectif est de rendre ces quantités strictement négatives, pour aboutir à une contradiction avec la Proposition \ref{prop_c_c_s_positives}. Tout ce qui contribue à les \emph{diminuer} est donc bon à prendre.\medskip

Fixons donc notre élément effectif $V$ de $K_\infty$ et un conducteur $N$ ($\in \{1,2\}$). On peut s'interroger, avec ${\rm C^o}$, sur l'existence d'une représentation $\pi$ de conducteur $N$ et telle que $\Ll(\pi_\infty)=V$. On peut en fait s'interroger sur l'existence de \emph{plusieurs} telles représentations. De même que la discussion du paragraphe \ref{Quantites_calculables} consistait à voir comment interdire l'existence d'\emph{au moins une} telle représentation, on peut, un entier naturel $r$ étant fixé, voir comment interdire l'existence d'\emph{au moins $r$} telles représentations (non isomorphes).\ps 

Soient donc $\varpi_1,\cdots,\varpi_r$ ces $r$ putatives représentations et considérons \[
\pi=\frac{\varpi_1+\cdots+\varpi_r}{r},
\]
élément de $\R[\Pi_{\rm alg}]$. Alors, pour $\pi'$ représentation connue, on a :
\begin{align*}
{\rm C^o}(\pi,\pi')&=\Phi(0) \delta(\pi,\pi') + \frac{\log 2}{2} \aar_2(\pi,\pi')-\mathrm{B}_\infty(\pi,\pi') \\
				 &=\frac{1}{r} \left( 0 + r\frac{\log 2}{2} \aar_2(\varpi_1,\pi')-r\mathrm{B}_\infty(\varpi_1,\pi') \right) \\
				 &={\rm C^o}(\varpi_1,\pi'),
\end{align*}
puisque la calcul de $\aar_2$ ne dépend que de la dimension et du conducteur, et que celui de $\B_\infty$ ne dépend que de la composante archimédienne, qui sont les mêmes pour toutes les $\varpi_i$. Par ailleurs :
\begin{align*}
{\rm C^o}(\pi,\pi)&=\Phi(0) \delta(\pi,\pi) + \frac{\log 2}{2} \aar_2(\pi,\pi)-\mathrm{B}_\infty(\pi,\pi) \\
				 &=\frac{1}{r^2} \left( r\Phi(0) + r^2\frac{\log 2}{2} \aar_2(\varpi_1,\varpi_1)-r^2\mathrm{B}_\infty(V,V) \right) \\
				 &=\frac{1}{r} \Phi(0) + \frac{\log 2}{2} \aar_2(\varpi_1,\varpi_1)-\mathrm{B}_\infty(\varpi_1,\varpi_1).
\end{align*}

On a donc ${\rm C^o}(\pi,\cdot) \leq {\rm C^o}(\varpi_1,\cdot)$ sur les éléments effectifs et donc plus de chances d'obtenir une contradiction avec $r$ représentations qu'avec une seule. Le Lemme \ref{lemme_C_flexible} s'applique d'ailleurs \emph{mutatis mutandis}, si bien que la technique s'applique encore avec ${\rm C}$. 
Nous pourrons ainsi faire diminuer $\m_N(V)$ au-delà de la borne donnée par le Théorème \ref{Multiplicité de Taibi}.

Cette stratégie de considérer simultanément plusieurs représentations non isomorphes ayant les mêmes caractéristiques vaut évidemment également pour ${\rm C^s}$, même si nous l'utiliserons peu en pratique (voir néanmoins la Remarque qui termine le paragraphe \ref{En_poids motivique 13}).

\subsection{Technique utilisée}

Notre objectif est donc d'interdire l'existence de certaines représentations en utilisant l'existence avérée d'autres représentations données. La première chose dont nous devons disposer est donc une liste de représentations \emph{connues}, ceci sera détaillé au début du Chapitre \ref{Chapitre_calculs}.\ps 

La classe de fonctions test que l'on utilisera est précisée : ce sont les $\F_\lambda$ avec $\lambda \in 
\frac{1}{10} \Z \cap [1;12]$. (Parfois, les calculs se limiteront à un nombre plus restreint de paramètres, quand les premiers \emph{sondages} indiquent qu'il faut chercher la contradiction à un endroit précis.)\ps 

Il faut ensuite une stratégie pour trouver une combinaison linéaire effective \og contredisante \fg{}. Cette stratégie est exposée avec beaucoup de clarté au §2.4 de \cite{Chen-Ta} dans le cadre de l'étude des représentations automorphes cuspidales algébriques de conducteur 1. Nous en disons quelques mots rapides avant d'exposer les différences auxquelles nous devons prêter attention.

Par homogénéité, on peut supposer que ce vecteur \og contredisant \fg est dans l'intersection de la sphère unité et du pan d'espace positif. On découpe alors ce pan de sphère en \og faces \fg{} de dimension plus petite et on est alors ramené à évaluer le minimum d'une forme quadratique sur chacune de ces faces. On construit la matrice de Gram associée à l'une des formes ${\rm C^o}, {\rm C}$ ou ${\rm C^s}$ et on se ramène alors à déterminer la plus petite valeur propre d'une matrice symétrique réelle \emph{associée à un vecteur propre effectif} (\ie à coefficients tous positifs ou nuls). Si celle-ci est strictement négative, on a la contradiction (à la Proposition \ref{prop_c_c_s_positives}) voulue.\ps 

Les différences notoires avec les calculs de \cite{Chen-Ta} sont les suivantes.
\begin{itemize}
\item Nous avons affaire à des représentations de conducteur 1 ou 2, il faut donc déterminer le conducteur d'une paire ${\rm N}(\pi_1 \times \pi_2)$ dans ce cadre-là, ce qui est assuré par la Proposition \ref{prop_cond_2}.
\item Il nous faut calculer la quantité $\B_2^{\rm calc}(\pi_1 \times \pi_2)$ selon la Définition \ref{defi_B_2_calc}, quantité non nulle seulement si $\pi_1$ et $\pi_2$ sont de conducteur 2.
\item Nous utilisons les représentations de conducteur 1 comme représentations connues, mais seulement jusqu'au poids motivique 21 (inclus). Elles sont alors toutes autoduales et la seule représentation non symplectique est la représentation triviale. 
\item La Proposition \ref{type I sp} affirme qu'une représentation autoduale de conducteur 2 est nécessairement symplectique. Ainsi, si $\pi$ est une telle représentation, la quantité ${\rm e}^\perp(\pi \times \pi')$ est nulle sauf peut-être pour la représentation triviale (soit $\pi'$ n'est pas autoduale et il n'y a rien à dire, soit $\pi'$ est autoduale et dans les cas étudiés, la seule non symplectique est la représentation triviale, voir notes de bas de page 2 et 3 \emph{supra}). La détermination de ${\rm e}^\perp(\pi \times \1)$, demande alors de connaître son signe local en 2, $\eps_2(\pi)$, en plus de sa composante archimédienne. \newline
\end{itemize}

Les algorithmes utilisent des routines et certaines fonctions de \cite{Chen-Ren} en PARI-GP, mais ils ont été codés indépendamment de ceux de \cite{Chen-Ta}. Il a ainsi pu être utile de \og tester \fg{} nos algorithmes sur les représentations connues de conducteur 1, dans une optique de vérification. Nous reprenons d'ailleurs à notre compte la remarque qui clôt le § 2.4.4. de \cite{Chen-Ta}, à savoir qu'il n'est pas nécessaire de fournir de justification pour nos algorithmes : pour interdire l'existence d'une représentation, il suffit de trouver une fonction test et une combinaison linéaire \emph{ad hoc} ; seule l'exactitude du calcul avec \emph{cette} fonction test et \emph{cette} combinaison linéaire est alors à discuter. Et elle ne l'est pas en pratique, tant les résultats obtenus sont loin des erreurs d'arrondi.

%




\chapter{Calculs combinant la formule explicite et les résultats d'Arthur}\label{Chapitre_calculs}

Tous les outils sont désormais en place pour démontrer les Théorèmes \ref{thm_2_w17}, \ref{thm_2_w19} et \ref{thm_p}. Nous précisons aussi les étapes qui nous amènent à formuler la Conjecture \ref{thm_2_w21}. Pour alléger les notations, nous écrirons \og représentation de $\GL_n$ \fg{} pour \og représentation automorphe cuspidale algébrique de $\GL_n$ sur $\Q$ \fg{}.

On rappelle d'ailleurs que ce sont des représentations \emph{centrées} donc de caractère central trivial, ce qui impose d'une part une contrainte sur les composantes archimédiennes (\cf §\ref{En_poids motivique pair}), et permet d'autre part d'utiliser le Corollaire \ref{cor_GL_2_autoduale}. 

Il s'agit ici de détailler notre démarche pour illustrer la façon dont on aboute la technique \emph{constructive} des Chapitres \ref{Théorie d'Arthur pour SO-2n+1} et \ref{Lien avec des objets classiques} avec la technique \emph{limitative} du Chapitre \ref{La_formule explicite de Riemann-Weil-Mestre}. De même qu'au Chapitre \ref{La_formule explicite de Riemann-Weil-Mestre}, nous nous limitons dans ce préambule et jusqu'au §\ref{Conjectures pour le poids motivique $21$} au cas du conducteur 2. En particulier, une telle représentation, si elle est autoduale, est automatiquement symplectique (Proposition \ref{type I sp}). Cela nous permet d'alléger les notations, les modifications à apporter pour le cas du conducteur $p>2$ de type (I) sont alors immédiates (et résumées au §\ref{En conducteur p>2}). \ps 

Toutes les étapes de notre raisonnement se trouvent dans nos feuilles de calcul (avec les algorithmes mis en œuvre). Il nous a paru néanmoins éclairant d'illustrer les mécanismes en jeu dans ce Chapitre. \newline

Nous nous intéressons donc aux représentations de $\GL_n$ de conducteur 2 et, même si nos techniques \og capturent \fg{} également\footnote{L'inégalité du Corollaire \ref{cor_finitude_cond_2} est \emph{a fortiori} vérifiée pour les composantes archimédiennes de représentations existantes de conducteur 1. Donc il sera normal de \og voir apparaître \fg{} les composantes archimédiennes des représentations (connues) de conducteur 1.} les représentations de conducteur 1 (puisqu'elles généralisent les techniques de \cite{Chen-Ren}, \cite{Chen-Lannes}, \cite{Chen-Ta} et donc permettraient de redémontrer leurs résultats), nous les considérons ici comme des \emph{données}. Plus précisément, nous utiliserons la liste suivante (exhaustive) de représentations \emph{connues} de conducteur 1 de poids motivique $\leq 19$ (\cite{Chen-Lannes}, extrait du Théorème F, dont on reprend les notations) :
\begin{itemize}
\item la représentation triviale $\1$ ;
\item quatre représentations de $\GL_2$ associées à des formes modulaires paraboliques pour $\SL_2(\Z)$ notées $\Delta_{11}, \Delta_{15}, \Delta_{17}, \Delta_{19}$ ;
\item une représentation de $\GL_4$ notée $\Delta_{19,7}$.
\end{itemize}\ps 

Pour la Conjecture \ref{thm_2_w21}, nous utiliserons le \og complément \fg{} suivant en poids motivique $\leq 21$ (\emph{loc. cit.}) :
\begin{itemize}
\item une représentation de $\GL_2$ associées à une forme modulaire parabolique pour $\SL_2(\Z)$ notée $\Delta_{21}$ ;
\item trois représentations de $\GL_4$ notées $\Delta_{21,5}, \Delta_{21,9}, \Delta_{21,13}$. \newline
\end{itemize} 

Par analogie avec ces notations, si $\underline{w}=(w_1,\cdots,w_n)$ est un $n$-uplet d'entiers naturels rangés par ordre décroissant et si $\eps \in \{+,-\}$, nous noterons $\E_{\underline{w}}^\eps$ une représentation de $\GL_{2n}$, de poids $\{\pm \frac{w_1}{2},\cdots, \pm \frac{w_n}{2}\}$, de conducteur 2 et de signe local $\eps$ en 2. Cela suffira en général à caractériser \emph{la seule représentation} avec ces propriétés ; quand ce n'est pas le cas, nous rajouterons des lettres en exposant (par exemple $\E_{21,7}^{-,a},\,\E_{21,7}^{-,b}$). \newline

Précisons maintenant la démarche, dans une présentation algorithmique.
\begin{enumerate}
\item[{\bf Étape 1}] Pour chaque poids motivique $w$, identifier les $V \in K_\infty^{\leq w}$ possibles (\ie correspondants à une représentation de $\GL_n$ de poids motivique $w$ exactement et de conducteur 2). Ils sont en nombre fini par le Corollaire \ref{cor_finitude_cond_2}. Les éléments trouvés doivent également vérifier l'inégalité \eqref{inegalite_J_F_cond_2} pour $\lambda=\log 2$.
\item[{\bf Étape 2}] Pour chaque $V$ dans cette liste, déterminer le nombre maximal de $\pi$ potentielles de conducteur 2 avec $\pi_\infty \simeq V$ avec le Théorème \ref{Multiplicité de Taibi}.
\item[{\bf Étape 3}] Confronter, par la \emph{méthode géométrique} du §\ref{Version géométrique}, l'existence de telles représentations aux représentations déjà connues.
\item[{\bf Étape 4}] Dans le cas où l'on a affaire à une $\pi$ autoduale régulière (donc symplectique par la Proposition \ref{type I sp}), vérifier son existence par sa \og trace \fg dans un paramètre d'Arthur global selon le Corollaire \ref{cor_762}.
\item[{\bf Étape 5}] Mettre à jour la liste des représentations connues.\newline
\end{enumerate}

À cause du Lemme \ref{lemme_pi_infty_i_w}, nous traitons les poids motiviques pairs et impairs séparément dans le cas du conducteur $2$ (dans le cas du conducteur $p>2$, le poids motivique impair fait partie des \emph{hypothèses} du Théorème \ref{thm_p}).

\section{En poids motivique impair $\leq 13$}
\subsection{En poids motivique impair $\leq 7$}
Nous savons déjà par la Proposition \ref{prop_facile_w<6} qu'il n'existe pas de représentation de $\GL_n$ de conducteur 2 et de poids motivique (impair) inférieur ou égal à 5.

Le premier poids impair à considérer est donc 7.

\paragraph*{Étape 1}
Il s'agit de considérer les éléments de $K_\infty^{\leq 7}$ de poids motivique exactement 7 qui vérifient l'inégalité du Corollaire \ref{cor_finitude_cond_2}.

La forme bilinéaire $\B_\infty^{\F_2}$ est définie positive sur $K_\infty^{\leq 7}$ (où $\F_2$ désigne la fonction d'Odlyzko \eqref{Odlyzko} de paramètre $\lambda=2$). Le calcul effectif du Corollaire \ref{cor_finitude_cond_2} nous donne alors une liste réduite à un élément : $I_7$. L'inégalité doit être vérifiée pour toute fonction test positive et $\Phi$-positive, on la \og teste \fg{} pour $\F_\lambda$ avec $\lambda \in \frac{1}{10} \Z \cap [1;12]$ et on conserve notre unique élément potentiel.

De plus, on a bien $\mathrm{J}_{{\rm F}_{\log 2}} (I_7) \leq \frac{\log 2}{2}$.



\paragraph*{Étape 2}

On peut alors compter combien (au plus) de représentations de $\GL_2$ de conducteur $2$ ont $I_7$ comme composante archimédienne. Le Théorème \ref{Multiplicité de Taibi} 2. pour $\F_2$ nous donne qu'il y en a au plus une.

Notons $\pi$ cette putative représentation et remarquons qu'elle est nécessairement autoduale, à la fois par le Corollaire \ref{cor_mult_autoduale} et par le Corollaire \ref{cor_GL_2_autoduale}.

\paragraph*{Étape 3}
Considérons le signe local en 2 de $\pi$. Ce signe vaut $\pm 1$ par autodualité et la méthode géométrique avec $\F_1$ et la représentation triviale interdit le signe $-1$ (qui imposerait un zéro en $s=\frac{1}{2}$ à la fonction $\Lambda$ correspondante pour lequel il n'y a \og pas de place \fg).

S'il existe une représentation de $\GL_n$, de conducteur 2 et de poids motivique 7, alors $n=2$, ladite représentation est autoduale et de signe local en 2 égal à $+1$. La méthode géométrique, en utilisant les représentations \emph{connues} de conducteur 1 (outre la représentation triviale), ne nous fournit pas de contradiction pour la fonction d'Odlyzko (avec des paramètres variables).

\paragraph*{Étape 4}
Par le Corollaire \ref{cor_dim_s_k} (et son amélioration avec le signe), il y a autant de représentations automorphes cuspidales autoduales algébriques de $\GL_2$ de poids motivique 7 et de signe local en 2 égal à $+1$ que de formes modulaires paraboliques normalisées nouvelles pour le groupe $\Gamma_0(2)$ de \emph{poids modulaire} 8 et de signe d'Atkin-Lehner égal à $+1$. Et il y a bien une telle forme modulaire parabolique (\cite{LMFDB}) et donc une telle représentation de $\GL_2$.

\paragraph*{Étape 5}
Nous ajoutons à la liste des représentations connues la représentation $\E_7^+$ qui est donc \emph{la seule} représentation de $\GL_n$ (pour $n$ quelconque) de conducteur 2 et de poids motivique impair $\leq 7$.

%

\subsection{En poids motivique $9$}
Tout se passe rigoureusement de la même façon.
\paragraph*{Étape 1} La forme bilinéaire $\B_\infty^{\F_2}$ est définie positive sur $K_\infty^{\leq 9}$. Le calcul effectif du Corollaire \ref{cor_finitude_cond_2} nous donne alors une liste réduite à un élément : $I_9$ et on a bien $\mathrm{J}_{{\rm F}_{\log 2}} (I_9) \leq \frac{\log 2}{2}$. 
\paragraph*{Étape 2}  Le Théorème \ref{Multiplicité de Taibi} 2. pour $\F_2$ nous dit qu'il y en a au plus une représentation $\pi$ de $\GL_2$ de conducteur 2 avec $\Ll(\pi_\infty)=I_9$.
\paragraph*{Étape 3}  Considérons le signe local en 2 de $\pi$ (qui est autoduale si elle existe) : le signe $+1$ est interdit par \emph{croisement} avec la représentation triviale.
\paragraph*{Étape 4}  Il existe bien une telle représentation correspondant à une forme modulaire parabolique normalisée nouvelle pour le groupe $\Gamma_0(2)$ de \emph{poids modulaire} 10 et de signe d'Atkin-Lehner égal à $-1$.
\paragraph*{Étape 5}  Nous ajoutons à la liste des représentations connues la représentation $\E_9^-$.

%

\subsection{En poids motivique $11$}
\paragraph*{Étape 1} La forme bilinéaire $\B_\infty^{\F_2}$ est définie positive sur $K_\infty^{\leq 11}$. Le calcul effectif du Corollaire \ref{cor_finitude_cond_2} nous donne alors une liste réduite à un élément : $I_{11}$ et on a bien $\mathrm{J}_{{\rm F}_{\log 2}} (I_{11}) \leq \frac{\log 2}{2}$. \medskip

On pourrait alors aller directement à l'Étape 4 puisqu'une telle représentation (si elle existe) est autoduale par le Corollaire \ref{cor_GL_2_autoduale}, manifestement régulière. On conclurait alors par le fait que l'espace ${\rm S}_{12}(\Gamma_0(2))^{\rm new}$ est réduit à $\{0\}$. On peut en fait aboutir à la même conclusion \emph{purement} \og côté formule explicite \fg.

\paragraph*{Étape 2} On sait que $\m_1(I_{11})=1$. Le Théorème \ref{Multiplicité de Taibi} 3. avec $\F_{1,5}$ implique alors $\m_2(I_{11})<1$ et donc $\m_2(I_{11})=0$.

%

\subsection{En poids motivique $13$}\label{En_poids motivique 13}
\paragraph*{Étape 1} On trouve quatre éléments potentiels : $I_{13},\,I_{13}\oplus I_9,\,I_{13}\oplus I_7,\,I_{13}\oplus I_5$. 
\paragraph*{Étape 2} Le Théorème \ref{Multiplicité de Taibi} 2. nous donne ${\rm m}_2(I_{13})\leq 2$ et ${\rm m}_2(I_{13}\oplus I_v)\leq 1$ pour $v \in \{5;7;9\}$. Les représentations putatives de $\GL_4$ et de conducteur 2 correspondantes doivent être autoduales.
\paragraph*{Étape 3} Pour une telle représentation putative $\pi$ de $\GL_4$, on peut considérer ${\rm C^s}(\pi,\pi)$, en remarquant que la quantité $\B^{\rm calc}_2$ est la même\footnote{En fait, le calcul donne le même résultat quel que soit $\lambda$ de module 1, on n'a pas besoin de l'autodualité, simplement du fait -- acquis -- que $\pi^\vee\simeq\overline{\pi}$.} que $\eps_2(\pi)$ vaille 1 ou $-1$. Le calcul pour $\F_2$ est négatif, contredisant la Proposition \ref{prop_c_c_s_positives}. Ces représentations n'existent pas.

\paragraph*{Étape 4}  Il existe bien deux représentations de $\GL_2$ de poids motivique 13 correspondant à chacune des formes modulaires paraboliques normalisées nouvelles pour le groupe $\Gamma_0(2)$ de \emph{poids modulaire} 14 et de signe d'Atkin-Lehner égal à $+1$ (resp. $-1$).
\paragraph*{Étape 5}  Nous ajoutons à la liste des représentations connues les représentations $\E_{13}^+$ et $\E_{13}^-$.\medskip
%
%

\emph{Remarque :} Nous aurions pu avec la formule explicite et la quantité ${\rm C^s}$ voir à l'Étape 3 que chacun des signes locaux était \emph{autorisé} pour les représentations de $\GL_2$ de poids motivique 13, mais qu'il ne pouvait y en avoir deux de même signe.

\section{En poids motivique pair $\leq 16$}\label{En_poids motivique pair}
\subsection{En poids motivique pair $\leq 12$}
Nous savons déjà par la Proposition \ref{prop_facile_w<6} qu'il n'existe pas de représentation de $\GL_n$ de conducteur 2 et de poids motivique (pair) inférieur ou égal à 4.

Le premier poids pair à considérer est donc 6. L'Étape 1 avec ce poids motivique ainsi qu'avec les poids 8 et 10 retourne une liste vide. \emph{Il n'y a donc aucune} représentation automorphe cuspidale algébrique de $\GL_n$ (avec $n$ quelconque) de conducteur 2 et de poids motivique pair $\leq 10$.

Nous en arrivons alors au poids motivique 12.

\paragraph*{Étape 1} La forme bilinéaire $\B_\infty^{\F_2}$ est définie positive sur $K_\infty^{\leq 12}$. Le calcul effectif du Corollaire \ref{cor_finitude_cond_2} nous donne sur ce réseau un élément potentiel : $I_{12} \oplus\eps_{\C/\R}$ et on a bien $\mathrm{J}_{{\rm F}_{\log 2}} (I_{12} \oplus\eps_{\C/\R}) \leq \frac{\log 2}{2}$.

\paragraph*{Étape 2} Le Théorème \ref{Multiplicité de Taibi} 2. nous donne ${\rm m}_2(I_{12} \oplus\eps_{\C/\R})\leq 1$. Or la Proposition \ref{type I sp} nous dit que la putative représentation de $\GL_3$ de conducteur 2 et de composante archimédienne $I_{12}\oplus\eps_{\C/\R}$ ne peut pas être autoduale (puisque $n=3$ est impair). D'après le Corollaire \ref{cor_654}, on doit donc avoir $\m_2(I_{12}\oplus\eps_{\C/\R})$ pair et finalement $\m_2(I_{12}\oplus\eps_{\C/\R})=0$. \newline

Il n'existe donc aucune représentation de $\GL_n$ de conducteur 2 et de poids motivique $\leq 12$.\newline

\emph{Remarque :} On peut s'interroger sur le fait
qu'il faille \og attendre \fg{} le poids motivique pair 12 pour que l'inégalité du Corollaire \ref{cor_finitude_cond_2} fasse apparaître les premiers éléments non triviaux de $K_\infty$ (indépendamment de leur réalisation effective comme composante archimédienne d'une représentation de $\GL_n$ de conducteur 2) quand la même inégalité faisait apparaître les premiers éléments en poids motivique impair 7.
C'est que, au-delà de la décroissance de $w \mapsto \J^{\F_\lambda}(w)$ pour la fonction d'Odlyzko (et ce, quel que soit le paramètre $\lambda$), les éléments de $K_\infty$ doivent être de déterminant 1 pour bel et bien correspondre à des représentations de $\GL_n$ de caractère central trivial. Or $\det I_w=\eps_{\C/\R}^{w+1}$ si bien qu'on ne peut avoir $I_w$ (avec $w$ pair) seul.
C'est pourquoi il faut atteindre un poids motivique plus élevé pour faire apparaître les premiers éléments \og pairs \fg{}.

\subsection{En poids motivique $14$}

\paragraph*{Étape 1} La forme bilinéaire $\B_\infty^{\F_2}$ est définie positive sur $K_\infty^{\leq 14}$. Le calcul effectif du Corollaire \ref{cor_finitude_cond_2} nous donne sur ce réseau 6 éléments potentiels : $I_{14}\oplus\eps_{\C/\R}$ et $I_{14}\oplus I_v$ avec $v \in \{2;4;6;8;10\}$.

L'élément $I_{14}\oplus I_2$ est tout de suite éliminé car il ne vérifie pas l'inégalité \eqref{inegalite_J_F_cond_2} pour $\lambda=\log 2$.

\paragraph*{Étape 2} Pour $I_{14}\oplus\eps_{\C/\R}$, le Corollaire \ref{cor_654} nous dit que l'on doit avoir $\m_2(I_{14}\oplus\eps_{\C/\R})$ pair. Or le Théorème \ref{Multiplicité de Taibi} pour $\F_3$ nous donne $\m_2(I_{14}\oplus\eps_{\C/\R})\leq 1$ et donc $\m_2(I_{14}\oplus\eps_{\C/\R})=0$.

Pour les autres éléments, il faut utiliser, outre la Proposition \ref{type I sp}, la Proposition \ref{CL VIII 3.3}, qui affirme que l'on ne peut avoir de représentation autoduale de $\GL_4$ de conducteur $2$ et de composante archimédienne $I_{14}\oplus I_v$ avec $v \neq 14$. Là encore, le Corollaire \ref{cor_654} impose $\m_2(I_{14}\oplus I_v)$ pair quand les calculs du Théorème \ref{Multiplicité de Taibi} donnent, avec $\F_3$, $\m_2(I_{14}\oplus I_v)<2$ pour $v \in \{4;6;8;10\}$.\newline

%
%
%

Il n'existe donc pas de représentation de $\GL_n$ de conducteur 2 et de poids motivique 14.

\subsection{En poids motivique $16$}

\paragraph*{Étape 1} On trouve 12 éléments potentiels.

\paragraph*{Étape 2} L'argument de multiplicité paire vaut encore, sauf qu'il reste 5 éléments parmi ces 12 où le Théorème \ref{Multiplicité de Taibi} ne permet pas de descendre en-deçà de 2. Il n'est donc pas \emph{a priori} exclu que l'on ait affaire à une représentation non autoduale et à sa duale. Les 5 éléments \emph{récalcitrants} sont : $I_{16}\oplus\eps_{\C/\R}, \; I_{16}\oplus I_v \text{ avec } v \in \{4;6;8;10\}$.

\paragraph*{Étape 3} On peut, avec la méthode géométrique, faire diminuer $\m_2(V)$ au-delà de la borne donnée par le Théorème \ref{Multiplicité de Taibi} (c'est ce qui est discuté au §\ref{La question de la multiplicité}). Cela nous permet de montrer\footnote{Pour l'élément $I_{16}\oplus I_6$, on utilise le \emph{croisement} avec les représentations $\Delta_{15}$ et $\E_{15,5}^-$. L'existence de cette dernière ne sera prouvée qu'au §\ref{En_poids_motivique 15}, indépendamment de toute discussion en poids motivique pair. Le lecteur nous pardonnera cette légère anticipation.} que $\m_2(I_{16}\oplus I_v)<2$ pour $v \in \{4;6;10\}$ et donc d'exclure ces éléments.\ps 

Pour les deux derniers éléments, il faut développer une technique plus subtile.

\begin{prop}\label{prop_TK_ruse}
Soit $\pi$ une représentation automorphe cuspidale unitaire de $\GL_n$ sur $\Q$ de conducteur $2$, \emph{non autoduale}.

Alors on sait minorer $\B_2^{\rm calc}(\frac{\pi+\pi^\vee}{2},\frac{\pi+\pi^\vee}{2})$ (pour la fonction d'Odlyzko).
\end{prop}
\begin{proof}
On a $$\B_2^{\rm calc}\left(\frac{\pi+\pi^\vee}{2},\frac{\pi+\pi^\vee}{2}\right)=\frac{1}{4}\left( \B_2^{\rm calc}(\pi,\pi)+ 2\B_2^{\rm calc}(\pi,\pi^\vee)+\B_2^{\rm calc}(\pi^\vee,\pi^\vee) \right).$$

Soit $\varphi=\chi_1\oplus\cdots\oplus\chi_{n-2}\oplus (\psi\otimes U_2)$ le paramètre de Langlands de $\pi_2$. On a, par compatibilité de la correspondance de Langlands à la dualité $\Ll((\pi^\vee)_2)=\varphi^\vee$. On peut donc, selon la Définition \ref{defi_B_2_calc} déterminer chacun des termes.
\begin{align*}
\B_2^{\rm calc}(\pi,\pi)&=\sum_{k=1}^{+ \infty} F(k \log 2) \frac{\log(2)}{2^{\frac{k}{2}}} \re (\overline{\psi(2)^k}\psi(2)^k) \\
						&=\sum_{k=1}^{+ \infty} F(k \log 2) \frac{\log(2)}{2^{\frac{k}{2}}},
\end{align*}
et le calcul est le même pour $\B_2^{\rm calc}(\pi^\vee,\pi^\vee)$. Par ailleurs 
\[
\B_2^{\rm calc}(\pi,\pi^\vee)=\sum_{k=1}^{+ \infty} F(k \log 2) \frac{\log(2)}{2^{\frac{k}{2}}} \re (\psi(2)^{2k}),
\]
et contrairement au cas autodual où $\psi(2) \in \{\pm 1\}$, on ne sait rien sur $\psi(2)$ si ce n'est que c'est un nombre complexe de module 1.\ps 

Si $F$ est la fonction test sous-jacente, on peut alors définir la fonction 
\[
\begin{array}{ccccl}
\Theta_F & : & \mathbb{U} & \longrightarrow & \R \\
 & & z & \longmapsto & \re \left(\sum_{k=1}^{+ \infty} F(k \log 2) \frac{\log(2)}{2^{k/2}} z^{k} \right)\\
\end{array}
\]
où $\mathbb{U}$ désigne l'ensemble des nombres complexes de module 1. On remarque d'ailleurs que $\B_2^{\rm calc}(\pi,\pi)=\Theta_F(1)$.\ps 

Dans le cas où $F=\F_\lambda$ est la fonction d'Odlyzko, la somme est finie et la fonction $\Theta_{\F_\lambda}$ est continue, en particulier, elle est bornée et atteint ses bornes. Nous \emph{observons} pour les paramètres $\lambda$ qui nous intéressent (et cela peut sans doute se démontrer) que $\Theta_{\F_\lambda}(z)\geq \Theta_{\F_\lambda}(-1)$ pour tout $z\in \mathbb{U}$. On a alors :
\[
\B_2^{\rm calc}\left(\frac{\pi+\pi^\vee}{2},\frac{\pi+\pi^\vee}{2}\right) \geq \frac{\Theta_{\F_\lambda}(1)+\Theta_{\F_\lambda}(-1)}{2}.
\]

Dans les cas étudiés, nous \emph{observons} encore que cette quantité est strictement positive et on peut alors l'utiliser dans le calcul de ${\rm C^s}$.
\end{proof}

%
%
%
%

En pratique, on a cherché par la méthode géométrique avec ${\rm C}$ la combinaison linéaire \emph{la plus proche d'être contredisante} et on a calculé ${\rm C^s}$ pour cette même combinaison linéaire en ajoutant la contribution de la Proposition \ref{prop_TK_ruse}. Nous avons depuis intégré algorithmiquement la machinerie \emph{ad hoc} pour savoir traiter ce cas directement. \newline

Il n'existe donc pas de représentation de $\GL_n$ de conducteur 2 et de poids motivique 16. Comme l'indique la mobilisation de cette technique subtile supplémentaire, la méthode touche à ses limites. C'est d'ailleurs le dernier poids pair que nous sachions traiter.

\section{En poids motivique $15$ et $17$}
\subsection{En poids motivique $15$}\label{En_poids_motivique 15}
\paragraph*{Étape 1} On trouve 7 éléments potentiels :\ps 
\begin{itemize}
\item 1 de dimension 2 : $I_{15}$ ; \ps
\item 6 de dimension 4 : $I_{15}\oplus I_v$ avec $v \in \{3;5;7;9;11;13\}.$
\end{itemize}
\paragraph*{Étape 2} Le Théorème \ref{Multiplicité de Taibi} nous donne ${\rm m}_2(I_{15})\leq 1$ (c'est le 3. en utilisant ${\rm m}_1(I_{15})=1$) et ${\rm m}_2(I_{13}\oplus I_v)\leq 1$ pour $v \in \{3;5;7;9;11;13\}$. Les représentations putatives correspondantes doivent être autoduales.

\paragraph*{Étape 3} La méthode géométrique nous permet d'éliminer avec la quantité ${\rm C}$ les éléments $I_{15}\oplus I_v$ avec $v \in \{3;9;11;13\}$ par \emph{croisement} avec $\1, \Delta_{11}, \Delta_{15}$.\ps 

L'élément $I_{15}\oplus I_7$ est \og plus coriace \fg{} et on n'obtient pas de contradiction avec la quantité ${\rm C}$. Comme on sait que la putative représentation correspondante doit être autoduale, on peut utiliser la quantité ${\rm C^s}$. On obtient alors une contradiction pour chaque signe, et une telle représentation n'existe pas.\ps 

On ne peut pas éliminer l'élément $I_{15}\oplus I_5$, mais on peut interdire le signe local $+1$ par croisement avec la représentation triviale.

\paragraph*{Étape 4}  Il existe bien une représentation de $\GL_2$ de conducteur $2$ et de poids motivique 15 par \cite{LMFDB}, c'est $\E_{15}^+$. La combinaison de la formule \cite{Ibu-Kita} avec les Corollaires \ref{cor_dim_s_k_j} et \ref{cor_762} nous dit qu'il existe bien une représentation de $\GL_4$ de conducteur $2$ et de poids $\{\pm \frac{15}{2},\pm \frac{5}{2}\}$. Par l'Étape 3, on connaît son signe local. 

\paragraph*{Étape 5}  Nous ajoutons à la liste des représentations connues les représentations $\E_{15}^+$ et $\E_{15,5}^-$.\medskip


\emph{Remarque 1 :} Puisqu'à l'Étape 2, nous avons montré que toutes les représentations putatives étaient autoduales, et manifestement (très) régulières, nous aurions pu directement aller à l'Étape 4. La combinaison de la formule \cite{Ibu-Kita} avec les Corollaires \ref{cor_dim_s_k_j} et \ref{cor_762} nous aurait dit directement qu'il n'y avait pas d'autre représentation de $\GL_4$ que celle de poids $\{\pm \frac{15}{2},\pm \frac{5}{2}\}$. Comme mentionné au §\ref{Signe_4}, la formule \cite{Ibu-Kita} ne nous donne cependant pas le signe local et on voit toute l'utilité de la formule explicite.\smallskip

\emph{Remarque 2 :} À l'inverse, nous aurions pu vérifier le signe (mais pas l'existence) de la représentation de $\GL_2$ de poids motivique 15 avec la formule explicite. On peut bien interdire le signe $-1$, mais il faut mobiliser quatre représentations connues pour y arriver : $\1, \Delta_{15}, \E_{13}^+, \E_{13}^-$.
 

\subsection{En poids motivique $17$}\label{paragraphe_poids_17}
Comme le laisse entrevoir la Remarque 2 ci-dessus, les limites de la formule explicite commencent à se dessiner. Le poids motivique 17 est d'ailleurs le dernier poids pour lequel nous avons l'énoncé le plus général (Théorème \ref{thm_2_w17}).

\paragraph*{Étape 1} On trouve 23 éléments potentiels :\ps
\begin{itemize}
\item 1 de dimension 2 : $I_{17}$ ;\ps 
\item 8 de dimension 4 : $I_{17}\oplus I_v$ avec $v \in [1;15]\cap (2\Z+1)$ ;\ps 
\item 14 de dimension 6 : $I_{17}\oplus I_v\oplus I_w$ avec $v > w$.
\end{itemize}
\paragraph*{Étape 2} Le Théorème \ref{Multiplicité de Taibi} nous donne que toutes les représentations putatives de dimension 6 sont autoduales. On a également ${\rm m}_2(I_{17}\oplus I_v)\leq 1$ pour $v \in \{1;13;15\}$. Pour tous les autres éléments, on ne peut conclure à l'autodualité à cette étape.

\paragraph*{Étape 3} La méthode géométrique nous permet d'éliminer avec la quantité ${\rm C}$ tous les éléments de dimension 6 (sans utiliser leur autodualité en fait, puisque les calculs effectués avec ${\rm C}$ n'intègrent pas leur signe local, voir le Lemme \ref{lemme_C_flexible}) avec la liste suivante : $\{\1, \Delta_{11}, \Delta_{15}, \Delta_{17}\}$.

Concernant les 8 éléments de dimension 4, après les croisements avec la même liste $\{\1, \Delta_{11}, \Delta_{15}, \Delta_{17}\}$, il nous reste :
\begin{itemize}
\item $I_{17}\oplus I_3$ de multiplicité au plus 1 ;
\item $I_{17}\oplus I_v$ avec $v \in \{5;7;9\}$ de multiplicité au plus 2.
\end{itemize}

\paragraph*{Étape 4}  Il existe bien une représentation (et une seule) de $\GL_2$ de conducteur $2$ et de poids motivique 17 par \cite{LMFDB}, c'est $\E_{17}^-$. \smallskip

La combinaison de la formule \cite{Ibu-Kita} avec les Corollaires \ref{cor_dim_s_k_j} et \ref{cor_762} nous dit qu'il existe une représentation (et une seule) $\pi$ de $\GL_4$ de conducteur 2 \emph{autoduale} avec $\Ll(\pi_\infty) \simeq I_{17}\oplus I_5$.

On sait que ${\rm m}_2(I_{17}\oplus I_5) \leq 2$, ce qui correspond à quatre possibilités (par le fait qu'une représentation et sa contragrédiente ont les mêmes poids) :
\begin{enumerate}
\item aucune représentation ;
\item une seule représentation autoduale ;
\item une représentation non autoduale et sa duale ;
\item deux représentations autoduales.
\end{enumerate}

Les possibilités 1. et 3. ci-dessus sont déjà exclues puisque $\pi$ est autoduale, la possibilité 4. également puisqu'elle impliquerait qu'on trouve \emph{une autre} représentation autoduale. On conclut donc à l'existence d'\emph{une seule} représentation de $\GL_4$ de conducteur 2, de poids $\{\pm\frac{17}{2},\pm\frac{5}{2}\}$. Cette représentation est de plus autoduale. \smallskip

Les choses se passent de la même manière pour $I_{17}\oplus I_v$ avec $v \in \{7;9\}$ pour aboutir à la même conclusion.\smallskip

Pour $I_{17}\oplus I_3$ qui est de multiplicité au plus 1, on peut, utilisant son autodualité si elle existe :
\begin{itemize}
\item soit confronter son existence (avec chacun des signes locaux possibles) à celle des représentations $\1$ et $\Delta_{17}$ par la méthode géométrique avec la quantité ${\rm C^s}$, aboutissant à une contradiction ;
\item soit remarquer que, puisqu'elle est très régulière, elle devrait apparaître, via la même gymnastique, dans la formule \cite{Ibu-Kita} ce qui n'est pas le cas. \smallskip
\end{itemize} 

Nous avons donc établi la liste exhaustive des représentations de conducteur 2 et de poids motivique 17, notons cependant qu'on ne connaît pas le signe local des représentations de $\GL_4$. Comme pour $\E_{15,5}^-$, la formule explicite (c'est un retour à l'Étape 3) nous permet de tester chacun des signes et d'interdire pour chacun des trois cas le signe local en 2 qui donnerait un facteur epsilon global égal à $-1$ (et donc un zéro en $s=\frac{1}{2}$ à la fonction $\Lambda$ correspondante pour lequel il n'y a \og pas de place \fg).

\paragraph*{Étape 5}  Nous ajoutons à la liste des représentations connues les représentations $\E_{17}^-,\,\E_{17,5}^+,\,\E_{17,7}^-,\,\E_{17,9}^+$.\bigskip

\emph{Remarque :} On voit que l'on touche aux limites de la méthode. Jusqu'à présent, on arrivait à éliminer presque toutes les représentations par la seule formule explicite. Les putatives représentations résiduelles étaient entièrement caractérisées (autoduales, de signe local prescrit) et on les \og retrouvait \fg{} dans des objets classiques. Ici, nous avons dû, pour la dernière étape, \emph{utiliser} l'existence d'objets classiques pour pouvoir conclure.

\section{En poids motivique $19$}\label{En_poids_motivique 19}
L'Étape 1 est la seule qui reste en place : on trouve 127 éléments potentiels.

Du fait de ce nombre élevé, il faut non seulement \og industrialiser \fg{} les Étapes suivantes d'un point de vue algorithmique, mais il faut surtout les intriquer (avec des allers-retours entre les Étapes). Nous indiquons ici le déroulé de notre démonstration qui mélange les Étapes selon une dynamique moins linéaire que pour les poids motiviques plus petits.\medskip

La technique la plus efficace étant la méthode géométrique du §\ref{Version géométrique}, notamment en utilisant les représentations connues de conducteur 1, nous réalisons une première série de croisements (c'est un saut à l'Étape 3) avec $\lambda \in \frac{1}{10} \Z \cap [1;5]$ (il n'y a rien à espérer pour $\lambda>5$ d'après les premiers sondages, et les temps de calcul commencent à être significatifs) et la liste $\{\1, \Delta_{11}, \Delta_{15}, \Delta_{17}, \Delta_{19}, \Delta_{19,7}\}$. Il est intéressant d'ailleurs de voir que cette dernière représentation (qui est la première de dimension 4 à apparaître en conducteur 1) joue un rôle crucial dans la constitution de combinaisons linéaires \og contredisantes \fg. \smallskip

Il nous reste alors 17 éléments de $K_\infty^{19}$ : \ps
\begin{itemize}
\item 1 de dimension 2 : $I_{19}$ ; \ps 
\item 6 de dimension 4 : $I_{19}\oplus I_v$ avec $v \in [3;13]\cap (2\Z+1)$ ;\ps 
\item 10 de dimension 6 : $I_{19}\oplus I_v\oplus I_w$ avec $v > w$.\smallskip
\end{itemize}

La comparaison de cette liste avec la liste pour le poids motivique 17 au paragraphe précédent peut donner l'impression que la formule explicite nous donne un nombre de représentations légèrement inférieur pour le poids 19 que pour le poids 17. Il n'en est rien, la liste des 23 éléments du paragraphe \ref{paragraphe_poids_17} est à rapprocher des 127 représentations que nous mentionnons et pour lesquelles nous avons mis en place une première élimination par croisements géométriques \og élémentaires \fg{}.

L'autre grande différence est que les bornes de multiplicité fournies par le Théorème \ref{Multiplicité de Taibi} ne sont plus aussi bonnes (c'est l'Étape 2), et qu'on ne peut donc plus conclure (si rapidement, en tout cas) à l'autodualité automatique des représentations putatives correspondantes.\medskip

Il nous faut alors, comme pour la fin du traitement du poids motivique 17, \emph{utiliser} autant que possible les informations sur les représentations déjà connues et faire un saut à l'Étape 4. Ainsi, on sait déjà par le Corollaire \ref{cor_dim_s_k} (et son amélioration \emph{signée}) avec les formules de dimensions \cite{LMFDB} que l'on a exactement deux représentations de $\GL_2$ de conducteur 2 et de poids motivique 19 : $\E_{19}^+$ et $\E_{19}^-$.

Par ailleurs, on sait relier les dimensions de \cite{Ibu-Kita} à des représentations de $\GL_4$.

\begin{center}
\begin{tabular}{| c |c|}
\hline
   $(w,v)$	& $\dim {\rm S}_{j,k}(\Gamma^{\rm para}(2))$\\
   \hline
   $(19,3)$	& 1 \\
   \hline
   $(19,5)$	& 1 \\
   \hline
      $(19,7)$	& 2 \\
   \hline
      $(19,9)$	& 2 \\
   \hline
      $(19,11)$	& 1 \\
   \hline
      $(19,13)$	& 1 \\
   \hline
 \end{tabular}\newline

Dimensions données par \cite{Ibu-Kita} (on rappelle que $j=v-1, \, k=\frac{w-v}{2}+2$)
\end{center}

On a au moins 6 représentations de $\GL_4$ de conducteur 2 : $\E_{19,3}^?,\, \E_{19,5}^?,\, \E_{19,9}^{?,a},\linebreak \E_{19,9}^{?,b},\, \E_{19,11}^?,\, \E_{19,13}^?$ par la combinaison des Corollaires \ref{cor_761} et \ref{cor_dim_s_k_j}. On remarque que l'on n'a pas de représentation autoduale de $\GL_4$ de conducteur 2 et de poids $\{\pm\frac{19}{2},\pm\frac{7}{2}\}$ par la même combinaison de Corollaires, avec l'existence\footnote{En fait, on peut montrer mieux avec la formule explicite et le Théorème \ref{Multiplicité de Taibi} 3. : le fait que ${\rm m}_1(I_{19}\oplus I_7)=1$ impose que ${\rm m}_2(I_{19}\oplus I_7)<1$. Il n'y a donc aucune représentation correspondante de conducteur 2, \emph{sans hypothèse d'autodualité}.} de $\Delta_{19,7}$ de conducteur 1. On ne connaît pas le signe local par ces formules, d'où le point d'interrogation en exposant.\medskip 

On peut alors, comme précédemment, tâcher de déterminer ce signe par la méthode géométrique (Étape 3) avec la quantité ${\rm C^s}$. On y parvient dans tous les cas, sauf pour la représentation autoduale de $\GL_4$ de conducteur 2 et de poids $\{\pm\frac{19}{2},\pm\frac{11}{2}\}$ pour laquelle les deux signes locaux en 2 semblent possibles.

On utilise alors la deuxième technique évoquée au paragraphe \ref{Signe_4}, \ie on construit un paramètre d'Arthur pour $\SO_7$ (du {\bf Cas 2}) selon le Corollaire \ref{cor_761} et on regarde la dimension (signée) donnée par la Proposition \ref{dim_inv_2_so}.
On lit dans les tables \cite{tabledimSOchenevier} que $\dim {\rm S}_{U_{(19,15,11)}}(7)^+=1$ et $\dim {\rm S}_{U_{(19,15,11)}}(7)^-=0$. Or on peut construire le paramètre d'Arthur très régulier $\E_{19,11}^?+\Delta_{15}$ qui est bien de multiplicité 1 selon le Corollaire \ref{cor_761} : on en déduit donc que le signe local en 2 de $\E_{19,11}^?$ est $+1$. \medskip


Parmi les 10 de dimension 6, on peut parfois faire baisser la multiplicité par croisements géométriques à 1, alors l'autodualité nous permet d'utiliser ${\rm C^s}$ avec les signes $1$ et $-1$, ce qui parfois suffit à trouver une contradiction.\smallskip

Illustrons un cas intéressant. On trouve avec le Théorème \ref{Multiplicité de Taibi} 2. puis avec les croisements selon le paragraphe \ref{La question de la multiplicité} que ${\rm m}_2(I_{19}\oplus I_{17}\oplus I_3)\leq 1$. Si une représentation $\pi$ de $\GL_6$ de conducteur 2 correspondante existe, alors elle est autoduale. En particulier, elle fournit un paramètre d'Arthur (régulier mais non très régulier) selon le Corollaire \ref{cor_761}. Le Corollaire \ref{cor_762} nous dit que dans le cas \emph{régulier}, on a {\it au moins autant} de représentations automorphes discrètes de $\SO_7$ avec les bons poids et les bons invariants que de paramètres d'Arthur globaux $\psi$ symplectiques de dimension $6$ tempérés, tels que considérés au Corollaire \ref{cor_761}.

Or, on a $\dim {\rm S}_{U_{(19,17,3)}}(7)=1$ (on pourrait regarder le signe, mais c'est hors de notre propos ici). Pour interdire l'existence de $\pi$, il suffit de montrer qu'il existe \emph{déjà} un paramètre d'Arthur global qui intervient pour ces invariants. Et c'est bien le cas avec $\E_{19,3}^+ + \Delta_{17}$.\medskip

On peut en fait, en combinant toutes ces étapes, éliminer toutes les représentations qui ne seraient pas très régulières. L'Étape 4 nous fournit alors toutes les représentations autoduales (régulières) et, si $V$ est un élément de $K_\infty^{\leq 19}$ \emph{restant}, on peut encore conclure si le majorant de ${\rm m}_2(V)$ est d'au plus une unité supérieur aux représentations autoduales connues (comme cela a été discuté pour la multiplicité de $I_{17}\oplus I_5$ \emph{supra}).

Il reste néanmoins quelques $V$ récalcitrants et c'est pourquoi nous avons rajouté l'hypothèse \og autoduale \fg{} dans l'énoncé du Théorème \ref{thm_2_w19}. À noter que cette hypothèse est également nécessaire du fait qu'on ne peut pas exclure l'existence de représentations non autoduales en poids motivique \emph{pair} 18.\bigskip

Nous sommes arrivés à l'Étape 5. On ajoute les représentations suivantes :\ps  
\begin{itemize}
\item $2$ pour $\GL_2$ :  $\E_{19}^+,\, \E_{19}^-$ ; \smallskip
\item $6$ pour $\GL_4$ :  $\E_{19,3}^+,\,\E_{19,5}^-,\, \E_{19,9}^+,\,\E_{19,9}^-,\,\E_{19,11}^+,\, \E_{19,13}^-$ ; \smallskip
\item $2$ pour $\GL_6$ :  $\E_{19,13,3}^-,\, \E_{19,13,5}^+$ ; \smallskip
\end{itemize}

Par ailleurs, pour les éléments suivants $V$ de $K_\infty^{\leq 19}$, on ne peut exclure qu'il existe une représentation $\pi$ de conducteur 2 \emph{non autoduale} telle que $\Ll(\pi_\infty)=V$ (et alors $\pi^\vee$ est une autre représentation de conducteur 2 non autoduale vérifiant la même condition) :
\[
I_{19}\oplus I_5,\;I_{19}\oplus I_9,\;I_{19}\oplus I_{11},\;I_{19}\oplus I_{13},\;I_{19}\oplus I_{15}\oplus I_3,\;I_{19}\oplus I_{15}\oplus I_5.
\]

%
%

\section{Conjectures pour le poids motivique $21$}\label{Conjectures pour le poids motivique $21$}
Nous adoptons ici un style beaucoup plus discursif : quoique nous soyons assez certain de la véracité de la Conjecture \ref{thm_2_w21}, elle fait appel à beaucoup d'extensions de nos travaux, que nous préférons exposer en en limitant la technicité. Le plus fort argument en faveur de ces extensions (et de notre Conjecture) est la confirmation sans faille qui leur est apportée par les tables \cite{tabledimSOchenevier}.\bigskip

Nous avons vu dans les paragraphes précédents que nos méthodes touchaient à leurs limites. Précisons un peu les choses dans cet ordre d'idées. L'Étape 1 donne ici 2421 éléments potentiels. Comme pour le poids motivique 19, on a intérêt à essayer tout de suite d'éliminer certains de ces éléments avec la méthode géométrique du §\ref{Version géométrique}. On considère la liste $\mathcal{L}$ constituée des représentations connues suivantes :
\begin{itemize}
\item toutes les représentations de conducteur 1 de poids motivique $\leq 21$ (qui sont listées en préambule de ce Chapitre) ;
\item toutes les représentations autoduales de conducteur 2 de poids motivique $\leq 19$ (qui sont celles du Théorème \ref{thm_2_w19}).\ps 
\end{itemize}

Après croisement avec \emph{une seule} représentation de $\mathcal{L}$, il ne nous reste déjà plus que 997 éléments. Les croisements avec deux représentations de $\mathcal{L}$ permettent de descendre à 605 éléments. Si l'on essaie avec plus de deux représentations, les temps de calcul deviennent sensiblement plus importants (plusieurs dizaines d'heures) pour aucun résultat probant.

Il demeure en fait beaucoup d'éléments de $K_\infty^{\leq 21}$ non réguliers (en particulier avec plusieurs fois $I_{21}$) qu'on ne sait pas gérer à l'Étape 4, quand bien même on parviendrait à montrer qu'ils sont autoduaux.\medskip

Notre Conjecture \ref{thm_2_w21} porte donc une double restriction : on s'intéresse aux représentations automorphes cuspidales de $\GL_n$ \emph{autoduales} et \emph{régulières}. On peut alors en théorie tout traiter selon les Chapitres \ref{Théorie d'Arthur pour SO-2n+1} et \ref{Lien avec des objets classiques} avec les limitations suivantes :
\begin{enumerate}
\item on ne sait que partiellement gérer les éléments réguliers, non très réguliers (\cf Corollaire \ref{cor_762}) ;
\item on n'a pas accès directement au signe local des éléments de $\GL_4$ (\cf §\ref{Signe_4}) ;
\item on n'a pas d'analogue des Corollaires \ref{cor_761} et \ref{cor_762} qui fasse intervenir
des représentations de $\GL_{2n}$ avec $n>4$.
\end{enumerate}

La première limitation nous invite à conjecturer une formule de multiplicité d'Arthur similaire à celle du \ref{thm_mult_Arthur} pour des paramètres d'Arthur non génériques (comme $\psi=\Delta_{21}+\1[6]$ ou $\psi=\Delta_{19}[3]$), avec le caractère $\eps_\psi$ qui n'est plus nécessairement trivial. On a affaire à des paquets d'Arthur locaux non tempérés pour lesquels conjecturalement \og tout se passe bien \fg{} mais sans que l'on sache le démontrer.\ps 

La deuxième limitation est levée selon les stratégies indiquées \emph{loc. cit.} : soit avec la formule explicite, soit en considérant des paramètres d'Arthur globaux faisant intervenir l'élément en question.\ps 

La troisième limitation appelle plus de commentaires. Un espace quadratique $(V,q)$ de dimension $2n+1$ sur $\Q$ relevant du {\bf Cas 2} du §\ref{Groupes étudiés et leurs représentations} n'existe que pour $2n+1 \equiv \pm 1 \bmod 8$. Si $2n+1 \equiv \pm 3 \bmod 8$, on peut construire un espace quadratique $(V,q)$ sur $\Q$ vérifiant :
\begin{itemize}
\item  $q \otimes_\Q \Q_p$ est d'indice de Witt maximal $n$ pour tout nombre premier $p>2$. Ainsi le groupe spécial orthogonal local associé $\bm{\SO_{V_p}}$ est déployé sur $\qp$ pour tout nombre premier $p>2$ et on peut utiliser les résultats de la Première Partie.
\item  $q \otimes_\Q \Q_2$ est d'indice de Witt $n-1$. Plus exactement $V\otimes_\Q \Q_2$ est la somme de $n-1$ plans hyperboliques et d'un espace (unique à isométrie près) anisotrope de dimension 3. Le groupe $\bm{\SO_{V_2}}$ n'est \emph{pas déployé} mais est forme intérieure d'un groupe déployé.
\item la forme quadratique $q \otimes_\Q \R$ est de signature $(2n+1,0)$ (resp. $(0,2n+1)$) donc définie positive (resp. définie négative) et le groupe spécial orthogonal associé $\bm{\SO_{V_\infty}} \simeq \SO_{2n+1,0}/\R \simeq \SO_{0,2n+1}/\R$ est compact. \newline
\end{itemize}
C'est ce qu'on appellera le \textbf{Cas 3}. Nous conjecturons que l'on peut développer une théorie du groupe épiparamodulaire (et de son sous-groupe paramodulaire) similaire à celle qui a été développée au Chapitre \ref{Le groupe paramodulaire}, avec un analogue du Théorème \ref{thm_A}. On peut alors conjecturer l'analogue de la Proposition \ref{dim_inv_2_so}

\begin{defi}
Soient $m \equiv \pm 3\bmod 8$, $V$ le $\Q$-espace vectoriel quadratique ${\rm I}_m \otimes \Q$ et $\mathbf{G}= \bm{\SO_V}$ le $\Q$-groupe algébrique spécial orthogonal de $V$ (relevant donc du {\bf Cas 3}).

Soient $\underline{w} \in {\rm W}_{m}$ et $\eps \in \{+,-\}$. On note $\Pi_{\underline{w}}^{2,\eps}(\mathbf{G})$ l'ensemble des représentations automorphes cuspidales algébriques $\pi$ de $\mathbf{G}$, de poids $\underline{w}$ de conducteur $2$, de signe local $\eps$. Plus précisément,
\begin{itemize}
\item $\pi_\infty \simeq U_{\underline{w}}$ au sens de la Définition \ref{defUw} ;\smallskip
\item $\pi_p$ est non ramifiée pour $p>2$ ; \smallskip
\item $\pi_2^{\K_0(2)}=\{0\}$ et $\pi_2^{(\J(2),\eps)}\neq \{0\}$. \smallskip
\end{itemize}

On pose $$\Pi_{\underline{w}}^{2}(\mathbf{G})=\Pi_{\underline{w}}^{2,+}(\mathbf{G}) \coprod \Pi_{\underline{w}}^{2,-}(\mathbf{G}).$$ 
\end{defi}

\begin{conj}\label{conj_dim_inv_2_so}
Soient $m \equiv \pm 3\bmod 8$, $\underline{w} \in {\rm W}_{m}$ et $\eps \in \{+,-\}$. Alors
\[
{\rm S}_{U_{\underline{w}}}(m)^\eps \simeq \bigoplus_{\pi \in \Pi_{\underline{w}}(\mathbf{G})} \pi_2^{(\J(2),-\eps)}.
\]

Si l'on suppose de plus que $\underline{w}$ est \emph{très régulier} (\ie $|w_i-w_j|>2$ pour $i\neq j$), alors 
\[
\dim {\rm S}_{U_{\underline{w}}}(m)^\eps=\left|\Pi_{\underline{w}}^{2,-\eps}(\mathbf{G}) \right|.
\]
\end{conj}

Il faut bien noter que, selon cette conjecture, les représentations non ramifiées n'apparaissent pas.\medskip

La table $m=3$ de \cite{tabledimSOchenevier} montre cette \og inversion de signe \fg{} à l'œuvre et on retrouve bien toutes les représentations de $\GL_2$ de conducteur 2 connues.

\paragraph*{Heuristique pour $m=5$.}\label{Heuristique pour m=5}
Décrivons maintenant ce qui nous amène à conjecturer une formule de multiplicité et donc un analogue du Corollaire \ref{cor_761} pour le groupe $\SO_5$ relevant du {\bf Cas 3}.\ps 

Les caractères $\chi_v$ pour $v\neq 2,\infty$ sont triviaux (ceci n'est pas conjectural, on retombe sur les groupes et paquets locaux de la Proposition \ref{prop_752}).\ps 

Le caractère $\chi_\infty$ est calculé comme indiqué à la Proposition \ref{prop_calcul_chi_infty} pour les groupes $\SO_{2n+1}$ \emph{compacts à l'infini}. Nous résumons avec les notations \emph{loc. cit.} en écrivant $\underline{\eps}=(+,-)$, la règle pour calculer $\chi_\infty$ étant alors la même. \ps 

Le caractère $\chi_2$ est désormais trivial pour une représentation de conducteur 1 et non trivial (donc égal à $-1$) pour une représentation du conducteur 2. Détaillons un peu l'heuristique.

Le calcul du caractère pour les paquets de représentations du groupe $\SO_5(\Q_2)$ déployé impose que ledit caractère soit trivial sur le centre. Puisque le groupe $\SO_5(\Q_2)$ non déployé en est une forme intérieure, et selon les conjectures générales de la correspondance de Langlands pour les groupes non quasi-déployés, on doit avoir que le caractère pour les paquets de représentations du groupe $\SO_5(\Q_2)$ non déployé est \emph{non trivial sur le centre}. Or, selon le parallélisme que l'on suppose entre notre étude et ce qu'il se passe pour le groupe paramodulaire \emph{non déployé}, les paquets correspondants doivent encore être des singletons, si bien que ledit caractère prend la valeur $-1$.\ps 

Nous adoptons des notations analogues à celles du Corollaire \ref{cor_761} en différenciant néanmoins graphiquement les représentations de conducteur 1 ou 2. Nous noterons $\pi^w$ (resp. $\varpi^w$) une représentation (autoduale symplectique) de $\GL_2$ de poids $\{\pm \frac{w}{2}\}$ de conducteur 1 (resp. 2). Les poids seront rangés dans l'ordre alphabétique strictement croissant, ainsi lorsqu'on écrit $\pi^w+\pi^v$, il est sous-entendu que $w>v$. De même $\varpi^{w,v}$ désigne une représentation autoduale symplectique de $\GL_4$ de poids $\{\pm \frac{w}{2}, \pm \frac{v}{2}\}$ et de conducteur 2.

Calculons, dans l'esprit du Théorème \ref{thm_mult_Arthur} et avec nos \og nouveaux \fg{} caractères la formule de multiplicité.\bigskip

Soit $\psi=\pi^w+\varpi^v$ et soit $\pi \in \Pi(\psi)$. Alors on a $\chi_\infty(s_1)=\chi_2(s_1)=1$ et $\chi_\infty(s_2)=\chi_2(s_2)=-1$. Le paramètre étant générique, on a $\eps_\psi(s_1)=\eps_\psi(s_2)=1$. En particulier, les caractères $\prod_v \chi_v$ et $\eps$ coïncident et $\pi$ est bien automorphe discrète (et de multiplicité 1 dans le spectre discret) : ce paramètre \emph{doit bien être considéré}.

Soit maintenant $\psi=\varpi^w+\pi^v$ et soit $\pi \in \Pi(\psi)$. Alors on a $\chi_\infty(s_1)=1$ et $\chi_2(s_1)=-1$ tandis que $\chi_\infty(s_2)=-1$ et $\chi_2(s_2)=1$. Le paramètre étant générique, on a $\eps_\psi(s_1)=\eps_\psi(s_2)=1$. En particulier, les caractères $\prod_v \chi_v$ et $\eps$ ne sont pas égaux et $\pi$ n'est pas automorphe discrète : ce paramètre \emph{ne doit pas être considéré}.\ps 

Cette gymnastique nous fournit la liste de paramètres suivante (dans l'esprit du Corollaire \ref{cor_761}) :
\begin{itemize}
\item $\varpi^{w,v}$ ;
\item $\pi^w+\varpi^v$ ;
\item $\pi^w +\eta[2]$ avec $\eps_{\rm glob}(\pi^w)=+1$ ;
\item $\varpi^w +\1[2]$ avec $\eps_{\rm glob}(\varpi^w)=-1$ ;
\item $\eta[4]$. \newline
\end{itemize}

On retrouve alors exactement toutes les dimensions de \cite{tabledimSOchenevier}, en particulier, on repère que l'on peut ainsi accéder \emph{directement} au signe local en 2 des représentations de $\GL_4$ de conducteur 2 (comme annoncé au §\ref{Signe_4}).

\paragraph*{La liste de la Conjecture \ref{thm_2_w21}.}
Toutes les représentations listées à l'Annexe \ref{Annexe_poids_21} sont de dimension $\leq 8$. L'existence de celles qui sont très régulières est donc \emph{démontrée} par l'Étape 4. Pour celles qui sont régulières, mais pas très régulières, on peut en général conclure avec le même genre d'arguments que ceux employés pour l'élément $I_{19}\oplus I_{17}\oplus I_3$ au paragraphe \ref{En_poids_motivique 19}.\ps 

Comme mentionné \emph{supra}, nous n'avons \emph{pas besoin} de l'étude conjecturale du {\bf Cas 3} pour déterminer le signe local des représentations de $\GL_4$. On conclut soit avec la formule explicite, soit en considérant des paramètres d'Arthur globaux pour $\SO_7$ ou $\SO_9$ (du {\bf Cas 2}) faisant intervenir l'élément en question. À noter que toutes ces techniques donnent des résultats cohérents.

Un point plus intéressant de la Conjecture \ref{thm_2_w21} est de savoir qu'il n'y a \emph{pas d'autres} représentations automorphes cuspidales algébriques \emph{autoduales régulières} de $\GL_n$ de conducteur 2 que celles indiquées. On pourrait, fort de nos conjectures pour le {\bf Cas 3} étudier l'ensemble des éléments des tables \cite{tabledimSOchenevier} et conclure qu'on arrive à expliquer toutes les dimensions avec les représentations de la Conjecture \ref{thm_2_w21} (et celles de conducteur 1 bien sûr). En fait, on peut faire mieux, c'est-à-dire moins.\medskip

La formule explicite ne donne pas de résultats très concluants en général pour le poids motivique 21, mais on peut s'interroger sur les résultats pour les seules représentations autoduales régulières. On peut énumérer les $2^{10}$ éléments réguliers de $K_\infty^{\leq 21}$ de poids motivique exactement 21 et leur appliquer les Étapes 1 et 3 ci-dessus. Les temps de calculs sont alors considérablement réduits et, après des croisements relativement simples, il ne nous reste que 128 éléments, qui sont tous de dimension $\leq 12$. Ainsi, on doit certes considérer les paramètres au-delà de ceux du Corollaire \ref{cor_761}, mais seulement pour $\SO_{11}$ et $\SO_{13}$ (relevant du {\bf Cas 3}). Par ailleurs, il suffit alors pour éliminer les représentations putatives correspondantes d'avoir au moins autant de paramètres que les dimensions indiquées. Ainsi, c'est seulement une \emph{petite partie} des lignes des tables $m=11$ et $m=13$ de \cite{tabledimSOchenevier} qu'il faut considérer.

Afin de vérifier notre Conjecture \ref{thm_2_w21}, nous avons néanmoins énuméré tous les paramètres correspondants pour toutes les lignes de ces tables (ainsi que pour $m \leq 9$), ce qui corrobore bien l'existence des représentations de la Conjecture \ref{thm_2_w21}, tout en nous donnant un degré d'assurance élevé quant au fait que la liste \emph{loc. cit.} est bien exhaustive.

\section{En conducteur $p>2$}\label{En conducteur p>2}
Nous sommes ici très bref sur la façon de traiter le cas du conducteur $p>2$. Nous avons simplement parallélisé ce qui pouvait l'être immédiatement, c'est pourquoi l'on s'intéresse aux représentations \emph{de poids motivique impair}. En effet, selon la Proposition \ref{CL VIII 3.3}, les représentations autoduales (qui sont les seules que nous sachions étudier) correspondantes sont symplectiques, et en particulier de type (I) au sens de la Proposition-Définition \ref{prop_types_I_et_II}.

L'efficacité de la méthode limitative avec la formule explicite décroît de façon spectaculaire et nous n'arrivons alors qu'à traiter des petits poids, pour lesquels les seules représentations à considérer sont de dimension $2$ ou $4$. Ainsi, bien que les résultats du Corollaire \ref{cor_761} s'appliquent encore, on ne les utilise pas au-delà des cas $\SO_3$ déployé et $\SO_5$ déployé. Il n'est donc pas nécessaire de chercher un analogue au paragraphe \ref{Formes automorphes de conducteur 2_GC} pour \emph{nourrir} cette machinerie.\medskip

On remarque également que chaque conducteur $p$ se gère séparément, \ie on ne fait intervenir dans les calculs que des représentations de conducteur $p$ et, bien sûr, des représentations de conducteur 1. Cela s'explique par le fait que l'inégalité d'Henniart est saturée pour des représentations de conducteur différent (c'est le deuxième cas de la Proposition \ref{prop_cond_2} qui vaut encore en conducteur $p$ de type (I) \emph{mutatis mutandis}) : les quantités en jeu annihilent complètement l'efficacité de la formule explicite.\medskip

Terminons enfin par observer qu'il existe bien des représentations de conducteur $p$ et de type (II), comme on peut le voir avec le caractère de Legendre en $p$ étendu en caractère de Hecke sur $\aq^\times$. Mais elles sont de poids motivique \emph{pair} (en l'occurrence 0).

\appendix
\part*{Annexes}
\chapter{Tables de représentations}\label{Tables_de_representations}

\section{Représentations automorphes cuspidales algébriques de $\GL_n$ de conducteur 2 en poids motivique $\leq 17$ }
\begin{center}
\begin{tabular}{ |c || *{2}{c|} c|}
\hline 
   Poids motivique & $\GL_2$ & $\GL_4$ & $\GL_6$\\
   \hline
   7				& $\E_7^+$ &		&			\\
   \hline
    9				& $\E_9^-$ &		&			\\
   \hline
    11				& 			 &		&			\\
   \hline
    13				& $\bm{\E_{13}^+},\, \E_{13}^-$  &	 	&			\\
   \hline
    15				& $\E_{15}^+$  &	$\E_{15,5}^-$	&			\\
   \hline
    17				& $\E_{17}^-$  &	$\E_{17,5}^+,\, \E_{17,7}^-,\, \E_{17,9}^+$	&		\\
   \hline

 \end{tabular}\newline

Représentations automorphes cuspidales algébriques de $\GL_n$ de conducteur 2 en poids motivique $\leq 17$ 
\end{center}

Elles sont toutes autoduales très régulières, et de signe epsilon global égal à 1, sauf $\bm{\E_{13}^+}$ justement mise en gras (de signe $-1$).

\section{Représentations automorphes cuspidales algébriques autoduales de $\GL_n$ de conducteur 2 en poids motivique $\leq 19$ }
\begin{center}
\begin{tabular}{ |c || *{2}{c|} c|} 
\hline
   Poids motivique & $\GL_2$ & $\GL_4$ & $\GL_6$\\
   \hline
   7				& $\E_7^+$ &		&			\\
   \hline
    9				& $\E_9^-$ &		&			\\
   \hline
    11				& 			 &		&			\\
   \hline
    13				& $\bm{\E_{13}^+},\, \E_{13}^-$  &	 	&			\\
   \hline
    15				& $\E_{15}^+$  &	$\E_{15,5}^-$	&			\\
   \hline
    17				& $\E_{17}^-$  &	$\E_{17,5}^+,\, \E_{17,7}^-,\, \E_{17,9}^+$	&		\\
   \hline
   \multirow{2}{*}{19}  & $\E_{19}^+$ &	$\E_{19,3}^+,\,\E_{19,5}^-,\, \bm{\E_{19,9}^+}$	& $\E_{19,13,3}^-$\\
   & $\bm{\E_{19}^-}$ &	$\E_{19,9}^-,\,\E_{19,11}^+,\, \E_{19,13}^-$	& $\E_{19,13,5}^+$\\

   \hline
 \end{tabular}\newline

Représentations automorphes cuspidales algébriques autoduales de $\GL_n$ de conducteur 2 en poids motivique $\leq 19$ 
\end{center}

Elles sont toutes très régulières, et de signe epsilon global égal à 1, sauf celles en gras (de signe epsilon global égal à $-1$).

%
%

\section{Représentations automorphes cuspidales autoduales algébriques régulières de $\GL_n$ de conducteur 2 et de poids motivique $21$ }\label{Annexe_poids_21}

\begin{center}
\begin{tabular}{| *{4}{c|}}
\hline 
    $\GL_2$ & $\GL_4$ & $\GL_6$ & $\GL_8$ \\
   \hline
	$\bm{\E_{21}^+}$ & $\E_{21,3}^-$ & $\E_{21,13,3}^+$ & $\E_{21,15,9,3}^+$ \\
	$\E_{21}^-$  & $\E_{21,5}^+$ & $\E_{21,13,5}^- $ & $\E_{21,17,9,3}^-$ \\
	& $\E_{21,7}^{-,a}$ & $\E_{21,13,7}^+$ & $\E_{21,17,11,3}^+$ \\
	& $\E_{21,7}^{-,b}$ & $\bm{\E_{21,13,7}^-}$ & $\E_{21,17,11,5}^-$ \\
	& $\E_{21,9}^{+,a}$	 & $\E_{21,15,3}^-$					& $\E_{21,17,13,3}^-$ \\
	& $\E_{21,9}^{+,b}$	 & $\E_{21,15,5}^+$					&$\E_{21,17,13,5}^+$	\\
	& $\bm{\E_{21,11}^+}$ & $\bm{\E_{21,15,7}^+}$ & $\E_{21,19,11,1}^+$	\\
	& $\E_{21,11}^{-,a}$ & $\E_{21,15,7}^-$				& $\E_{21,19,11,3}^-$ \\
	& $\E_{21,11}^{-,b}$ & $\bm{\E_{21,15,9}^-}$	& $\E_{21,19,13,1}^-$ \\
	& $\E_{21,13}^+$	& $\E_{21,17,3}^+$					& $\E_{21,19,13,3}^+$ \\
	& $\E_{21,15}^-$	& $\E_{21,17,5}^-$					& $\E_{21,19,15,3}^-$ \\
	& $\E_{21,17}^+$	& $\E_{21,17,7}^{+,a}$ 				& $\E_{21,19,15,5}^+$\\
	&					& $\E_{21,17,7}^{+,b}$	& \\
	&					& $\E_{21,17,9}^-$ 		& \\
	\hline
 \end{tabular}\newline

Représentations automorphes cuspidales algébriques autoduales de $\GL_n$ de conducteur 2 de poids motivique $21$ (liste conjecturale)
\end{center}

\section{Représentations automorphes cuspidales algébriques de $\GL_n$ de conducteur 3 en poids motivique impair $\leq 13$ }
\begin{center}
\begin{tabular}{ |c || *{1}{c|} c|} 
\hline
   Poids motivique & $\GL_2$ & $\GL_4$ \\
      \hline
   5				& $(5)^-$ &		\\
   \hline
   7				& $(7)^+$ &		\\
   \hline
    9				& $\bm{(9)^+},\,(9)^-$ &		\\
   \hline
    11				& $(11)^+$ &		\\
   \hline
    13				& $\bm{(13)^+},\,(13)^-$ &	$(13,5)^+,\,(13,7)^-$	\\
   \hline

 \end{tabular}\newline

Représentations automorphes cuspidales algébriques autoduales de $\GL_n$ de conducteur 3 en poids motivique impair $\leq 13$ 
\end{center}

Elles sont toutes très régulières, et de signe epsilon global égal à 1, sauf celles en gras (de signe epsilon global égal à $-1$).\bigskip

L'écriture $(w)^\eps$ désigne une représentation automorphe de $\GL_2$ de conducteur 3, de poids $\{\pm\frac{w}{2}\}$ et de signe local en 3 égal à $\eps$.

L'écriture $(w,v)^\eps$ désigne une représentation automorphe de $\GL_4$ de conducteur 3, de poids $\{\pm\frac{w}{2},\pm\frac{v}{2}\}$ et de signe local en 3 égal à $\eps$.

Nous utiliserons encore ces notations pour le conducteur $p\geq3$.

\section{Représentations automorphes cuspidales algébriques de $\GL_n$ de conducteur 5 en poids motivique impair $\leq 11$ }
\begin{center}
\begin{tabular}{ |c || *{1}{c|} c|} 
\hline
   Poids motivique & $\GL_2$ & $\GL_4$ \\
      \hline
   3				& $(3)^+$ &		\\
      \hline
   5				& $(5)^-$ &		\\
   \hline
   7				& $(7)^{+,a},\,(7)^{+,b},\,\bm{(7)^{-}}$ &		\\
   \hline
    9				& $\bm{(9)^+},\,(9)^{-,a},\,(9)^{-,b}$ &	$(9,5)^+$	\\
   \hline
    11				& $(11)^{+,a},\,(11)^{+,b},\,\bm{(11)^{-}}$ &	$(11,5)^-$	\\
   \hline

 \end{tabular}\newline

Représentations automorphes cuspidales algébriques autoduales de $\GL_n$ de conducteur 5 en poids motivique impair $\leq 11$ 
\end{center}

Elles sont toutes très régulières, et de signe epsilon global égal à 1, sauf celles en gras (de signe epsilon global égal à $-1$).

\section{Représentations automorphes cuspidales algébriques de $\GL_n$ de conducteur 7 en poids motivique impair $\leq 7$ }
\begin{center}
\begin{tabular}{ |c || *{1}{c|} c|} 
\hline
   Poids motivique & $\GL_2$ & $\GL_4$ \\
      \hline
   3				& $(3)^+$ &		\\
      \hline
   5				& $\bm{(5)^+},\,(5)^{-,a},\,(5)^{-,b}$ &		\\
   \hline
   7				& $(7)^{+,a},\,(7)^{+,b},\,\bm{(7)^{-}}$ &		\\
   \hline

 \end{tabular}\newline

Représentations automorphes cuspidales algébriques autoduales de $\GL_n$ de conducteur 7 en poids motivique impair $\leq 7$ 
\end{center}

Elles sont toutes très régulières, et de signe epsilon global égal à 1, sauf celles en gras (de signe epsilon global égal à $-1$).

\section{Représentations automorphes cuspidales algébriques de $\GL_n$ de conducteur 11 en poids motivique impair $\leq 7$ }
\begin{center}
\begin{tabular}{ |c || *{1}{c|} c|} 
\hline
   Poids motivique & $\GL_2$ & $\GL_4$ \\
   \hline
   1				& $(1)^-$ &		\\
   \hline
   3				& $(3)^{+,a},\,(3)^{+,b}$ &		\\
   \hline
	5				& $\bm{(5)^+},\,(5)^{-,a},\,(5)^{-,b},\,(5)^{-,c}$ &		\\
\hline
  \multirow{2}{*}{7}				& $(7)^{+,a},\,(7)^{+,b},\,(7)^{+,c},\,(7)^{+,d}$ &	$(7,3)^+$	\\
  									& $\bm{(7)^{-,a}},\,\bm{(7)^{-,b}}$ &	$(7,5)^-$	\\ 
\hline

 \end{tabular}\newline


Représentations automorphes cuspidales algébriques autoduales de $\GL_n$ de conducteur 11 en poids motivique impair $\leq 7$ 
\end{center}

Elles sont toutes régulières, et de signe epsilon global égal à 1, sauf celles en gras (de signe epsilon global égal à $-1$).

\section{Représentations automorphes cuspidales algébriques de $\GL_n$ de conducteur 13 en poids motivique impair $\leq 5$ }
\begin{center}
\begin{tabular}{ |c || *{1}{c|} } 
\hline
   Poids motivique & $\GL_2$ \\

   \hline
   3				& $(3)^{+,a},\,(3)^{+,b},\,\bm{(3)^-}$		\\
 \hline
    5				& $\bm{(5)^{+,a}},\,\bm{(5)^{+,b}},\,(5)^{-,a},\,(5)^{-,b},\,(5)^{-,c}$		\\
\hline

 \end{tabular}\newline


Représentations automorphes cuspidales algébriques autoduales de $\GL_n$ de conducteur 13 en poids motivique impair $\leq 5$ 
\end{center}

Elles sont toutes régulières, et de signe epsilon global égal à 1, sauf celles en gras (de signe epsilon global égal à $-1$).

\section{Représentations automorphes cuspidales algébriques de $\GL_n$ de conducteur 17 en poids motivique impair $\leq 3$ }
\begin{center}
\begin{tabular}{ |c || *{1}{c|} } 
\hline
   Poids motivique & $\GL_2$ \\

   \hline
   1				& $(1)^-$ \\
   \hline
   3				& $(3)^{+,a},\,(3)^{+,b},\,(3)^{+,c},\,\bm{(3)^-}$		\\
 \hline

 \end{tabular}\newline


Représentations automorphes cuspidales algébriques autoduales de $\GL_n$ de conducteur 17 en poids motivique impair $\leq 3$ 
\end{center}

Elles sont toutes régulières, et de signe epsilon global égal à 1, sauf celles en gras (de signe epsilon global égal à $-1$).

\chapter{Démonstration de la formule explicite de Riemann-Weil-Mestre}\label{Démonstration de la formule explicite de Riemann-Weil-Mestre}
Cette Annexe vise à démontrer la formule explicite de Riemann-Weil-Mestre, dans le cadre défini par Jean-François Mestre dans \cite{Mes}.

Comme mentionné en \ref{Enonce_general}, sa démonstration invoque \og des techniques classiques \fg{} d'analyse complexe sans référence spécifique. Voici donc la démonstration détaillée, avec l'ensemble des étapes nécessaires. Pour les rares lemmes que nous n'avons pas démontrés, nous renvoyons à des énoncés précis de la littérature.

Enfin, pour donner à cette Annexe une certaine autonomie, nous rappelons les définitions du paragraphe \ref{Enonce_general} \emph{in extenso} dans le corps du texte.


\section{La fonction $\Gamma$ d'Euler}
On définit la fonction $\Gamma$ d'Euler par la formule suivante :
\begin{align} 
\frac{1}{\Gamma(z)}
&=\lim\limits_{n \rightarrow +\infty} \frac{z(z+1)\cdots (z+n)}{n! n^z} \\
&=\lim\limits_{n \rightarrow +\infty} \frac{z(z+1)\cdots (z+n)}{1 \cdot 2 \cdots (n+1)}n^{1-z}. \label{dev_gamma_euler}
\end{align}

Ceci définit bien une fonction méromorphe sur $\mathbb{C}$ avec des pôles simples aux entiers négatifs. La fonction $\Gamma$ prend des valeurs de tout argument, si bien qu'on ne peut \emph{a priori} utiliser aucune détermination du logarithme pour définir $\log \Gamma$ globalement sur $\mathbb{C} \setminus \mathbb{Z}_-$.

Deux solutions différentes sont proposées par \cite{Bbki} et \cite{Rem}, qui toutes les deux imposent de se restreindre à $\mathbb{C} \setminus \mathbb{R}_- $, où la détermination principale du logarithme est bien définie.

\cite{Rem} propose de partir de la dérivée logarithmique de la fonction $\Gamma$ (bien définie car les pôles sont des réels négatifs) et de l'intégrer le long du segment $[1;z]$. On pose alors :
\[
l(z)=\int_{[1;z]} \frac{\Gamma'(s)}{\Gamma(s)} \dd s,
\]
ce qui définit \emph{un} logarithme de la fonction $\Gamma$ au sens où :
\[
e^{l(z)}=\Gamma(z) \; \; \forall z \in \mathbb{C} \setminus \mathbb{R}_-.
\]

Cette version permet d'énoncer le plus de résultats concernant la fonction $\Gamma$. Pour notre étude, nous suivrons plutôt \cite{Bbki}, qui, quoique moins précis, suffira à nos estimations. On utilise donc la détermination principale du logarithme \emph{ \og en convenant que lorsqu'un logarithme, dans cette formule, porte sur un nombre réel négatif, il a l'une ou l'autre des deux valeurs limites (différant de $2\pi i$) de la détermination principale du logarithme en ce point \fg{}} \cite{Bbki} [VII.12]

Pour alléger les notations, on notera systématiquement `` = '' dans les formules qui suivent, sachant que ce signe est à considérer comme signifiant la congruence modulo $2 i \pi  \mathbb{Z}$ 
, ce qui, quoique déroutant dans les développements asymptotiques de précision supérieure à $O(1)$, s'avérera inoffensif pour les estimées effectives dont nous ferons usage.

\begin{prop}\label{stirl_ex}
On a, pour $z \in \mathbb{C} \setminus \mathbb{R}_- $, avec les conventions citées :
\begin{equation*}
\log \Gamma(z)= (z-\frac{1}{2}) \log z - z +\frac{1}{2} \log (2 \pi) - \int_0^{\uparrow+ \infty} \frac{P(u)}{z+u} \dd u,
\end{equation*}
où $P$ désigne la fonction \og dents de scie \fg{} définie comme la fonction 1-périodique, donnée par $u \mapsto u-\frac{1}{2}$ sur $[0;1[$.
\end{prop}

Le développement \eqref{dev_gamma_euler} nous donne :
\begin{equation*}
- \log \Gamma(z)=\lim\limits_{n \rightarrow +\infty} \left(\sum_{k=0}^n \log \left(\frac{z+k}{1+k} \right)+(1-z) \log n \right).
\end{equation*}
On utilise alors le lemme suivant :
\begin{lemme}
Soit $f \in {\rm C}^1(\mathbb{R}_+,\mathbb{C})$. Alors, pour tout entier naturel $n$
\begin{equation*}
\sum_{k=0}^n f(k) - \int_0^n f(u) \dd u= \frac{f(0)+f(n)}{2}+\int_0^n f'(u) P(u) \dd u,
\end{equation*}
où $P$ désigne la fonction \og dents de scie \fg{} de la Proposition \ref{stirl_ex}.
\end{lemme}
\begin{proof}
Soit $k \in \mathbb{N}^*$, on a :
\begin{align*}
f(k)-\int_{k-1}^k f(u) \dd u
&=f(k)-\int_{k-1}^k f(u) \dd(u-k+\frac{1}{2}) \\
&=f(k)-\frac{f(k)+f(k-1)}{2}+\int_{k-1}^k f'(u) (u-k+\frac{1}{2}) \dd u \\
&=\frac{f(k)-f(k-1)}{2}+\int_{k-1}^k f'(u) P(u) \dd u.
\end{align*}
On somme alors pour $k$ allant de 1 à $n$ et on conclut en ajoutant $f(0)$.
\end{proof}

On considère maintenant la fonction $f : u \mapsto \log (z+u)$ pour $z$ dans le domaine $-\pi +\delta \leq \mathrm{arg}\, z \leq \pi - \delta$, où $\delta \in ]0;\pi[$. On obtient :
\begin{align*}
\sum_{k=0}^n \log(z+k)
&=\int_0^n \log(z+u) \dd u + \frac{\log (z)+\log (z+n)}{2} +\int_0^n \frac{P(u)}{z+u} \dd u \\
&=\left[ (z+u) \log (z+u)-u \right]_0^n + \frac{\log (z)+\log (z+n)}{2} +\int_0^n \frac{P(u)}{z+u} \dd u \\
&=(z+n) \log(z+n) -n - z \log(z) + \frac{\log (z)+\log (z+n)}{2} +\int_0^n \frac{P(u)}{z+u} \dd u 
\\
&=-\left(z-\frac{1}{2}\right) \log (z)+\frac{1}{2} \log(z+n)+(z-1)\log(z+n) \\ &\qquad \qquad + (n+1)\log(z+n) -n +\int_0^n \frac{P(u)}{z+u} \dd u.
\end{align*}
En particulier, pour $z=1$, on obtient :
\[
\sum_{k=0}^n \log(1+k)=\frac{1}{2} \log(n+1)+ (n+1)\log(1+n) -n +\int_0^n \frac{P(u)}{1+u} \dd u,
\]
et en soustrayant cette dernière égalité à l'égalité précédente, on a :
\begin{multline*}
\sum_{k=0}^n \log \frac{z+k}{1+k} = -\left(z-\frac{1}{2}\right) \log (z)+\frac{1}{2} \log \frac{z+n}{n+1}+(z-1)\log(z+n) \\ + (n+1)\log \frac{z+n}{1+n} +\int_0^n \left(\frac{P(u)}{z+u}- \frac{P(u)}{1+u} \right) \dd u.
\end{multline*}

En utilisant le fait que, pour $n$ assez grand, $\log (s+n)=\log n+ O\left(\frac{1}{n}\right)$, on a :
\begin{multline*}
\lim\limits_{n \rightarrow +\infty}\left( (1-z) \log n + \sum_{k=0}^n \log \frac{z+k}{1+k} \right)
=-\left(z-\frac{1}{2}\right) \log (z) \\ +\lim\limits_{n \rightarrow +\infty} \left( (n+1)\log \frac{z+n}{1+n} +\int_0^n \left(\frac{P(u)}{z+u}- \frac{P(u)}{1+u} \right) \dd u \right).
\end{multline*}
On reconnaît $-\log \Gamma(z)$ dans le membre de gauche et on a, par ailleurs :
\[
(n+1)\log \frac{z+n}{1+n}=(n+1)\log \left(1+\frac{z-1}{1+n} \right) \underset{n \rightarrow + \infty}{\longrightarrow} z-1.
\]
Pour conclure à la Proposition \ref{stirl_ex}, il suffit de calculer $\lim\limits_{n \rightarrow +\infty}\int_0^n \frac{P(u)}{1+u} \dd u$, ce qui fait l'objet du 
\begin{lemme}\label{int_imp}
Avec les notations précédentes, on a :
\[
\int_0^{\uparrow +\infty} \frac{P(u)}{1+u} \dd u=\frac{1}{2} \log (2 \pi) -1.
\]
\end{lemme}

Avant de donner la démonstration du lemme, donnons un corollaire de la Proposition \ref{stirl_ex}.

\begin{cor}\label{log_gam}
On a, pour $z \in \mathbb{C} \setminus \mathbb{R}_- $, en considérant la détermination principale du logarithme complexe : 
\begin{equation*}
\log \Gamma(z)= (z-\frac{1}{2}) \log z - z +\frac{1}{2} \log (2 \pi) + O\left(\frac{1}{|z|}\right)
\end{equation*}
uniformément dans tout domaine de la forme $-\pi +\delta \leq \mathrm{arg}\, z \leq \pi - \delta$
\end{cor}

\begin{proof}
On a :
\begin{align*}
\int_0^{\uparrow+ \infty} \frac{P(u)}{z+u} \dd u
&= \sum_{n=0}^{+ \infty} \int_n^{n+1} \frac{P(u)}{z+u} \dd u \\
&=\sum_{n=0}^{+ \infty} \int_0^1 \frac{P(u)}{z+n+u} \dd u \\
&=\sum_{n=0}^{+ \infty} \int_0^1 \left(1-\frac{\frac{1}{2}+z+n}{z+n+u} \right) \dd u \\
&=\sum_{n=0}^{+ \infty} \left( 1- \left(\frac{1}{2}+z+n \right) \log \left(1+\frac{1}{z+n} \right) \right).
\end{align*}
Cette écriture nous permet déjà de démontrer le Lemme \ref{int_imp} :
\begin{proof} (du Lemme \ref{int_imp})

Il s'agit d'effectuer le calcul pour $z=1$, considérons donc les sommes partielles jusqu'à l'indice $N-1$ pour $N \in \mathbb{N}^*$ :
\begin{align*}
\int_0^N \frac{P(u)}{1+u} \dd u
&=\sum_{n=0}^{N-1} \left( 1- \left(\frac{1}{2}+1+n \right) \log \left(1+\frac{1}{1+n} \right) \right) \\
&=\sum_{n=0}^{N-1} \left( 1- \left(n+\frac{3}{2} \right) \log \left(\frac{n+2}{n+1} \right) \right) \\
&=N-\sum_{n=0}^{N-1} \left(n+\frac{3}{2} \right) \log (n+2)+\sum_{n=0}^{N-1} \left(n+\frac{1}{2} \right) \log (n+1)+  \sum_{n=0}^{N-1}\log (n+1) \\
&=N-\left(N+\frac{1}{2} \right) \log (N+1)+ \log(N!).
\end{align*}

On utilise alors la formule de Stirling dans sa forme précise, telle qu'on peut l'obtenir via les intégrales de Wallis :
\[
N! = \sqrt{2 \pi N} \left( \frac{N}{e}\right)^N e^{o(1)}.
\]
On a donc :
\begin{align*}
\int_0^N \frac{P(u)}{1+u} \dd u
&= N-\left(N+\frac{1}{2} \right) \left(\log (N) +\frac{1}{N}+o(\frac{1}{N})\right)+\frac{1}{2} \log(2\pi) \\ &\qquad \qquad +\frac{1}{2} \log(N) +N \log (N) -N + o(1)\\
&=-1+\frac{1}{2} \log(2\pi)+o(1),
\end{align*}
ce qui conclut.
\end{proof}

On reprend la démonstration du Corollaire \ref{log_gam}.

On a, pour $n$ assez grand, et uniformément dans le domaine considéré :
\begin{align*}
& 1- \left(\frac{1}{2}+z+n \right) \log \left(1+\frac{1}{z+n} \right) \\
& \qquad= -\left(\frac{1}{2}+z+n \right) \left(\frac{1}{z+n}-\frac{1}{2(z+n)^2}+\frac{1}{3(z+n)^3}+o\left(\frac{1}{n^3} \right) \right)+1 \\
& \qquad=-1+\frac{1}{2(z+n)}-\frac{1}{3(z+n)^2}-\frac{1}{2(z+n)}+\frac{1}{4(z+n)^2}+1+ o\left(\frac{1}{n^2} \right) \\
& \qquad=-\frac{1}{12(z+n)^2}+o\left(\frac{1}{n^2} \right)
\end{align*}
et on conclut par théorèmes de comparaison pour les séries.
\end{proof}

Posons $z=r e^{i \theta}=\sigma+it$ (les lettres $r,\theta,\sigma,t$ correspondront toujours à ces décompositions dans la suite). Alors :
\begin{align*}
\log \Gamma(z) 	&= (\sigma+it-\frac{1}{2})(\log r+i \theta)-(\sigma+it)+\frac{1}{2} \log (2 \pi) +O(\frac{1}{r}) \\
        				&= (\sigma-\frac{1}{2})\log r - t \theta-\sigma+\frac{1}{2} \log (2 \pi) + i(\cdots) +O(\frac{1}{r}),
\end{align*}
d'où l'on déduit :
\begin{equation}\label{Gamma_estim_pour_ordre}
|\Gamma(z)|=r^{\sigma-\frac{1}{2}}e^{-t \theta} \sqrt{2 \pi} e^{-\sigma} e^{O(\frac{1}{r})}
\end{equation}
uniformément dans tout domaine de la forme $-\pi +\delta \leq \mathrm{arg}\, z \leq \pi - \delta$ lorsque $r=|z| \longrightarrow + \infty$

En particulier, si on fixe une bande verticale de largeur finie 
\begin{equation*}
- \infty < \sigma_0 \leq \mathrm{Re}(s) \leq \sigma_1 < + \infty,
\end{equation*}
on a, lorsque $|t| \longrightarrow +\infty$ :
\begin{equation*}
|\Gamma(z)| \sim |t|^{\sigma-\frac{1}{2}}e^{-\frac{|t| \pi}{2}} \sqrt{2 \pi} e^{-\sigma} 
\end{equation*}
uniformément dans la bande (mais suffisamment loin des pôles aux entiers négatifs) -- en utilisant le fait que, dans cette situation, $\theta~\longrightarrow~\pm \frac{\pi}{2}$ et $r \sim |t|$.

On introduit maintenant la fonction \emph{digamma} d'Euler, que nous utiliserons également dans la suite.

\begin{defi}
On définit la fonction \emph{digamma} comme la dérivée logarithmique de la fonction $\Gamma$ d'Euler. On pose donc, pour $z \in \mathbb{C} \setminus \mathbb{Z}_-$,
\begin{equation*}
\psi(z)=\frac{\Gamma'(z)}{\Gamma(z)}
\end{equation*}
\end{defi}

\begin{prop} \label{digamma_estim}
On a, uniformément dans toute bande verticale de largeur bornée incluse dans le demi-plan $\re s >0$ :
\[
\psi (z)=O(\log |t|).
\]
\end{prop}

\begin{proof}
Il suffit de dériver l'expression de la Proposition \ref{stirl_ex}, on obtient :
\begin{align*}
\psi(z)&=\frac{z-\frac{1}{2}}{z}+\log z -1+\int_0^{+ \infty} \frac{P(u)}{(z+u)^2} \dd u \\
&=-\frac{1}{2z}+\log z +O\left( \frac{1}{z}\right)\\
&=\log z +O\left( \frac{1}{z}\right)
\end{align*}
\end{proof}

Nous avons également une formule intégrale exacte pour la fonction digamma.

\begin{prop} \label{digamma_expl}
On a, pour tout complexe $z$ de partie réelle strictement positive :
\begin{equation*}
\psi(z) = \int_0^{\infty} \left( \frac{e^{-t}}{t} - \frac{e^{-zt}}{1-e^{-t}}\right)\dd t
\end{equation*}
\end{prop}
\begin{proof}
Cette formule est due à Gauss. On en trouve une preuve dans \cite{WW} §12.3.
\end{proof}

\section{Fonctions $\Lambda$}
\subsection{Notations}
Nous rappelons maintenant la notion de fonction $\Lambda$ au sens de \cite{Mes}.
\begin{defi}\label{B_defi préfonction lambda}
\emph{(Définition \ref{defi préfonction lambda})}

Soit $M \in \mathbb{N} $, soit $A \in \mathbb{R}^*$, soit $c \in \mathbb{R}_+$.

On se donne $(a_i)_{1 \leq i \leq M} \in (\mathbb{R}_+)^M$, $(a'_i)_{1 \leq i \leq M} \in \mathbb{C}^M$ tels que $\mathrm{Re}(a'_i) \geq 0$ et $\mathrm{Re}(a_i+a'_i) > 0$ pour tout $i$. 

On se donne, pour tout $p$ premier $(\alpha_j(p))_{1 \leq j \leq M'} \in \mathbb{C}^{M'}$ tel que, pour tout $j$, $| \alpha_j(p)| \leq p^c$.

Une \emph{pré-fonction $\Lambda$} est une fonction méromorphe sur $\mathbb{C}$ vérifiant
\begin{enumerate}
\item $\Lambda$ n'a qu'un nombre fini de pôles.
\item La fonction $\Lambda$ diminuée de ses parties singulières est bornée dans toute bande verticale
\begin{equation*}
- \infty < \sigma_0 \leq \mathrm{Re}(s) \leq \sigma_1 < + \infty.
\end{equation*}
\item Pour $\mathrm{Re}(s) > 1+c$, on a le développement en produit suivant :
\begin{equation} \label{prod_eul+gamma}
\Lambda(s)=A^s \prod_{i=1}^M \Gamma(a_i s+a'_i) \prod_p \prod_{j=1}^{M'} \frac{1}{1-\alpha_j(p)p^{-s}}.
\end{equation}
\end{enumerate}
On appellera fonction $L$ associée à la pré-fonction $\Lambda$ la fonction définie pour $\mathrm{Re}(s) > 1+c$ par :
\begin{equation*}
L(s)=\prod_p \prod_{j=1}^{M'} \frac{1}{1-\alpha_j(p)p^{-s}}.
\end{equation*}
On pose également :
\begin{equation*}
G(s)=A^s \prod_{i=1}^M \Gamma(a_i s+a'_i).
\end{equation*}
\end{defi}

\begin{lemme} La fonction $L$, \emph{a priori} définie seulement pour $\mathrm{Re}(s) > 1+c$ admet de façon évidente un prolongement méromorphe à $\mathbb{C}$ comme quotient des fonctions méromorphes $\Lambda$ et $G$.

De plus, elle n'a qu'un nombre fini de pôles.
\end{lemme}
\begin{proof}
Il suffit de remarquer que la fonction $G$ ne s'annule pas sur $\mathbb{C}$ et que $\Lambda$ n'a qu'un nombre fini de pôles par hypothèse.
\end{proof}

Ce qu'il manque, par rapport aux fonctions $\Lambda$ usuelles, c'est l'équation fonctionnelle. On va ici, l'introduire, par le biais d'un couple de pré-fonctions $\Lambda$.
\begin{defi}
\emph{(Définition \ref{Defi fonctions lambda})}

Considérons un couple de pré-fonctions $\Lambda$, $\Lambda_1$ et $\Lambda_2$. Pour fixer les notations, écrivons la condition 3. de la Définition \ref{B_defi préfonction lambda} pour chacune d'entre elles. 
Pour $\mathrm{Re}(s) > 1+c_1$ (resp. $\mathrm{Re}(s) > 1+c_2$), on a les développements en produit absolument convergent suivants :
\begin{align*}
\Lambda_1(s)&=A^s \prod_{i=1}^{M_1} \Gamma(a_i s+a'_i) \prod_p \prod_{j=1}^{M'_1} \frac{1}{1-\alpha_j(p)p^{-s}} , \\
\Lambda_2(s)&=B^s \prod_{i=1}^{M_2} \Gamma(b_i s+b'_i) \prod_p \prod_{j=1}^{M'_2} \frac{1}{1-\beta_j(p)p^{-s}}.
\end{align*}
On dit alors que $(\Lambda_1,\Lambda_2)$ forme un couple de fonctions $\Lambda$ s'il existe un complexe non nul $w$ tel que :

\begin{enumerate}[(i)]

\item $\Lambda_1(s)=w \Lambda_2(1-s)$ \hfill (équation fonctionnelle)
\item $\sum_{i=1}^{M_1} a_{i}= \sum_{i=1}^{M_2} b_{i}$ \hfill (compatibilité des facteurs archimédiens)
\end{enumerate}
\end{defi}

Comme on l'a dit au paragraphe \ref{Les fonctions Lambda sont des fonctions Lambda}, on peut choisir le même réel positif $c$ pour les deux pré-fonctions $\Lambda$, ainsi que les mêmes entiers $M$ et $M'$, quitte à poser des $a_i=0$, $ a'_i=1$ et $\alpha_j(p)=0$.

\begin{ex}
La fonction $\Lambda$ associée à la fonction $\zeta$ de Riemann :
\begin{equation*}
\Lambda(s)=\pi^{-\frac{s}{2}} \Gamma(\frac{s}{2}) \prod_p \frac{1}{1-p^{-s}}
\end{equation*}
est une pré-fonction $\Lambda$ (pour $M=M'=1$, $c=0$) : elle admet exactement deux pôles simples (en $s=0$ et $s=1$) et la condition de bornitude dans toute bande verticale s'obtient classiquement par l'écriture de $\Lambda$ comme somme de deux intégrales convergentes.

C'est en fait une fonction $\Lambda$ car elle vérifie l'équation fonctionnelle $\Lambda(s)=\Lambda(1-s)$ (on a donc $\Lambda=\tilde{\Lambda}$ et $w=1$)
\end{ex}

On se donne un couple de fonctions $\Lambda$ associées, que l'on note $\Lambda_1$ et $\Lambda_2$. On a donc un réel $c >0$, deux entiers $M$ et $M'$ et on fixe pour l'ensemble du mémoire les notations liées au développement en produit pour $\mathrm{Re}(s)>1+c$ :
\begin{equation*}
\Lambda_1(s)=A^s \prod_{i=1}^M \Gamma(a_i s+a'_i) \prod_p \prod_{j=1}^{M'} \frac{1}{1-\alpha_j(p)p^{-s}}
\end{equation*}
\begin{equation*}
\Lambda_2(s)=B^s \prod_{i=1}^M \Gamma(b_i s+b'_i) \prod_p \prod_{j=1}^{M'} \frac{1}{1-\beta_j(p)p^{-s}}
\end{equation*}
et on note $L_1$ et $L_2$ les produits eulériens associés, ainsi que $G_1$ et $G_2$ les produits \emph{archimédiens} associés.

\subsection{Ordre des fonctions $\Lambda$}

On définit maintenant l'ordre d'une fonction entière.
\begin{defi}
Soit $f$ une fonction \emph{entière}.

On dit que $f$ est \emph{d'ordre fini} s'il existe $M,C>0$ et $\lambda>0$ tels que :
\begin{equation*}
\forall z \in \mathbb{C},|f(z)| \leq M e^{C |z|^\lambda}
\end{equation*}
La borne inférieure des $\lambda$ pour lesquels on trouve une telle majoration s'appelle \emph{l'ordre de $f$}
\end{defi}

On dira, quelque peu improprement qu'une fonction holomorphe est d'ordre fini \emph{dans une certaine zone} (bande verticale, demi-plan, etc.) si elle y vérifie une inégalité du même type.

\begin{thm} \emph{Théorème de Weierstrass}

Soit $f$ une fonction entière d'ordre au plus $1$ dont les zéros, comptés avec multiplicité, sont ordonnés de sorte que $|z_1| \leq |z_2| \leq \cdots$. Alors il existe des constantes $A, B$ telles que :
\begin{equation*}
f(z)=e^{A+Bz} \prod_{n=1}^{+\infty} (1-\frac{z}{z_n}) e^{\frac{z}{z_n}}
\end{equation*}
et le produit est absolument convergent.
\end{thm}
\begin{proof}
On renvoie à \cite{Murty} Section 6.2. Il faut commencer par montrer des propriétés sur la répartition des zéros (\og il n'y en a pas trop \fg) qui assurent l'absolue convergence du produit infini puis voir que le quotient de la fonction de départ par le produit infini s'écrit sous la forme \og exponentielle d'un polynôme de degré au plus 1 \fg{}.
\end{proof}

On prendra garde que le théorème de Weierstrass s'applique aux fonctions \emph{entières}, on ne peut espérer de formulation aussi simple dans le cas des fonctions d'ordre fini \emph{dans une certaine zone}.

\begin{thm}\label{lambda_d'ordre_1}
Une fonction $\Lambda$ corrigée de ses singularités éventuelles est d'ordre au plus $1$.
\end{thm}

\begin{proof}
On énumère les pôles de $\Lambda$ (en nombre fini par définition) comptés avec multiplicité : $\mu_1, \cdots \mu_n$. On pose :
\begin{equation}\label{xi}
\xi(s)=\Lambda(s)\prod_{i=1}^n (s-\mu_i).
\end{equation}
Alors $\xi$ est bien définie sur $\mathbb{C}$ tout entier et y est holomorphe, on peut donc s'intéresser à son ordre.\ps 

Plaçons-nous d'abord sur le demi-plan $\mathrm{Re}(s) \geq c+2$. On a alors, d'après \eqref{Gamma_estim_pour_ordre} :
$$
|\Gamma(s)|=O(e^{|s|^{1+\varepsilon}}) \; \; \forall \varepsilon >0
$$
et donc $|G(s)|=O(e^{|s|^{1+\varepsilon}})$ pour tout $\varepsilon >0$.

Par ailleurs, l'expression de $L$ comme produit eulérien (convergent pour $\mathrm{Re}(s) \geq c+1$) nous assure que $L(s)$ est borné dans le même demi-plan $\mathrm{Re}(s) \geq c+2$.

Enfin le facteur \og correctif \fg{} $\prod_{i=1}^n (s-\mu_i)$ est un $O(s^\varepsilon)$ pour tout $\varepsilon >0$ (les polynômes sont d'ordre zéro).

On a donc, sur le demi-plan $\mathrm{Re}(s) \geq c+2$ :
\begin{equation*}
|\xi(s)|=O(e^{|s|^{1+\varepsilon}}) \; \; \forall \varepsilon >0.
\end{equation*}
L'équation fonctionnelle vérifiée par $\Lambda$ nous assure que cette majoration vaut également sur le demi-plan $\mathrm{Re}(s) \leq -c-1$.\ps 

Plaçons-nous maintenant sur la bande $-c-1 \leq \mathrm{Re}(s) \leq c+2$. Par définition, on sait que $\Lambda$ diminuée de ses parties singulières y est bornée. On a donc :
\begin{align*}
\xi(s)&=\prod_{i=1}^n (s-\mu_i) \left(\Lambda(s)-\sum_{i=1}^{n'} \frac{r_i}{(s-\mu_i)^{d_i}}+\sum_{i=1}^{n'} \frac{r_i}{(s-\mu_i)^{d_i}} \right) \\
&=(\mathrm{fonction \; born\acute{e}e}) \times (\mathrm{polyn\hat{o}me})+(\mathrm{polyn\hat{o}me})
\end{align*}
où l'on a supposé que les pôles \emph{distincts} étaient indicés de 1 à $n'$ (avec $n' \leq n$ pour d'éventuelles multiplicités), donc on a sur cette même bande :
$$
|\xi(s)|=O(e^{|s|^{\varepsilon}}) \; \; \forall \varepsilon >0.
$$

Finalement, $\xi$ est d'ordre fini, et même d'ordre au plus 1.
\end{proof}

\section{Majorations}
Nous allons avoir besoin de différentes majorations pour la suite, nous suivons la stratégie générale adoptée dans \cite{ANT}.
\begin{prop} \label{borne_l}
Pour toute bande verticale de largeur finie, il existe un réel $m \geq 0$ tel que l'on ait $L_i(s)=O(|t|^m)$ pour $|t|$ assez grand (à cause des éventuels pôles qu'il faut éviter).
\end{prop}

On aura besoin pour cela du théorème suivant.
\begin{thm} \emph{Théorème de Phragmen-Lindelöf}

Soit $f$ holomorphe dans une demi-bande verticale
$$
a \leq \sigma \leq b \;  \; \mathrm{ et } \;\; t \geq t_0 >0.
$$

On suppose que
\begin{enumerate}
\item $f$ est d'ordre fini dans la demi-bande.
\item Il existe un réel $w \geq 0$ tel que l'on ait $f(s)=O(|t|^w)$ sur les bords de la demi-bande.
\end{enumerate}
Alors $f(s)=O(|t|^w)$ dans toute la demi-bande verticale.
\end{thm}

\emph{Remarque :} Il suffit de supposer la majoration de bord sur les deux demi-droites $\sigma=a$ et $\sigma=b$, la continuité sur le segment $[a+i t_0 ; b+ i t_0]$ assurant que $f$ y est bornée, donc en particulier est un $O(|t|^w)$ quel que soit $w \geq 0$.
\begin{proof}
Quitte à remplacer $f$ par $s \mapsto \dfrac{f(s)}{s^w}$, on peut supposer que $f$ est bornée sur le bord de la demi-bande et il s'agit alors de voir que $f$ est bornée dans la demi-bande. 

La fonction $f$ est d'ordre fini donc il existe $\alpha \geq 0$ et $B>0$ tels que $f(s) \leq B e^{|s|^\alpha}$ dans la demi-bande.

On choisit $m \in \mathbb{N}$ tel que $m \equiv 2 \; [4]$ et $m > \alpha$ 

On a alors $s^m=(r e^{i \theta})^m=r^m (\cos (m \theta) + i \sin (m \theta))$ et, quitte à augmenter $t_0$, on peut supposer que, pour tout $s$ dans la bande, $\frac{\pi}{4} < \mathrm{arg}\,s <\frac{3 \pi}{4}$ (et donc $\cos (m \theta)\leq \eta<0$). 

On pose, pour $\varepsilon >0$ :
\begin{equation*}
g_\varepsilon : s \longmapsto f(s) e^{\varepsilon s^m}
\end{equation*}
On a alors :
\begin{align*}
|g_\varepsilon  (s) | 	& \leq B e^{r^\alpha} e^{\varepsilon r^m \cos (m \theta)} \\
					& \leq B e^{r^\alpha+ \eta \varepsilon r^m}
\end{align*}
et $r^\alpha+ \eta \varepsilon r^m \longrightarrow - \infty$ lorsque $t=\mathrm{Im}(s) \longrightarrow + \infty$, uniformément dans la demi-bande.

Donc $g_\varepsilon  (s) \longrightarrow 0$ lorsque $t \longrightarrow + \infty$ et on peut en particulier trouver $t_1 > t_0$ tel que, pour tout $t_2 \geq t_1, |g_\varepsilon  (s) | \leq B$ sur le segment horizontal $[a+i t_2 ; b+ i t_2]$. 

On a $|e^{\varepsilon s^m}| \leq 1$ dans la demi-bande donc en particulier sur ses bords, et $|g_\varepsilon|$ est majorée par la même constante que $|f|$ sur les bords, disons $B'$. Quitte à remplacer $B$ par $\mathrm{max}(B, B')$, on a $g_\varepsilon$ bornée par $B$ sur les bords du rectangle délimité par les droites $\sigma=a, \sigma=b, t=t_0, t=t_2$.

Donc, par le principe du maximum, $|g_\varepsilon(s)| \leq B$ dans ce même rectangle. Comme $t_2$ peut être choisi arbitrairement grand, ceci vaut dans toute la demi-bande verticale. 

On a donc, à $s$ fixé et quel que soit $\varepsilon >0$ :
\begin{equation*}
|f(s)| \leq B e^{-\varepsilon r^m \cos (m \theta)}
\end{equation*}
(on remarque que $\cos (m \theta)<0$). On obtient donc  $|f(s)| \leq B$ en faisant tendre $\varepsilon$ vers 0, ce qui conclut.
\end{proof}

\begin{lemme}\label{133}
Les fonctions $L_i$ sont bornées le long de toute droite verticale $\mathrm{Re}(s)= \sigma$ avec $\sigma > 1+c$.
\end{lemme}
\begin{proof}
On considère l'expression de $L_1$ comme produit eulérien, absolument convergent pour $\mathrm{Re}(s)>1+c$ :
\begin{equation*}
L_1(s)=\prod_p \prod_{j=1}^{M'} \frac{1}{1-\alpha_j(p)p^{-s}}.
\end{equation*}

On a, pour $j$ entre 1 et $M'$ :
\begin{equation*}
1 - p^{c-\mathrm{Re}(s)} \leq 1-|\alpha_j(p)p^{-s}| \leq |1-\alpha_j(p)p^{-s}| \leq 1+ |\alpha_j(p)|p^{-\mathrm{Re}(s)} \leq 1 + p^{c-\mathrm{Re}(s)}.
\end{equation*}

La condition $\mathrm{Re}(s)>1+c$ nous assure que les deux expressions extrêmes sont les facteurs d'un produit absolument convergent, pris pour $p$ parcourant l'ensemble des nombres premiers.

Notons $G(s)$ et $H(s)$ les valeurs de ces produits (qui sont donc des quantités réelles, \emph{ne dépendant que de $\mathrm{Re}(s)$}). On a :

\begin{align*}
0  < G(s)  \leq &\prod_p |1-\alpha_j(p)p^{-s}|  \leq H(s)  < +\infty \\
0  < G(s)^{M'} \leq &\prod_{j=1}^{M'} \prod_p |1-\alpha_j(p)p^{-s}|  \leq H(s)^{M'}  < +\infty
\end{align*}
d'où l'on déduit que $L_1(s)$ (ainsi que son inverse) est bornée le long de toute droite verticale $\mathrm{Re}(s)= \sigma$ avec $\sigma > 1+c$. Le même raisonnement s'applique évidemment à $L_2$.
\end{proof}
Il reste à considérer le cas d'une droite verticale $\mathrm{Re}(s)= \sigma$ avec $\sigma <-c$. Si on ne dispose pas de formule explicite pour la fonction $L_1$, on peut exploiter le lien entre les fonctions $\Lambda_1$ et $\Lambda_2$ qui nous donne :
\begin{lemme}
Sur toute droite verticale $\mathrm{Re}(s)= \sigma$ avec $\sigma <-c$, on a :
$$
L_i(s)=O(|t|^{\gamma})
$$
pour un certain $\gamma \geq 0$.
\end{lemme}
\begin{proof}
On a :
\begin{align*}
L_1(s) 	&= \frac{\Lambda_1(s)}{A^s \prod_{i=1}^M \Gamma(a_i s+a'_i)} \\
		&= w \frac{B^{1-s}}{A^s} L_2(1-s) \prod_{i=1}^M \frac{\Gamma(b_i (1-s)+b'_i)}{\Gamma(a_i s+a'_i)}
\end{align*}
Or, sur une droite verticale donnée, $L_2(1-s)$ est bornée par le calcul précédent et $\frac{1}{w} \frac{A^{1-s}}{B^s}$ est constant en module sur cette même droite. Il reste donc à considérer le facteur lié aux fonctions $\Gamma$.\ps 

L'équivalent calculé au \eqref{Gamma_estim_pour_ordre} nous donne :
\begin{align*}
|\Gamma(a_i s+a'_i)|	&=  |\Gamma(a_i (\sigma+it)+a'_i)| \\
						&=  |\Gamma(a_i \sigma + \mathrm{Re}(a'_i) +i (ta_i+\mathrm{Im}(a'_i)))| \\
						& \sim |t'|^{\sigma'-\frac{1}{2}}e^{-\frac{|t'|\pi}{2}} \sqrt{2 \pi} e^{-\sigma'} \\
						& \sim J |t|^\alpha e^{-\frac{a_i|t|\pi}{2}}
\end{align*}
où l'on a noté $\sigma'=a_i \sigma+\mathrm{Re}(a'_i)$ et $t'=a_i t+\mathrm{Im}(a'_i)$, avec $J>0$ et $\alpha \in \mathbb{R}$, l'équivalent étant pris lorsque $|t| \rightarrow + \infty$ sur cette droite.\ps 

De même :
\begin{align*}
|\Gamma(b_i (1-s)+b'_i)|&=  |\Gamma(b_i (1-\sigma-it)+b'_i)| \\
						&=  |\Gamma(b_i (1-\sigma) + \mathrm{Re}(b'_i) +i (-tb_i+\mathrm{Im}(b'_i)))| \\
						& \sim |t''|^{\sigma''-\frac{1}{2}}e^{-\frac{|t''|\pi}{2}} \sqrt{2 \pi} e^{-\sigma''} \\
						& \sim K |t|^\beta e^{-\frac{b_i|t|\pi}{2}}
\end{align*}
où l'on a noté $\sigma''=b_i (1-\sigma)+ \mathrm{Re}(b'_i)$ et $t''=-tb_i+\mathrm{Im}(b'_i)$, avec $K>0$ et $\beta \in \mathbb{R}$.
\newline

On a alors :
\begin{align*}
\left| \prod_{i=1}^M \frac{\Gamma(b_i (1-s)+b'_i)}{\Gamma(a_i s+a'_i)} \right| 
&\sim \prod_{i=1}^M \frac {K_i |t|^{\beta_i} e^{-\frac{b_i|t|\pi}{2}}}{J_i |t|^{\alpha_i}e^{-\frac{a_i|t|\pi}{2}}} \\
& \sim C |t|^\gamma e^{\frac{|t|\pi}{2} \sum_{i=1}^M (a_i-b_i)} \\
& \sim C |t|^\gamma
\end{align*}
d'après l'hypothèse 2 concernant les fonctions $\Lambda$. 
\end{proof}

\begin{proof} (de la Proposition \ref{borne_l})

Une bande verticale de largeur donnée peut toujours être incluse dans une bande : 
$$- \infty < \sigma_0 \leq \mathrm{Re}(s) \leq \sigma_1 < + \infty$$
avec $\sigma_0 < -c$ et $\sigma_1 >1+c$.

Les pôles de $L_i$ sont les pôles de $\Lambda_i$ qui sont en nombre fini, donc $L_i$ est holomorphe sur la demi-bande $t \geq t_0$ pour $t_0$ assez grand. On peut alors appliquer le Théorème de Phragmen-Lindelöf et conclure que $L_i(s)=O(|t|^\alpha)$ dans cette demi-bande pour $t$ assez grand. En considérant $L_i(1-s)$, on a le résultat pour la demi-bande de partie imaginaire négative.
\end{proof}

Ces majorations vont nous permettre de localiser les zéros, ou plus exactement, d'identifier des zones sans zéros, qui permettront de trouver des chemins d'intégration \emph{pratiques}.

\begin{lemme} \label{lem_zer}
Soit $f$ une fonction holomorphe dans le disque fermé de centre $z_0$ et de rayon $R$. Soit $0<r<R$.

On suppose que $f$ a au moins $n$ zéros dans le disque fermé de centre $z_0$ et de rayon $r$, et que $f(z_0) \neq 0$.

Alors $$\left(\frac{R}{r}\right)^n \leq \frac{M}{|f(z_0)|},$$ où $$M=\sup_{|z-z_0|=R} |f(z)|.$$
\end{lemme}

\begin{proof}
On peut supposer $z_0=0$ et qu'on a exactement $n$ zéros comptés avec multiplicité dans le petit disque, notés $a_i$.

On pose, pour $|z| \leq R$ :
\[
g(z)=f(z)\prod_{i=1}^n \frac{R^2-\overline{a_i}z}{R(z-a_i)}.
\]

La fonction $g$ est holomorphe sur le grand disque et le principe du maximum nous donne $|g(0)| \leq \|g\|_R$ (sup sur le grand cercle).\ps 

Si $|z|=R$, on a, en écrivant $z=Re^{i \theta}$ :
\[
\left| \frac{R^2-\overline{a_i}z}{R(z-a_i)} \right|
=\left| \frac{R^2-\overline{a_i}Re^{i\theta}}{R(Re^{i \theta}-a_i)} \right|
= \frac{|R-\overline{a_i}e^{i\theta}|}{|Re^{i\theta}-a_i|} 
=\frac{|R-{a_i}e^{-i\theta}|}{|e^{i\theta}||R-a_ie^{-i\theta}|}
=1
\]
d'où $\|g\|_R=\|f\|_R=M$ avec les notations de l'énoncé.\ps 

Par ailleurs,
\[
|g(0)|
=|f(0)| \prod_{i=1}^n \left|\frac{R^2}{Ra_i}\right|
=|f(0)| \prod_{i=1}^n \frac{R^2}{R|a_i|}
\geq |f(0)| \prod_{i=1}^n \frac{R}{r},
\]
ce qui conclut.
\end{proof}

\begin{prop} \label{136}
Pour $T$ assez grand, le nombre de zéros de $\Lambda_i$ dans chacun des deux rectangles délimités par :
$$ -c \leq \sigma \leq 1+c$$
et
$$ T \leq |t| \leq T+1$$
est $O(\log T)$.
\end{prop}

\begin{proof}
On traite le cas du rectangle de partie imaginaire positive, l'autre cas étant similaire.
On va utiliser les résultats concernant la fonction $L_i$ associée.

Tout zéro éventuel de $\Lambda_i$ est un zéro de $L_i$ (puisque la fonction $\Gamma$ ne s'annule pas sur $\mathbb{C}$, la réciproque est fausse). On majore donc le nombre de zéros de $L_i$.

Comme dans la démonstration de la Proposition \ref{borne_l}, la finitude du nombre de pôles de $\Lambda_i$ (donc de $L_i$) nous assure que $L_i$ est holomorphe pour $\mathrm{Im}(s) \geq t_0$ pour un certain $t_0$. 

Dans la suite, on considère $T>t_0+(2+2c)+1$.

On pose \begin{align*}
z_0&=(2+c)+i(T+\frac{1}{2})\, , \\
r&=\sqrt{(2+2c)^2+\left(\frac{1}{2}\right)^2} \, , \\
R&=er \,.
\end{align*}

Alors $T+\frac{1}{2}-R>t_0$ si bien que la fonction $L_i$ est holomorphe dans le disque fermé de centre $z_0$ et de rayon $R$, le Lemme \ref{lem_zer} nous donne :
$$\left(\frac{R}{r}\right)^n \leq \frac{M}{|L_i(z_0)|},$$
où $n$ désigne le nombre de zéros dans le disque de centre $z_0$ et de rayon $r$, qui contient le rectangle considéré.
Or, on a vu que $|L_i(s)|\geq \alpha >0$ sur la bande verticale $\mathrm{Re}(s)=2+c$ (ce qui justifie également que $L_i(z_0) \neq 0$), et $L_i(s)=O(|t|^m)$ donc :
\[
M=\sup_{|s-z_0|=R}|L_i(s)| \leq \sup_{|s-z_0|=R} K |t|^m \leq K |T+\frac{1}{2}+r|^m \leq K' T^m.
\]

On a donc \{nombre de zéros dans le rectangle considéré\} $\leq n  \log(\frac{R}{r}) \leq \log K'+m \log T - \log \alpha \leq \tilde{K} \log T$.
\end{proof}

\begin{cor} \label{T_m_evite}
Il existe $\eta >0$ tel que pour $m \in \mathbb{N}$ suffisamment grand en valeur absolue, on trouve $T_m \in ]m ; m+1 [$ avec $\Lambda_i$ sans zéros dans chacune des bandes horizontales 
\begin{align*}
T_m-\frac{\eta }{\log (m)} &\leq \mathrm{Im}(s) \leq T_m + \frac{\eta }{\log (m)} \\
-T_m-\frac{\eta }{\log (m)} &\leq \mathrm{Im}(s) \leq -T_m + \frac{\eta }{\log (m)}
\end{align*}
\end{cor}
\begin{proof}
Soit $m_0 \in \mathbb{N}$ suffisamment grand pour que la Proposition \ref{136} s'applique. Soit $m \geq m_0$.

On se place sur la bande horizontale $m \leq \mathrm{Im}(s) \leq m+1$.

D'après la formule \eqref{prod_eul+gamma} et les calculs du Lemme \ref{133}, $\Lambda_1$ n'a pas de zéro pour $\mathrm{Re}(s)> 1+c$. En considérant l'équation fonctionnelle et $\Lambda_2$, on voit qu'il n'y a pas non plus de zéro pour $\mathrm{Re}(s)< -c$.

Tous les zéros éventuels sont donc dans le rectangle délimité par
$$ -c \leq \sigma \leq 1+c$$
et
$$ m \leq t \leq m+1$$ et la proposition précédente nous dit que leur nombre est inférieur à $C \log (m)$ pour une constante réelle $C$ fixée.

De même les zéros éventuels dans la bande horizontale $-m-1 \leq \mathrm{Im}(s) \leq -m$ sont au plus au nombre de $C \log (m)$.

Par le principe de Dirichlet, il existe donc un intervalle (en partie imaginaire) de longueur $\dfrac{1}{2(C+1) \log m}$ ne contenant ni zéro de $\Lambda_i$ ni le conjugué d'un tel zéro . En posant $T_m$ le centre de cet intervalle et $\eta=\dfrac{1}{4(C+1)}$, on a le résultat.
\end{proof}

\begin{prop}\label{lambda_vertical}
Soit $a>0$. Alors, le long de toute droite verticale $\mathrm{Re}(s)=1+c+a$, on a :
$$
\frac{\Lambda'}{\Lambda}(s)=O(\log |t|)
$$
\end{prop}

\begin{proof}
Il suffit d'écrire
$$
\frac{\Lambda'}{\Lambda}(s)=\log A +\sum_{i=1}^M \psi(a_is+a'_i)+\frac{L'}{L}(s)
$$
Par un calcul fait plus bas, dans la section \ref{Partie ultramétrique} \og Partie ultramétrique \fg{}, on a $\frac{L'}{L}(s)$ bornée le long de toute droite verticale fixée.

Par ailleurs, la Proposition \ref{digamma_estim} nous assure que chacun des termes de type \emph{digamma} apparaissant est un $O(\log |s|)$, ce qui conclut.
\end{proof}

\begin{prop} \label{estim_lambda_log}
On a, pour $-c-1 \leq \sigma \leq 2+c$ :
$$ \left| \frac{\Lambda'}{\Lambda}(\sigma+iT_m) \right| \leq D (\log |m|)^2$$
où $T_m$ est défini par le Corollaire \ref{T_m_evite} pour $m$ assez grand.
\end{prop}

\begin{proof}
Soit $m$ suffisamment grand pour que le Corollaire \ref{T_m_evite} s'applique, qui nous donne un $T_m$.

On a $s=\sigma+iT_m$ et on pose $s_0=2+c+iT_m$.

On utilise la fonction $\xi$ introduite en \eqref{xi} qui correspond à la fonction $\Lambda$ corrigée de ses singularités éventuelles :
$$
\xi(s)=\Lambda(s)\prod_{i=1}^n (s-\mu_i)
$$
où les $\mu_i$ correspondent aux pôles de $\Lambda$ comptés avec multiplicité. Ils sont en nombre fini par définition.

Les zéros de $\xi$ sont les zéros de $\Lambda$, et $\xi$ étant entière d'ordre au plus 1 d'après le Théorème \ref{lambda_d'ordre_1}, elle admet d'après le théorème de Weierstrass le développement suivant :
$$
\xi(s)=e^{A+Bs}\prod_{\rho} (1-\frac{s}{\rho})e^{\frac{s}{\rho}}
$$
où $\rho$ parcourt les zéros de $\Lambda$.

On a donc :
$$
\Lambda(s)=e^{A+Bs}\prod_{\rho} (1-\frac{s}{\rho})e^{\frac{s}{\rho}}\prod_{i=1}^n \frac{1}{s-\mu_i}
$$
puis 
\begin{align}\label{dev_lambda_log}
\frac{\Lambda'}{\Lambda}(s)&=B+\sum_\rho (\frac{\frac{-1}{\rho}}{1-\frac{s}{\rho}} + \frac{1}{\rho})- \sum_{\mu} \frac{1}{s-\mu} \notag \\
&=B+\sum_\rho (\frac{1}{s-\rho} + \frac{1}{\rho})- \sum_{\mu} \frac{1}{s-\mu}
\end{align}

En particulier :
$$
\frac{\Lambda'}{\Lambda}(s)-\frac{\Lambda'}{\Lambda}(s_0)=\sum_\rho (\frac{1}{s-\rho}-\frac{1}{s_0-\rho})-\sum_\mu (\frac{1}{s-\mu}-\frac{1}{s_0-\mu})
$$
Tous les calculs s'effectuent hors des pôles et la somme (finie) portant sur les $\mu$ est donc majorée en valeur absolue par un réel fixé $K_1$ indépendamment de $m$ (supposé toujours suffisamment grand).

Si l'on considère la droite $\re(s)=2+c$, alors la proposition précédente nous assure que $\left| \frac{\Lambda'}{\Lambda}(s_0) \right| \leq  K_2 \log |m|$ pour une constante $K_2$ indépendante de $m$ (on utilise $T_m \sim m$).

Il reste à considérer la somme portant sur les zéros de $\Lambda$. On a :
$$
\left| \frac{1}{s-\rho}-\frac{1}{s_0-\rho} \right|=\frac{|s_0-s|}{|s-\rho||s_0-\rho|} \leq \frac{3+2c}{|s-\rho||s_0-\rho|}.
$$
Par ailleurs, en décomposant $\rho=\beta + i \gamma$, on a
$$
|s-\rho|=|\sigma+i T_m - \beta -i \gamma| \geq |T_m - \gamma|
$$
et de même pour $|s_0-\rho|$
d'où :
$$
\left| \frac{1}{s-\rho}-\frac{1}{s_0-\rho} \right| \leq \frac{3+2c}{(T_m - \gamma)^2}.
$$

On va maintenant séparer, dans la somme, les zéros tels que $|T_m - \gamma|<1$ (dont on notera l'ensemble $Z_1$) des autres (dont on notera l'ensemble $Z_2$).

Si $|T_m - \gamma| \geq 1$ alors $|T_m - \gamma|^2+1 \leq 2 |T_m - \gamma|^2$ et 
$$
\sum_{\rho \in Z_2} \left| \frac{1}{s-\rho}-\frac{1}{s_0-\rho} \right| \leq \sum_{\rho \in Z_2} 2\frac{3+2c}{|T_m - \gamma|^2+1}
$$
qui est inférieur à $K_3 \log |m|$ d'après le lemme suivant. \ps 

Pour les zéros dans $Z_1$, on a $|2+c+iT_m-\rho|=|2+c+iT_m-\beta - i\gamma|\geq 2+c-\beta \geq 1$ car $\beta$, partie réelle d'un zéro de $\Lambda$, est dans l'intervalle $[-c;1+c]$.

Donc $\sum_{\rho \in Z_1} \left|\frac{1}{s_0-\rho} \right| \leq \# Z_1 \times 1 \leq K_4 \log |m|$ d'après la Proposition \ref{136}.

À cette étape, on peut donc écrire :
$$
\frac{\Lambda'}{\Lambda}(s)=\sum_{\rho \in Z_1} \frac{1}{s-\rho}+R(s)
$$
avec $|R(s)| \leq K_1+ K_2 \log |m| + K_3 \log |m| + K_4 \log |m| \leq K_5 \log |m| $.

Or on a choisi $s$ de partie imaginaire $T_m$, de sorte que, par le Corollaire \ref{T_m_evite}, 
$$|s-\rho|=|\sigma+iT_m-\beta-i\gamma| \geq |T_m -  \gamma| \geq \frac{\eta}{\log |m|}$$
d'où en utilisant de nouveau $\# Z_1 \leq K_4 \log |m|$ :
$$
\left| \frac{\Lambda'}{\Lambda}(s) \right| \leq K_4 \log |m| \times \eta \log |m| + K_5 \log |m| \leq D (\log |m|)^2.
$$
\end{proof}

\begin{lemme}
On a, pour $\rho$ parcourant tous les zéros de $\Lambda$, et pour $t$ suffisamment grand en valeur absolue :
$$
\sum_{\rho} \frac{1}{1+(t-\gamma)^2}=O(\log |t|)
$$
Cette majoration est \emph{a fortiori} vérifiée pour $\rho$ parcourant seulement l'ensemble $Z_2$.
\end{lemme}

\begin{proof}
De \eqref{dev_lambda_log} et des majorations déjà faites concernant la somme portant sur les pôles, on tire :
$$
\frac{\Lambda'}{\Lambda}(s)=\sum_\rho (\frac{1}{s-\rho} + \frac{1}{\rho})+O(1).
$$

Pour $s=2+c+it$, on obtient, en utilisant la majoration de la Proposition \ref{lambda_vertical} le long de la droite $\re(s)=2+c$, et en prenant les parties réelles :
$$
\sum \re(\frac{1}{s-\rho} + \frac{1}{\rho}) \leq A \log |t|.
$$
Or 
\begin{align*}
\re(\frac{1}{s-\rho})&=\re(\frac{1}{2+c+it-\beta-i\gamma}) \\
					 &=\frac{2+c+\beta}{(2+c+\beta)^2+(t-\gamma)^2} \\
					 & \geq \frac{1}{(2+c+\beta)^2+(t-\gamma)^2} \\
					 & \geq \frac{1}{(2+2c)^2+(t-\gamma)^2}
\end{align*}
en utilisant que $-c \leq \beta \leq 1+c$ (partie réelle d'un zéro de $\Lambda$).

Par ailleurs,
$$
\re \frac{1}{\rho}=\frac{\beta}{|\rho|^2}
$$
et, $\beta$ étant borné, il suffit de voir que $\sum \frac{1}{|\rho|^2}$ converge pour conclure que la somme des $\frac{1}{\rho}$ est bornée en module (donc \emph{a fortiori} sa partie réelle est bornée).

Or les zéros de $\Lambda$ sont les zéros de $\xi$ et cette dernière fonction est d'ordre au plus 1 ; un résultat général concernant les fonctions d'ordre fini (\emph{cf.} \cite{Murty} Exercice 6.1.6) assure que $\sum \frac{1}{|\rho|^{1+\varepsilon}} < + \infty$ pour tout $\varepsilon >0$.

On peut donc écrire
$$
\sum_{\rho} \frac{1}{(2+2c)^2+(t-\gamma)^2} \leq A' \log |t|
$$
et on conclut en remarquant que, de $(2+2c)^2>1$ on tire $(2+2c)^2+(t-\gamma)^2 \leq (2+2c)^2(1+(t-\gamma)^2)$.
\end{proof}

\section{Formules explicites}
\subsection{Fonction test et formule des résidus}\label{Fonction explicite}

\begin{defi} \emph{(\cite{Mes})}\label{Définition fonction explicite} \emph{(Définition \ref{defi fonction explicite})}

Soit $c\geq 0$. Soit $F:\R \rightarrow \R$, mesurable, paire. On dit que $F$ est une fonction test (de niveau $c$) si : 
\begin{enumerate}[(i)]
\item il existe $\varepsilon >0$ tel que $x \mapsto F(x)e^{(\frac{1}{2}+c+\varepsilon)|x|} \in {\rm L}^1(\mathbb{R})$,
\item il existe $\varepsilon >0$ tel que $x \mapsto F(x)e^{(\frac{1}{2}+c+\varepsilon)|x|} $ soit à variation bornée et normalisée (i.e. égale en chaque point à la moyenne de ses limites à gauche et à droite),
\item la fonction $x \mapsto \dfrac{F(x)-F(0)}{x}$ est à variation bornée.
\end{enumerate}
\end{defi}

On se donne $F$ une fonction test et un réel $\varepsilon$ tel que les conditions $(ii)$ et $(iii)$ soient remplies. On se donne un entier naturel $m$ suffisamment grand pour être redevable du Corollaire \ref{T_m_evite}, qui a vocation à tendre vers $+ \infty$ (et $T_m$ avec).
On pose $\Psi_{\varepsilon}(x)=F(x)e^{(\frac{1}{2}+c+\varepsilon)|x|}$ (qui est également paire).\ps 

On pose, pour $s$ complexe avec $\mathrm{Re}(s) \in [-c- \frac{\varepsilon}{2};1+c+\frac{\varepsilon}{2}]$ :
$$ \Phi(s)= \int_{- \infty}^{+ \infty} F(x) e^{(s-\frac{1}{2})x} dx$$

Cette intégrale est bien définie car on a, pour $x \geq 0$ :
\begin{align*}
|F(x) e^{(s-\frac{1}{2})x}|
&=|F(x)| e^{(\sigma-\frac{1}{2})x} \\
& \leq |F(x)| e^{(1+c+\frac{\varepsilon}{2}-\frac{1}{2})x}  \\
& \leq e^{-\frac{\varepsilon}{2}x}|\Psi_{\varepsilon}(x)|
\end{align*}
La condition $(iii)$ nous donne que la fonction $\Psi_{\varepsilon}$ est en particulier bornée et alors, la fonction $x \mapsto F(x) e^{(s-\frac{1}{2})x}$ est non seulement intégrable au voisinage de $+ \infty$ mais également à décroissance rapide.

De même, pour $x<0$, on a 
$$
|F(x) e^{(s-\frac{1}{2})x}| \leq e^{\frac{\varepsilon}{2}x}|\Psi_{\varepsilon}(x)|
$$
donc, là encore, on a intégrabilité et décroissance rapide au voisinage de $- \infty$.
\newline

La fonction $\Phi$ est donc holomorphe, comme intégrale à paramètre, et la formule des résidus appliquée à $s \mapsto \Phi(s) \frac{\Lambda_1'(s)}{\Lambda_1(s)}$ sur le bord du rectangle
$$ \left[-c- \frac{\varepsilon}{2};1+c+\frac{\varepsilon}{2}\right] \times [-T_m;T_m]$$
donne :
\begin{equation}\label{resid}
 \frac{1}{2i\pi} \int \Phi(s) \frac{\Lambda_1'(s)}{\Lambda_1(s)} \dd s = \sum_{|\mathrm{Im}(\rho)|<T_m} \Phi(\rho) - \sum_{|\mathrm{Im}(\mu)|<T_m} \Phi(\mu)
\end{equation}
où $\rho$ (resp. $\mu$) parcourt les zéros (resp. les pôles) de $\Lambda_1$ de partie réelle comprise entre $-c- \frac{\varepsilon}{2}$ et $1+c+\frac{\varepsilon}{2}$ comptés avec multiplicité. Il n'y a pas de zéros sur les bords du rectangle d'après le Corollaire \ref{T_m_evite} (pour les bords horizontaux) ainsi que par la formule \eqref{prod_eul+gamma} (pour les bords verticaux).

\emph{Remarque :} En réalité, on a considéré tous les zéros et tous les pôles de la fonction $\Lambda_1$ car on a déjà vu que l'écriture en produit interdisait l'existence de zéros ou de pôles dans le demi-plan $\mathrm{Re}(s)>1+c$, l'équation fonctionnelle interdisant l'existence dans le demi-plan \og symétrique \fg{} $\mathrm{Re}(s)<-c$.

\begin{prop}\label{Phi_estim} Avec les notations précédentes, on a :
$\Phi(s)=O\left(\dfrac{1}{|t|}\right)$ lorsque $|t| \longrightarrow + \infty$
\end{prop}

\begin{proof}
Il s'agit d'effectuer une intégration par parties \og fonction à variation bornée contre fonction continue \fg

On rappelle qu'on a la série d'implications suivante : \newline
\og VB sur \fg $\mathbb{R} \Rightarrow$ \og VB sur chaque segment \fg $\Rightarrow$ \og différence de deux fonctions croissantes sur chaque segment \fg $\Rightarrow$ \og dérivable presque partout \fg

On peut donc écrire : 
\begin{align*}
\Phi(s) &= \int_{- \infty}^{+ \infty} F(x) e^{(\sigma-\frac{1}{2})x}e^{itx} \, \dd x \\
		&= \left[F(x)e^{(\sigma-\frac{1}{2})x}\frac{e^{itx}}{it}\right]_{- \infty}^{+ \infty} - \frac{1}{it} \int_{- \infty}^{+ \infty} e^{itx} \,\dd (F(x)e^{(\sigma-\frac{1}{2})x})
\end{align*}

Or, on a vu plus haut que la fonction $x \mapsto F(x) e^{(\sigma-\frac{1}{2})x}$ est à décroissance rapide à l'infini.

Le premier terme est donc égal à $0$ et la seconde intégrale est bornée  par la variation totale, si bien qu'on obtient le résultat recherché.
\end{proof}

\begin{cor}
Les parties horizontales de l'intégrale (\ref{resid}) tendent vers $0$ quand $m$ (ou $T_m$) tend vers l'infini.
\end{cor}

Pour alléger les notations, on note désormais $T$ au lieu de $T_m$ et on sera amené à raisonner avec le fait que $T$ tend vers l'infini lorsque $m$ tend vers l'infini.
\begin{proof}
Il suffit de combiner la majoration précédente avec l'estimée de la Proposition \ref{estim_lambda_log}
\end{proof}

On s'intéresse donc aux parties verticales de l'intégrale (\ref{resid}) et on remarque que l'équation fonctionnelle reliant $\Lambda_1$ et $\Lambda_2$ permet d'écrire, en posant $a=c+\frac{\varepsilon}{2}$ :
$$\int_{-a+iT}^{-a-iT} \Phi(s) \frac{\Lambda_1'(s)}{\Lambda_1(s)} \dd s
=\int_{1+a-iT}^{1+a+iT} \Phi(1-u) \frac{\Lambda_1'(1-u)}{\Lambda_1(1-u)} (- \dd u)$$
en faisant le changement de variable $u=1-s$.

Or, on a :
\begin{equation}\label{Phi(s)=Phi(1-s)}
\begin{aligned}
\Phi(1-s)	&=\int_{- \infty}^{+ \infty} F(x) e^{(1-s-\frac{1}{2})x} \dd x \\
			&=\int_{- \infty}^{+ \infty} F(x) e^{(\frac{1}{2}-s)x} \dd x \\
			&=\int_{- \infty}^{+ \infty} F(-x) e^{(s-\frac{1}{2})x} \dd x \\
			&=\Phi(s)
\end{aligned}
\end{equation}
par parité de $F$.\ps 

Finalement, en utilisant le fait que $w \Lambda_2'(s)=- \Lambda_1'(1-s)$, on obtient que les parties verticales de l'intégrale (\ref{resid}) peuvent s'écrire :
\begin{align*}
&\int_{1+a-iT}^{1+a+iT} \left(\Phi(s) \frac{\Lambda_1'(s)}{\Lambda_1(s)}
+ \Phi(1-s) \frac{\Lambda_2'(s)}{\Lambda_2(s)} \right) \dd s \\
&=\int_{1+a-iT}^{1+a+iT} \Phi(s) \left(\frac{\Lambda_1'(s)}{\Lambda_1(s)}
+  \frac{\Lambda_2'(s)}{\Lambda_2(s)} \right) \dd s 
\end{align*}

Comme tout se passe désormais pour $\mathrm{Re}(s)>1+c$, on peut utiliser l'écriture en produit donnée par (\ref{prod_eul+gamma}), particulièrement commode pour le calcul de la dérivée logarithmique, ce qui donne :
\begin{equation} \label{decomp_arch_ultr}
\frac{\Lambda_1'(s)}{\Lambda_1(s)}=
\underbrace{\log A +\sum_{i=1}^M \psi(a_is+a'_i)}_{\frac{G_1'(s)}{G_1(s)}}
\underbrace{-\sum_p \sum_{j=1}^{M'} \frac{\alpha_j(p) \log(p)}{1-\alpha_j(p) p^{-s}}}_{\frac{L_1'(s)}{L_1(s)}}
\end{equation}
Une telle décomposition est aussi valable pour la dérivée logarithmique de $\Lambda_2$.

\subsection{Partie ultramétrique} \label{Partie ultramétrique}
On s'intéresse à l'intégrale de $\Phi$ contre la partie ultramétrique de \eqref{decomp_arch_ultr} le long du segment $[1+a-iT;1+a+iT]$.

Or 
\begin{align*}
\frac{L_1'(s)}{L_1(s)}
&=-\sum_p \sum_{j=1}^{M'} \frac{\alpha_j(p) \log(p)}{1-\alpha_j(p) p^{-s}} \\
&= -\sum_p \sum_{j=1}^{M'} \alpha_j(p) \log(p) p^{-s} \sum_{k=0}^{+ \infty} \alpha_j(p)^k p^{-ks} \\
&=-\sum_p \sum_{j=1}^{M'} \sum_{k=1}^{+ \infty} \log(p)\alpha_j(p)^k p^{-ks}.
\end{align*}

On a donc :
\begin{align*}
&\frac{1}{2 i \pi} \int_{1+a-iT}^{1+a+iT} \Phi(s) \frac{L_1'(s)}{L_1(s)} \dd s \\
= &-\frac{1}{2 \pi} \int_{-T}^T \left(\int_\mathbb{R} F(x) e^{(\frac{1}{2}+a+it)x} \dd x \right) \sum \log(p)\alpha_j(p)^k p^{-k(1+a+it)} \dd t \\
= &-\frac{1}{2 \pi} \int_{-T}^T \sum \left(\int_\mathbb{R} F(u+k \log p) e^{(\frac{1}{2}+a+it)u}  p^{k(\frac{1}{2}+a+it)} \dd u \right) \log(p)\alpha_j(p)^k p^{-k(1+a+it)} \dd t \\
= &-\frac{1}{2 \pi} \int_{-T}^T \sum \left(\int_\mathbb{R} F(u+k \log p) e^{(\frac{1}{2}+a+it)u}  \dd u \right) \frac{\log(p)\alpha_j(p)^k}{p^{\frac{k}{2}}}  \dd t \\
= &-\frac{1}{2 \pi} \int_{-T}^T \left(\int_\mathbb{R} H(u) e^{itu}  \dd u \right) \dd t \\
\end{align*}
avec le changement de variable affine $x= u+k \log p$ et en posant :
\begin{align*}
H(u)&= \sum F(u+k \log p) e^{(\frac{1}{2}+a)u} \frac{\log(p)\alpha_j(p)^k}{p^{\frac{k}{2}}} \\
&= \sum F_a(u+k \log p) \frac{\log(p)\alpha_j(p)^k}{p^{k(1+a)}}
\end{align*}
où $F_a : x \mapsto F(x) e^{(\frac{1}{2}+a)x}$.

Il reste à justifier l'interversion \og somme-intégrale \fg{} qui a fait apparaître la fonction $H$.

La somme sur l'indice $j$ étant finie, elle n'a pas d'incidence sur la convergence, on suppose pour la clarté des calculs que $M'=1$ et on omet à partir de maintenant l'indice $j$. Il reste donc une double sommation sur $k \geq 1$ et sur $p$ premier.

La fonction $F$ étant une fonction test, on a $F_a$ à variation bornée, et en particulier bornée en valeur absolue sur $\mathbb{R}$ par une constante $W$. Alors
\begin{align*}
\sup_{u \in \mathbb{R}}\left| F_a(u+k \log p) \frac{\log(p)\alpha(p)^k}{p^{k(1+a)}} \right|
&\leq W \frac{\log p (p^{c})^k}{p^{k(1+c+\frac{\varepsilon}{2})}} \\
&\leq W \frac{\log p}{p^{k(1+\frac{\varepsilon}{2})}}
\end{align*}
en utilisant $a=c+ \frac{\varepsilon}{2}$ ainsi que l'hypothèse sur les fonctions $\Lambda$ concernant le module des $\alpha(p)$.

On somme alors en $k$, à $p$ fixé :
\begin{align*}
\sum_{k=1}^{+ \infty} W \frac{\log p}{p^{k(1+\frac{\varepsilon}{2})}}
&= W \frac{\log p}{p^{1+\frac{\varepsilon}{2}}} \sum_{k=0}^{+ \infty} \left(\frac{1}{p^{1+ \frac{\varepsilon}{2}}} \right) ^k \\
&= W \frac{\log p}{p^{1+\frac{\varepsilon}{2}}-1} \\
&\sim_{p \rightarrow + \infty} W \frac{\log p}{p^{1+\frac{\varepsilon}{2}}}.
\end{align*}

On obtient donc le terme général d'une série convergente en $p$, si bien que la double somme converge et donc la série définissant $H$ converge normalement sur $\mathbb{R}$.

En remarquant que si $F$ est une fonction test, alors $F_a \in {\rm L}^1(\mathbb{R})$, on a, via le même calcul avec une intégration en plus, légitimé l'interversion ci-dessus, et montré que $H \in {\rm L}^1(\mathbb{R})$. De même, on a $F_a$ à variation bornée et donc, par sous-additivité de la variation totale, $H$ à variation bornée.

On a donc $\int_\mathbb{R} H(u) e^{itu}  \dd u= \widehat{H}(t)$ où $\widehat{\cdot}$ désigne la transformée de Fourier sur le groupe localement compact $\mathbb{R}$.

On utilise une version faible de l'inversion de Fourier
\begin{thm} \emph{(Pringsheim)}

Soit $f$ une fonction complexe de la variable réelle, à variation bornée. On suppose que $f(t)$ tend vers 0 lorsque $t$ tend vers $\pm \infty$. Alors :
$$
\frac{1}{2 \pi} \lim_{T \rightarrow +\infty} \int_{t=-T}^{t=T} e^{-ity} \left( \int_{x=-\infty}^{x=+\infty} f(x) e^{itx} \dd x \right) \dd t= \frac{f(y)^+ + f(y)^-}{2}.
$$
\end{thm}
\begin{proof}
On peut y voir un analogue du théorème de Dirichlet concernant la convergence ponctuelle des séries de Fourier. On trouve la preuve dans \cite{Riesz}.
\end{proof}
On a alors :
\begin{align*}
\frac{1}{2 i \pi} \int_{1+a-iT}^{1+a+iT} \Phi(s) \frac{L_1'(s)}{L_1(s)} \dd s
& \quad=-\frac{1}{2 \pi} \int_{-T}^T \widehat{H}(t) e^{-it \cdot 0}  \dd t \\
& \underset{T \rightarrow + \infty}{\longrightarrow}-H(0) \; \; \; \mbox{\footnotesize en utilisant le théorème précédent} \\
& \quad =-\sum_p \sum_{j=1}^{M'} \sum_{k=1}^{+ \infty} F(k \log p) \frac{\log(p)\alpha_j(p)^k}{p^{\frac{k}{2}}}
\end{align*}

On utilise le fait que $H$ est normalisée au sens du {\it (ii)} de la Définition \ref{Définition fonction explicite}, ce qui découle immédiatement du fait que $F$ est normalisée (ce qui vient à son tour de la définition d'une fonction explicite).

\subsection{Partie archimédienne}
On considère maintenant à l'intégrale de $\Phi$ contre la partie archimédienne de \eqref{decomp_arch_ultr} le long du segment $[1+a-iT;1+a+iT]$.

On commence par remarquer qu'on peut, puisque $T$ va tendre vers $+ \infty$, remplacer le segment $[1+a-iT;1+a+iT]$ par $[\frac{1}{2}-iT;\frac{1}{2}+iT]$. En effet, intégrons le long des bords du rectangle délimité par ces deux segments, où il n'y a pas de pôles :

\begin{align*}
0 	&= \int \Phi(s) \frac{G'(s)}{G(s)} \dd s \; \; \; \mbox{\footnotesize par le théorème des résidus} \\
	&= \int_{1+a-iT}^{1+a+iT} + \int_{\frac{1}{2}-iT}^{\frac{1}{2}+iT} + 2\left(O(\log T) \times o\left(\frac{1}{T}\right)\right),
\end{align*}
les majorations à droite provenant des Propositions \ref{digamma_estim} et \ref{Phi_estim}.
Ainsi la différence entre les deux intégrales tend vers 0 lorsque $T$ tend vers $+ \infty$.\ps 

Commençons par le terme correspondant à $A^s$, on a :
\begin{align*}
\int_{\frac{1}{2}-i T}^{\frac{1}{2}+i T} \Phi(s) \log A \;\dd s
&\quad= \int_{t=- T}^{t=T} \Phi(\frac{1}{2}+it) \log A \;i \; \dd t \\
&\quad= i \log A \int_{t=- T}^{t=T} \int_{x=- \infty}^{x=+\infty} F(x) e^{itx} \dd x \; \dd t \\
&\quad = i \log A \int_{t=- T}^{t=T} \widehat{F}(t) e^{-it \cdot 0} \dd t \\
&\underset{T \rightarrow + \infty}{\longrightarrow}  2 i \pi (\log A) F(0)
\end{align*}
par le même argument d'inversion de Fourier \og faible \fg{}.

On pose $\varphi (t)=\Phi(\frac{1}{2}+it)$, on doit alors calculer les limites lorsque $T$ tend vers l'infini des intégrales :
$$
\int_{\frac{1}{2}-i T}^{\frac{1}{2}+i T} \Phi(s) \psi(a_is+a'_i) \dd s
=i\int_{t=- T}^{t=T} \varphi(t) \psi(\sigma_i +i a_i t) \dd t
$$
où $\sigma_i=\frac{a_i}{2}+a'_i >0$ d'après les hypothèses faites sur $a_i$ et $a'_i$. Pour la lisibilité des calculs qui suivent, on omettra les indices $i$.

On va avoir besoin du
\begin{lemme} \label{lem_122}
Soit $k \in {\rm L}^1(\mathbb{R}) \cap {\rm L}^2(\mathbb{R})$. On pose
$$ \rho(t)=\int_{- \infty}^{+ \infty} k(x) \frac{1-e^{itx}}{x} \dd x $$

Soit $\gamma$ la transformée de Fourier de $ x \mapsto \dfrac{F(x)-F(0)}{x}$.
On suppose que $\lim\limits_{t \rightarrow +\infty} \rho(t) \gamma(t)=0$.

Alors $\rho \varphi \in {\rm L}^1(\mathbb{R})$ et 

$$ \frac{1}{2 \pi} \int_{- \infty}^{+ \infty} \rho(t) \varphi(t)dt=\int_{- \infty}^{+ \infty} k(x) \frac{F(0)-F(x)}{x} \dd x.$$
\end{lemme}
\begin{proof}
La démonstration se trouve dans \cite{Poit}, Lemme 2, p.5.
\end{proof}

On a, en utilisant la Proposition \ref{digamma_expl} ($\sigma >0$) :
\begin{align*}
\psi(\sigma+iat)- \psi (\sigma)
&= - \int_0^{+ \infty} \frac{e^{-(\sigma + iat)x}-e^{-\sigma x}}{1-e^{-x}} \, \dd x \\
&= - \int_0^{+ \infty}  \frac{e^{-\sigma x}}{1-e^{-x}} (e^{-iatx}-1) \, \dd x \\
&= + \int_0^{+ \infty}  \frac{e^{-\sigma \frac{y}{a}}}{1-e^{-\frac{y}{a}}}  (1-e^{-ity}) \, \frac{\dd y}{a} \\
&= \int_0^{+ \infty} e^{- \frac{\sigma y}{a}} \frac{y}{a} \frac{1}{1-e^{-\frac{y}{a}}} \frac{1-e^{-ity}}{y}\, \dd y \\
&= \int_{- \infty}^{+ \infty} k(y) \frac{1-e^{-ity}}{y}\, \dd y
\end{align*}

avec le changement de variable $y=ax$ et en posant

$$
k : y \longmapsto
\begin{cases}
e^{- \frac{\sigma y}{a}} \frac{y}{a} \frac{1}{1-e^{-\frac{y}{a}}} &\text{ si }y> 0 \\
0 &\text{ si }y \leq 0
\end{cases}
$$


On a alors $k \in {\rm L}^1(\mathbb{R}) \cap {\rm L}^2(\mathbb{R})$ et, en reprenant les notations du Lemme \ref{lem_122}, on a, modulo la vérification $\lim\limits_{t \rightarrow +\infty} \rho(t) \gamma(t)=0$ (pour laquelle on renvoie encore à \cite{Poit} §1) :
\begin{align*}
\frac{1}{2 \pi} \int_{- \infty}^{+ \infty} (\psi(\sigma+iat)- \psi (\sigma)) \varphi(t)\, \dd t
&= \int_{- \infty}^{+ \infty} k(x) \frac{F(0)-F(x)}{x} \,\dd x \\
&= \int_0^{+ \infty} e^{- \frac{\sigma x}{a}} \frac{x}{a} \frac{1}{1-e^{-\frac{x}{a}}} \frac{F(0)-F(x)}{x} \,\dd x \\
&= \int_0^{+ \infty}  \frac{e^{- \sigma y}}{1-e^{-y}} (F(0)-F(ay)) \,\dd y \\
\end{align*}

Enfin, en utilisant le même calcul d'inversion de Fourier que ci-dessus, on a :
\begin{align*}
\frac{1}{2 \pi} \int_{- \infty}^{+ \infty}  \psi (\sigma) \varphi(t)\, \dd t
&= \psi (\sigma) \times F(0) \\
&=  F(0)  \int_0^{+ \infty}  \left(\frac{e^{-y}}{y}-\frac{e^{- \sigma y}}{1-e^{-y}} \right) \, \dd y
\end{align*}
d'où :
\begin{equation*}
\frac{1}{2 \pi} \int_{- \infty}^{+ \infty} \psi(\sigma+iat) \varphi(t)dt
=- \int_0^{+ \infty} \left( \frac{F(ay)e^{- \sigma y}}{1-e^{-y}} - F(0) \frac{e^{-y}}{y} \right) \dd y
\end{equation*}

\subsection{Formule}
On prend la limite quand $m$ tend vers l'infini (donc $T_m$ tend vers l'infini) dans l'équation \eqref{resid} : le membre de gauche admet une limite par les calculs des sections précédentes, et la somme sur les pôles également car c'est une somme finie d'après l'hypothèse faite sur $\Lambda$.

On a donc montré que $\sum_{|\mathrm{Im}(\rho)|<T_m} \Phi(\rho)$ admet une limite quand $m$ tend vers l'infini qu'on notera désormais sans plus de précautions $\sum \Phi(\rho)$.

On peut finalement, en conservant les notations précédentes et en utilisant la parité de $F$, écrire la formule explicite, pour une paire quelconque de fonctions $\Lambda$. C'est le Théorème \ref{formule_explicite}.
\begin{multline*}\label{formule_explicite_ann}
\sum \Phi(\rho) - \sum \Phi(\mu)=F(0) \log(AB)\\
-\sum_p \sum_{j=1}^{M'} \sum_{k=1}^{+ \infty} F(k \log p) \frac{\log(p)}{p^{\frac{k}{2}}}(\alpha_j(p)^k+\beta_j(p)^k)\\
-\sum_{i=1}^M (I(a_i,a'_i)+I(b_i,b'_i))
\end{multline*}
où $$I(a,a')=\int_0^{+ \infty} \left( \frac{F(ay)e^{- (\frac{a}{2}+a') y}}{1-e^{-y}} - F(0) \frac{e^{-y}}{y} \right) \dd y.$$

\bibliographystyle{alpha}
\newpage
\bibliography{test_bibli_2018}

\end{document}